\documentclass[reqno]{amsart}

\RequirePackage[dvipsnames]{xcolor}
\usepackage{amsfonts, fancyhdr, amsmath, amssymb, hyperref, url, amsthm, subfigure, xypic, tikz-cd, verbatim, enumitem,lscape, longtable}

\usepackage{todonotes}

\numberwithin{equation}{section}

\theoremstyle{norm}
\newtheorem{thm}{Theorem}[section]
\newtheorem{theorem}[thm]{Theorem}
\newtheorem{lem}[thm]{Lemma}

\newtheorem{prop}[thm]{Proposition}
\newtheorem{df}[thm]{Definition}
\newtheorem{definition}[thm]{Definition}

\newtheorem{cor}[thm]{Corollary}
\newtheorem{corollary}[thm]{Corollary}

\newtheorem{rem}[thm]{Remark}
\newtheorem{remark}[thm]{Remark}
\newtheorem{exam}[thm]{Example}
\newtheorem{quest}[thm]{Question}
\newtheorem{warn}[thm]{Warning}
\newtheorem{fact}[thm]{Fact}

\newtheorem*{slogan}{Slogan}
\newtheorem{notation}[thm]{Notation}

\newcommand{\done}{\bar{\mathfrak{d}}_1}
\newcommand{\dthree}{\bar{\mathfrak{d}}_3}
\newcommand{\sone}{\bar{s}_1}
\newcommand{\sthree}{\bar{s}_3}
\newcommand{\rone}{\bar{r}_1}
\newcommand{\grone}{\gamma \bar{r}_1}
\newcommand{\rthree}{\bar{r}_3}
\newcommand{\grthree}{\gamma \bar{r}_3}
\newcommand{\m}[1]{\underline{#1}}
\newcommand{\mZ}{\m{\mathbb Z}}
\newcommand{\HZ}{H\mZ}
\newcommand{\BP}{BP} 
\newcommand{\BPG}[1][G]{\BP^{(\!(#1)\!)}}
\newcommand{\BPC}{\BPG[C_{4}]} 
\newcommand{\BPR}{\BP_{\mathbb R}}
\newcommand{\BPone}{\BPC\!\langle 1 \rangle}
\newcommand{\BPtwo}{\BPC\!\langle 2 \rangle}
\newcommand{\BPCeightone}{\BP^{(\!(C_8)\!)} \langle 1 \rangle}
\newcommand{\MU}{MU}
\newcommand{\MUR}{\MU_{\mathbb R}}
\newcommand{\MUG}[1][G]{\MU^{(\!(#1)\!)}}
\newcommand{\BPGm}{\BPG\!\langle m\rangle}

\DeclareMathOperator{\Gal}{\textup{Gal}}

\DeclareMathOperator{\SliceSS}{\textup{SliceSS}}
\DeclareMathOperator{\HFPSS}{\textup{HFPSS}}

\DeclareMathOperator{\Spec}{\text{Spec}}

\title[The Slice Spectral Sequence of a Height-4 Theory]{The Slice Spectral Sequence of a $C_4$-Equivariant Height-4 Lubin--Tate Theory}
\author{Michael A. Hill}
\address{Department of Mathematics, UCLA, Los Angeles, CA 90095}
\email{mikehill@math.ucla.edu}

\author{XiaoLin Danny Shi}
\address{Department of Mathematics, University of Chicago, Chicago, IL 60637}
\email{dannyshi@math.uchicago.edu}

\author{Guozhen Wang}
\address{Shanghai Center for Mathematical Sciences, Fudan University, Shanghai, China 200433}
\email{wangguozhen@fudan.edu.cn}

\author{Zhouli Xu}
\address{Department of Mathematics, UC San Diego, La Jolla, CA 92093}
\address{Department of Mathematics, Massachusetts Institute of Technology, Cambridge, MA 02142}
\email{xuzhouli@ucsd.edu}

\begin{document}
%%%%%%%%%%%Abstract

\maketitle
\begin{abstract}
We completely compute the slice spectral sequence of the $C_4$-spectrum $BP^{(\!(C_4)\!)}\langle 2 \rangle$.  This spectrum provides a model for a height-4 Lubin--Tate theory with a $C_4$-action induced from the Goerss--Hopkins--Miller theorem.  In particular, our computation shows that $E_4^{hC_{12}}$ is 384-periodic.
\end{abstract}

\setcounter{tocdepth}{1}
\tableofcontents

\section{Introduction}
Chromatic homotopy theory is a powerful tool to study periodic phenomena in stable homotopy theory by analyzing the algebraic geometry of smooth one-parameter formal groups.  More precisely, the moduli stack of formal groups has a stratification by height, which corresponds to localization with respect to the Morava $K$-theories $K(n)$, $n \geq 0$.  As the height increases, this stratification carries increasingly more information about the stable homotopy category, but also becomes increasingly harder to understand.  

At height 0, localizing with respect to $K(0)$ corresponds to rationalization.  At height $n \geq 1$, the $K(n)$-local sphere $L_{K(n)}S^0$ is equivalent to $E_n^{h\mathbb{G}_n}$ \cite{DevinatzHopkins}, where $E_n$ is the height-$n$ Lubin--Tate theory and $\mathbb{G}_n$ is a profinite group called the Morava stabilizer group.  One can analyze $E_n^{h\mathbb{G}_n}$ by further decomposing it into smaller building blocks of the form $E_n^{hG}$, where $G$ is a finite subgroup of $\mathbb{G}_n$.  

At height 1, the $K(1)$-local sphere $L_{K(1)} S^0$ is completely understood via this method by works of Bousfield~\cite{Bousfield}, Adams--Baird, and Ravenel \cite{RavenelLocalization}.  At height 2, the $K(2)$-local sphere has been the subject of extensive research, starting from the chromatic spectral sequence of Miller--Ravenel--Wilson \cite{MillerRavenelWilson}.  Significant progress towards understanding the $K(2)$-local sphere have since been made by the works of Shimomura--Yabe~\cite{ShimomuraYabe} and Shimomura--Wang~\cite{ShimomuraWang1, ShimomuraWang2} (see also \cite{BehrensSE2}), the computation of topological modular forms (an important building block of $L_{K(2)}S^0$) by Hopkins--Mahowald \cite{HopkinsMahowald}, and the resolution of the $K(2)$-local sphere at the prime $p =3$ by Goerss--Henn--Mahowald--Rezk \cite{GoerssHennMahowaldRezk}.

Classically, the homotopy fixed point spectra $E_n^{hG}$ are computed by using the homotopy fixed point spectral sequence.  However, when $p =2$ and the height $n$ is bigger than 2, the spectra $E_n^{hG}$ are very difficult to compute: there is no convenient description of the $G$-action on $\pi_*E_n$ (other than the unpublished work of Hill--Hopkins--Ravenel), so it is hard to compute the $E_2$-page of its homotopy fixed point spectral sequence.  In fact, the only known cases before our result in this paper are when $|G| = 2\ell$, where $\ell \equiv 1 \pmod{2}$.  Even worse, the $G$-action on $E_n$ is constructed purely from obstruction theory \cite{HopkinsMiller, GoerssHopkins}, so there is no systematic method to compute the differentials.  There have been attempts to understanding this by using topological automorphic forms \cite{BehrensLawsonTAF}, but with limited computational success.  

The main result of this paper is the first height 4 computation of a building block to the $K(4)$-local sphere $L_{K(4)}S^0$ at the prime $p=2$.  Our computation uses Hill--Hopkins--Ravenel's slice spectral sequence in equivariant homotopy theory, which is a crucial tool in their solution of the Kervaire invariant one problem \cite{HHR}.  
\begin{thm}\label{thm:MainTheorem1}
\hfill
\begin{enumerate}
\item There exists a height-4 Lubin--Tate theory $E_4$ with coefficient ring 
\begin{eqnarray*} 
\pi_*E_4 &\cong& W(\mathbb{F}_{2^4})[C_4 \cdot r_1, C_4 \cdot u][C_4 \cdot u^{-1}]^{\wedge}_{\mathfrak{m}},
\end{eqnarray*}
where $|r_1| = |u| = 2$, $\mathfrak{m} = (C_4 \cdot r_1, C_4 \cdot (u - \gamma u))$, and $C_4 \cdot x$ denotes the set $\{ x, \gamma x\}$ with the generator $\gamma \in C_4$ sending $x \mapsto \gamma x$ and ${\gamma x \mapsto -x}$.  Furthermore, there is a subgroup ${G = C_4 \times \left(\Gal(\mathbb{F}_{2^4}/\mathbb{F}_2) \ltimes C_3\right)}$ inside the Morava stabilizer group $\mathbb{G}_4$ such that the isomorphism above is a $G$-equivariant isomorphism. 
\item There is a $C_4$-equivariant homotopy commutative ring map 
\[\phi: \BPC \longrightarrow E_4\]
such that $\pi_*^e \phi: \pi_*^e \BPC \to \pi_* E_4$ is the map determined by sending
\[r_{2^i -1} \longmapsto \left\{\begin{array}{ll} r_1 & i = 1, \\ 
u^3 & i = 2, \\ 
0 & i \geq 3.  \end{array}\right. \]
\item After inverting the element 
\[D_2:= N(\bar{v}_4)N(\rthree)N(\bar{r}_3^2 + \bar{r}_3 (\gamma \bar{r}_3) + (\gamma \bar{r}_3)^2) \in \pi_{24\rho_4}^{C_4} \BPC,\]
there is a factorization 
\[\begin{tikzcd}
\BPC \ar[r] \ar[d] & E_4 \\ 
D_2^{-1}\BPC \ar[ru, dashed]&
\end{tikzcd}\]
of the $C_4$-equivariant orientation through $D_2^{-1}\BPC$.  
\end{enumerate}
\end{thm}

\begin{thm}\label{thm:MainTheorem2}
We compute all the differentials in the slice spectral sequence of $\BPtwo$ (see Figure~\ref{fig:E4C4}).  The spectral sequence terminates after the $E_{61}$-page and has a horizontal vanishing line of filtration 61.
\end{thm}

\begin{figure}
\begin{center}
\makebox[\textwidth]{\includegraphics[trim={0cm 20cm 0cm 19.5cm}, clip, page = 1, scale = 0.23]{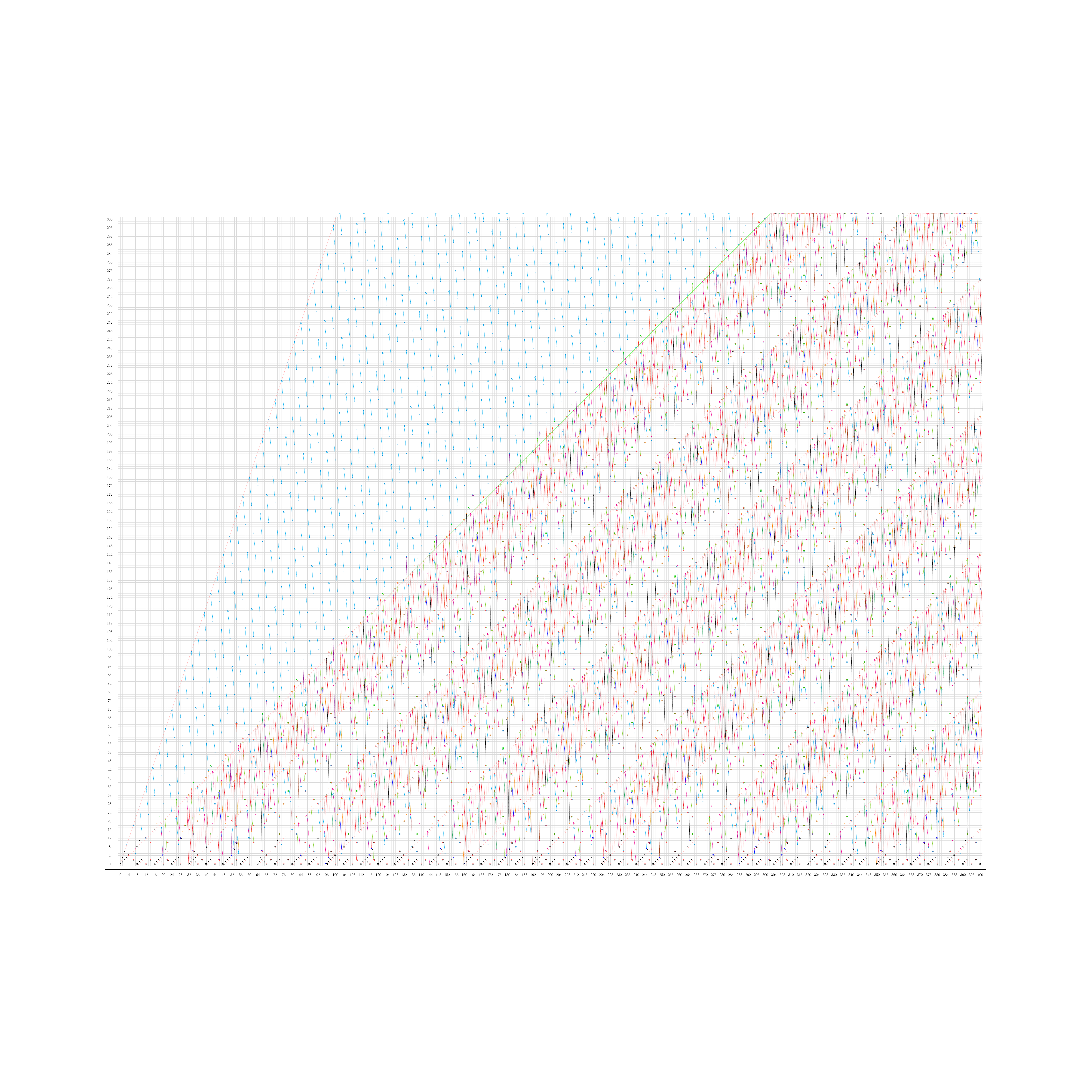}}
\caption{The slice spectral sequence of $BP^{((C_4))}\langle 2 \rangle$.} 
\hfill
\label{fig:E4C4}
\end{center}
\end{figure}

\begin{thm}\label{thm:MainTheorem3}
After inverting the element $D_2 \in \pi_{24\rho_4}^{C_4} \BPC$ in Theorem~\ref{thm:MainTheorem1}, the $C_4$-spectrum $\Psi:= D_2^{-1}\BPtwo$ has three periodicities: 
\begin{enumerate}
\item $S^{3\rho_4} \wedge \Psi \simeq \Psi$;
\item $S^{8-8\sigma} \wedge \Psi \simeq \Psi$;
\item $S^{32+32\sigma-32\lambda} \wedge \Psi \simeq \Psi$.
\end{enumerate}
Together, these three periodicities imply that $\Psi^{C_4}$ and $E_4^{hC_{12}}$ are 384-periodic theories.
\end{thm}

The spectrum $E_4^{hC_{12}}$ is one of the building blocks of the $K(4)$-local sphere $L_{K(4)}S^0$.  Our approach bypasses the previously mentioned difficulties surrounding the homotopy fixed point spectral sequence by using the $C_4$-equivariant spectra $\BPC$ and $\BPtwo$.  These equivariant spectra exploit the connections between the geometry of Real bordism theories and the obstruction-theoretic actions on the Lubin--Tate theories.  Roughly speaking, the $C_4$-spectrum $\BPtwo$ encodes the universal example of a height-4 formal group law with a $C_4$-action extending the formal inversion action.

 %This is because when the prime $p$ is big compared to the height $n$, the Adams--Novikov spectral sequence for the sphere collapses, so analyzing the $K(n)$-local sphere becomes purely algebraic (see \cite{BehrensSE2, Lader, ShimomuraYabe, BarthelTomerNat, PiotrChromatic}).  However, when $p$ is small with respect to $n$, the problem is no longer purely algebraic.  

At height 2, Hill, Hopkins, and Ravenel \cite{HHRKH} studied the slice spectral sequence of the spectrum $D_1^{-1}\BPone$.  They showed that $D_1^{-1}\BPone$ is 32-periodic and is closely related to a height-2 Lubin--Tate theory, which has also been studied by Behrens--Ormsby \cite{BehrensOrmsby} as $\textup{TMF}_0(5)$.  

The spectrum $\Psi$, which is 384-periodic, is a height-4 generalization of $D_1^{-1}\BPone$.  This can be viewed as a different approach than that of Behrens and Lawson \cite{BehrensLawsonTAF} to generalizing $\textup{TMF}$ with level structures to higher heights, with the advantage of being completely computable. 

%(In general, computations at the prime $p =2$ are difficult.  This is because when the prime $p$ is small with respect to the height $n$, the $E_2$-page of the $K(n)$-local Adams--Novikov spectral sequence for the sphere is very dense, and has many nontrivial differentials.  As a result, the problem of understanding the $K(n)$-local sphere is no longer purely algebraic. ) } 

\subsection{Motivation and main results}\label{subsec:motivationMainResults}
In 2009, Hill, Hopkins, and Ravenel \cite{HHR} proved that the Kervaire invariant elements $\theta_j$ do not exist for $j \geq 7$.  A key construction in their proof is the spectrum $\Omega$, which detects all the Kervaire invariant elements in the sense that if $\theta_j \in \pi_{2^{j+1}-2} S^0$ is an element of Kervaire invariant 1, then the Hurewicz image of $\theta_j$ under the map $\pi_*S^0 \to \pi_*\Omega$ is nonzero (see also \cite{HaynesKervaire, HHRCDM1, HHRCDM2} for surveys on the result).  

The detecting spectrum $\Omega$ is constructed using equivariant homotopy theory as the fixed points of a $C_8$-spectrum $\Omega_\mathbb{O}$, which in turn is a chromatic-type localization of $\MU^{(\!(C_8)\!)} := N_{C_2}^{C_8} \MU_\mathbb{R}$.  Here, $N_{C_2}^{C_8}(-)$ is the Hill--Hopkins--Ravenel norm functor and $\MU_\mathbb{R}$ is the Real cobordism spectrum of Landweber \cite{LandweberMUR}, Fujii \cite{FujiiMUR}, and Araki \cite{Araki}.  The underlying spectrum of $\MU_\mathbb{R}$ is $\MU$, with the $C_2$-action coming from the complex conjugation action on complex manifolds. 

To analyze the $G$-equivariant homotopy groups of $\MUG$, Hill, Hopkins, and Ravenel generalized the $C_2$-equivariant filtration of Hu--Kriz \cite{HuKriz} and Dugger \cite{DuggerKR} to a $G$-equivariant Postnikov filtration for all finite groups $G$.  They called this the \textit{slice filtration}.  Given any $G$-equivariant spectrum $X$, the slice filtration produces the slice tower $\{P^*X\}$, whose associated slice spectral sequence strongly converges to the $RO(G)$-graded homotopy groups $\pi_\bigstar^G X$.  

For $G = C_{2^n}$, the $G$-spectrum $\MUG$ are amenable to computations.  Hill, Hopkins, and Ravenel proved that the slice spectral sequences for $\MUG$ and its equivariant localizations have especially simple \(E_2\)-terms.  Furthermore, they proved the Gap Theorem and the Periodicity Theorem, which state, respectively, that $\pi_i^{C_8} \Omega_\mathbb{O} = 0$ for $-4 < i < 0$, and that there is an isomorphism $\pi_*^{C_8} \Omega_\mathbb{O} \cong \pi_{*+256}^{C_8} \Omega_\mathbb{O}.$
The two theorems together imply that
$$\displaystyle \pi_{2^{j+1}-2} \Omega = \pi_{2^{j+1}-2}^{C_8} \Omega_\mathbb{O} = 0$$
for all $j \geq 7$, from which the nonexistence of the corresponding Kervaire invariant elements follows.

%In the solution to the Kervaire Invariant One problem, Hill--Hopkins--Ravenel introduced a family of equivariant spectra: the norms of the Fujii--Landweber Real bordism spectrum \(\MUG\). The crux of the argument was an analysis of a particular chromatic-type localization of \(\MU^{((C_8))}\). This spectrum detected all of the Kervaire classes. Moreover, its homotopy groups were accessible via the slice spectral sequence, which, for the norms of \(\MU_{\mathbb R}\) has an especially simple \(E_2\)-term. This gives us a motivating slogan:

The solution of the Kervaire invariant one problem gives us a motivating slogan: 
\begin{slogan}
The homotopy groups of the fixed points of \(\MUG\) as \(|G|\) grows are increasingly good approximations to the stable homotopy groups of spheres.
\end{slogan}

%Unfortunately, these are never (even asymptotically) a perfect approximation. The element \(\eta^{3}\) is never present in the homotopy groups of the fixed points of any \(\MUG\) for a cyclic \(2\)-group \(G\).

To explain the slogan, we unpack some of the algebraic geometry around \(\MUG\) when $G = C_{2^n}$. The spectrum underlying \(\MUG\) is the smash product of \(2^{n-1}\)-copies of \(\MU\), and so the underlying homotopy ring co-represents the functor which associates to a (graded) commutative ring a formal group law and a sequence of \((2^{n-1}-1)\) isomorphisms:
\[
F_1\xrightarrow{f_1}F_2\xrightarrow{f_2}\dots\xrightarrow{f_{2^{n-1}-1}} F_{2^{n-1}}.
\]
The underlying homotopy ring has an action of \(C_{2^n}\), and by canonically enlarging our moduli problem, we can record this as well. We extend our sequence of isomorphisms by one final isomorphism from the final formal group law back to the first, composing the inverses to the isomorphisms already given with the formal inversion. This gives us our moduli problem: a map from the underlying homotopy of \(\MU^{(\!(C_{2^n})\!)}\) to a graded commutative \(C_{2^n}\)-equivariant ring \(R\) is given by a formal group law \(F\) together with isomorphisms
\[
f_{i+1}\colon{\gamma^{i}}^\ast F\to {\gamma^{(i+1)}}^\ast F, \,\,\, 0 \leq i \leq 2^{n-1}-1
\]
such that the composite of all of the \(f_{i}\) is the formal inversion on \(F\). 

If \(F\) is a formal group law over a ring \(R\) that has an action of \(C_{2^{k}}\) extending the action of \(C_{2}\) given by formal inversion, then \(F\) canonically defines a sequence of formal groups as above. Simply take all of the maps \(f_{i}\) to be the identity unless we pass a multiple of  \(2^{n-k}\), in which case, take the corresponding element of \(C_{2^{k}}\). In this way, we see that the stack \(\Spec(\pi_{\ast}^{e}\MU^{(\!(C_{2^n})\!)})/\!/C_{2^{n}}\)  provides a cover of the moduli stack of formal groups in a way that reflects the automorphism groups which extend the formal inversion action and which are isomorphic to subgroups of \(C_{2^{n}}\). 

As an immediate, important example, we consider the universal deformation \(\Gamma_{m}\) of a fixed height-\(m\) formal group law \(F_{m}\) over an algebraically closed field \(k\) of characteristic \(p\). Lubin and Tate \cite{LubinTate} showed that the space of deformations is Ind-representable by a pro-ring abstractly isomorphic to 
\[
\mathbb W(k)[\![u_{1},\dots,u_{m-1}]\!][u^{\pm 1}]=:{E_{m}}_\ast,
\]
over which \(\Gamma_{m}\) is defined.  Here, \(\mathbb W(k)\) is the \(p\)-typical Witt vectors of \(k\), \(|u_{i}|=0\), and \(|u|=2\). 

By naturality, the ring ${E_m}_\ast$ is acted on by the Morava stabilizer group \(\mathbb G_{m}\), the automorphism group of \(F_{m}\).  Hewett \cite{Hewett} showed that if \(m=2^{n-1}(2r+1)\), then there is a subgroup of the Morava stabilizer group isomorphic to \(C_{2^{n}}\). In particular, associated to \(\Gamma_{m}\) and the action of a generator of \(C_{2^{n}}\), we have a \(C_{2^{n}}\)-equivariant map 
\[
\pi_{\ast}^{e}\MU^{(\!(C_{2^n})\!)} \longrightarrow {E_m}_\ast.
\]

Topologically, this entire story can be lifted.  The formal group law \(\Gamma_{m}\) is Landweber exact, and hence there is a complex orientable spectrum \(E_{m}\) which carries the universal deformation \(\Gamma_{m}\).  The Goerss--Hopkins--Miller Theorem \cite{HopkinsMiller, GoerssHopkins} proves that \(E_{m}\) is a commutative ring spectrum and that the automorphism group of \(E_{m}\) as a commutative ring spectrum is homotopy equivalent to the Morava stabilizer group. In particular, we may view \(E_{m}\) as a commutative ring object in naive \(G\)-spectra. The functor 
\[
X\longmapsto F(EG_{+},X)
\] 
takes naive equivalences to genuine equivariant equivalences, and hence allows us to view \(E_{m}\) as a genuine \(G\)-equivariant spectrum. The commutative ring spectrum structure on \(E_{m}\) gives an action of a trivial \(E_{\infty}\)-operad on \(F(EG_{+},E_{m})\).  Work of Blumberg--Hill \cite{BlumbergHill} shows that this is sufficient to ensure that \(F(EG_{+},E_{m})\) is actually a genuine equivariant commutative ring spectrum, and hence it has norm maps.

The spectra $E_m^{hG}$ are the building blocks of the $p$-local stable homotopy category.  In particular, the homotopy groups $\pi_* E_m^{hG}$ assemble to the stable homotopy groups of spheres.  To be more precise, the chromatic convergence theorem \cite{RavenelOrangeBook} exhibits the $p$-local sphere spectrum $S^0_{(p)}$ as the inverse limit of the chromatic tower 
$$\cdots \longrightarrow L_{E_m} S^0 \longrightarrow L_{E_{m-1}} S^0 \longrightarrow \cdots \longrightarrow L_{E_0} S^0,$$
where each $L_{E_m}S^0$ is assembled via the chromatic fracture square
$$\begin{tikzcd}
L_{E_m}S^0 \ar[d] \ar[r] & L_{K(m)}S^0 \ar[d] \\ 
L_{E_{m-1}} S^0 \ar[r] & L_{E_{m-1}}L_{K(m)} S^0.
\end{tikzcd}$$
Here, $K(m)$ is the $m$th Morava $K$-theory.

Devinatz and Hopkins \cite{DevinatzHopkins} proved that $L_{K(m)}S^0 \simeq E_m^{h\mathbb{G}_m}$, and, furthermore, that the Adams--Novikov spectral sequence computing $L_{K(m)}S^0$ can be identified with the associated homotopy fixed point spectral sequence for $E_m^{h\mathbb{G}_m}$.  One may further analyze $L_{K(m)}S^0$ by using the spectra $\{E_m^{hG}\, |\, G \subset \mathbb{G}_m \text{ finite}\}$.  A comprehensive description of these techniques can be found in \cite{Henn2007} and \cite{AgnesToby}.  

At height 1, $E_1^{hC_2}$ is $KO_2^\wedge$, the 2-adic completion of real $K$-theory.  The Morava stabilizer group $\mathbb{G}_1$ is isomorphic to $\mathbb{Z}_2^\times$.  %Bousfield~\cite{Bousfield}, Adams--Baird, and Ravenel \cite{RavenelLocalization} showed there is a finite resolution of length one for $L_{K(1)}S^0$.

At height 2, the homotopy fixed points $E_2^{hG}$ are the $K(2)$-localizations of topological modular forms and variants of topological modular forms with level structures.  Computations of the homotopy groups of these spectra are done by Hopkins--Mahowald~\cite{HopkinsMahowald}, Bauer~\cite{BauerTMF}, Mahowald--Rezk~\cite{MahowaldRezk}, Behrens--Ormsby~\cite{BehrensOrmsby}, Hill--Hopkins--Ravenel~\cite{HHRKH}, and Hill--Meier~\cite{HillMeier}.  %Finite resolutions of the $K(2)$-local sphere $L_{K(2)}S^0$ have also been constructed.  See works of Goerss--Henn--Mahowald--Rezk~\cite{GoerssHennMahowaldRezk}, Goerss--Henn--Mahowald~\cite{GoerssHennMahowald}, Behrens~\cite{BehrensModular}, Henn--Karamanov--Mahowald~\cite{HennKaramanovMahowald} at the prime 3, works of Henn~\cite{Henn2007} and Lader~\cite{Lader} at the prime $p \geq 5$, and works of Beaudry~\cite{BeaudryDualityResolution, BeaudryChromaticSplitting}, Bobkova--Goerss~\cite{BobkovaGoerss}, Beaudry--Goerss--Henn~\cite{BeaudryGoerssHenn} at the prime $p= 2$.  For higher heights, finite resolutions are not known to exist in general. 

At height $(p -1)$ and at the prime $p$, $E_{p-1}^{hC_p}$ has been computed by Hopkins, Miller, and Nave.  The result of this computation is used by Nave to prove the nonexistence of certain Smith--Toda complexes \cite{NaveThesis, Nave}. 
 
For higher heights $m >2$ and when $p=2$, the homotopy fixed points $E_m^{hG}$ are notoriously difficult to compute.  One of the chief reasons that these homotopy fixed points are so difficult to compute is because the group actions are constructed purely from obstruction theory.  This stands in contrast to the norms of $\MU_{\mathbb{R}}$, whose actions are induced from geometry.  

Recent work of Hahn--Shi \cite{HahnShi} establishes the first known connection between the obstruction-theoretic actions on Lubin--Tate theories and the geometry of complex conjugation.  More specifically, there is a Real orientation for any of the \(E_{m}\): there are \(C_{2}\)-equivariant maps
\[
\MU_{\mathbb R}\longrightarrow i_{C_{2}}^{\ast}E_{m}.
\]
Using the norm-forget adjunction, such a map can be promoted to a $G$-equivariant map
\[
\MUG\longrightarrow N_{C_{2}}^{G}i_{C_{2}}^{\ast}E_{m}\longrightarrow E_{m}.
\]
By construction, since the original map \(\MU_{\mathbb R}\to E_{m}\) classified \(\Gamma_{m}\) as a Real formal group law, this \(G\)-equivariant map recovers the algebraic map ${\pi_{\ast}^{e}\MU^{(\!(C_{2^n})\!)} \longrightarrow {E_m}_\ast}$.  

As a consequence of the Real orientation theorem, the fixed point spectra $(\MU^{(\!(C_{2^n})\!)})^{C_{2^n}}$ and $E_{2^{n-1}m}^{hC_{2^n}}$ can be assembled into the following diagram:

\begin{equation}\label{diagram:DetectionTower}
\begin{tikzcd}
&& \vdots \ar[d]&& \\
&&(\MU^{(\!(C_{2^n})\!)})^{C_{2^n}} \ar[d] \ar[r] &  E_{2^{n-1}m}^{hC_{2^n}}\\
&& \vdots \ar[d] &\\ 
S^0 \ar[rr] \ar[rrd] \ar[rrdd] \ar[rruu]&&(\MU^{(\!(C_8)\!)})^{C_{8}} \ar[d] \ar[r]  & E_{4m}^{hC_8}\\ 
&&(\MU^{(\!(C_4)\!)})^{C_{4}} \ar[d] \ar[r] &  E_{2m}^{hC_4}\\ 
&&(\MU_\mathbb{R})^{C_{2}} \ar[r] & E_m^{hC_2}.
\end{tikzcd}
\end{equation}

The existence of equivariant orientations renders computations that rely on the slice spectral sequence tractable.  Using differentials in the slice spectral sequence of $\MU_\mathbb{R}$ and the Real orientation $\MU_\mathbb{R} \to E_m$, Hahn--Shi computed $E_m^{hC_2}$ at arbitrarily large heights $m$.  

An example of a Real orientable theory that was previously known is Atiyah's Real $K$-theory.  In 1966, Atiyah \cite{AtiyahKR} formalized the connection between complex $K$-theory ($KU$) and real $K$-theory ($KO$).  Analogous as in the case of $MU_\mathbb{R}$, the complex conjugation action on complex vector bundles induces a natural $C_2$-action on $KU$, and this produces a $C_2$-spectrum $K_\mathbb{R}$ called Atiyah's Real $K$-theory.  The theory $K_\mathbb{R}$ interpolates between complex and real $K$-theory in the sense that the underlying spectrum of $K_\mathbb{R}$ is $KU$, and its $C_2$-fixed points is $KO$.  The $RO(C_2)$-graded homotopy groups $\pi_\bigstar^{C_2} K_\mathbb{R}$ has two periodicities: a $\rho_2$-periodicity that corresponds to the complex Bott-periodicity, and a 8-periodicity that corresponds to the real Bott-periodicity. 

In \cite{HHRKH}, Hill, Hopkins, and Ravenel computed the slice spectral sequence of a $C_4$-equivariant height-2 theory that is analogous to $K_\mathbb{R}$.  To introduce this theory, note that the height of the formal group law \(\Gamma_{m}\) is at most \(m\) and the ring \({E_m}_\ast\) is \(2\)-local. We can therefore pass to \(2\)-typical formal group laws (and hence \(\BP\)), and our map
\[
\BP\longrightarrow E_{m}
\]
classifying the formal group law descends to a map
\[
v_m^{-1}\BP \longrightarrow E_{m}.
\]
Equivariantly, we have a similar construction, which we will review in more detail in Section~\ref{sec:Preliminaries}.  The \(C_{2^{n}}\)-equivariant map 
\[
\BP^{(\!(C_{2^n})\!)}\longrightarrow E_{m}
\]
will factor through a localization of $\BP^{(\!(C_{2^n})\!)}$.  To study the Hurewicz images of $E_m$, it therefore suffices to study the various localization of the quotients of $\BP^{(\!(C_{2^n})\!)}$.  

Hill--Hopkins--Ravenel computed the homotopy Mackey functors of $\BPone$.  This spectrum gives a model of a height-2 Lubin--Tate theory with $C_4$-action coming from the automorphisms of its formal group law.  More precisely, there exists a height-2 Lubin--Tate theory $E_2$ with coefficient ring 
\[\pi_*E_2 \cong W(\mathbb{F}_{2^2})[C_4 \cdot u][C_4 \cdot u^{-1}]^\wedge_{\mathfrak{m}}, \]
where $|u| = 2$ and $\mathfrak{m} = (C_4 \cdot (u - \gamma u))$.  Furthermore, there is a subgroup $C_4$ inside the Morava stabilizer group $\mathbb{G}_2$ such that the isomorphism above is a $C_4$-equivariant isomorphism.  

There is a $C_4$-equivariant homotopy commutative ring map 
\[\phi: \BPC \longrightarrow E_2\]
such that $\pi_*^e \phi: \pi_*^e \BPC \to \pi_*E_2$ is the map determined by sending 
\[r_{2^i -1} \longmapsto \left\{\begin{array}{ll} r_1 & i = 1, \\ 
0 & i \geq 2.  \end{array}\right. \]
After inverting the element 
\[D_1: = N(\bar{v}_2) N(\rone) \in \pi_{4\rho_4}^{C_4} \BPC,\]
there is a factorization 
\[\begin{tikzcd}
\BPC \ar[r, "\phi"] \ar[d] & E_2 \\ 
D_1^{-1}\BPC \ar[ru, dashed]&
\end{tikzcd}\]
of the $C_4$-equivariant orientation $\phi$ through $D_1^{-1}\BPC$.  

The slice spectral sequence of $\BPone$ degenerates after the $E_{13}$-page and has a horizontal vanishing line of filtration 13.  The $C_4$-spectrum $D_1^{-1}\BPone$ has three periodicities: 
\begin{enumerate}
\item $S^{\rho_4} \wedge D_1^{-1}\BPone \simeq D_1^{-1}\BPone$;
\item $S^{4-4\sigma}\wedge D_1^{-1}\BPone \simeq D_1^{-1}\BPone$;
\item $S^{8+8\sigma-8\lambda} \wedge D_1^{-1}\BPone \simeq D_1^{-1}\BPone$.
\end{enumerate}
Together, these three periodicities combine to imply that $D_1^{-1}\BPone$ and $E_2^{hC_4}$ are 32-periodic theories.

To this end, Theorem~\ref{thm:MainTheorem1}, Theorem~\ref{thm:MainTheorem2}, and Theorem~\ref{thm:MainTheorem3} show that $\BPtwo$ provides a model for a height-4 Lubin--Tate theory with a $C_4$-action coming from the automorphisms of its formal group law and give a complete computation of the slice spectral sequence of $\BPtwo$.

When $G = C_2$, Li--Shi--Wang--Xu \cite{HurewiczImages} analyzed the bottom layer of tower~(\ref{diagram:DetectionTower}) and showed that the Hopf-, Kervaire-, and $\bar{\kappa}$-families in the stable homotopy groups of spheres are detected by the map 
$$S^0 \longrightarrow (\MU_\mathbb{R})^{C_2}.$$
As we increase the height $m$, an increasing subset of the elements in these families is detected by the map 
$$S^0 \longrightarrow E_m^{hC_2}.$$

Since $(\BPone)^{C_4}$ is closely related to $\textup{TMF}_0(5)$, one can study its Hurewicz images via the Hurewicz images of $\textup{TMF}$ (see \cite{BehrensOrmsby, HHRKH}).  In particular, there are elements detected by the $C_4$-fixed points $(\BPone)^{C_4}$ that are not detected by $(\MU_\mathbb{R})^{C_2}$.  

In general, it is difficult to determine all the Hurewicz images of $(\BPG)^G$.  Computations of Hill \cite{HillEtaCubed} have shown that the class $\eta^3 \in \pi_3 S^0$ is not detected by $(\BP^{(\!(C_{2^n})\!)})^{C_{2^n}}$ for any $n \geq 1$.  However, this element is detected by $(\BP^{(\!(Q_8)\!)})^{Q_8}$, where $Q_8$ is the quaternion group.  It is a current project to understand the Hurewicz images of the $G$-fixed points of $\BPG$ and its various quotients when $G = C_{2^n}$ and $Q_8$.  

Note that after inverting a specific generator $D \in \pi_{19\rho_{C_8}}^{C_8} \BP^{(\!(C_8)\!)}$, $D^{-1} \BP^{(\!(C_8)\!)}$ is the detecting spectrum $\Omega_{\mathbb{O}}$ of Hill--Hopkins--Ravenel \cite{HHR}.  Since we have completely computed the slice spectral sequence of $(\BPtwo)^{C_4}$, the following questions are of immediate interest: 

\begin{quest}
What are the Hurewicz images of $(\BPtwo)^{C_4}$ and $E_4^{hC_4}$?
\end{quest}

\begin{quest}
What are the homotopy groups and the Hurewicz images of $(\BP^{(\!(C_4)\!)}\langle n \rangle)^{C_4}$ and $E_{2n}^{hC_4}$ for $n \geq 3$? 
\end{quest}

\begin{quest}
What are the $C_8$-fixed points of $\BPCeightone$? 
\end{quest}

\begin{quest}
What are the Hurewicz images of $(\BPCeightone)^{C_8}$ and $E_4^{hC_8}$? 
\end{quest}

\begin{quest}
What are the Hurewicz images of the inverse limit $\varprojlim (\MU^{(\!(C_{2^n})\!)})^{C_{2^n}}$ of the detection tower (see \cite{HillEtaCubed} and \cite{HurewiczImages})? 
\end{quest}

%{\color{blue} should say more of our motivations?  Detecting the last Kervaire class.  The input for the $C_8$-spectral sequence? }

%This is a \(C_{4}\)-equivariant theory for which the underlying height is at most \(4\), which some kind of \(C_{4}\)-equivariant height is at most \(2\). The slice spectral sequence for this is still manageable, but it has several interesting and unexpected subtleties. }

\subsection{Summary of the contents}
We now turn to a summary of the contents of this paper.  Section~\ref{sec:Preliminaries} provides the necessary background on $\MUG$.  In particular, we define the Hill--Hopkins--Ravenel theories $\BPG\langle m \rangle$ (Defintion~\ref{df:HHRTheories}), describe the $E_2$-pages of their slice spectral sequences, and prove Theorem~\ref{thm:MainTheorem1} and Theorem~\ref{thm:MainTheorem3}.

The rest of the paper are dedicated to proving Theorem~\ref{thm:MainTheorem2}.  In Section~\ref{sec:SliceSSBPone}, we review Hill--Hopkins--Ravenel's computation of $\SliceSS(\BPone)$.  Our proofs for some of the differentials are slightly different than those appearing in \cite{HHRKH}.  The computation is presented in a way that will resemble our subsequent computation for $\SliceSS(\BPtwo)$.  

Section~\ref{sec:slicesBPtwo} describes the slice filtration of $\BPtwo$.  We organize the slice cells of $\BPtwo$ into collections called \textit{$\BPone$-truncations} and \textit{$i_{C_2}^*\BPone$-truncations}.  This is done to facilitate later computations.  In Section~\ref{sec:C2BPtwoSliceSS}, we compute the $C_2$-slice spectral sequence of $i_{C_2}^*\BPtwo$.  

From Section~\ref{sec:InducedDiffBPone} forward, we focus our attention on computing the $C_4$-slice spectral sequence of $\BPtwo$.  Section~\ref{sec:InducedDiffBPone} proves that all the differentials in $C_4$-$\SliceSS(\BPtwo)$ of length $\leq 12$, as well as some of the $d_{13}$-differentials, can be induced from $C_4$-$\SliceSS(\BPone)$ via the quotient map $\BPtwo \to \BPone$.  In Section~\ref{sec:higherDifferentialsI} we prove all the $d_{13}$ and $d_{15}$ differentials by using the restriction map, the transfer map, and multiplicative structures. 

In Section~\ref{sec:Norm}, we prove differentials on the classes $u_{2\lambda}a_\sigma$, $u_{4\lambda}a_\sigma$, $u_{8\lambda}a_\sigma$, and $u_{16\lambda}a_\sigma$ by norming up $C_2$-equivariant differentials in $C_2$-$\SliceSS(\BPtwo)$.  Using these differentials, we prove the Vanishing Theorem (Theorem~\ref{thm:aboveFiltration61die}), which states that a large portion of the classes that are above filtration 96 on the $E_2$-page must die on or before the $E_{61}$-page.  The Vanishing Theorem is of great importance for us because it establishes a bound on the differentials that can possibly occur on a class.  

Sections~\ref{sec:HigherDifferentialsIV}, \ref{sec:HigherDifferentialsV}, and \ref{sec:HigherDifferentialsVI} prove all of the the remaining differentials in the slice spectral sequence.  The slice spectral sequence degenerates after the $E_{61}$-page and has a horizontal vanishing line of filtration 61 at the $E_\infty$-page.  Section~\ref{sec:SummaryOfDifferentials} gives a summary of all the differentials.  

\subsection{Acknowledgements}
The authors would like to thank Agn\`{e}s Beaudry, Mark Behrens, Irina Bobkova, Mike Hopkins, Hana Jia Kong, Lennart Meier, Haynes Miller, Doug Ravenel, and Mingcong Zeng for helpful conversations.  The authors are grateful to Hood Chatham, both for helpful conversations and for his spectral sequence program, which greatly facilitated our computations and the production of this manuscript.  Finally, the authors would like to thank the anonymous referee for many helpful comments and suggestions.  The first author was supported by the National Science Foundation under Grant No. DMS-1811189.  The fourth author was supported by the National Science Foundation under Grant No. DMS-1810638 and DMS-2043485.

\section{Preliminaries}\label{sec:Preliminaries}
\subsection{The slice spectral sequence of $\MUG$}
Let $\MUR$ be the Real cobordism spectrum, and $G$ be the cyclic group of order $2^n$.  The spectrum $\MUG$ is defined as 
$$\MUG := N_{C_2}^{G} \MUR,$$
where $N_H^G(-)$ is the Hill--Hopkins--Ravenel norm functor \cite{HHR}.  The underlying spectrum of $\MUG$ is the smash product of $2^{n-1}$-copies of $MU$.  

Hill, Hopkins, and Ravenel \cite[Section 5]{HHR} constructed generators 
\[
\bar{r}_i \in \pi_{i\rho_2}^{C_2} \MUG
\]
such that 
\[
\pi_{\ast\rho_{2}}^{C_{2}}\MUG\cong\mathbb Z_{(2)}[\bar{r}_{1},\gamma\bar{r}_{1},\dots,\gamma^{2^{n-1}-1}\bar{r}_{1},\bar{r}_{2},\dots].
\]
Here \(\gamma\) is a generator of \(C_{2^{n}}\), and the Weyl action is given by 
\[
\gamma\cdot \gamma^{j}\bar{r}_{i}=\begin{cases}
\gamma^{j+1}\bar{r}_{i} & 0\leq j\leq 2^{n-1}-2 \\
(-1)^{i}\bar{r}_{i} & j=2^{n-1}-1.
\end{cases}
\]
%Here, $r_i \in \pi_{2i}^u \MUG$ is the underlying homotopy element of $\bar{r}_i$, and $G \cdot \bar{r}_i$ stands for the sequence 
%$$(r_i, \gamma r_i, \ldots, \gamma^{2^{n-1}-1} r_i),$$
%where $\gamma$ is the generator of $G$.  

%The existence of these $\bar{r}_i$-generators makes possible the construction of the refinement 
%$$S^0[G \cdot \bar{r}_1, G\cdot \bar{r}_2, \ldots] \longrightarrow \MUG,$$
%from which the Slice Theorem and the Reduction Theorem for $\MUG$ follow.  

Adjoint to the maps 
\[
\bar{r}_{i}\colon S^{i\rho_{2}}\to i_{C_{2}}^{\ast}\MUG
\]
are associative algebra maps from free associative algebras
\[
S^{0}[\bar{r}_{i}]=\bigvee_{j\geq 0} \big(S^{i\rho_{2}}\big)^{\wedge j}\to i_{C_{2}}^{\ast}\MUG,
\]
and hence \(G\)-equivariant associative algebra maps
\[
S^{0}[G\cdot\bar{r}_{i}]=N_{C_{2}}^{G}S^{0}[\bar{r}_{i}]\to \MUG.
\]
Smashing these all together gives an associative algebra map
\[
A:=S^{0}[G\cdot\bar{r}_{1},\dots]=\bigwedge_{i=1}^{\infty} S^{0}[G\cdot\bar{r}_{i}]\to\MUG.
\]

For \(\MUG\) and the quotients below, the slice filtration is the filtration associated to the powers of the augmentation ideal of \(A\), by the Slice Theorem of \cite{HHR}.

The classical Quillen idempotent map $\MU \longrightarrow \BP$ can be lifted to a $C_2$-equivariant map 
\[
\MUR \longrightarrow \BP_{\mathbb R},
\]
where $\BP_{\mathbb R}$ is the Real Brown--Peterson spectrum.  Taking the norm $N_{C_2}^G(-)$ of this map produces a $G$-equivariant map 
\[
\MUG \to \BPG=: N_{C_2}^G \BP_{\mathbb R}.
\]
%The spectrum $BP^{((G))}$ can also be constructed as the quotient 
%$$\MUG/(G \cdot \bar{r}_k \,|\, k \neq 2^i-1) = \MUG \wedge_{A} A',$$ 
%where
%\begin{eqnarray*}
%A &=& S^0[G \cdot \bar{r}_1, G \cdot \bar{r}_2, \ldots], \\
%A' &=& S^0[G \cdot \bar{r}_1, G \cdot \bar{r}_3, G \cdot \bar{r}_7, \ldots].
%\end{eqnarray*}
Using the techniques developed in \cite{HHR}, it follows that $\BPG$ has refinement 
\[
S^0[G \cdot \bar{r}_1, G \cdot \bar{r}_3, G \cdot \bar{r}_7, \ldots] \longrightarrow \BPG.
\]
We can also produce truncated versions of these norms of \(\BPR\), wherein we form quotients by all of the \(\bar{r}_{2^{m}-1}\) for all \(m\) sufficiently large. For each \(m> 0\), let
\[
A_{m}=\bigwedge_{j=m}^{\infty}S^0[G\cdot \bar{r}_{2^{j}-1}].
\]

\begin{definition}[Hill--Hopkins--Ravenel theories] \label{df:HHRTheories}
For each \(m\geq 0\), let 
\[
\BPGm=\BPG\wedge_{A_{m+1}}S^{0}.
\]
\end{definition}
The Reduction Theorem of \cite{HHR} says that for all \(G\), \(\BPG\!\langle0\rangle=\HZ\), and \cite{HHRKH} studied the spectrum \(\BPone\) (a computation we review below). 
\begin{remark}
Although the underlying homotopy groups of \(\BPGm\) is a polynomial ring:
\[
\pi_{\ast}^{e}\BPGm\cong \mathbb Z_{(2)}[r_{1},\gamma r_{1},\dots, \gamma^{2^{n-1}-1} r_{1},\dots,\gamma^{2^{n-1}-1}r_{2^{m}-1}],
\]
we do not know that \(\BPGm\) has even an associative multiplication. It is, however, canonically an \(\MUG\)-module, and hence the slice spectral sequence will be a spectral sequence of modules over the slice spectral sequence for  \(\MUG\).
\end{remark}

The same arguments as for \(\BPG\) allow us to determine the slice associated graded for \(\BPGm\) for any \(m\). 

\begin{theorem}
The slice associated graded for \(\BPGm\) is the graded spectrum
\[
S^0[G\cdot\bar{r}_{1},\dots, G\cdot\bar{r}_{2^{m}-1}]\wedge \HZ,
\]
where the degree of a summand corresponding to a polynomial in the \(\bar{r}_{i}\) and their conjugates is just the underlying degree.
\end{theorem}

\begin{corollary}
The slice spectral sequence for the \(RO(G)\)-graded homotopy of \(\BPGm\) has \(E_{2}\)-term the \(RO(G)\)-graded homology of \(S^0[G\cdot\bar{r}_{1},\dots,G\cdot\bar{r}_{2^{m}-1}]\), with coefficients in the constant Mackey functor \(\underline{\mathbb Z}\).
\end{corollary}

Since the slice filtration is an equivariant filtration, the slice spectral sequence is a spectral sequence of \(RO(G)\)-graded Mackey functors. Moreover, the slice spectral sequence for \(\MUG\) is a multiplicative spectral sequence, and the slice spectral sequence for \(\BPGm\) is a spectral sequence of modules over it in Mackey functors. 

\subsection{The slice spectral sequence for \texorpdfstring{\(\BPtwo\)}{BPC4Two}}
From now on, we restrict attention to the case \(G=C_{4}\) and \(\BPtwo\). We will use the slice spectral sequence to compute the integer graded homotopy Mackey functors of \(\BPtwo\). To describe this, we describe in more detail the \(E_{2}\)-term of the slice spectral sequence.

\begin{notation}
Let \(\sigma\) denote the \(1\)-dimensional sign representation of \(C_{4}\), and let \(\lambda\) denote the \(2\)-dimensional irreducible representation of \(C_{4}\) given by rotation by \(\pi/2\). Let \(\sigma_{2}\) denote the \(1\)-dimensional sign representation of \(C_{2}\). Finally, let \(1\) denote the trivial representation of dimension \(1\).
\end{notation}

The homology groups of a representation sphere with coefficients in \(\underline{\mathbb Z}\) are generated by certain products of Euler classes and orientation classes for irreducible representations. 

\begin{definition}
For any representation \(V\) for which \(V^{G}=\{0\}\), let \(a_{V}\colon S^{0}\to S^{V}\) denote the Euler class of the representation \(V\). Let \(a_{V}\) also denote the corresponding Hurewicz image in \(\pi_{-V}\HZ\).
\end{definition}

The following definition can be found in \cite[Definition~3.12]{HHR}. 
\begin{definition} \label{def:defininguV}
If \(V\) is an orientable representation of \(G\), then a choice of orientation for $V$ gives an isomorphism 
$\underline{H}_{\dim V}^G(S^V; \mZ) \cong \mZ$.  In particular, the restriction map 
\[H_{\dim V}^G(S^V; \mZ) \longrightarrow H_{\dim V}(S^{\dim V}; \mathbb Z) \cong \mathbb{Z}\]
is an isomorphism.  Let
\[
u_{V}\in H_{\dim V}^G(S^{V};\mZ)
\]
be the generator which restricts to the element \(1 \in \mathbb{Z} \) under the restriction isomorphism.
\end{definition}

For the group $C_{4}$, the sign representation $\sigma$ is not orientable.  However, its restriction to $C_2$ is the trivial representation, which is orientable.  This gives an element $u_\sigma \in \pi_{1-\sigma} H\mZ(C_4/C_2)$.  The action of the generator $\gamma \in C_4/C_2$ on $\pi_{1-\sigma} H\mZ(C_4/C_2)$ sends $u_\sigma$ to $-u_\sigma$ (see \cite[Theorem~2.13]{HHRKH}).

The element $u_\sigma$ is useful for computations involving the transfer classes.  Let $X$ be a $C_4$-spectrum.  For any $V \in RO(C_4)$ that restricts to $W \in RO(C_2)$, we have maps 
$$\begin{tikzcd}
\pi_V^{C_4} X \ar[r, shift left, "res"] & \ar[l, shift left, "tr_V"] \pi_W^{C_2} X. 
\end{tikzcd}$$
Note that we include $V$ as an index in $tr_V$ because its target is not uniquely determined by its source.  More precisely, given an element $x \in \pi_W^{C_2} X$, we can transfer $x$ to an element in $\pi_{V+k(1-\sigma)}^{C_4}X$ for any integer $k$.  In the later sections of the paper, the element $u_\sigma$ will often be used as a placeholder corresponding to the element $1 \in \pi_0 H\mZ(C_4/C_2)$ in order to indicate the target degree of the transfer.  In particular, if $V$ does not have any copies of $\sigma$, then we will use $tr(u_\sigma^k x)$ to denote the transfer $tr_{V + k(1-\sigma)}(x)$.

At $C_4$, the $a_V$ and $u_V$ classes satisfy a number of relations:

\begin{enumerate}%[leftmargin=*]
\item $2a_\sigma = 2a_{\sigma_2} = 4a_\lambda = 0$;
\item $res_{C_{2}}^{C_{4}}(a_\sigma) = 0$, $res_{C_{2}}^{C_{4}}(a_\lambda) = a_{2\sigma_2}$, $res_{C_{2}}^{C_{4}}(u_{2\sigma}) = u_{\sigma}^2 = 1$, $res_{C_{2}}^{C_{4}}(u_\lambda) = u_{2\sigma_2}$; 
%\item $N(a_{\sigma_2}) = a_\lambda$, $u_{\sigma}N(u_{\sigma_2}) = u_\lambda$;
\item $u_\lambda a_{2\sigma}= 2a_\lambda u_{2\sigma}$ (gold relation);
%\item $tr(res(a)b) = a\cdot tr(b)$ (Frobenius). %%This shouldn't go here.
\end{enumerate}

These allow us to identify all of the elements in the homology groups of representation spheres.

%\vspace{0.1in}

\subsection{Tambara structure}
%
%From now on, we specialize to the case $G = C_4$.  Let $X$ be a commutative ($\mathbb{E}_\infty$) $C_4$-spectrum.  There are maps
%$$\begin{tikzcd}
%\pi_\bigstar^{C_4} X \ar[rr, bend left = 40, "\text{res}"] && \ar[ll, bend left = 40, "\text{tr}"]  \ar[ll, swap, "N_{C_2}^{C_4}"]  \pi_\bigstar^{C_2} X. 
%\end{tikzcd}$$
%The norm map above is the internal norm, which sends a $C_2$-equivariant map $S^\star \longrightarrow i_{C_2}^*X$ to the $C_4$-equivariant map 
%$$N_{C_2}^{C_4}(S^\star) \longrightarrow N_{C_2}^{C_4}(i_{C_2}^*X) \longrightarrow X.$$
%These maps induce maps of towers between the $C_4$-slice tower and the $C_2$-slice tower of $X$.  These correspond to maps of spectral sequences
%$$\begin{tikzcd}
%C_4\text{-}\SliceSS(X) \ar[rr, bend left = 30, "\text{res}"] && \ar[ll, bend left = 30, "\text{tr}"]  \ar[ll, swap, dashed, "N_{C_2}^{C_4}"]  C_2\text{-}\SliceSS(X). 
%\end{tikzcd}$$
A multiplicative spectral sequence of Mackey functors can equivalently be thought of a kind of Mackey functor object in spectral sequences. In particular, we can view this as being \(3\) spectral sequences:
\begin{enumerate}
\item a multiplicative spectral sequence computing the \(C_{4}\)-fixed points,
\item a multiplicative spectral sequence computing the \(C_{2}\)-fixed points, and
\item a (collapsing) multiplicative spectral sequence computing the underlying homotopy.
\end{enumerate}
The restriction and transfer maps in the Mackey functors can then be viewed as maps of spectral sequences connecting these, with the restriction maps being maps of DGAs, and the transfer maps being maps of DGMs over these DGAs.

For commutative ring spectra like \(\MUG\), we have additional structure on the \(RO(G)\)-graded homotopy groups given by the norms. If \(R\) is a \(G\)-equivariant commutative ring spectrum, then we have a multiplicative map
\[
N_{H}^{G}\colon \pi_{V}^{H}(R)\to \pi_{Ind_{H}^{G}V}^{G}(R)
\]
which takes a map
\[
S^{V}\to i_{H}^{\ast}R
\]
to the composite
\[
S^{Ind_{H}^{G}V}\cong N_{H}^{G}(S^{V})\to N_{H}^{G}i_{H}^{\ast}R\to R,
\]
where the final map is the counit of the norm-forget adjunction. The norm maps are not additive, but they do satisfy certain explicitly describable formulae which encode the norms of sums and of transfers. At the level of \(\pi_{0}\), this data is traditionally called a ``Tambara functor'', studied by Brun for equivariant commutative ring spectra, and more generally, this \(RO(G)\)-graded version was used by Hill, Hopkins, and Ravenel in their analysis of the slice spectral sequence \cite{Brun, HHR, HHRKH}.

In the slice spectral sequence, the norms play a more subtle role. The norm from \(H\) to \(G\) scales slice filtration by \(|G/H|\), just as multiplication scales degree. In particular, it will not simply commute with the differentials. We have a formula, however, for the differentials on key multiples of norms.

\begin{thm}
\label{thm:NormFormula}
Let $R$ be a commutative $G$-spectrum.  Let $d_r(x) = y$ be a $d_r$-differential in the $C_2$-slice spectral sequence.  If both $a_\sigma N_{C_2}^{C_4} x$ and $N_{C_2}^{C_4}y$ survive to the $E_{2r-1}$-page, then $d_{2r-1}(a_\sigma N_{C_2}^{C_4} x) = N_{C_2}^{C_4}y$ in $C_4$-$\SliceSS(R)$ (see \cite[Corollary 4.8]{HHRKH}).  
\end{thm}

\begin{proof}
The $d_r$-differential can be represented by the diagram 
$$\begin{tikzcd}
S^V \ar[r] \ar[d, "y"]  & D(1+V) \ar[r] \ar[d] & S^{1+V} \ar[d, "x"] \\ 
P_{s+r-1}^{C_2} R \ar[r] &P_s^{C_2} R \ar[r] & P_{s}^{C_2}X/P_{s+r-1}^{C_2}R. 
\end{tikzcd}$$
Let $W = \text{Ind}_H^G V$.  Applying the norm functor $N_{C_2}^{C_4}(-)$ yields the new diagram 
$$\begin{tikzcd}
S^W \ar[r] \ar[d, "N_{C_2}^{C_4} y"]  & D(1+\sigma+W) \ar[r] \ar[d] & S^{1+\sigma+W} \ar[d, "N_{C_2}^{C_4}x"] \\ 
N_{C_2}^{C_4} P_{s+r-1}^{C_2} R \ar[r] &N_{C_2}^{C_4} P_s^{C_2} R \ar[r] & N_{C_2}^{C_4}(P_{s}^{C_2}X/P_{s+r-1}^{C_2} R). 
\end{tikzcd}$$
Both rows of the this diagram are no longer cofiber sequences.  We can enlarge this diagram so that both the top and the bottom rows are cofiber sequences: 
$$\begin{tikzcd}
S^W \ar[d,"="] \ar[r] \ar[ddd, bend right = 40]& D(1+W) \ar[d, "a_\sigma"] \ar[r] \ar[ddd, bend right = 40]& S^{1+W} \ar[d, "a_\sigma"] \ar[ddd, bend right = 40] &\\ 
S^W \ar[r] \ar[d, "N_{C_2}^{C_4} y"]  & D(1+\sigma+W) \ar[r] \ar[d] & S^{1+\sigma+W} \ar[d, "N_{C_2}^{C_4}x"]& \\ 
N_{C_2}^{C_4} P_{s+r-1}^{C_2} R \ar[r]  &N_{C_2}^{C_4} P_s^{C_2} R \ar[r] & N_{C_2}^{C_4}(P_{s}^{C_2}R/P_{s+r-1}^{C_2}R)  \ar[rd] &\\
N_{C_2}^{C_4} P_{s+r-1}^{C_2} R \ar[r] \ar[u,swap, "id"] \ar[d] &N_{C_2}^{C_4} P_s^{C_2} R \ar[r] \ar[d] \ar[u,swap, "id"]& N_{C_2}^{C_4}(P_{s}^{C_2}R)/N_{C_2}^{C_4}(P_{s+r-1}^{C_2}R) \ar[u] \ar[d] & P_{2s}^{C_4} R/P_{2s+r-1}^{C_4} R\\
P_{2s+2r-2}^{C_4} R \ar[r] &P_{2s}^{C_4} R \ar[r] &P_{2s}^{C_4}R/P_{2s+2r-2}^{C_4} R \ar[ru] &
\end{tikzcd}$$
The first, fourth, and fifth rows are cofiber sequences.  The third vertical map from the fourth row to the third row is induced by the first two vertical maps.  The third long vertical map from the first row to the fourth row is induced from the first two long vertical maps.  

The composite map from the first row to the fifth row predicts a $d_{2r-1}$-differential in the $C_4$-slice spectral sequence.  The predicted target is $N_{C_2}^{C_4}y$.  Therefore, this class must die on or before the $E_{2r-1}$-page.  If both this class and $a_\sigma N_{C_2}^{C_4} x$ survive to the $E_{2r-1}$-page, then 
$$d_{2r-1}(a_\sigma N_{C_2}^{C_4} x) = N_{C_2}^{C_4}y.$$
\end{proof}

\begin{remark}
The slice spectral sequence is actually a spectral sequence of graded Tambara functors in the sense that the differentials are actually genuine equivariant differentials in the sense of \cite{HillAQ}. We will not need this in what follows, however.
\end{remark}

\subsection{Formal group law formulas} \label{subsec:fglFormulas}
Consider the $C_2$-equivariant map 
$$\BPR \longrightarrow i_{C_2}^* N_{C_2}^{C_4} \BPR = i_{C_2}^* \BPC$$
coming from the norm-restriction adjunction.  Post-composing with the quotient map $\BPC \to \BPtwo$ produces the $C_2$-equivariant map 
$$\BPR \to i_{C_2}^* \BPtwo,$$
which, after taking $\pi_{*\rho_2}^{C_2}(-)$, is a map 
$$\mathbb{Z}_{(2)}[\bar{v}_1, \bar{v}_2, \ldots] \longrightarrow \mathbb{Z}_{(2)}[\rone, \grone, \rthree, \grthree]$$
of polynomial algebras.  Here, the $\bar{v}_i$-generators are the Araki generators. 

Let $S :=  \pi^{C_2}_{*\rho_2} \BPtwo= \mathbb{Z}_{(2)}[\rone, \grone, \rthree, \grthree]$.  By an abuse of notation, let $\bar{v}_i \in S$ denote the image of $\bar{v}_i \in \pi_{(2^i-1)\rho_2}^{C_2} \BPR$ under the map above.  Our next goal is to relate the $\bar{v}_i$-generators to the $\bar{r}_i$-generators. 

Let $\bar{F}$ be the $C_2$-equivariant formal group law corresponding to the map $\pi_{*\rho_2}^{C_2} \BPR \to \pi_{*\rho_2}^{C_2} \BPtwo$.  By definition, its 2-series is 
$${[}2]_{\bar{F}}(\bar{x}) = 2\bar{x} +_{\bar{F}} \bar{v}_1 \bar{x}^2 +_{\bar{F}} \bar{v}_2 \bar{x}^4 +_{\bar{F}} \bar{v}_3 \bar{x}^8 +_{\bar{F}} \bar{v}_4 \bar{x}^{16} + \cdots.$$
Let $\bar{m}_i \in 2^{-1}S$ be the coefficients of the logarithm of $\bar{F}$: 
$$\log_{\bar{F}}(\bar{x}) = \bar{x} + \bar{m}_1 \bar{x}^2 + \bar{m}_2 \bar{x}^4 + \bar{m}_3 \bar{x}^8 + \bar{m}_4 \bar{x}^{16} + \cdots.$$
Taking the logarithm of both sides of the 2-series produces the equation
$$2 \log_{\bar{F}}(\bar{x}) = \log_{\bar{F}}(2 \bar{x}) + \log_{\bar{F}}( \bar{v}_1 \bar{x}^2) + \log_{\bar{F}}(\bar{v}_2 \bar{x}^4) + \log_{\bar{F}}(\bar{v}_3 \bar{x}^8) + \cdots.$$
Expanding both sides of the equation using the power series expansion of the logarithm and comparing coefficients, we obtain the equations
\begin{eqnarray}
2 \bar{m}_1 &=& 4 \bar{m}_1 + \bar{v}_1 \label{eq:miCoefficientsvi} \\
2 \bar{m}_2 &=& 16 \bar{m}_2 + \bar{m}_1 \bar{v}_1^2 + \bar{v}_2 \nonumber \\ 
2 \bar{m}_3 &=& 2^8 \bar{m}_3 + \bar{m}_2 \bar{v}_1^4 + \bar{m}_1 \bar{v}_2^2 + \bar{v}_3 \nonumber\\ 
2 \bar{m}_4 &=& 2^{16} \bar{m}_4 + \bar{m}_3 \bar{v}_1^8 + \bar{m}_2 \bar{v}_2^4 + \bar{m}_1 \bar{v}_3^2 + \bar{v}_4 \nonumber\\
&\vdots& \nonumber
\end{eqnarray}
Rearranging, we obtain the relation
\begin{equation} \label{eqn:v_iequals2m_i}
\bar{v}_i = 2\bar{m}_i \pmod{M_{i}}
\end{equation}
for all $i \geq 1$.  Here, $M_{i}$ is the $S$-submodule of $2^{-1}S$ (regarded as a $S$-module) that is generated by the elements $2$, $\bar{m}_1$, $\bar{m}_2$, $\ldots$, $\bar{m}_{i-1}$.  In other words, an element in $M_{i}$ is of the form 
$$s_0 \cdot 2 + s_1 \cdot \bar{m}_1 + \cdots + s_{i-1} \bar{m}_{i-1}$$
where $s_j \in S$ for all $0 \leq j \leq i-1$. 

\begin{lem}\label{lem:miSameAsvi}
Let $I_{i} \subset S$ denote the ideal $(2, \bar{v}_1, \cdots \bar{v}_{i-1})$.  Then 
$$M_{i} \cap S = I_{i}.$$
\end{lem}
\begin{proof}
We will prove the claim by using induction on $i$.  The base case when $i \geq1$ is straight forward: an element in $M_1$ is of the form $s_0 \cdot 2$, where $s_0 \in S$.  Therefore $M_1 \cap S = (2) = I_1$.  

Now, suppose that $M_{i-1} \cap S = I_{i-1}$.  Furthermore, suppose that the element 
$$m = s_0 \cdot 2 + s_1 \cdot \bar{m}_1 + \cdots + s_{i-1} \bar{m}_{i-1} \in M_i$$
is also in $S$.  From the equations in (\ref{eq:miCoefficientsvi}), it is straightforward to see that $\bar{m}_k$ has denominator exactly $2^k$ for all $k \geq 1$.  In the expression for $m$, only the last term $s_{i-1} \bar{m}_{i-1}$ has denominator $2^{i-1}$.  All the other terms have denominators at most $2^{i-2}$.  Since $m \in S$, $s_{i-1}$ must be divisible by 2.  In other words, $s_{i-1} = 2s_{i-1}'$ for some $s_{i-1}' \in S$.  Using equation~(\ref{eqn:v_iequals2m_i}), $m$ can be rewritten as 
\begin{eqnarray*}
m &=& s_0 \cdot 2 + s_1 \cdot \bar{m}_1 + \cdots + s_{i-2}\bar{m}_{i-2}+ 2s_{i-1}' \bar{m}_{i-1} \\ 
&=& s_0 \cdot 2 + s_1 \cdot \bar{m}_1 + \cdots + s_{i-2}\bar{m}_{i-2}+ s_{i-1}' (2\bar{m}_{i-1}) \\ 
&\in& s_0 \cdot 2 + s_1 \cdot \bar{m}_1 + \cdots + s_{i-2}\bar{m}_{i-2} + s_{i-1}' (\bar{v}_{i-1} + M_{i-1}) \\
&\in& M_{i-1} + s_{i-1}' (\bar{v}_{i-1} + M_{i-1}) \\
&=& M_{i-1} + s_{i-1}' \bar{v}_{i-1}.
\end{eqnarray*}
Therefore, $m = x + s_{i-1}' v_{i-1}$ for some $x \in M_{i-1}$.  Since $m \in S$ and $s_{i-1}' v_{i-1} \in S$, $x \in S$ as well.  The induction hypothesis now implies that $x \in I_{i-1}$.  It follows from this that $m \in I_i$, as desired. 
\end{proof}
%\begin{eqnarray*}
%\bar{v}_1 &=& 2 \bar{m}_1 \pmod{2} \\ 
%\bar{v}_2 &=& 2 \bar{m}_2 \pmod{2, \bar{m}_1} \\ 
%\bar{v}_3 &=& 2 \bar{m}_3 \pmod{2, \bar{m}_1, \bar{m}_2} \\ 
%\bar{v}_4 &=& 2 \bar{m}_4 \pmod{2, \bar{m}_1, \bar{m}_2, \bar{m}_3} \\
%&\vdots&
%\end{eqnarray*}

\begin{thm}\label{thm:fglFormulas}
We have the following relations: 
\begin{eqnarray*}
\bar{v}_1 &=& \bar{r}_1 + \gamma \bar{r}_1 \pmod{2}, \\ 
\bar{v}_2 &=& \bar{r}_1^3 + \bar{r}_3 + \gamma \bar{r}_3 \pmod{2, \bar{v}_1}, \\
\bar{v}_3 &=& \bar{r}_1 (\bar{r}_3^2 + \bar{r}_3 (\gamma \bar{r}_3) + (\gamma \bar{r}_3)^2) \pmod{2, \bar{v}_1, \bar{v}_2}, \\ 
\bar{v}_4 &=& \bar{r}_3^4 (\gamma \bar{r}_3) \pmod{2, \bar{v}_1, \bar{v}_2, \bar{v}_3}.
\end{eqnarray*}
\end{thm}
\begin{proof}
To obtain the formulas in the statement of the theorem, we need to establish relations between the generators $\{\bar{r}_1, \gamma \bar{r}_1, \bar{r}_3, \gamma \bar{r}_3\}$ and the $\bar{m}_i$-generators.  The $\bar{r}_i$ generators, by definition, are the coefficients of the strict isomorphism from $\bar{F}$ to $\bar{F}^\gamma$ (see \cite[Section 5]{HHR}):
$$\begin{tikzcd}
&\bar{F}^{\text{add}}& \\ 
\bar{F} \ar[ru, "\log_{\bar{F}}"] \ar[rr, swap, "\bar{x} +_{\bar{F}^\gamma} \bar{r}_1 \bar{x}^2 +_{\bar{F}^\gamma} \bar{r}_3 \bar{x}^4" ]& & \bar{F}^\gamma \ar[lu, swap, "\log_{\bar{F}^\gamma}"]
\end{tikzcd}$$
Here, $\log_{\bar{F}^\gamma}$ is the logarithm for the formal group law $\bar{F}^\gamma$, and its power-series expansion is 
$$\log_{\bar{F}^\gamma}(x) = \bar{x} + (\gamma \bar{m}_1) \bar{x}^2 + (\gamma\bar{m}_2 )\bar{x}^4 + (\gamma\bar{m}_3) \bar{x}^8 + (\gamma \bar{m}_4) \bar{x}^{16} + \cdots.$$

The commutativity of the diagram implies that 
\begin{eqnarray*}\label{eqn:log}
\log_{\bar{F}}(\bar{x}) &=& \log_{\bar{F}^\gamma}(\bar{x} +_{\bar{F}^\gamma} \bar{r}_1 \bar{x}^2 +_{\bar{F}^\gamma} \bar{r}_3 \bar{x}^4) \\ 
&=& \log_{\bar{F}^\gamma}(\bar{x}) + \log_{\bar{F}^\gamma}(\bar{r}_1 \bar{x}^2) + \log_{\bar{F}^\gamma}(\bar{r}_3 \bar{x}^4)
\end{eqnarray*}
Expanding both sides according to the logarithm formulas, we get 
\begin{eqnarray*}
\bar{x}+\bar{m}_1 \bar{x}^2 + \bar{m}_2 \bar{x}^4 + \bar{m}_3 \bar{x}^8 + \bar{m}_4 \bar{x}^{16} + \cdots &=& \bar{x} + (\gamma \bar{m}_1 + \bar{r}_1) \bar{x}^2 + (\gamma \bar{m}_2 + (\gamma \bar{m}_1) \bar{r}_1^2 + \bar{r}_3) \bar{x}^4 \\ 
&+& (\gamma \bar{m}_3 + (\gamma \bar{m}_2) \bar{r}_1^4 + (\gamma \bar{m}_1)\bar{r}_3^2) \bar{x}^8 \\ 
&+& (\gamma \bar{m}_4 + (\gamma \bar{m}_3) \bar{r}_1^8 + (\gamma \bar{m}_2) \bar{r}_3^4 )\bar{x}^{16} + \cdots
\end{eqnarray*}
Comparing coefficients, we obtain the relations
\begin{eqnarray*}
\bar{m}_1 - \gamma \bar{m}_1 &=& \bar{r}_1 \\ 
\bar{m}_2 - \gamma \bar{m}_2 &=& (\gamma \bar{m}_1)\bar{r}_1^2 + \bar{r}_3 \\ 
\bar{m}_3 - \gamma \bar{m}_3 &=& (\gamma \bar{m}_2) \bar{r}_1^4 + (\gamma \bar{m}_1) \bar{r}_3^2 \\ 
\bar{m}_4 - \gamma \bar{m}_4 &=& (\gamma \bar{m}_3) \bar{r}_1^8 + (\gamma \bar{m}_2) \bar{r}_3^4 
\end{eqnarray*}
We can also apply $\gamma$ to the relations above to obtain more relations
\begin{eqnarray*}
\gamma \bar{m}_1 + \bar{m}_1 &=& \gamma \bar{r}_1 \\ 
\gamma \bar{m}_2 + \bar{m}_2 &=& -\bar{m}_1(\gamma \bar{r}_1)^2 + \gamma \bar{r}_3 \\ 
\gamma \bar{m}_3 + \bar{m}_3 &=& -\bar{m}_2 (\gamma \bar{r}_1)^4 - \bar{m}_1 (\gamma \bar{r}_3)^2 \\
\gamma \bar{m}_4 + \bar{m}_4 &=& - \bar{m}_3 (\gamma \bar{r}_1)^8 - \bar{m}_2 (\gamma \bar{r}_3)^4.  
\end{eqnarray*}

These relations together produce the following formulas: 
\begin{eqnarray*} 
\bar{v}_1 &=& 2 \bar{m}_1 \pmod{M_1}  \\
&=& \bar{r}_1 + \gamma \bar{r}_1 \pmod{M_1}; \\ 
\bar{v}_2 &=& 2 \bar{m}_2 \pmod{M_2} \\ 
&=& (\gamma \bar{m}_1) \bar{r}_1^2 + \bar{r}_3 + \gamma \bar{r}_3 \pmod{M_2}\\ 
&=& (\gamma \bar{m}_1 - \bar{m}_1) \bar{r}_1^2 + \bar{r}_3 + \gamma \bar{r}_3 \pmod{M_2}\\ 
&=& -\bar{r}_1^3 + \bar{r}_3 + \gamma \bar{r}_3 \pmod{M_2}\\ 
&=& \bar{r}_1^3 + \bar{r}_3 + \gamma \bar{r}_3 \pmod{M_2}; \\ 
\bar{v}_3&=& 2 \bar{m}_3 \pmod{M_3} \\ 
&=& (\gamma \bar{m}_2) \bar{r}_1^4 + (\gamma \bar{m}_1) \bar{r}_3^2 -\bar{m}_2 (\gamma \bar{r}_1)^4 - \bar{m}_1 (\gamma \bar{r}_3)^2  \pmod{M_3} \\
&=& (\gamma \bar{m}_2) \bar{r}_1^4 + (\gamma \bar{m}_1) \bar{r}_3^2 \pmod{M_3} \\ 
&=& (\gamma \bar{m}_2 + \bar{m}_2 ) \bar{r}_1^4 + (\gamma \bar{m}_1 + \bar{m}_1)\bar{r}_3^2  \pmod{M_3}\\ 
&=& (\gamma \bar{r}_3) \bar{r}_1^4 + (\gamma \bar{r}_1)\bar{r}_3^2 \pmod{M_3}\\ 
&=& \bar{r}_1 \left((\gamma \bar{r}_3) (\bar{r}_3 + \gamma \bar{r}_3) + \bar{r}_3^2\right) \pmod{2M_3}\\ 
&=& \bar{r}_1(\bar{r}_3^2 + \bar{r}_3 (\gamma \bar{r}_3) + (\gamma \bar{r}_3)^2) \pmod{M_3}; \\
\bar{v}_4 &=& 2\bar{m}_4 \pmod{M_4} \\ 
&=& (\gamma \bar{m}_3) \bar{r}_1^8 + (\gamma \bar{m}_2) \bar{r}_3^4 - \bar{m}_3 (\gamma \bar{r}_1)^8 - \bar{m}_2 (\gamma \bar{r}_3)^4  \pmod{M_4} \\ 
&=& (\gamma \bar{m}_3) \bar{r}_1^8 + (\gamma \bar{m}_2) \bar{r}_3^4  \pmod{M_4} \\ 
&=& (\gamma \bar{m}_3 + \bar{m}_3) \bar{r}_1^8 + (\gamma \bar{m}_2 + \bar{m}_2) \bar{r}_3^4 \pmod{M_4} \\ 
&=& (\gamma \bar{r}_3) \bar{r}_3^4  \pmod{M_4}. 
\end{eqnarray*}
These formulas, combined with Lemma~\ref{lem:miSameAsvi}, give the desired formulas. 
\end{proof}

\subsection{Lubin--Tate Theories}
We will now study the relationship between $\BPtwo$ and Lubin--Tate theories.  By analyzing the $C_4$-equivariant formal group law associated to $\BPtwo$, the following theorem shows that $\BPtwo$ provides a model for a height-4 Lubin--Tate theory with a $C_4$-action coming from the automorphisms of its formal group law. 

\begin{thm}[Theorem~\ref{thm:MainTheorem1}]\label{thm:BPC4-2isE4hC12} \hfill
\begin{enumerate}
\item There exists a height-4 Lubin--Tate theory $E_4$ with coefficient ring 
\begin{eqnarray*}
\pi_*E_4 &\cong& W(\mathbb{F}_{2^4})[C_4 \cdot r_1, C_4 \cdot u][C_4 \cdot u^{-1}]^{\wedge}_{\mathfrak{m}},
\end{eqnarray*}
where $|r_1| = |u| = 2$ and $\mathfrak{m} = (C_4 \cdot r_1, C_4 \cdot (u - \gamma u))$.  Furthermore, there is a subgroup ${G = C_4 \times \left(\Gal(\mathbb{F}_{2^4}/\mathbb{F}_2) \ltimes C_3\right)}$ inside the Morava stabilizer group $\mathbb{G}_4$ such that the isomorphism above is a $G$-equivariant isomorphism. 
\item There is a $C_4$-equivariant homotopy commutative ring map 
\[\phi: \BPC \longrightarrow E_4\]
such that $\pi_*^e \phi: \pi_*^e \BPC \to \pi_* E_4$ is the map determined by sending
\[r_{2^i -1} \longmapsto \left\{\begin{array}{ll} r_1 & i = 1, \\ 
u^3 & i = 2, \\ 
0 & i \geq 3.  \end{array}\right. \]
\item After inverting the element 
\[D_2:= N(\bar{v}_4)N(\rthree)N(\bar{r}_3^2 + \bar{r}_3 (\gamma \bar{r}_3) + (\gamma \bar{r}_3)^2) \in \pi_{24\rho_4}^{C_4} \BPC,\]
there is a factorization 
\[\begin{tikzcd}
\BPC \ar[r] \ar[d] & E_4 \\ 
D_2^{-1}\BPC \ar[ru, dashed]&
\end{tikzcd}\]
of the $C_4$-equivariant orientation through $D_2^{-1}\BPC$.  
\end{enumerate}
\end{thm}

\begin{remark}\rm
Since the appearance of the first draft of this paper, Beaudry, the first author, the second author and Zeng proved a general formula for the images of the $v_i$-generators in $\pi_*^e BP^{(\!(n)\!)} \langle m \rangle$ \cite[Theorem~1.1]{BHSZ}.  This is a generalization of Theorem~\ref{thm:fglFormulas}.  As a result, they generalized Theorem~\ref{thm:BPC4-2isE4hC12} and proved that the formal group law associated to $\BP^{(\!(C_{2^n})\!)}\langle m \rangle$ gives a model of a height $h =2^{n-1}m$ Lubin--Tate theory with a $C_{2^n}$-action coming from the automorphism of its formal group law. 
\end{remark}

\begin{proof}[Proof of Theorem~\ref{thm:BPC4-2isE4hC12}]
We will prove (1) following the framework developed in the proof of \cite[Theorem~1.5]{BHSZ} (of which (1) is a special case of) and refer the readers to the corresponding theorems in \cite{BHSZ} for the proofs of (2) and (3).  

Recall that the underlying homotopy group of $\BPtwo$ is 
\[\pi_*^e \BPtwo = \mathbb{Z}_{(2)}[C_4 \cdot r_1, C_4 \cdot r_3].\]  
The composite map 
\[\begin{tikzcd}
BP_* \ar[r, "\pi_*\eta_L"]& \pi_*^e \BPC \ar[r] & \pi_*^e \BPtwo
\end{tikzcd}\]
defines a formal group law on $\pi_*^e \BPtwo$ which we will denote by $\mathcal{F}$.  

Define $R$ to be the ring 
\[R := W(\mathbb{F}_{2^4})[C_4 \cdot r_1, C_4 \cdot u][C_4 \cdot u^{-1}]^\wedge_{\mathfrak{m}},\]
where $|r_1| = |u| = 2$ and $\mathfrak{m} = (C_4 \cdot r_1, C_4 \cdot (u - \gamma u))$.  Note that $R$ is the completion of an extension of $\pi_* \BPtwo$, and there is a map $\pi_*\BPtwo \to R$ which sends $r_1 \mapsto r_1$ and $r_3 \mapsto u^3$.  The composite map 
\[BP_* \longrightarrow \pi_*^e \BPtwo \longrightarrow R\]
defines a formal group law over $R$ which we will continue to denote by $\mathcal{F}$.  

The formulas established in Theorem~\ref{thm:fglFormulas} imply that 
\begin{itemize}
\item $(2, v_1, v_2, \ldots)$ forms a regular sequence in $R$.
\item In the ring $R$, the ideal $I_4 = (2, v_1, v_2, v_3)$ is equal to the maximal ideal $\mathfrak{m}$. 
\item $R/I_4 = R/\mathfrak{m} \cong K := \mathbb{F}_{2^4}[u^\pm]$. 
\item The formal group law $\Gamma_4 := p^*\mathcal{F}$, where $p: R \to R/I_4 = K$ is the quotient map, has height exactly 4 over $K$. 
\end{itemize}
All together, these facts imply that $(R, \mathcal{F})$ is a universal deformation of $(K, \Gamma_4)$.  

The action of $C_4$ on $R$ is defined by 
\begin{eqnarray*}
f_\gamma(x) &=& \gamma x, \\
f_\gamma(\gamma x) &=& -x 
\end{eqnarray*}
for $x = r_1, u$.  The group $\Gal(\mathbb{F}_{2^4}/\mathbb{F}_2)$ acts on $R$ via its action on the coefficients $W(\mathbb{F}_{2^4})$.  The group $C_3$ acts on $R$ by 
\begin{eqnarray*}
f_{\zeta}(u) &=& \zeta^{-1}u, \\ 
f_{\zeta}(r_1) &=& r_1 
\end{eqnarray*}
for every $\zeta \in C_3$.  All together, these three actions combine to give an action of $G$ on $R$.  Applying the Landweber exact functor theorem and the Goerss--Hopkins--Miller theorem finishes the proof of (1).  

Statement (2) follows from \cite[Theorem~1.7]{BHSZ}.  Statement (3) follows from \cite[Theorem~1.8]{BHSZ}. 
\end{proof}

\begin{remark}\label{rem:psiComputingHFPSS}\rm
It can be shown \cite[Proposition~5.4]{BHSZ} that the ${(\Gal(\mathbb{F}_{2^4}/\mathbb{F}_2) \ltimes C_3)}$-homotopy fixed point spectral sequence (HFPSS) for $E_4$ collapses, and there is an isomorphism 
\[\pi_* E_4^{h\Gal(\mathbb{F}_{2^4}/\mathbb{F}_2) \ltimes C_3} \cong \mathbb{Z}_2[C_4 \cdot r_1, C_4 \cdot r_3][C_4 \cdot r_3^{-1}]^\wedge_{\mathfrak{m}'}, \]
where $\mathfrak{m}' = (C_4 \cdot r_1, C_4 \cdot (r_3 - \gamma r_3))$.  Furthermore, there is a $C_4$-equivariant homotopy commutative ring map (see \cite[Theorem~5.5]{BHSZ})
\[\psi: \BPC \longrightarrow E_4^{h\Gal(\mathbb{F}_{2^4}/\mathbb{F}_2) \ltimes C_3}\]
such that $\pi_*^e \psi: \pi_*^e \BPC \to \pi_*^e E_4^{h\Gal(\mathbb{F}_{2^4}/\mathbb{F}_2) \ltimes C_3}$ is the map determined by 
\[r_{2^i -1} \longmapsto \left\{\begin{array}{ll} r_{2^i-1} & i = 1,2, \\ 
0 & i \geq 3.  \end{array}\right. \]
Therefore, $E_4^{h \,C_4 \times (\Gal(\mathbb{F}_{2^4}/\mathbb{F}_2) \ltimes C_3)}$ can be computed by using the diagram 
\[\begin{tikzcd}
\SliceSS(\BPC) \ar[r] \ar[d] & \HFPSS(\BPC) \ar[r] & \HFPSS(E_4^{h\Gal(\mathbb{F}_{2^4}/\mathbb{F}_2) \ltimes C_3}) \\
\SliceSS(\BPtwo)
\end{tikzcd}\]
In the diagram above, $\SliceSS(\BPtwo)$ has the advantage that it contains all the information about the differentials in $\HFPSS(E_4^{h\Gal(\mathbb{F}_{2^4}/\mathbb{F}_2) \ltimes C_3})$ while having all of its classes being finite ($\mathbb{Z}/2$, $\mathbb{Z}/4$, and $\mathbb{Z}$) instead of power series. 
\end{remark}

We will now prove Theorem~\ref{thm:MainTheorem3}.

\begin{thm}\label{thm:BPtwo384}
The spectrum $D_2^{-1}\BPtwo$ is $384$-periodic. 
\end{thm}
\begin{proof}
This is a direct consequence of the discussion in \cite[Section~9]{HHR}.  There are three periodicities for $D_2^{-1}\BPtwo$: 
\begin{enumerate}
\item $\displaystyle S^{3\rho_4} \wedge D_2^{-1}\BPtwo \simeq D_2^{-1}\BPtwo$. 
\item $\displaystyle S^{8-8\sigma} \wedge D_2^{-1}\BPtwo \simeq D_2^{-1}\BPtwo$. 
\item $S^{32 + 32\sigma - 32\lambda} \wedge D_2^{-1}\BPtwo \simeq D_2^{-1}\BPtwo$. 
\end{enumerate}
The first periodicity is induced from $N(\bar{r}_3)$, which has been inverted.  For the second periodicity, the Slice Differentials Theorem of Hill--Hopkins--Ravenel \cite[Theorem~9.9]{HHR} shows that there are differentials 
$$d_{2^{i+2}-3}(u_{2\sigma}^{2^{k-1}}) = N(\bar{r}_{2^{k}-1})a_\sigma^{2^{k+1}-1}, i \geq 1$$
in the slice spectral sequence of $\BPC$.  These differentials will produce all the differentials in the region of the slice spectral sequence where there are only contributions from the regular slice cells.  Since we have inverted $N(\bar{r}_3) \in D_2$, the arguments in \cite[Theorem~9.16]{HHR} show that the classes $u_{2\sigma}^{2^{k-1}}$ for $k \geq 3$ are all permanent cycles in the slice spectral sequence of $D_2^{-1}\BPC$ (their predicted targets have all been killed by $d_{13}$-differentials).  In particular, when $k = 3$, the class $u_{8\sigma}$ is a permanent cycle.  This induces a $(8-8\sigma)$-periodicity in $D_2^{-1}\BPC$ and therefore in $D_2^{-1}\BPtwo$.  

For the third periodicity, note that since $N(\bar{v}_4)$ has been inverted, Theorem~9.16 in \cite{HHR} shows that the class $u_{32\sigma_2}$ is a permanent cycle in the $C_2$-slice spectral sequence of $D_2^{-1}\BPtwo$.  Therefore, the norm 
$$N(u_{32\sigma_2}) = \frac{u_{32\lambda}}{u_{32\sigma}}$$
is a permanent cycle in the $C_4$-slice spectral sequence (for the norm formula, see \cite[Lemma~4.9]{HHRKH}).  Here, note that the class $u_{32\sigma}$ is not invertible, so we cannot divide by it in general.  The formula above is a notational convention to reflect the fact that $N(u_{32\sigma_2})u_{32\sigma} = u_{32\lambda}$.  Combining these three periodicities produces the desired 384-periodicity: 
\begin{eqnarray*}
&&32\cdot (3 \rho_4) + 24 \cdot (8- 8\sigma) + 3 \cdot (32 + 32 \sigma - 32\lambda) \\
&=& 32\cdot (3 + 3\sigma + 3\lambda) + 24 \cdot (8- 8\sigma) + 3 \cdot (32 + 32 \sigma - 32\lambda) \\
&=& 384. 
\end{eqnarray*}
\end{proof}

\begin{remark}\rm
After computing the $E_\infty$-page of $C_4$-$\SliceSS(\BPtwo)$ completely (see Figure~\ref{fig:E4C4EinftyPage}), we learn that 384 is the minimal periodicity for $D_2^{-1}\BPtwo$.  In particular, it is not 128 or 192 periodic.
\end{remark}

The proof of Theorem~\ref{thm:BPtwo384} also shows that the elements $N(\rthree)$, $u_{8\sigma}$, and $\frac{u_{32\lambda}}{u_{32\sigma}}$ are permanent cycles in the the slice spectral sequence of $D_2^{-1}\BPC$.  The map 
\[\psi:\BPC \longrightarrow E_4^{h\Gal(\mathbb{F}_{2^4}/\mathbb{F}_2) \ltimes C_3}\]
in Remark~\ref{rem:psiComputingHFPSS} factors through $D_2^{-1}\BPC$ (see \cite[Remark~6.1]{BHSZ}), and the maps of spectral sequences 
\[\SliceSS(D_2^{-1}\BPC) \longrightarrow \HFPSS(D_2^{-1}\BPC) \longrightarrow \HFPSS(E_4^{h\Gal(\mathbb{F}_{2^4}/\mathbb{F}_2) \ltimes C_3})\]
imply that the elements $N(\rthree)$, $u_{8\sigma}$, and $\frac{u_{32\lambda}}{u_{32\sigma}}$ are permanent cycles in the homotopy fixed point spectral sequence of $E_4^{h\Gal(\mathbb{F}_{2^4}/\mathbb{F}_2) \ltimes C_3}$.  Therefore, we have $(3\rho_4)$, ${(8-8\sigma)}$, and ${(32+32\sigma - 32\lambda)}$-periodicities.  These combine to give a 384-periodicity for $E_4^{h\, C_4 \times (\Gal(\mathbb{F}_{2^4}/\mathbb{F}_2) \ltimes C_3)}$.  In particular, the spectrum $E_4^{hC_{12}}$ is 384-periodic.  This finishes the proof of Theorem~\ref{thm:MainTheorem3}.

\begin{remark}\rm
The careful reader may worry about the choices present in the construction of \(\BPtwo\) or the more general quotients of \(\BPG\). The terse answer is that the slice spectral sequence only cares about the indecomposables in the underlying homotopy ring of \(\BPG\), not the particular lifts. As an example of this, consider the class \(\bar{r}_{3}\) for \(C_{2^{n}}\). This is only well-defined modulo the ideal generated by 2, \(\bar{r}_{1}\), and its conjugates. Consider now the differential on the class \(u_{2\sigma}^{2}\) coming from \cite[Theorem~9.9]{HHR}:
\[
u_{2\sigma}^{2}\mapsto N_{C_{2}}^{C_{2^{n}}}(\bar{r}_{3}) a_{\sigma}^{4}a_{3\bar{\rho}}.
\]
Since multiplication by \(a_{\sigma}\) annihilates the transfer, the norm is additive after being multiplied by \(a_{\sigma}\). Moreover, the norm of \(2\) is killed by \(a_{\sigma}\) and the norm of \(\bar{r}_{1}\) is killed by \(a_{\sigma}^{3}\), so any possible indeterminacy in the definition of \(\bar{r}_{3}\) results in the exact same differentials.  Our computation applies to any form of $\BPtwo$. 
%We believe this to be generic behavior and hence do not distinguish between a form of \(\BPtwo\) and \(\BPtwo\).
\end{remark}

%%%%%%%%%%%%%%%%%%%%%%%%
%%%%%%%%%%%%%%%%%%%%%%%%
\section{The slice spectral sequence of $\BPC\langle 1 \rangle$} \label{sec:SliceSSBPone}
The $C_4$-equivariant refinement of $\BPone$ is 
$$S^0[\bar{r}_1, \gamma \bar{r}_1] \longrightarrow \BPone.$$
(See \cite[Section~5.3]{HHR} for the definition of a refinement.)  The proofs of the slice theorem and the reduction theorem in \cite{HHR} apply to $\BPone$ as well, from which we deduce its slices:
$$\left\{\begin{array}{ll}
\bar{r}_1^i \gamma \bar{r}_1^i: S^{i\rho_4} \wedge \HZ, & i \geq 0 \, \, \, \text{($4i$-slice)},\\
\bar{r}_1^i \gamma \bar{r}_1^i (\bar{r}_1^j, \gamma \bar{r}_1^j): {C_4}_+ \wedge_{C_2} S^{(2i+j)\rho_2} \wedge \HZ, & i \geq 0, j \geq 1 \, \, \, \text{(induced $(4i+2j)$-slice)}.
\end{array} \right.$$

%%%
\subsection{The $C_2$-slice spectral sequence}
The $C_2$-spectrum $i_{C_2}^*\BPone$ has no odd slice cells, and its $(2k)$-slice cells are indexed by the monomials
$$\{\bar{r}_1^i \gamma \bar{r}_1^j \,|\, i, j \geq 0, \, i + j = k \}.$$
Let $\bar{v}_i \in \pi_{i \rho_2}^{C_2} \BPR$ be the $C_2$-equivariant lifts of the Araki $v_i$-generators for $\pi_*BP$.  We can also regard them as elements in $\pi_{i\rho_2}^{C_2} \BPone$ via the map 
$$\BPR \stackrel{\eta_L}{\longrightarrow} i_{C_2}^* \BPC \longrightarrow i_{C_2}^* \BPone.$$
In \cite[Section 7]{HHRKH}, Hill, Hopkins, and Ravenel proved
\begin{eqnarray*}
\bar{v}_1 \pmod{2} &=& \bar{r}_1 + \gamma \bar{r}_1, \\ 
\bar{v}_2 \pmod{2, \bar{v}_1} &=&  \bar{r}_1^3, \\
\bar{v}_i \pmod{2, \bar{v}_1, \ldots, \bar{v}_{i-1}}&=& 0, \,\,\, i \geq 3. 
\end{eqnarray*}

In $C_2$-$\SliceSS(\BPR)$, all the differentials are known.  They are determined by the differentials
$$d_{2^{i+1}-1}(u_{2^i\sigma_2}) = \bar{v}_i a_{\sigma_2}^{2^{i+1}-1}, \,\,\, i \geq 1,$$
and multiplicative structures \cite[Theorem~9.9]{HHR}.  This, combined with the formulas above, implies that in $C_2$-$\SliceSS(i_{C_2}^*\BPone)$, all the differentials are determined by
\begin{eqnarray*}
d_3(u_{2\sigma_2}) &=& \bar{v}_1 a_{\sigma_2}^3 = (\bar{r}_1 + \gamma \bar{r}_1)a_{\sigma_2}^3, \\
d_7(u_{4\sigma_2}) &=& \bar{v}_2 a_{\sigma_2}^7 = \bar{r}_1^3 a_{\sigma_2}^7,
\end{eqnarray*}
and multiplicative structures.  The class $u_{8\sigma_2}$ is a permanent cycle.  The $C_2$-slice spectral sequence of $i_{C_2}^*\BPone$ is shown in Figure~\ref{fig:E2C2}.  
\begin{figure}
\begin{center}
\makebox[\textwidth]{\includegraphics[trim={0cm 11cm 0cm 4cm}, clip, scale = 0.8]{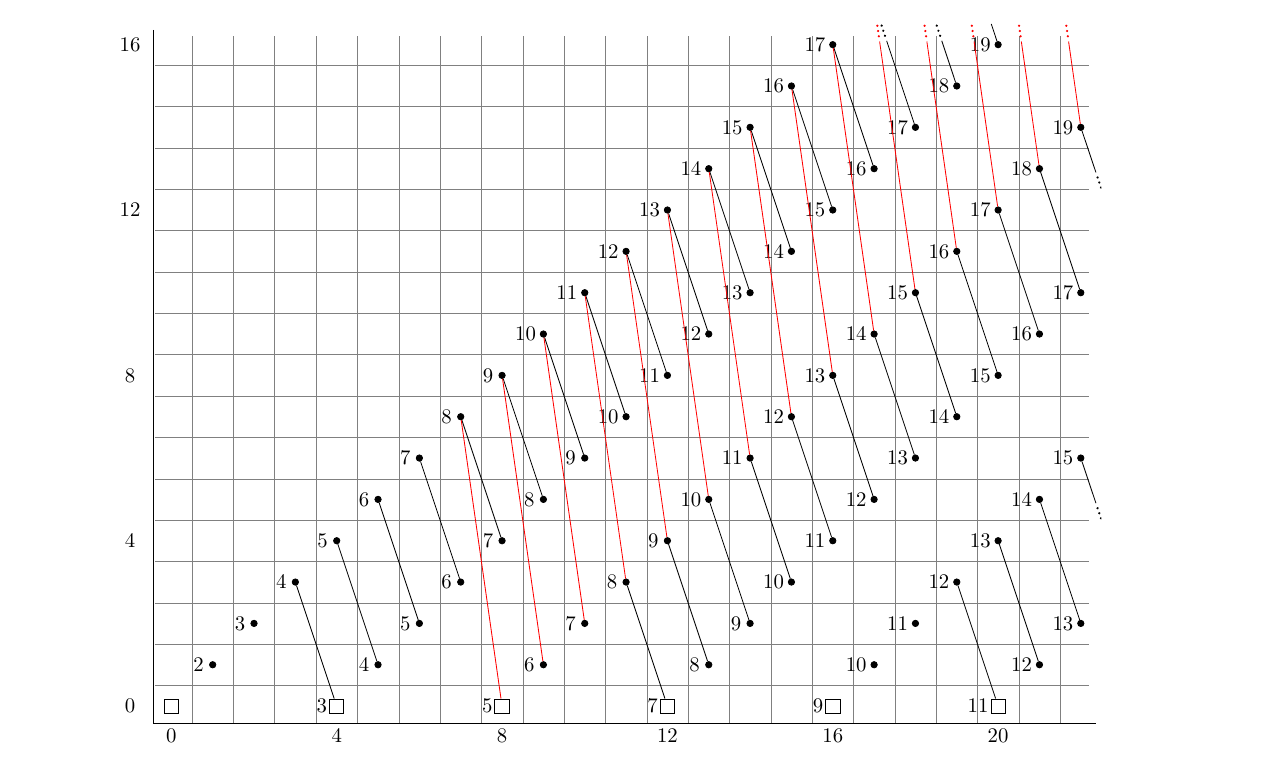}}
\end{center}
\begin{center}
\caption{The $C_2$-slice spectral sequence for $i_{C_2}^*\BPone$.  The $\square$-classes are $\mathbb{Z}$'s, the $\bullet$-classes are $\mathbb{Z}/2$'s.  The numbers next to the classes indicate the number of copies of such classes. }
\hfill
\label{fig:E2C2}
\end{center}
\end{figure}

%%%
\subsection{Organizing the slices, $d_3$-differentials}
We can organize the slices into the following table 
{\renewcommand{\arraystretch}{1.3}
\begin{equation}
\begin{array}{c|c|c|c} \label{array:BPoneSlices}
\bar{\mathfrak{d}}_1^0 &\bar{\mathfrak{d}}_1^1 &\bar{\mathfrak{d}}_1^2 & \cdots \\ \hline 
\color{blue}  \bar{\mathfrak{d}}_1^0 \bar{s}_1^1 & \color{blue}  \bar{\mathfrak{d}}_1^1 \bar{s}_1^1 & \color{blue}  \bar{\mathfrak{d}}_1^2 \bar{s}_1^1 & \cdots \\ 
\color{blue} \bar{\mathfrak{d}}_1^0 \bar{s}_1^2 &\color{blue}  \bar{\mathfrak{d}}_1^1 \bar{s}_1^2 &\color{blue}  \bar{\mathfrak{d}}_1^2 \bar{s}_1^2 & \cdots \\ 
\vdots & \vdots & \vdots & \ddots
\end{array}
\end{equation}
where $\done:= N(\bar{r}_1)$, and $\bar{s}_1^i := tr(\rone^i)$.  The first row consists of non-induced slices and the rest of the rows are all induced slices.  Also note that with the definition above, $res(\done) = \bar{r}_1 \gamma \bar{r}_1$ and $res(\bar{s}_1^i) = (1 + \gamma) \rone^i = \rone^i + \grone^i$.  As a warning, note that $\bar{s}_1^i \neq (\bar{s}_1)^i$. 

\begin{thm}\label{thm:BPoned3}
$d_3(u_\lambda) = tr(\rone a_{\sigma_2}) a_\lambda$.  
\end{thm}
\begin{proof}
The restriction of $u_\lambda$ is $res(u_\lambda) = u_{2\sigma_2}$.  In the $C_2$-slice spectral sequence, the class $u_{2\sigma_2}$ supports a nonzero $d_3$-differential 
$$d_3(u_{2\sigma_2}) = (\bar{r}_1 + \gamma \bar{r}_1)a_{3\sigma_2}.$$  
Therefore, $u_\lambda$ must support a differential of length at most 3.  For degree reasons, this differential must be a $d_3$-differential.  Natuality implies that 
$$d_3(u_\lambda) = tr(\rone a_{3\sigma_2}) = tr(\rone a_{\sigma_2}res(a_\lambda)) = tr(\rone a_{\sigma_2}) a_\lambda,$$
as desired. 
\end{proof}

To organize the $C_4$-slices in table~\ref{array:BPoneSlices}, we separate them into columns.  Each column consists of one non-induced slice cell, $\done^i$, and all the induced slice cells of the form $\done^i \bar{s}_1^j$, where $j \geq 1$.  

In light of Theorem~\ref{thm:BPoned3}, each column can be treated as an individual unit with respect to the $d_3$-differentials.  More precisely, the leading terms of any of the $d_3$-differentials are classes coming from the homotopy groups of slices belonging to the same column.  When drawing the slice spectral sequence of $\BPone$, we first produce the $E_2$-page of each column individually, together with their $d_3$-differentials (See Figure~\ref{fig:E2C4d3diffType1} and ~\ref{fig:E2C4d3diffType2}).  Afterwards, we combine the $E_5$-pages of every column all together into one whole spectral sequence.

\begin{figure}
\begin{center}
\makebox[\textwidth]{\includegraphics[trim={0cm 11cm 0cm 4cm}, clip, scale = 0.8, page = 1]{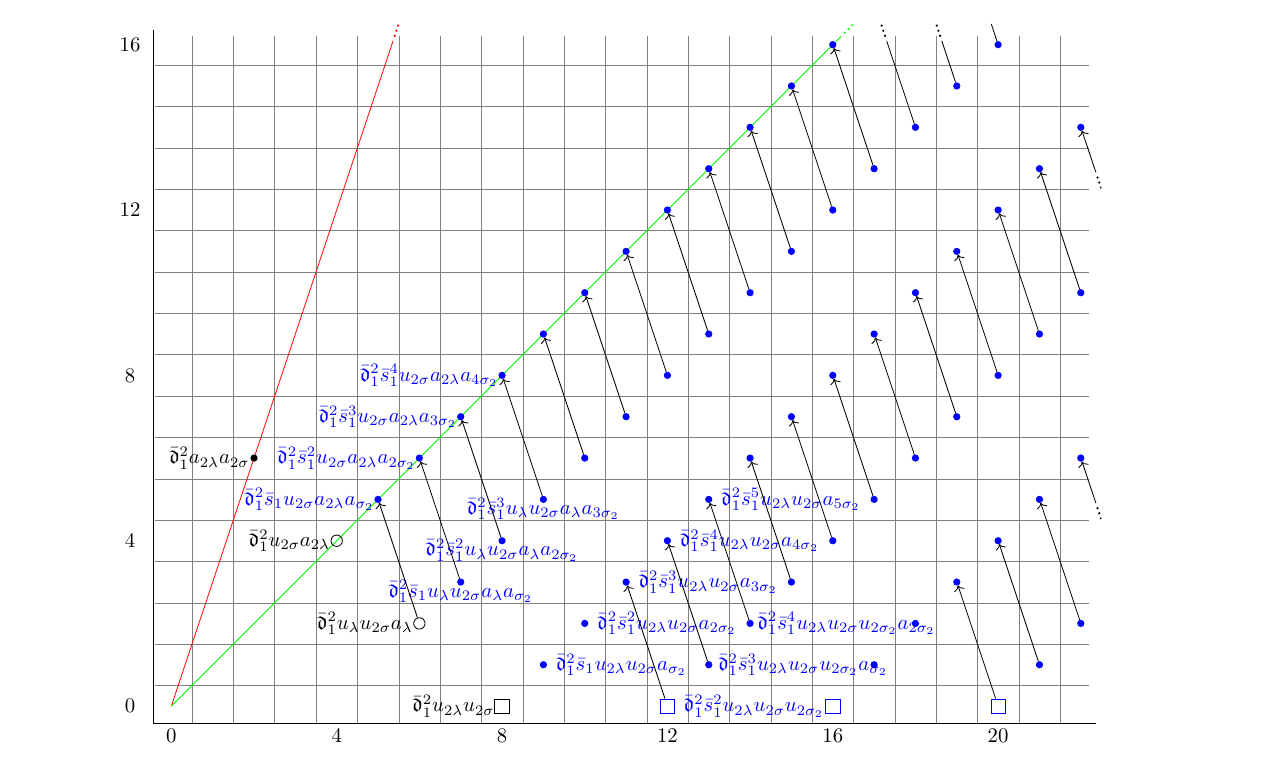}}
\end{center}
\begin{center}
\caption{$d_3$-differentials within the column containing $\done^2$.}
\hfill
\label{fig:E2C4d3diffType1}
\end{center}
\end{figure}

\begin{figure}
\begin{center}
\makebox[\textwidth]{\includegraphics[trim={0cm 11cm 0cm 4cm}, clip, scale = 0.8, page = 2]{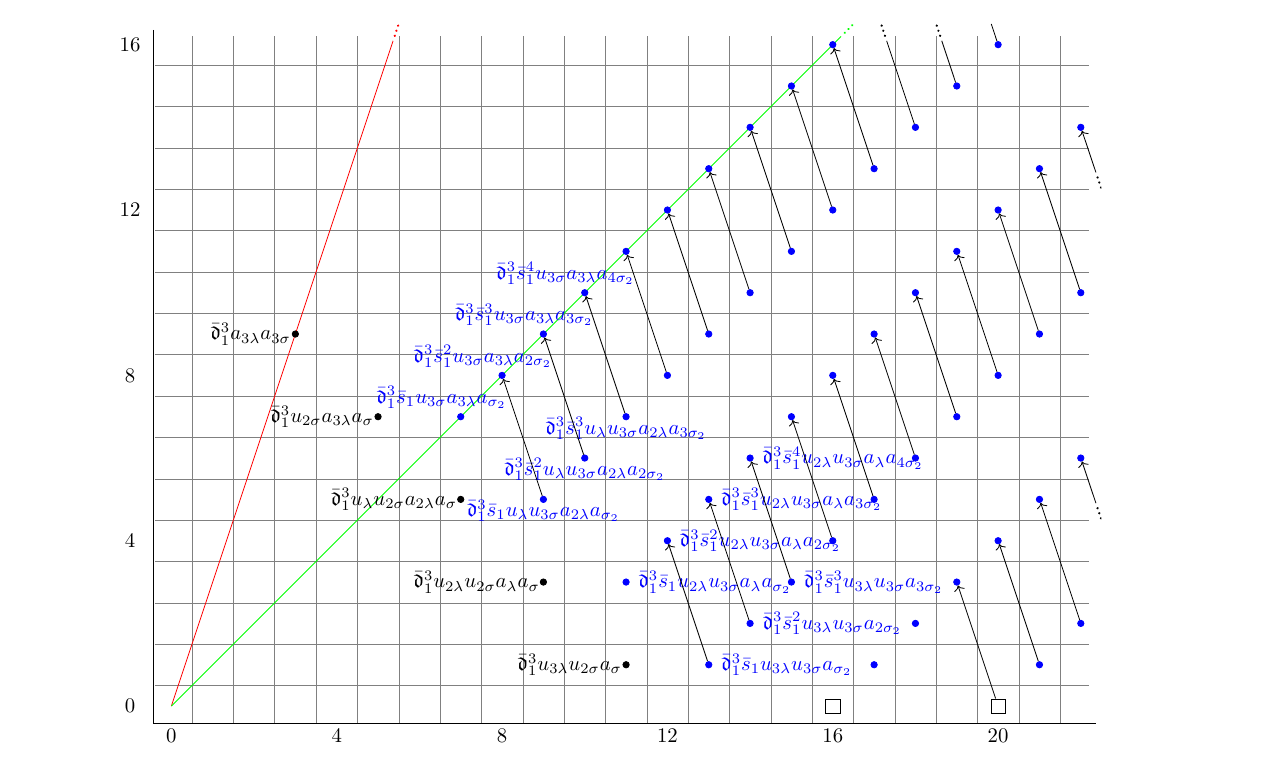}}
\end{center}
\begin{center}
\caption{$d_3$-differentials within the column containing $\done^3$.}
\hfill
\label{fig:E2C4d3diffType2}
\end{center}
\end{figure}

%%%%%%%%%%%%%%%%%%%%%%%
%%%%%%%%%%%%%%%%%%%%%%%
\begin{notation}[Naming the transfer classes]\rm
For classes that are associated with the induced slice cells, we denote
\[\done^m \sone^n u_\lambda^i u_\sigma^j u_{2\sigma_2}^k a_\lambda^r a_{\sigma_2}^\ell :=  tr_V(\rone^{m+n} \grone^m u_{2\sigma_2}^{i+k}a_{\sigma_2}^{2r+\ell}).\]  
Here, $n \equiv \ell \pmod{2}$ and 
\begin{eqnarray*}
V &=& m(1 + \sigma + \lambda) + n \left(1 + \frac{\lambda}{2}\right) + i (2 - \lambda) + j (1 - \sigma) + k (2 - \lambda) - r\lambda - \frac{\ell}{2}\lambda  \\
&=& (m + n + 2i + j + 2k)  + (m - j) \sigma + \left(m + \frac{n-\ell}{2} - i - k - r \right)\lambda.
\end{eqnarray*}
In the notation above, the term $\done^m \sone^n$ indicates the slice cell that the class is associated to.  

This notation has the advantage that it is easy to see all the different possible expressions for the same transfer class (in light of the Frobenius relation ${tr(res(a)\cdot b) = a\cdot tr(b)}$).  More specifically, for any $m'$, $m''$, $i'$, $i''$, $j'$, $j''$ even, $r'$, and $r''$ satisfying the equalities $m' + m'' = m$, $i' + i'' = i$, $j' + j'' = j$, and $r' + r'' = r$, we can rewrite the expression above as 
\begin{eqnarray*}
\done^m \sone^n u_\lambda^i u_\sigma^j u_{2\sigma_2}^k a_\lambda^r a_{\sigma_2}^\ell &=& tr_V(\rone^{m+n} \grone^m u_{2\sigma_2}^{i+k}a_{\sigma_2}^{2r+\ell}) \\
&=& tr_V(res(\done^{m''}u_\lambda^{i''}u_\sigma^{j''}a_\lambda^{r''}) \cdot \rone^{m'+n}\grone^{m'}u_{2\sigma_2}^{i'+k}a_{\sigma_2}^{2r'+\ell}) \\
&=& (\done^{m''}u_\lambda^{i''}u_\sigma^{j''}a_\lambda^{r''}) \cdot tr_{V'}(\rone^{m'+n}\grone^{m'}u_{2\sigma_2}^{i'+k}a_{\sigma_2}^{2r'+\ell}) \\ 
&=& (\done^{m''}u_\lambda^{i''}u_\sigma^{j''}a_\lambda^{r''}) \cdot (\done^{m'}\sone^n u_\lambda^{i'}u_\sigma^{j'}u_{2\sigma_2}^{k}a_\lambda^{r'}a_{\sigma_2}^{\ell}),
\end{eqnarray*}
where
\[V' = (m' + n + 2i' + j' + 2k)  + (m' - j') \sigma + \left(m' + \frac{n-\ell}{2} - i' - k - r' \right)\lambda. \]
For example, consider the class 
\[\done^2 \sone^2 u_{2\lambda}u_{2\sigma}u_{2\sigma_2} = tr_{12}(\rone^4 \grone^2 u_{6\sigma_2})\]
in bidegree $(12,0)$.  There are many ways to rewrite this class, two of which are
\begin{eqnarray*}
\done^2 u_{2\lambda}u_{2\sigma} \cdot \sone^2 u_{2\sigma_2} &=& \done^2 u_{2\lambda}u_{2\sigma} tr_4(\rone^2 u_{2\sigma_2}) \\
&=& tr_{12}(res(\done^2 u_{2\lambda}u_{2\sigma} )\rone^2 u_{2\sigma_2})
\end{eqnarray*}
and
\begin{eqnarray*}
\done u_{2\lambda} \cdot \done \sone^2 u_{2\sigma} u_{2\sigma_2}  &=& \done u_{2\lambda} tr_{7- \sigma+\lambda}(\rone^3\grone u_{2\sigma_2}) \\
&=& tr_{12}(res(\done u_{2\lambda})\rone^3\grone u_{2\sigma_2}).
\end{eqnarray*}
\end{notation}

\begin{exam}[Computing the leading term]\rm
Some classes support $d_3$-differentials with targets the sum of two classes.  For example, consider the class 
\[{\done^2\sone^2u_{2\lambda}u_{2\sigma}u_{2\sigma_2} = tr_{12}(\rone^4 \grone^2 u_{6\sigma_2})}\]
in bidegree $(12,0)$.  In the $C_2$-slice spectral sequence, the class $\rone^4 \grone^2 u_{6\sigma_2}$ supports the $d_3$-differential 
\[d_3(\rone^4 \grone^2 u_{6\sigma_2}) = \rone^4 \grone^2 (\rone + \grone) u_{4\sigma_2}a_{3\sigma_2} = \rone^5 \grone^2 u_{4\sigma_2}a_{3\sigma_2} + \rone^4 \grone^3 u_{4\sigma_2}a_{3\sigma_2}. \]
Applying the transfer to the target shows that the class 
\[tr(\rone^5 \grone^2 u_{4\sigma_2}a_{3\sigma_2}) + tr(\rone^4 \grone^3 u_{4\sigma_2}a_{3\sigma_2}) = \done^2 \sone^3 u_{2\lambda}a_{3\sigma_2} + \done^3 \sone u_{2\lambda}a_{\lambda}a_{\sigma_2}\]
must be killed by a differential of length at most 3.  By naturality, the source of this differential is $tr(\rone^4 \grone^2 u_{6\sigma_2})$, and so 
\[d_3(\done^2\sone^2u_{2\lambda}u_{2\sigma}u_{2\sigma_2}) = \done^2 \sone^3 u_{2\lambda}a_{3\sigma_2} + \done^3 \sone u_{2\lambda}a_{\lambda}a_{\sigma_2}.\]
Note that the first term of the target is in the same column as the source, but the second term is not (it belongs to a slice cell in the next column)

This $d_3$-differential is introducing the relation $\done^2 \sone^3 u_{2\lambda} u_{2\sigma} a_{3\sigma_2} =  \done^3 \sone u_{2\lambda}u_{2\sigma}a_\lambda a_{\sigma_2}$ after the $E_3$-page.  As a convention, when we are drawing the slice spectral sequence, we only kill the leading term of the target:
$$d_3(\done^2\sone^2u_{2\lambda}u_{2\sigma}u_{2\sigma_2}) =  \done^2 \sone^3 u_{2\lambda} u_{2\sigma} a_{3\sigma_2}.$$
\end{exam}

%%%
\subsection{$d_5$-differentials}
\begin{thm}
$d_5(u_{2\sigma}) = \done a_\lambda a_{\sigma}^3$.
\end{thm}
\begin{proof}
This differential is given by Hill--Hopkins--Ravenel's Slice Differentials Theorem \cite[Theorem~9.9]{HHR}. 
\end{proof}
The Slice Differentials Theorem, combined with the fact that $N(\bar{r}_i) = 0$ for $i \geq 2$, implies that this $d_5$-differential produces all the $d_5$-differentials in the region that only have classes corresponding to the homotopy groups of regular slice cells (those of the form $\done^i$).  In particular, the class $u_{4\sigma}$ is a permanent cycle.  In the integer graded slice spectral sequence, this $d_5$-differential produces all the $d_5$-differentials between the line of slope 1 and the line of slope 3.  

\begin{thm}\label{thm:BPoned5}
The class $\done^2u_{2\lambda}u_{2\sigma}$ at $(8, 0)$ supports the $d_5$-differential
$$d_5(\done^2u_{2\lambda}u_{2\sigma}) = \done^3u_{\lambda}u_{2\sigma} a_{2\lambda}a_\sigma.$$
\end{thm}
\begin{proof}
The restriction of $\done^2u_{2\lambda}u_{2\sigma}$ is $res(\done^2u_{2\lambda}u_{2\sigma}) = \bar{r}_1^2\gamma\bar{r}_1^2 u_{4\sigma_2}$, which supports the $d_7$-differential 
$$d_7(\bar{r}_1^2\gamma\bar{r}_1^2 u_{4\sigma_2}) = \bar{r}_1^5 \gamma \bar{r}_1^2 a_{7\sigma_2}$$
in $C_2$-$\SliceSS(\BPone)$.  This implies that the class $\done^2u_{2\lambda}u_{2\sigma}$ must support a differential of length at most 7 in $C_4$-$\SliceSS(\BPone)$.  The only possible targets are the classes $\done^3u_\lambda u_{2\sigma} a_{2\lambda} a_\sigma$ at $(7,5)$ and $\done^3\sone u_{3\sigma} a_{3\lambda} a_{\sigma_2}$ at $(7,7)$ (see Figure~\ref{fig:E2C4d5}). 

\begin{figure}
\begin{center}
\makebox[\textwidth]{\includegraphics[trim={4cm 16cm 4.5cm 4cm}, clip, page = 1, scale = 1]{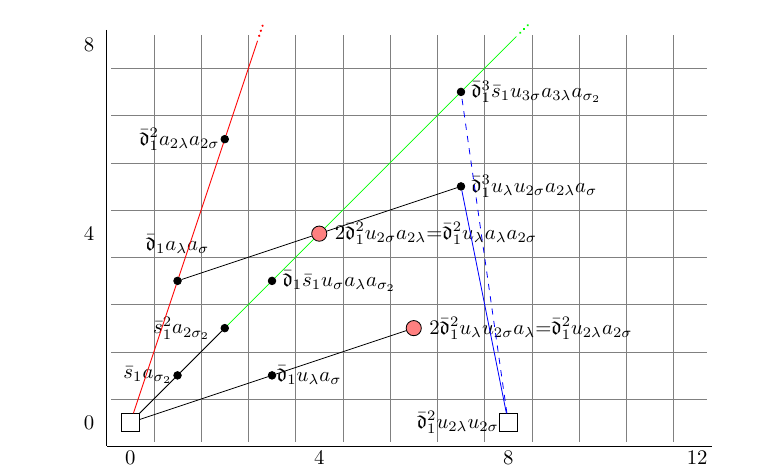}}
\end{center}
\begin{center}
\caption{$d_5$-differential on $\done^2u_{2\lambda}u_{2\sigma}$.}
\hfill
\label{fig:E2C4d5}
\end{center}
\end{figure}

To prove the desired $d_5$-differential, it suffices to show that the $d_7$-differential 
$$d_7(\bar{\mathfrak{d}}_1^2u_{2\lambda} u_{2\sigma}) = \bar{\mathfrak{d}}_1^3 \bar{s}_1 u_{3\sigma} a_{3\lambda} a_{\sigma_2}$$
does not exist.  For the sake of contradiction, suppose that this $d_7$-differential does occur.  By natuality, this differential must be compatible with the restriction map.  The left-hand side restricts to $\bar{r}_1^2 \gamma \bar{r}_1^2 u_{4\sigma_2}$, but the right-hand side restricts to 
$$res(\bar{\mathfrak{d}}_1^3 \bar{s}_1 u_{3\sigma} a_{3\lambda} a_{\sigma_2}) = \bar{r}_1^3 \gamma \bar{r}_1^3(\bar{r}_1+\gamma \bar{r}_1) a_{7\sigma_2} = 0 \neq \bar{r}_1^5 \gamma \bar{r}_1^2 a_{7\sigma_2}$$
because the $d_3$-differential $d_3(u_{2\sigma_2}) = (\bar{r}_1+\gamma \bar{r}_1) a_{3\sigma_2}$ introduced the relation $(\bar{r}_1+\gamma \bar{r}_1) a_{3\sigma_2} = 0$.  This is a contradiction.  
\end{proof}

\begin{cor}\label{cor:BPoned5}
$d_5(u_{2\lambda}) = \done u_\lambda a_{2\lambda}a_\sigma$. 
\end{cor}
\begin{proof}
Using the Leibniz rule, we have
\begin{eqnarray*}
d_5(\done^2u_{2\lambda}u_{2\sigma}) &=& \done^2 u_{2\lambda} d_5(u_{2\sigma}) + \done^2 u_{2\sigma} d_5(u_{2\lambda}) \\ 
&=& \done^2u_{2\lambda} \cdot \done a_\lambda a_{3\sigma} + \done^2 u_{2\sigma} d_5(u_{2\lambda}) \\ 
&=& 0 + \done^2 u_{2\sigma} d_5(u_{2\lambda}) \\ 
&=& \done^2 u_{2\sigma} d_5(u_{2\lambda}),
\end{eqnarray*}
where we have used the gold relation $u_{\lambda} a_{3\sigma} = 2u_{2\sigma} a_{\lambda} a_\sigma = 0$.  Theorem~\ref{thm:BPoned5} implies that $\done^2 u_{2\sigma} d_5(u_{2\lambda}) = \done^3u_{\lambda}u_{2\sigma} a_{2\lambda} a_\sigma$.  Rearranging, we obtain the equality 
$$\done^2u_{2\sigma}(d_5(u_{2\lambda}) - \done u_{\lambda} a_{2\lambda} a_\sigma) = 0,$$
from which the desired differential follows (multiplication by $\done^2 u_{2\sigma}$ is faithful on the $E_5$-page).
\end{proof}

All the other $d_5$-differentials are obtained from Theorem~\ref{thm:BPoned5} via multiplication with the classes 
\begin{enumerate}
\item $\done a_\lambda a_\sigma$ at $(1,3)$ (permanent cycle); 
\item $\done u_\lambda a_\sigma$ at $(3,1)$ (permanent cycle); 
\item $\done^2 u_{2\sigma}a_{2\lambda}$ at $(4,4)$ ($d_5(\done^2u_{2\sigma} a_{2\lambda}) = \done^3 a_{3\lambda}a_{3\sigma_2}$); 
\item $\done^4 u_{4\lambda}u_{4\sigma}$ at $(16, 0)$ ($d_5$-cycle).  
\end{enumerate}
and using the Leibniz rule (see Figure~\ref{fig:E2C4topd5}). 

\begin{figure}
\begin{center}
\makebox[\textwidth]{\includegraphics[trim={0cm 5cm 0cm 3cm}, clip, page = 2, scale = 0.9]{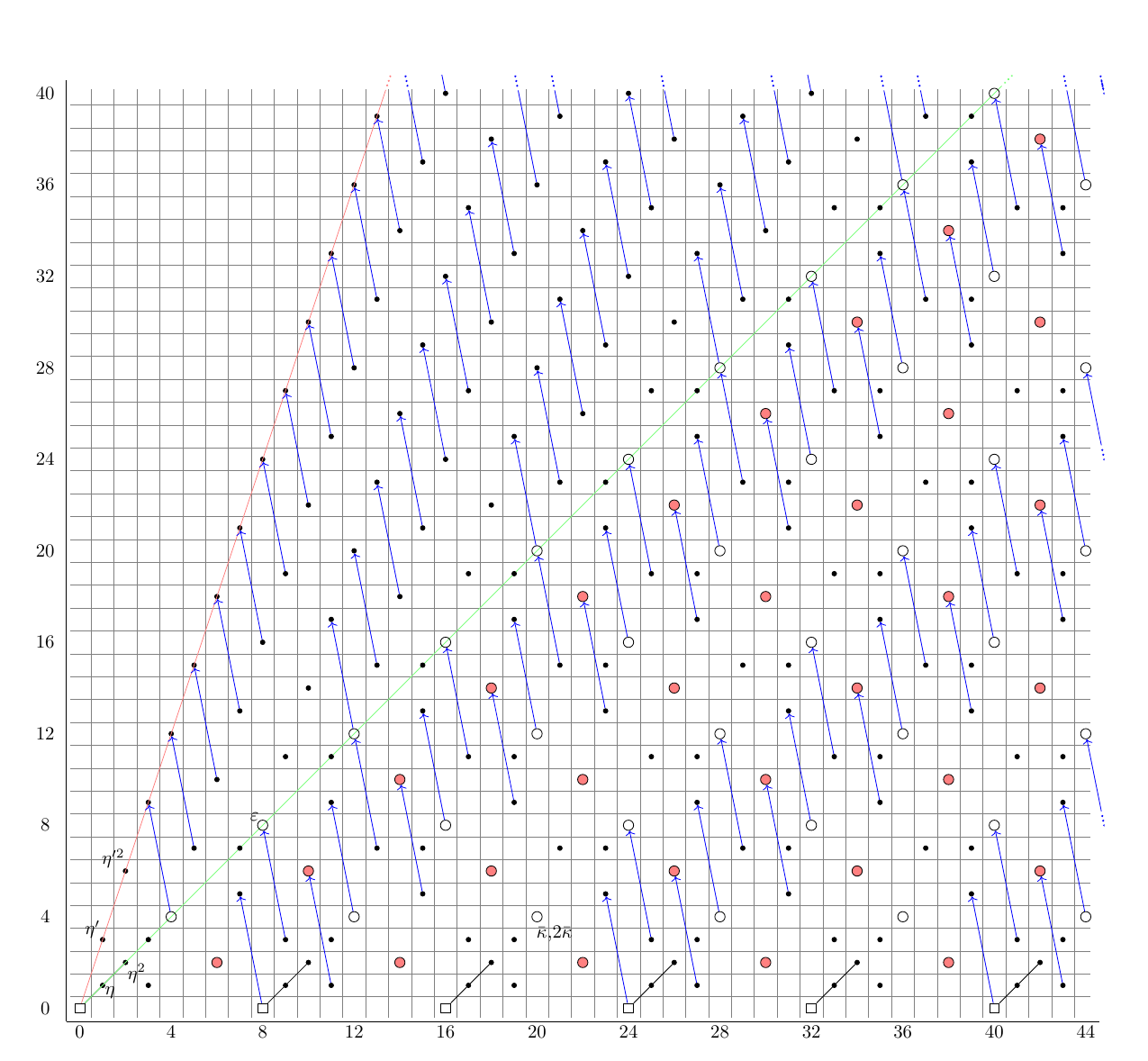}}
\end{center}
\begin{center}
\caption{$d_5$-differentials in $C_4$-$\SliceSS(\BPone)$.}
\hfill
\label{fig:E2C4topd5}
\end{center}
\end{figure}

\vspace{0.1in}

%The first formula of Theorem~\ref{thm:NormFormula} predicts the $d_3$-differential 
%$$d_3\left(\frac{u_{2\lambda}}{u_{2\sigma}}\right) = tr(u_{2\sigma_2} \cdot (\bar{r}_1 + \gamma \bar{r}_1) a_{3\sigma_2}).$$
%However, this prediction is void because the right-hand side is equal to 0:
%$$tr(res(u_\lambda a_\lambda) (\rone + \grone)a_{\sigma_2}) = u_\lambda a_\lambda tr(res(\sone a_{\sigma_2})) = u_\lambda a_\lambda \cdot 2 \sone a_{\sigma_2} = 0. $$
%This is due to the fact that $d_3(u_\lambda) \neq 0$, and so $u_{2\lambda}$ is a $d_3$-cycle.
There is an alternative way to prove Corollary~\ref{cor:BPoned5} by using the norm formula.  Start with the $d_3$-differential 
$$d_3(u_{2\sigma_2}) = (\bar{r}_1 + \gamma \bar{r}_1) a_{3\sigma_2}$$
in the $C_2$-slice spectral sequence.  Since $N(u_{2\sigma_2}) = u_{2\lambda}/u_{2\sigma}$ by \cite[Lemma~4.9]{HHRKH}, Theorem~\ref{thm:NormFormula} predicts the $d_5$-differential
$$d_5\left(a_\sigma \cdot \frac{u_{2\lambda}}{u_{2\sigma}}\right) = N_{C_2}^{C_4}(\bar{r}_1 + \gamma \bar{r}_1) a_{3\lambda}$$
in the $C_4$-slice spectral sequence.  Using the Leibniz rule, this formula predicts the $d_5$-differential 
\begin{eqnarray*} 
d_5 (a_\sigma u_{2\lambda}) = d_5\left(u_{2\sigma} \left(a_\sigma\cdot \frac{u_{2\lambda}}{u_{2\sigma}}\right)\right) &=& d_5(u_{2\sigma}) \left( a_\sigma \cdot\frac{ u_{2\lambda}}{u_{2\sigma}} \right)+ u_{2\sigma} d_5\left(a_\sigma \cdot\frac{u_{2\lambda}}{u_{2\sigma}}\right) \\ 
&=& (\bar{\mathfrak{d}}_1 a_\lambda a_{3\sigma})\left( a_\sigma \cdot \frac{u_{2\lambda}}{u_{2\sigma}} \right) + N_{C_2}^{C_4}(\bar{r}_1 + \gamma \bar{r}_1) u_{2\sigma} a_{3\lambda}\\
&=&  0 + N_{C_2}^{C_4}(\bar{r}_1 + \gamma \bar{r}_1) u_{2\sigma} a_{3\lambda}\\
&=& N_{C_2}^{C_4}(\bar{r}_1 + \gamma \bar{r}_1) u_{2\sigma} a_{3\lambda}.
\end{eqnarray*}

To compute $N_{C_2}^{C_4}(\bar{r}_1 + \gamma \bar{r}_1)$, note that
$$res(N_{C_2}^{C_4}(\bar{r}_1 + \gamma \bar{r}_1)) = (\bar{r}_1 + \gamma \bar{r}_1)(\gamma \bar{r}_1 - \bar{r}_1) = -(\bar{r}_1^2 - \gamma \bar{r}_1^2) = - (1 + \gamma) \bar{r}_1^2 u_{\sigma}^{-1}.$$
Therefore, $N_{C_2}^{C_4}(\bar{r}_1 + \gamma \bar{r}_1) = -tr(\bar{r}_1^2 u_{\sigma}^{-1})$.  Note that $u_\sigma^{-1}$ is used as a placeholder here to indicate that the target of the transfer is in degree $(2 + \lambda)- (1 - \sigma) = 1 + \sigma + \lambda$ (see discussion after Definition~\ref{def:defininguV}).  The target of the normed $d_5$-differential is 
$$-tr(\bar{r}_1^2 u_{-\sigma}) u_{2\sigma}a_{3\lambda} = -tr(\bar{r}_1^2u_\sigma a_{6\sigma_2}) = tr(\bar{r}_1^2u_\sigma a_{6\sigma_2}).$$
The last equality holds because multiplication by 2 kills transfer of classes with filtration at least 1.  

To identity this target with a more familiar expression, we add $tr(\rone \grone u_\sigma a_{6\sigma_2})$ to it and use the Frobenius relation: 
\begin{eqnarray*}
tr(\bar{r}_1^2u_\sigma a_{6\sigma_2}) + tr(\bar{r}_1 \grone u_\sigma a_{6\sigma_2}) &=& tr(\bar{r}_1(\bar{r}_1 + \grone)u_\sigma a_{6\sigma_2}) \\
&=& tr(\bar{r}_1u_\sigma a_{\sigma_2} res(tr(\bar{r}_1a_{\sigma_2})a_{2\lambda}) ) \\ 
&=& tr(\bar{r}_1 u_\sigma a_{\sigma_2}) tr(\bar{r}_1 a_{\sigma_2}) a_{2\lambda} \\
&=& 0. 
\end{eqnarray*}
The last expression is 0 on the $E_5$-page because $d_3(u_{\lambda}) = \sone a_\lambda a_{\sigma_2} = tr(\bar{r}_1 a_{\sigma_2}) a_\lambda$.  Therefore,
\begin{eqnarray*}
tr(\bar{r}_1^2u_\sigma a_{6\sigma_2}) &=& tr(\bar{r}_1 \grone u_\sigma a_{6\sigma_2}) \\ 
&=& tr(res(\done u_{2\sigma}a_{3\lambda}))  \, \, \, \, \, \, \,(res(\bar{\mathfrak{d}}_1) = \bar{r}_1 \gamma \bar{r}_1 u_\sigma^{-1})\\ 
&=& 2 \done u_{2\sigma} a_{3\lambda} \\ 
&=& \done u_{\lambda}a_{2\lambda} a_{2\sigma} \, \, \, \, \, \, \, (2u_{2\sigma}a_\lambda = u_{\lambda}a_{2\sigma}) .
\end{eqnarray*}
It follows that 
$$a_\sigma (d_5(u_{2\lambda}) - \done u_{\lambda}a_{2\lambda} a_{\sigma})= 0,$$
and $d_5(u_{2\lambda}) = \done u_{\lambda}a_{2\lambda} a_{\sigma}$ (multiplication by $a_\sigma$ is injective in this bidegree), as desired.

%%%
\subsection{$d_7$-differentials}

\begin{thm}\label{thm:BPoned7case1}
The classes $2\done^2u_{2\lambda}u_{2\sigma}$ at $(8, 0)$ and $2\done^4 u_{2\lambda} u_{4\sigma} a_{2\lambda}$ at (12, 4) support the $d_7$-differentials
\begin{eqnarray*}
d_7(2\done^2u_{2\lambda}u_{2\sigma}) &=& \done^3 \bar{s}_1 u_{3\sigma}a_{3\lambda}a_{\sigma_2},\\ 
d_7(2\done^4 u_{2\lambda} u_{4\sigma} a_{2\lambda}) &=& \done^5 \sone u_{5\sigma} a_{5\lambda}a_{\sigma_2}.
\end{eqnarray*}
\end{thm}
\begin{proof}
Consider the $d_7$-differential 
$$d_7(\rone^2 \grone^2 u_{4\sigma_2}) = \rone^5 \grone^2 a_{7\sigma_2} = \rone^4 \grone^3 a_{7\sigma_2}$$
in the $C_2$-slice spectral sequence (the last equality holds because $\rone a_{3\sigma_2} = \grone a_{3\sigma_2}$ after the $d_3$-differentials).  The transfer of the target is 
$$tr(\rone^4\grone^3a_{7\sigma_2}) = \done^3 \sone u_{3\sigma}a_{3\lambda} a_{\sigma_2}.$$ 
For degree reasons, this class must be killed by a differential of length exactly 7 (see Figure~\ref{fig:E2C4d7diagram1}).  Natuality implies that the source is
$$tr(\rone^2 \grone^2 u_{4\sigma_2}) = tr(res(\done^2u_{2\sigma}u_{2\lambda})) = 2\done^2 u_{2\sigma}u_{2\lambda}.$$

The second differential is proved using the same method, by applying the transfer to the $d_7$-differential 
$$d_7(\rone^4 \grone^4 u_{4\sigma_2} a_{4\sigma_2}) = \rone^7\grone^4 a_{11\sigma_2} = \rone^6 \grone^5 a_{11\sigma_2}.$$

\begin{figure}
\begin{center}
\makebox[\textwidth]{\includegraphics[trim={4cm 16cm 4.5cm 4cm}, clip, page = 2, scale = 1]{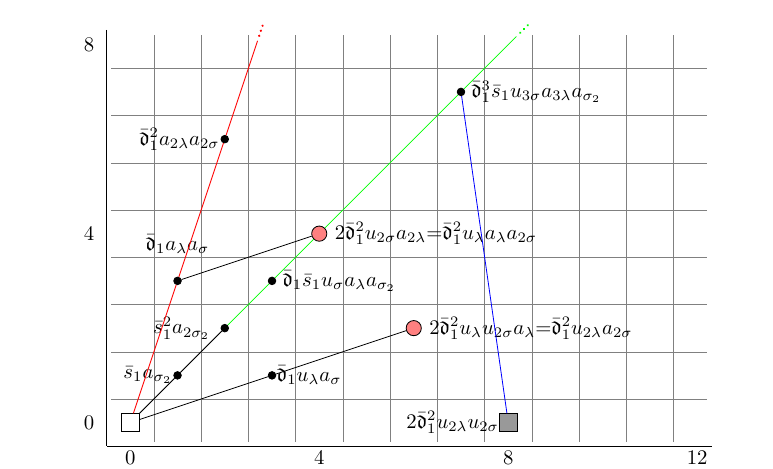}}
\end{center}
\begin{center}
\caption{$d_7$-differential on $2\done^2u_{2\lambda}u_{2\sigma}$.}
\hfill
\label{fig:E2C4d7diagram1}
\end{center}
\end{figure}
\end{proof}

\begin{cor}\label{cor:BPoned7case1}
The classes $2u_{2\lambda}$ and $2u_{2\lambda} u_{2\sigma}$ support the $d_7$-differentials
\begin{eqnarray*}
d_7(2u_{2\lambda}) &=& \done \sone u_\sigma a_{3\lambda} a_{\sigma_2}, \\ 
d_7(2u_{2\lambda} u_{2\sigma}) &=& \done \sone u_{3\sigma} a_{3\lambda} a_{\sigma_2}.
\end{eqnarray*}
\end{cor}
\begin{proof}
Since the classes $\done$, $u_{4\sigma}$, and $a_{\lambda}$ are permanent cycles, the second differential in Theorem~\ref{thm:BPoned7case1} can be rewritten as 
$$\done^4u_{4\sigma}a_{2\lambda}(d_7(2u_{2\lambda}) - \done \sone u_\sigma a_{3\lambda} a_{\sigma_2}) = 0,$$
from which the first differential follows.  The second differential is proven similarly by using the first differential in Theorem~\ref{thm:BPoned7case1}. \end{proof}

\begin{rem}\rm
Corollary~\ref{cor:BPoned7case1} can also be proved by applying the transfer to the $d_7$-differential 
$$d_7(u_{4\sigma_2}) = \bar{r}_1^3 a_{7\sigma_2} = \bar{r}_1^2 \gamma \bar{r}_1 a_{7\sigma_2}$$
in the $C_2$-slice spectral sequence.  
\end{rem}

\begin{rem}\rm
On the $E_7$-page of $C_4$-$\SliceSS(\BPone)$, there is more than one class at $(8, 0)$.  They are
\begin{enumerate}
\item $2\bar{\mathfrak{d}}_1^2u_{2\lambda} u_{2\sigma} = tr(res(\bar{\mathfrak{d}}_1^2u_{2\lambda} u_{2\sigma})) = tr(\bar{r}_1^2 \gamma \bar{r}_1^2 u_{4\sigma_2})$;
\item $\bar{\mathfrak{d}}_1 \bar{s}_1^2 u_\lambda u_\sigma u_{2\sigma_2} = tr(\bar{r}_1^3 \gamma \bar{r}_1 u_{4\sigma_2})$; 
\item $\bar{s}_1^4 u_{4\sigma_2} = tr(\bar{r}_1^4 u_{4\sigma_2})$.  
\end{enumerate}
Except for class (1), the classes (2) and (3) are ``grayed out'' on the upper-left of Hill, Hopkins, and Ravenel's original computation of $C_4$-$\SliceSS(D_1^{-1}\BPone)$~\cite[pg. 4]{HHRKH}.

On the $E_2$-page, there is more than one class at $(7, 7)$ as well:
\begin{enumerate}
\item $\bar{\mathfrak{d}}_1^3 \bar{s}_1 u_{3\sigma} a_{3\lambda} a_{\sigma_2} = tr(\bar{r}_1^4 \gamma \bar{r}_1^3 a_{7\sigma_2})$;
\item $\bar{\mathfrak{d}}_1^2 \bar{s}_1^3 u_{2\sigma} a_{2\lambda} a_{3\sigma_2} = tr(\bar{r}_1^5 \gamma \bar{r}_1^2 a_{7\sigma_2})$; 
\item $\bar{\mathfrak{d}}_1 \bar{s}_1^5 u_{\sigma} a_\lambda a_{5\sigma_2} = tr(\bar{r}_1^6 \gamma \bar{r}_1 a_{7\sigma_2})$; 
\item $\bar{s}_1^7 a_{7\sigma_2} = tr(\bar{r}_1^7a_{7\sigma_2})$.  
\end{enumerate}
Applying transfers to the following $d_3$-differentials in $C_2$-$\SliceSS(\BPone)$ yields $d_3$-differentials in $C_4$-$\SliceSS(\BPone)$: 
\begin{enumerate}
\item $d_3(\bar{r}_1^6 u_{2\sigma_2} a_{4\sigma_2}) = \bar{r}_1^6 (\bar{r}_1 + \gamma \bar{r}_1) a_{\sigma_2}^7$: transfer of this kills (3) + (4);  
\item $d_3(\bar{r}_1^5 \gamma \bar{r}_1 u_{2\sigma_2} a_{4\sigma_2}) = \bar{r}_1^5 \gamma \bar{r}_1 (\bar{r}_1 + \gamma \bar{r}_1) a_{7\sigma_2}$: transfer of this kills (2) + (3);  
\item $d_3(\bar{r}_1^4 \gamma \bar{r}_1^2 u_{2\sigma_2} a_{4\sigma_2}) = \bar{r}_1^4 \gamma \bar{r}_1^2  (\bar{r}_1 + \gamma \bar{r}_1) a_{7\sigma_2}$: transfer of this kills (1) + (2).  
\end{enumerate}
These $d_3$-differentials identify the four classes at $(7,7)$.  The transfer argument in Theorem~\ref{thm:BPoned7case1} shows that each of the three classes at $(8, 0)$ supports a $d_7$-differential, all killing the single remaining class at $(7, 7)$.  
\end{rem}

The proof of Hill--Hopkins--Ravenel's Periodicity theorem \cite[Section 9]{HHR} shows that the class $\done^8u_{8\lambda} u_{8\sigma}$ at $(32, 0)$ is a permanent cycle.  For degree reasons, the following classes are also permanent cycles and survive to the $E_\infty$-page: 
\begin{enumerate}
\item $\eta = \bar{s}_1a_{\sigma_2}$ at $(1, 1)$; 
\item $\eta^2 = \bar{s}_1^2a_{2\sigma_2}$ at $(2, 2)$; 
\item $\eta' = \done a_\lambda a_\sigma$ at $(1, 3)$;
\item $\eta'^2 = \done^2 a_{2\lambda} a_{2\sigma}$ at $(2,6)$;
\item $\epsilon= \done^4 u_{4\sigma}a_{4\lambda}$ at $(8, 8)$. 
\end{enumerate}
Their names come from the spherical classes that they detect in $\pi_*S^0$ \cite[Theorem~9.8]{HHRKH} under the Hurewicz map 
$$\pi_* S^0 \longrightarrow \pi_* \BPone^{C_4}.$$

\begin{thm}\label{thm:BPoned13}
The class $u_{4\lambda} a_\sigma$ supports the $d_{13}$-differential 
$$d_{13}(u_{4\lambda}a_\sigma) = \bar{\mathfrak{d}}_1^3 u_{4\sigma}a_\lambda^7.$$
\end{thm}
\begin{proof}
Applying the norm formula of Theorem~\ref{thm:NormFormula} to the $d_7$-differential 
$$d_7(u_{4\sigma_2}) = \rone^3 a_{7\sigma_2}$$
in the $C_2$-slice spectral sequence predicts the $d_{13}$-differential 
\begin{eqnarray*}
d_{13}(u_{4\lambda}a_\sigma) &=& u_{4\sigma} N_{C_2}^{C_4}(\rone^3 a_{7\sigma_2}) \\
&=& \done^3 u_{4\sigma} a_{7\lambda} 
\end{eqnarray*}
in the $C_4$-slice spectral sequence.  The target is not zero on the $E_{13}$-page because multiplying it by the permanent cycle $\done^5 u_{4\sigma} a_{\lambda}$ gives the nonzero class $\done^8 u_{8\sigma}a_{8\lambda}$ at $(16, 16)$.  Therefore, this $d_{13}$-differential exists.  
\end{proof}

Multiplying the differential in Theorem~\ref{thm:BPoned13} by the permanent cycle $\done^5 u_{4\sigma} a_{\lambda}$ produces a $d_{13}$-differential in the integer graded spectral sequence. 
\begin{cor}\label{cor:BPoned13}
The class $\done^5u_{4\lambda} u_{4\sigma} a_\lambda a_\sigma$ at $(17, 3)$ supports the $d_{13}$-differential 
$$d_{13}(\done^5u_{4\lambda} u_{4\sigma} a_\lambda a_\sigma) =  \done^8 u_{8\sigma}a_{8\lambda} = \epsilon^2. $$
\end{cor}

\begin{thm}\label{thm:BPoned7case2}
The class $\done^4 u_{4\lambda} u_{4\sigma}$ at $(16, 0)$ supports the $d_7$-differential 
$$d_7(\done^4 u_{4\lambda} u_{4\sigma}) =   \done^5 \sone u_{2\lambda} u_{5\sigma} a_{3\lambda} a_{\sigma_2}.$$
\end{thm}
\begin{proof}
The class $\eta' = \done a_\lambda a_\sigma$ is a permanent cycle.  By Corollary~\ref{cor:BPoned13}, the class $\done^4 u_{4\lambda} u_{4\sigma}$ at $(16, 0)$ must support a differential of length at most 13.  For degree reasons, the only possible target is $\done^5 \sone u_{2\lambda} u_{5\sigma} a_{3\lambda} a_{\sigma_2}$.
\end{proof}

\begin{cor}\label{cor:BPoned7case2}
The class $u_{4\lambda}$ supports the $d_7$-differential
$$d_7(u_{4\lambda}) = \done \sone u_{2\lambda} u_\sigma a_{3\lambda} a_{\sigma_2}.$$
\end{cor}
\begin{proof}
This follows directly from Theorem~\ref{thm:BPoned7case2} because 
$$\done^4 u_{4\sigma}(d_7 (u_{4\lambda}) -\done \sone u_{2\lambda} u_\sigma a_{3\lambda} a_{\sigma_2}) = 0$$
and multiplication by $\done^4 u_{4\sigma}$ is faithful on the $E_7$-page. 
\begin{comment} Consider the $d_7$-differential 
$$d_7(u_{4\sigma_2}) = \rone^3 a_{7\sigma_2} = \rone^2 \grone a_{7\sigma_2}$$
in the $C_2$-slice spectral sequence.  The first norm formula in Theorem~\ref{thm:NormFormula} predicts the $d_7$-differential 
\begin{eqnarray*}
d_7(u_{4\lambda}) &=& u_{4\sigma} \cdot tr(u_{4\sigma_2} \cdot \rone^2 \grone a_{7\sigma_2} u_{\sigma}^{-4}) \\
&=& tr(\rone^2 \grone u_{4\sigma_2} a_{7\sigma_2}) \\
&=& tr(res(\done u_\sigma u_{2\lambda} a_{3\lambda}) \rone a_{\sigma_2}) \\ 
&=& \done u_{2\lambda} u_\sigma a_{3\lambda} \cdot tr(\rone a_{\sigma_2}) \\ 
&=& \done \sone u_{2\lambda} u_\sigma a_{3\lambda} a_{\sigma_2}
\end{eqnarray*}
in the $C_4$-slice spectral sequence.  

The target is not zero on the $E_7$-page because multiplying it by the permanent cycle $\done^4 u_{4\sigma}$ gives the nonzero class $\done^5 \sone u_{2\lambda} u_{5\sigma} a_{3\lambda} a_{\sigma_2}$ at $(15, 7)$.  Therefore, this $d_7$-differential exists.
\end{comment}
\end{proof}

Once we have proven the $d_7$-differentials in Theorem~\ref{thm:BPoned7case1} and Theorem~\ref{thm:BPoned7case2}, all the other $d_7$-differentials are obtained via multiplication with the classes 
\begin{enumerate}
\item $\done^4u_{4\sigma}a_{4\lambda}$ at $(8, 8)$ (permanent cycle); 
\item $\done^4 u_{4\lambda} u_{4\sigma}$ at $(16, 0)$ (Theorem~\ref{thm:BPoned7case2});
\item $\done^8 u_{8\lambda}u_{8\sigma}$ at $(32, 0)$ ($d_7$-cycle).
\end{enumerate}
and using the Leibniz rule (see Figure~\ref{fig:E2C4topd7}). 

\begin{figure}
\begin{center}
\makebox[\textwidth]{\includegraphics[trim={0cm 5cm 0cm 3cm}, clip, page = 3, scale = 0.9]{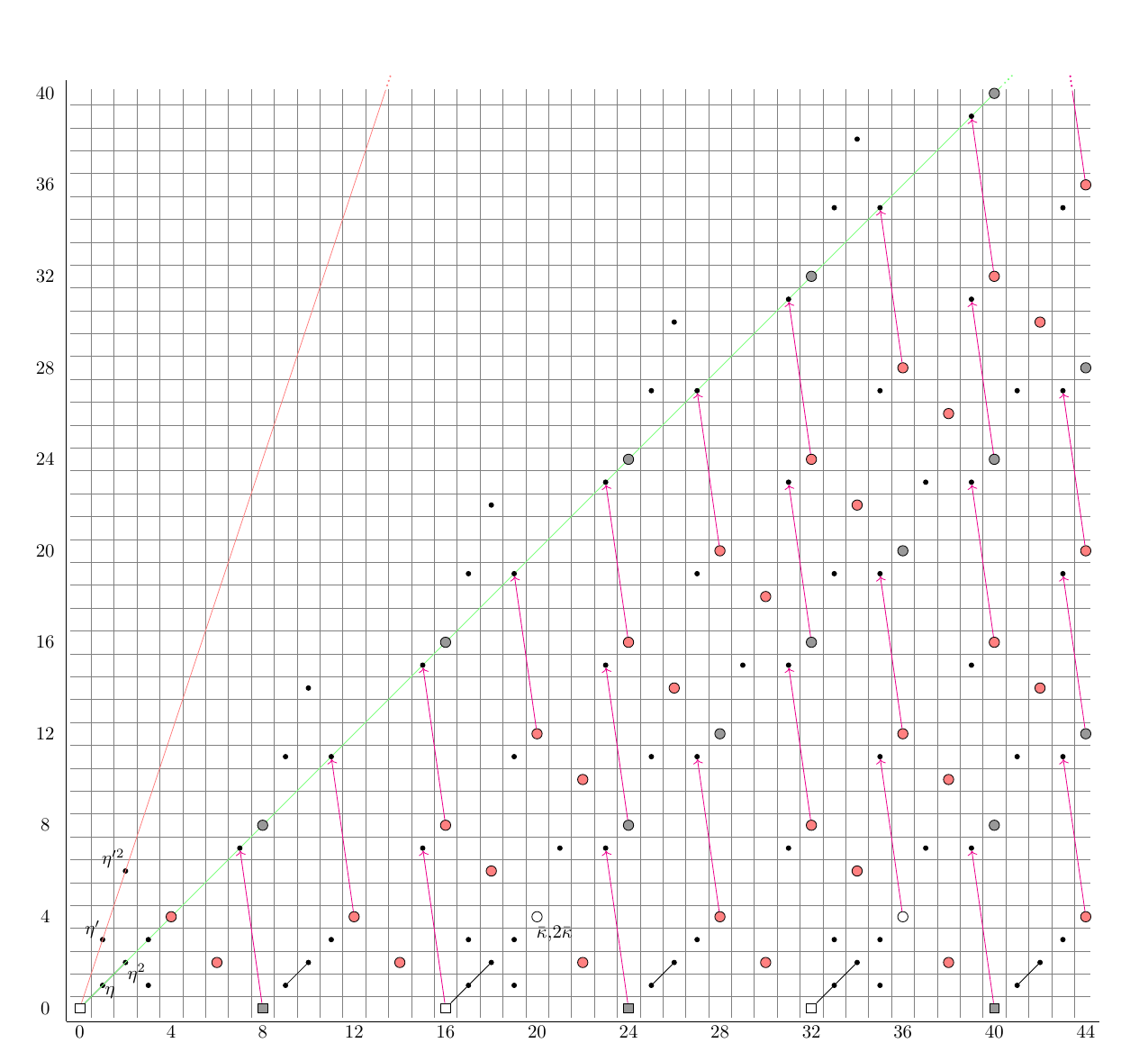}}
\end{center}
\begin{center}
\caption{$d_7$-differentials in $C_4$-$\SliceSS(\BPone)$.}
\hfill
\label{fig:E2C4topd7}
\end{center}
\end{figure}

%%%
\subsection{Higher differentials}
\begin{fact}\rm \label{fact:epsilon}
Multiplication by $\epsilon = \done^4u_{4\sigma}a_{4\lambda}$ is injective on the $E_2$-page.  The image of this multiplication map is the region defined by the inequalities  
\begin{eqnarray*}
s &\geq& 8,\\
3(s-8) &\leq& t-s -8. 
\end{eqnarray*}
In other words, this region consists of classes with filtrations at least 8 and these classes are all on or below the ray of slope 3, starting at $(8, 8)$.  Starting from the $E_5$-page, all the classes in this region are divisible by $\epsilon$.  Therefore, when $r \geq 5$, multiplication by $\epsilon$ induces a surjective map from the whole $E_r$-page to this region.
\end{fact}

\begin{lem} \label{lem:BPoneMainLem}
Let $d_r(x) = y$ be a nontrivial differential in $C_4$-$\SliceSS(\BPone)$.
\begin{enumerate}
\item The classes $\epsilon x$ and $\epsilon y$ are both nonzero on the $E_r$-page, and $d_r(\epsilon x ) = \epsilon y$.
\item If both $x$ and $y$ are divisible by $\epsilon$ on the $E_2$-page, then $x/\epsilon$ and $y/\epsilon$ both survive to the $E_r$-page, and $d_r(x/\epsilon) = y/\epsilon$.  
\end{enumerate}
\end{lem}
\begin{proof}
We will prove both statements by using induction on $r$, the length of the differential.  Both claims are true in the base case when $r=3$. 

Now suppose that both statements hold for all differentials with length $k<r$.  Given a nontrivial differential $d_r(x) = y$, we will first show that $\epsilon y$ survives to the $E_r$-page.  

If $\epsilon y$ supports a differential, then $y$ must support a differential as well.  This is a contradiction because $y$ is the target of a differential.  Therefore if $\epsilon y$ does not survive to the $E_r$-page, it must be killed by a differential $d_k(z) = \epsilon y$ where $k <r$.  By Fact~\ref{fact:epsilon}, $z$ is divisible by $\epsilon$.  The inductive hypothesis, applied to the differential $d_k(z) = \epsilon y$, shows that $d_k(z/\epsilon) = y$.  This is a contradiction because $d_r(x) = y$ is a nontrivial $d_r$-differential.  Therefore, $\epsilon y$ survives to the $E_r$-page.  

$$\begin{tikzcd}
& \epsilon y &\\
y \ar[ru, dashed, dash, "\cdot \epsilon"] && \\
&& \\ 
&&z \ar[luuu, swap, "d_k"]\\ 
&z/\epsilon \ar[ru, dashed, dash, "\cdot \epsilon"]\ar[luuu, swap, "d_k"]& \epsilon x \\ 
& x \ar[luuuu, "d_r"] \ar[ru, dashed, dash, "\cdot \epsilon"]& 
\end{tikzcd}$$

If $\epsilon x$ does not survive to the $E_r$-page, then it must be killed by a shorter differential as well.  This shorter differential introduces the relation $\epsilon x = 0$ on the $E_r$-page.  However, the Leibniz rule, applied to the differential $d_r(x) = y$, shows that 
$$d_r(\epsilon x) = \epsilon y \neq 0$$
on the $E_r$-page.  This is a contradiction.  It follows that $\epsilon x$ survives to the $E_r$-page as well, and it supports the differential 
$$d_r(\epsilon x) = \epsilon y.$$ 
This proves (1). 

To prove (2), note that if $y/\epsilon$ supports a differential of length smaller than $r$, then the induction hypothesis would imply that $y$ also supports a differential of the same length.  Similarly, if $y/\epsilon$ is killed by a differential of length smaller than $r$, then the induction hypothesis would imply that $y$ is also killed a by a differential of the same length.  Both scenarios lead to contradictions.  Therefore, $y/\epsilon$ survives to the $E_r$-page.  

We will now show that $x/\epsilon$ survives to the $E_r$-page as well.  Since $x$ supports a $d_r$-differential, $x/\epsilon$ must also support a differential of length at most $r$.  Suppose that $d_k (x/\epsilon) = z$, where $k < r$.  The induction hypothesis, applied to this $d_k$-differential, implies the existence of the differential $d_k(x) = \epsilon z$.  This is a contradiction.  
$$\begin{tikzcd}
&y& \\
y/\epsilon \ar[ru, dashed, dash, "\cdot \epsilon"]&\epsilon z&  \\ 
z\ar[ru, dashed, dash, "\cdot \epsilon"]&& \\
&&\\
&& x \ar[luuu, "d_k"]\ar[luuuu, swap, "d_r"]\\
&x/\epsilon \ar[luuu, "d_k"]\ar[ru, dashed, dash, "\cdot \epsilon"]&
\end{tikzcd}$$

It follows that $x/\epsilon$ survives to the $E_r$-page, and it supports a nontrivial $d_r$-differential.  Since $y/\epsilon$ also survives to the $E_r$-page, the Leibniz rule shows that 
$$d_r(x/\epsilon) = y/\epsilon,$$
as desired. 
\end{proof}

\begin{thm}\label{thm:BPoneMainThm}
Any class $x = \epsilon^2 a$ on the $E_2$-page of $C_4$-$\SliceSS(\BPone)$ must die on or before the $E_{13}$-page.   
\end{thm}
\begin{proof}
If the class $a$ is a $d_{13}$-cycle, then $x$ is a $d_{13}$-cycle as well.  Since $\epsilon^2$ is killed by a $d_{13}$-differential by Corollary~\ref{cor:BPoned13}, $\epsilon^2 a$ must be killed by a differential of length at most 13.  

Now suppose that the class $a$ is not a $d_{13}$-cycle and it supports the differential $d_r(a) = b$, where $r \leq 13$.  Applying Lemma~\ref{lem:BPoneMainLem}, we deduce that the class $x = \epsilon^2a$ must support the nontrivial $d_r$-differential 
$$d_r(\epsilon^2 a) = \epsilon^2 b,$$
and therefore  cannot survive to the $E_{13}$-page.  
\end{proof}

\begin{thm} \label{thm:BPoned11}
The class $\done^7 \bar{s}_1 u_{2\lambda} u_{7\sigma} a_{5\lambda} a_{\sigma_2}$ at $(19, 11)$ supports the $d_{11}$-differential 
$$d_{11}(\done^7 \bar{s}_1 u_{2\lambda} u_{7\sigma} a_{5\lambda} a_{\sigma_2}) =  \done^{10} u_{8\sigma} a_{10\lambda}a_{2\sigma}.$$
\end{thm}
\begin{proof}
The class $\done^{10} u_{8\sigma} a_{10\lambda}a_{2\sigma}$ at $(18, 22)$ is equal to 
$$\done^{10} u_{8\sigma} a_{10\lambda}a_{2\sigma} = \epsilon^2 (\done^2 a_{2\lambda}a_{2\sigma}). $$
By Theorem~\ref{thm:BPoneMainThm}, this class must die on or before the $E_{13}$-page.  For degree reasons, the only possibility is for it to be killed by a $d_{11}$-differential coming from the class $\done^7 \bar{s}_1 u_{2\lambda} u_{7\sigma} a_{5\lambda} a_{\sigma_2}$. 
\end{proof}

\begin{cor} \label{cor:BPoned11}
The class $\sone u_{2\lambda}u_{3\sigma}a_{\sigma_2}$ supports the $d_{11}$-differential 
$$d_{11}(\sone u_{2\lambda}u_{3\sigma}a_{\sigma_2}) = \done^3 u_{4\sigma}a_{5\lambda}a_{2\sigma}.$$
\end{cor}
\begin{proof}
The $d_{11}$-differential in Theorem~\ref{thm:BPoned11} can be rewritten as 
$$\done^7 u_{4\sigma} a_{5\lambda} (d_{11}(\sone u_{2\lambda}u_{3\sigma}a_{\sigma_2}) - \done^3 u_{4\sigma}a_{5\lambda}a_{2\sigma})= 0.$$
Since multiplication by $\done^7 u_{4\sigma} a_{5\lambda}$ is injective on the $E_{11}$-page, the claim follows.  
\end{proof}

\begin{cor}
The class $\done^6 u_{4\lambda} u_{6\sigma} a_{2\lambda}$ at $(20,4)$ is a permanent cycle that survives to the $E_\infty$-page.  In homotopy, it detects the class $\bar{\kappa} \in \pi_{20}S^0$.  
\end{cor}
\begin{proof}
For degree reasons, this class is a permanent cycle that survives to the $E_\infty$-page.  To show this detects $\bar{\kappa}$, consider the commutative diagram 
$$
\begin{tikzcd}
\pi_*S^0 \ar[r] \ar[d] \ar[rd] & \pi_* \BPone^{C_4} \ar[d, "res"] \\ 
\pi_* {\BPR^{C_2}} \ar[r] & {\pi_* (i_{C_2}^*\BPone)^{C_2}},
\end{tikzcd}
$$
where the bottom horizontal map is the composition 
$$\BPR \stackrel{\iota_L}{\longrightarrow} i_{C_2}^* \BPC \longrightarrow i_{C_2}^* \BPone.$$
It is proven in \cite[Section 6]{HurewiczImages} that $\bar{\kappa}$ is detected in the $C_2$-slice spectral sequence of $\BPR$ by the class $\bar{v}_2 u_{8\sigma_2}a_{4\sigma_2}$.  Since $\bar{v}_2 = \rone^3$, $\bar{\kappa}$ is detected in $C_2$-$\SliceSS(i_{C_2}^*\BPone)$ by the class $\rone^{12}u_{8\sigma_2}a_{4\sigma_2}$.  This is exactly the restriction of the class $\done^6 u_{4\lambda} u_{6\sigma} a_{2\lambda}$ because 
$$res(\done^6 u_{4\lambda} u_{6\sigma} a_{2\lambda}) = \rone^6 \grone^6 u_{8\sigma_2}a_{4\sigma_2} = \rone^{12} u_{8\sigma_2}a_{4\sigma_2}. $$
Therefore, $\bar{\kappa}$ is detected by $\done^6 u_{4\lambda} u_{6\sigma} a_{2\lambda}$, as desired. 
\end{proof}

As shown in Figure~\ref{fig:E2C4topd11}, all the other $d_{11}$-differentials are obtained from the $d_{11}$-differential in Theorem~\ref{thm:BPoned11} via multiplication with the permanent cycles $\epsilon$, $\bar{\kappa}$, and $\done^8u_{8\lambda} u_{8\sigma}$ (at $(32, 0)$).  

Similarly, all the other $d_{13}$-differentials are obtained from Corollary~\ref{cor:BPoned13} by using multiplicative structures with the classes $\eta'$, $\epsilon$, $\bar{\kappa}$, and $\done^8u_{8\lambda} u_{8\sigma}$ (see Figure~\ref{fig:E2C4topd13}).

\begin{figure}
\begin{center}
\makebox[\textwidth]{\includegraphics[trim={0cm 5cm 0cm 3cm}, clip, page = 4, scale = 0.9]{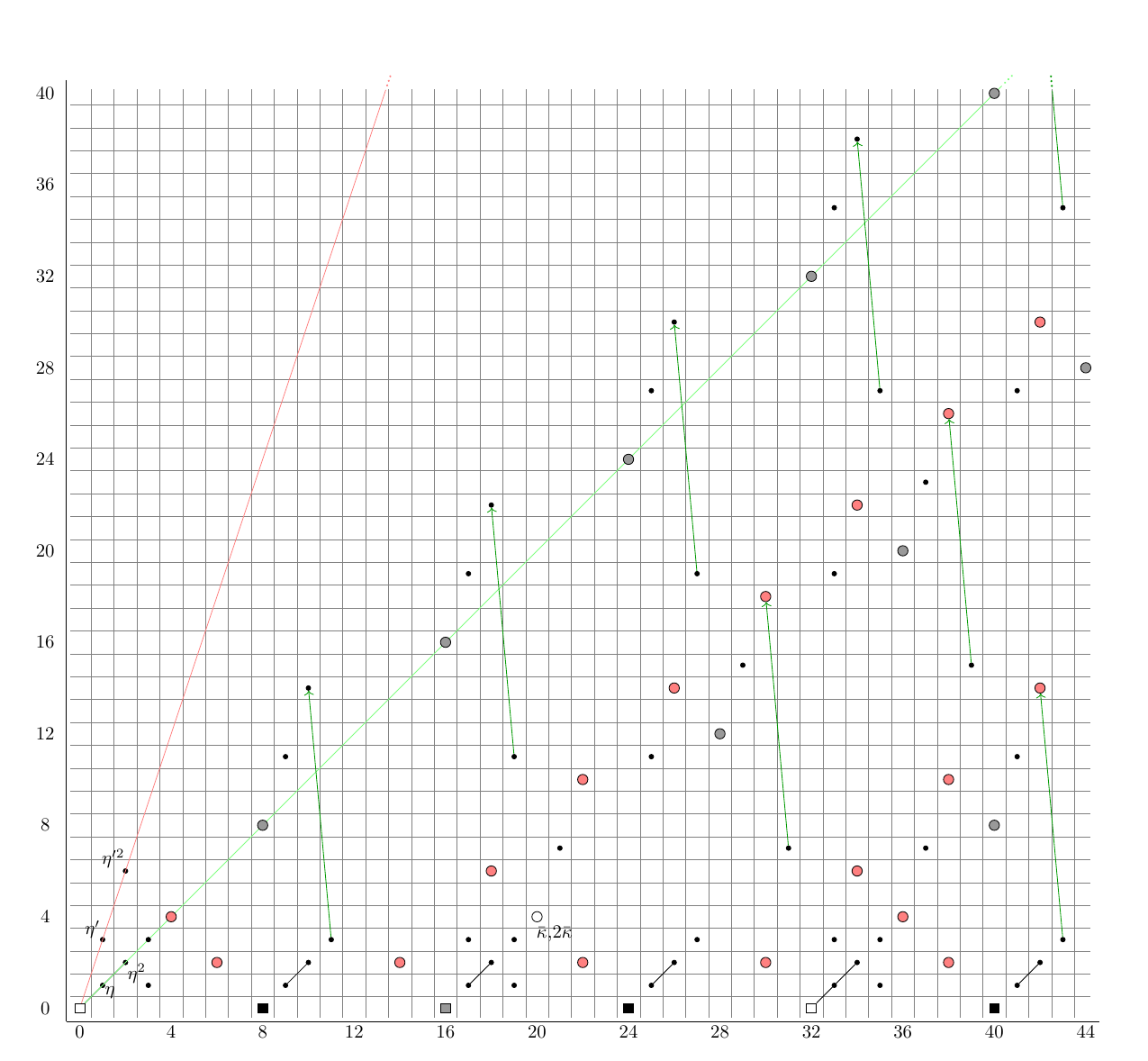}}
\end{center}
\begin{center}
\caption{$d_{11}$-differentials in $C_4$-$\SliceSS(\BPone)$.}
\hfill
\label{fig:E2C4topd11}
\end{center}
\end{figure}

\begin{figure}
\begin{center}
\makebox[\textwidth]{\includegraphics[trim={0cm 5cm 0cm 3cm}, clip, page = 5, scale = 0.9]{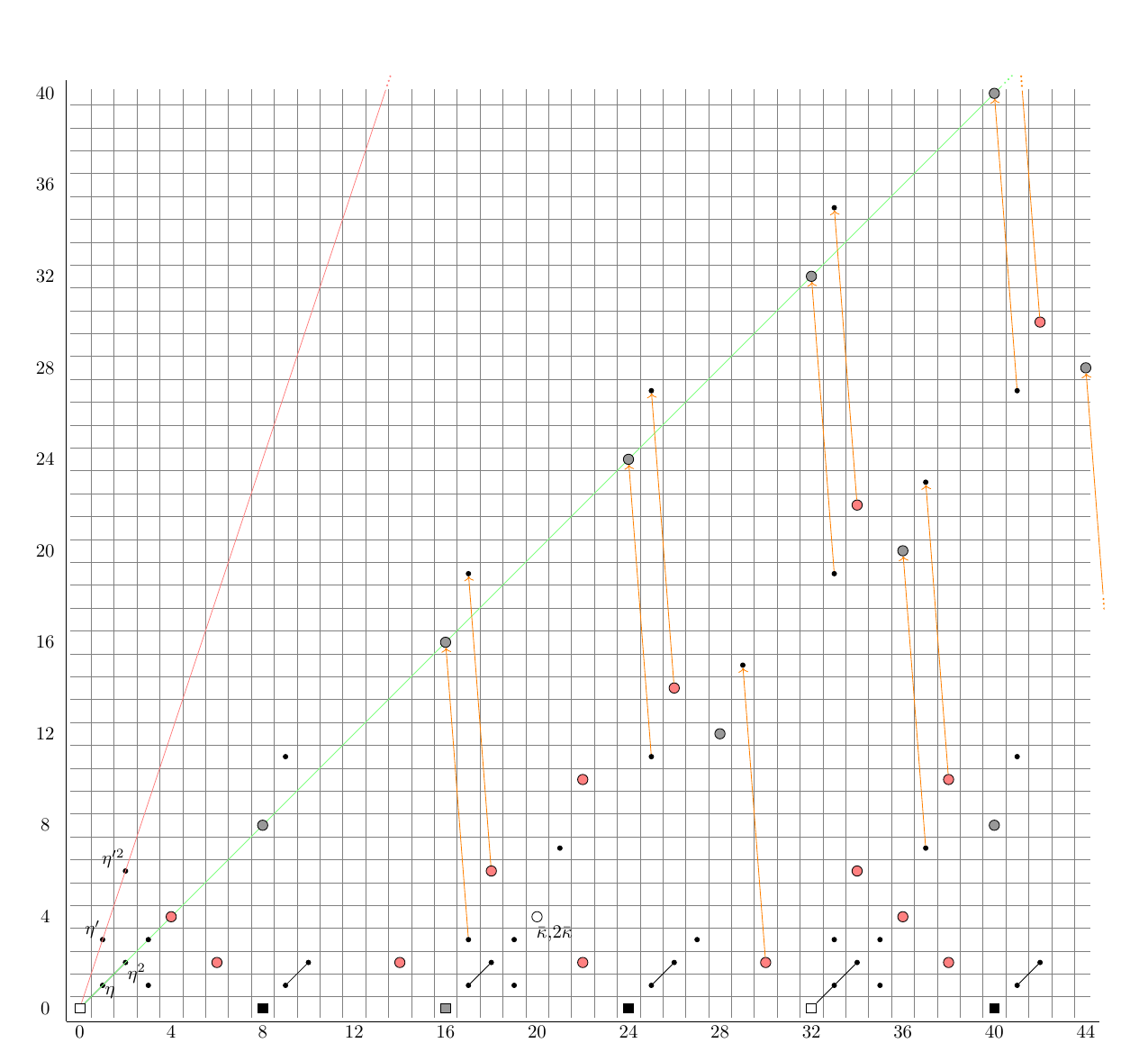}}
\end{center}
\begin{center}
\caption{$d_{13}$-differentials in $C_4$-$\SliceSS(\BPone)$.}
\hfill
\label{fig:E2C4topd13}
\end{center}
\end{figure}

\begin{figure}
\begin{center}
\makebox[\textwidth]{\includegraphics[trim={0cm 5cm 0cm 3cm}, clip, page = 1, scale = 0.9]{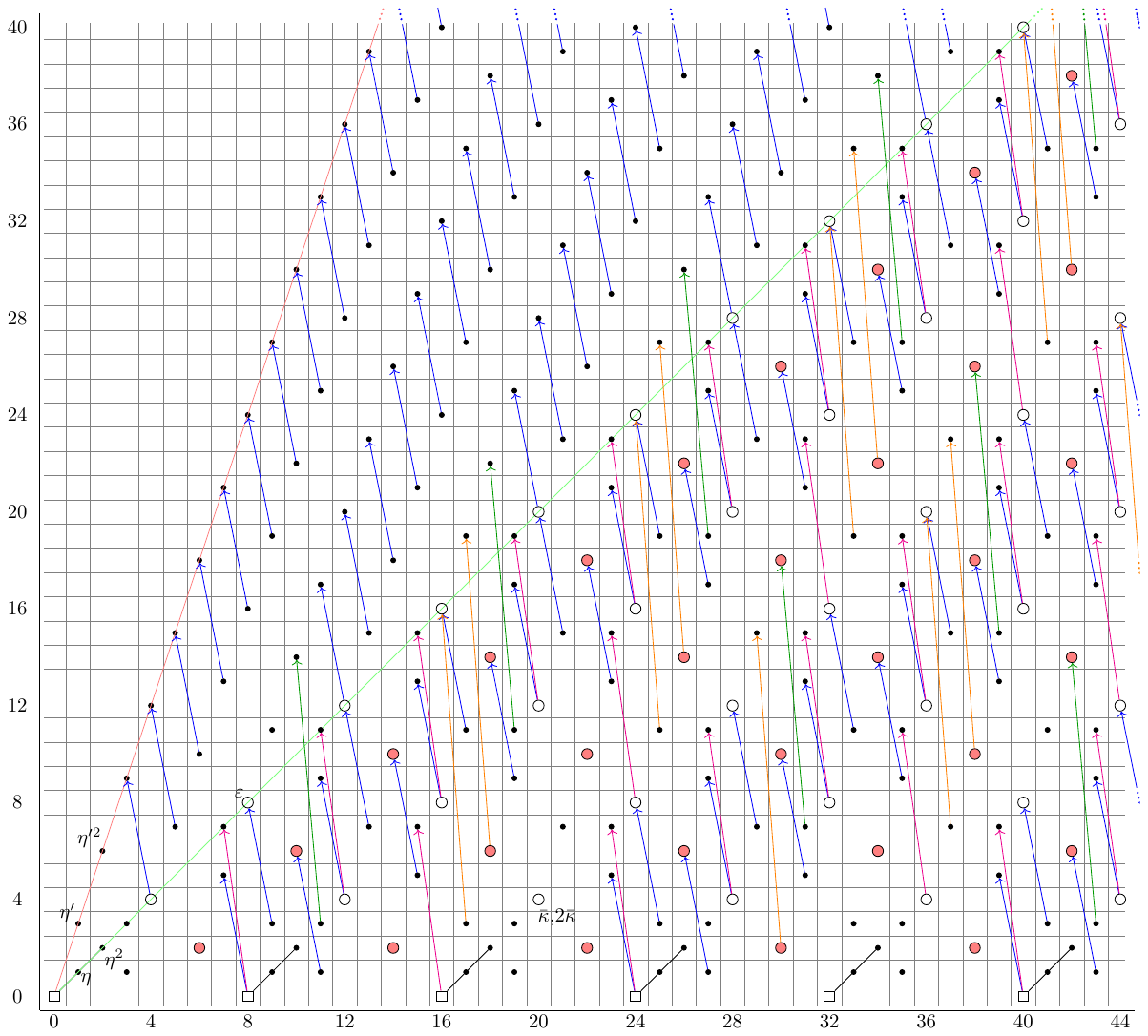}}
\end{center}
\begin{center}
\caption{Differentials in $C_4$-$\SliceSS(\BPone)$.}
\hfill
\label{fig:E2C4top1}
\end{center}
\end{figure}

\begin{figure}
\begin{center}
\makebox[\textwidth]{\includegraphics[trim={0cm 5cm 0cm 3cm}, clip, page = 6, scale = 0.9]{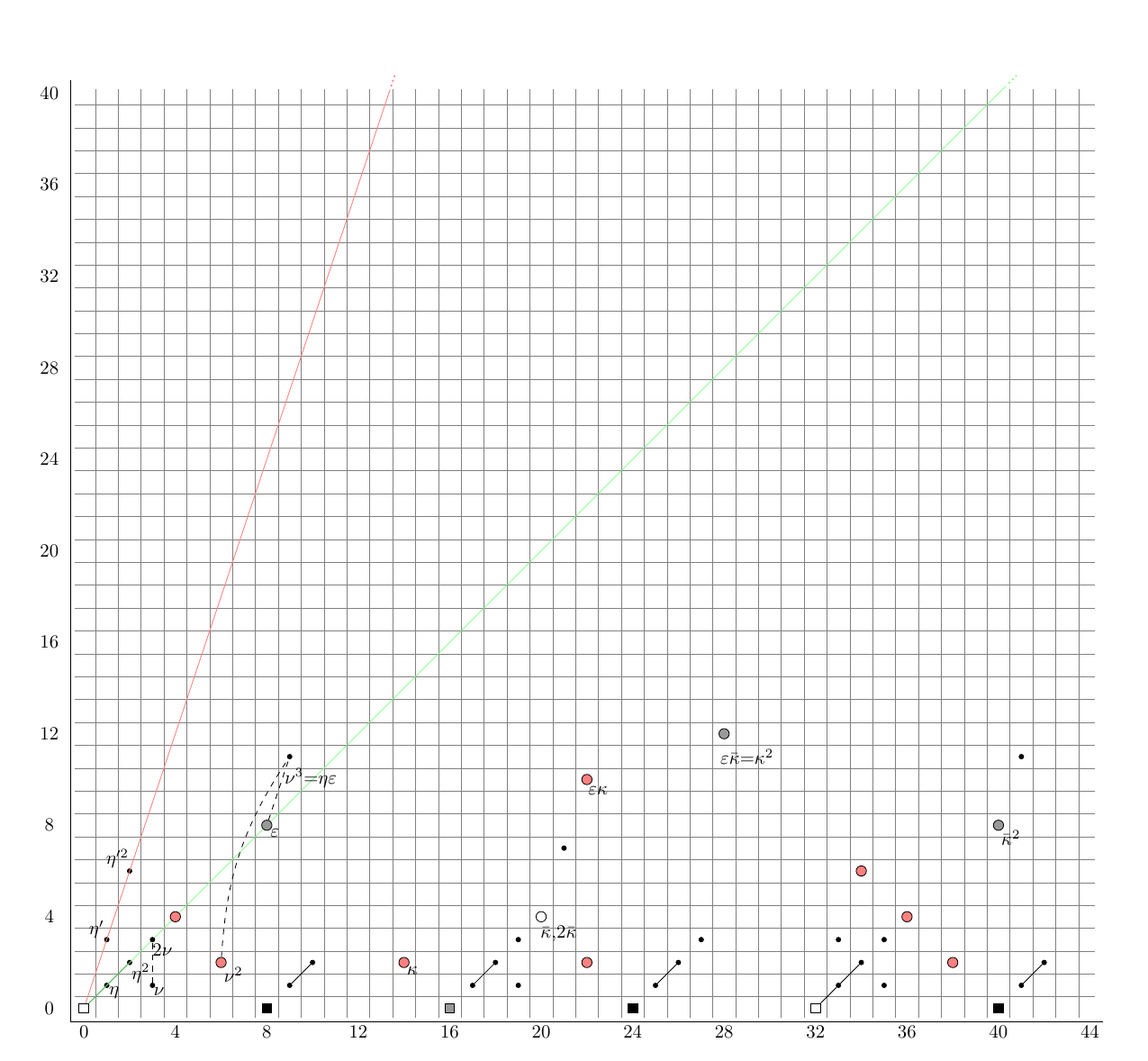}}
\end{center}
\begin{center}
\caption{$E_\infty$-page of $C_4$-$\SliceSS(\BPone)$.}
\hfill
\label{fig:E2C4top2}
\end{center}
\end{figure}

\subsection{Summary of differentials}
\begin{center}
\begin{tabular}{|c |l |l |}
\hline
Differential &Formula & Proof \\ \hline
$d_3$ & $d_3(u_\lambda) = \bar{s}_1a_\lambda a_{\sigma_2}$& Theorem~\ref{thm:BPoned3} (restriction) \\ 
& $d_3(u_{2\sigma_2}) = \bar{s}_1 a_{3\sigma_2}$ & \\ \hline
{\color{blue} $d_5$} & {\color{blue}$d_5(u_{2\sigma}) = \done a_\lambda a_{3\sigma}$} & \cite[Theorem~9.9]{HHR} (Slice Differentials Theorem)\\ \hline
{\color{blue} $d_5$} & {\color{blue}$d_5(u_{2\lambda}) = \done u_\lambda a_{2\lambda}a_\sigma$ }& Theorem~\ref{thm:BPoned5} and Corollary~\ref{cor:BPoned5} (restriction)\\ 
&{\color{blue}$d_5(u_{3\lambda}a_\sigma) = 2\done u_\lambda u_{2\sigma} a_{3\lambda}$ }& \\ \hline
{\color{magenta} $d_7$} & {\color{magenta}$d_7(2u_{2\lambda}) = \done \sone u_\sigma a_{3\lambda} a_{\sigma_2}$} & Theorem~\ref{thm:BPoned7case1} and Corollary~\ref{cor:BPoned7case1} (transfer) \\ 
&{\color{magenta}$d_7(2u_{2\lambda} u_{2\sigma}) = \done \sone u_{3\sigma} a_{3\lambda} a_{\sigma_2}$}& \\ \hline
{\color{magenta} $d_7$} & {\color{magenta}$d_7(u_{4\lambda}) = \done \sone u_{2\lambda} u_\sigma a_{3\lambda} a_{\sigma_2}$}& Theorem~\ref{thm:BPoned7case2} and Corollary~\ref{cor:BPoned7case2} (norm) \\ \hline
{\color{green!60!black} $d_{11}$} & {\color{green!60!black}$d_{11}(\sone u_{2\lambda}u_{3\sigma}a_{\sigma_2}) = \done^3 u_{4\sigma}a_{5\lambda}a_{2\sigma}$} & Theorem~\ref{thm:BPoned11} and Corollary~\ref{cor:BPoned11} \\
&{\color{green!60!black}$d_{11}(\sone u_{6\lambda}u_\sigma a_{\sigma_2}) = 2\done^3 u_{3\lambda} u_{4\sigma}a_{6\lambda}$} & (uses Theorem~\ref{thm:BPoneMainThm})\\ \hline
{\color{orange} $d_{13}$} & {\color{orange}$d_{13}(u_{4\lambda}a_\sigma) = \bar{\mathfrak{d}}_1^3 u_{4\sigma}a_{7\lambda}$} & Theorem~\ref{thm:BPoned13} and Corollary~\ref{cor:BPoned13} (norm)\\ 
&{\color{orange}$d_{13}(u_{8\lambda}u_{2\sigma}a_\sigma)= \done^3 u_{4\lambda}u_{6\sigma} a_{7\lambda}$}& \\ \hline
\end{tabular}
\end{center}

%%%%%%%%%%%%%%%%%%%%%%%%
%%%%%%%%%%%%%%%%%%%%%%%%
\section{The slice filtration of $\BPC\langle 2 \rangle$} \label{sec:slicesBPtwo}
The refinement of $\BPtwo$ is 
$$S^0[\rone, \grone, \rthree, \grthree] \longrightarrow \BPtwo.$$
Its slices are the following: 
$$\left\{\begin{array}{ll}
\rthree^i \grthree^i \rone^k \grone^k: S^{(3i+k)\rho_4} \wedge \HZ, & i, k \geq 0 \\
(\rthree^i \grthree^j \rone^k \grone^l, \rthree^j \grthree^i \rone^l \grone^k): {C_4}_+ \wedge_{C_2} S^{(3i+3j+k+l)\rho_2} \wedge \HZ, & i \neq j \text{ or } k \neq l.
\end{array} \right.$$
Similar to the slices of $\BPone$, we organize the slices for $\BPtwo$ in order to facilitate our computation.  

Consider the monomial $\rthree^i \grthree^j \rone^k \grone^l$.  When $i = j$, this monomial can be written as $\dthree^i \rone^k \grone^l$.  Fix a non-negative integer $i$.  After the $d_3$-differentials, the classes of filtration $\geq 3$ that are contributed by these slice cells are exactly the same as the classes on the $E_5$-page of $C_4$-$\SliceSS(\BPone)$, truncated at the line $(t-s) + s = 12i$.  For this reason, we will call this collection of slices $\dthree^i\BPone$.  Figure \ref{fig:E2C4trunclines} shows the truncation lines for the slices $\dthree^i \BPone$ and Figures \ref{fig:E2C4trunclines12} and \ref{fig:E2C4trunclines24} illustrate the classes contributed by the slices in $\dthree \BPone$ and $\dthree^2 \BPone$.  

\begin{figure}
\begin{center}
\makebox[\textwidth]{\includegraphics[trim={0cm 5cm 0cm 3cm}, clip, page = 1, scale = 0.9]{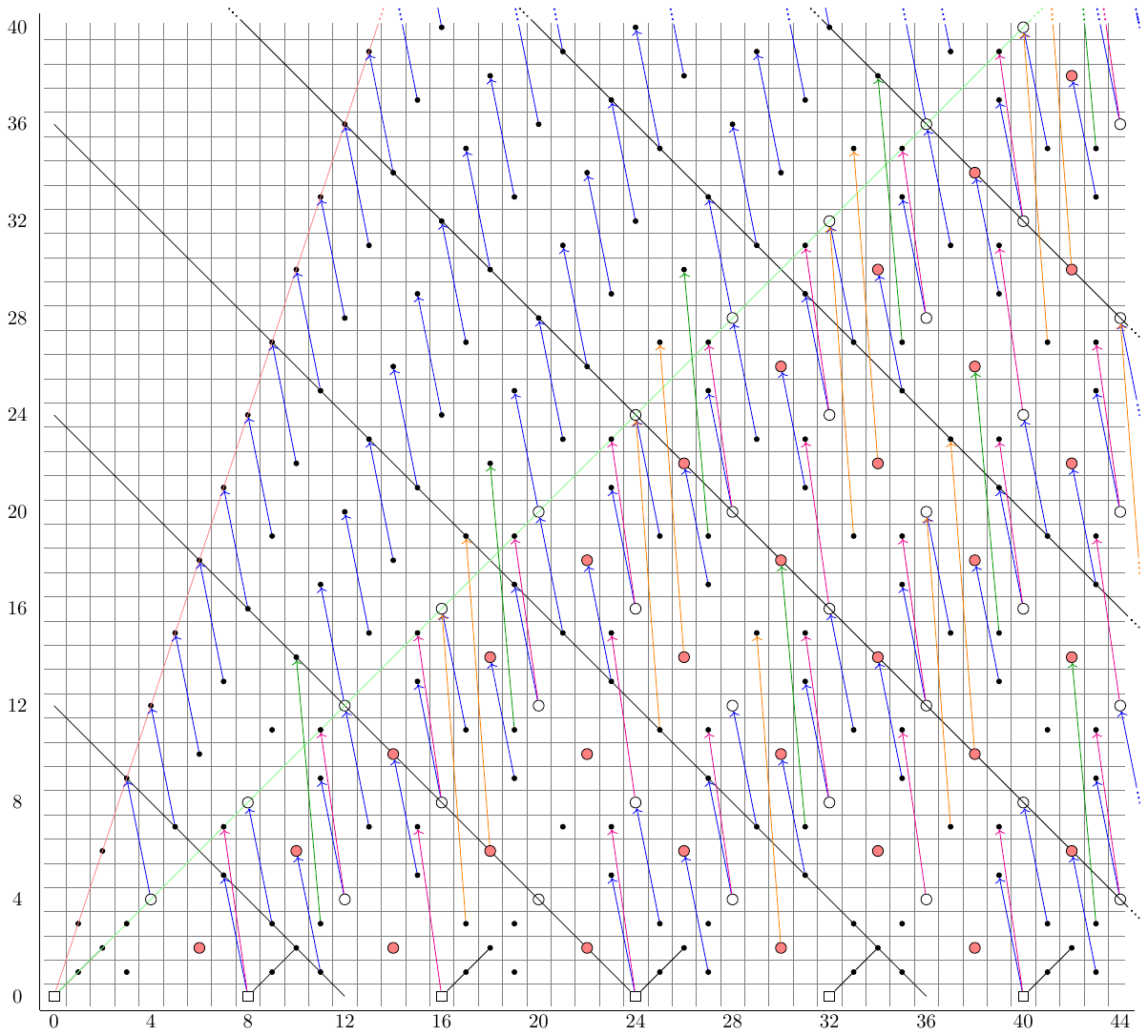}}
\end{center}
\begin{center}
\caption{Truncations lines for the slices $\dthree^i \BPone$. }
\hfill
\label{fig:E2C4trunclines}
\end{center}
\end{figure}

\begin{figure}
\begin{center}
\makebox[\textwidth]{\includegraphics[trim={0cm 5cm 0cm 3cm}, clip, page = 2, scale = 0.9]{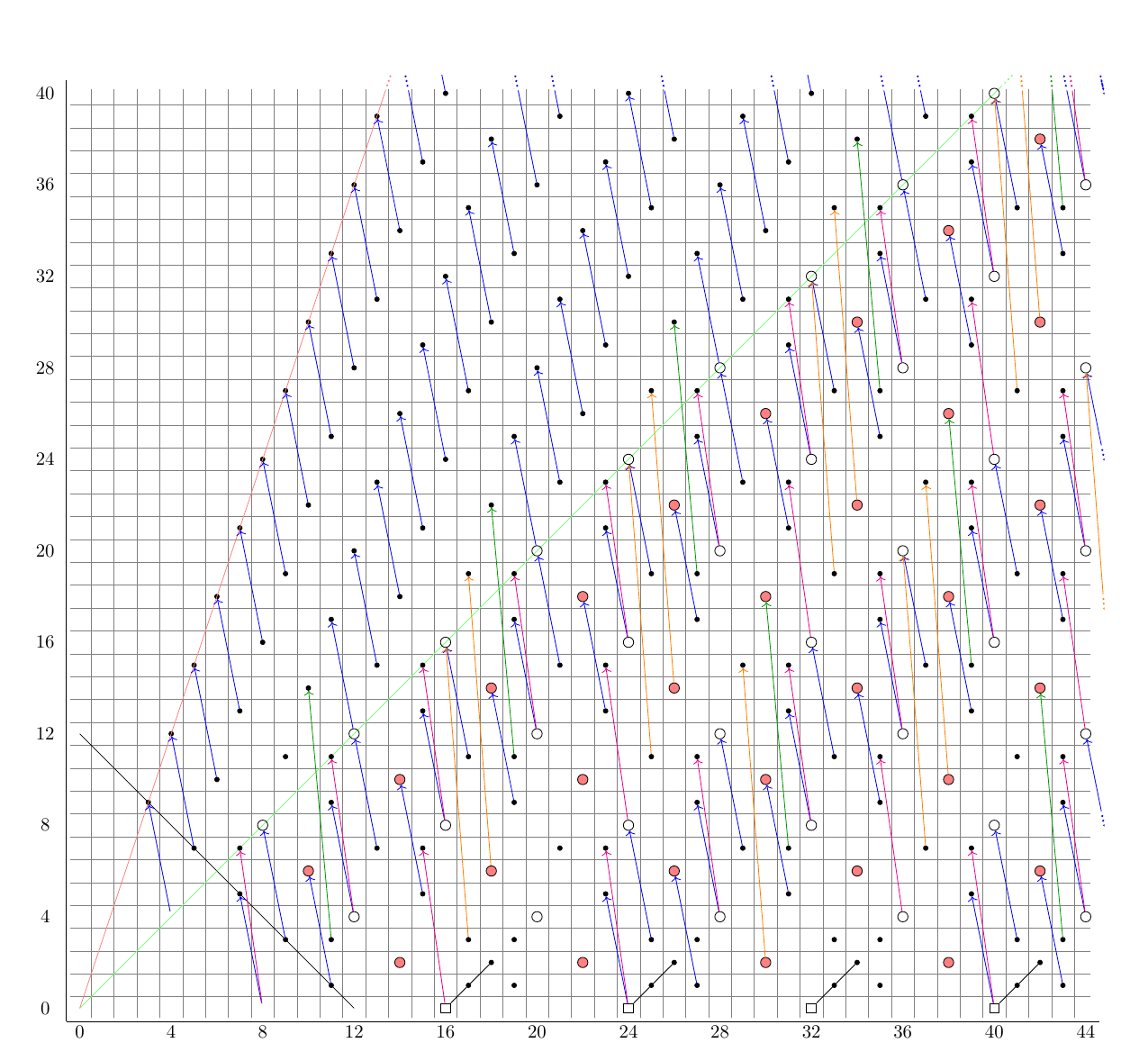}}
\end{center}
\begin{center}
\caption{Classes contributed by the slices in $\dthree \BPone$. }
\hfill
\label{fig:E2C4trunclines12}
\end{center}
\end{figure}

\begin{figure}
\begin{center}
\makebox[\textwidth]{\includegraphics[trim={0cm 5cm 0cm 3cm}, clip, page = 3, scale = 0.9]{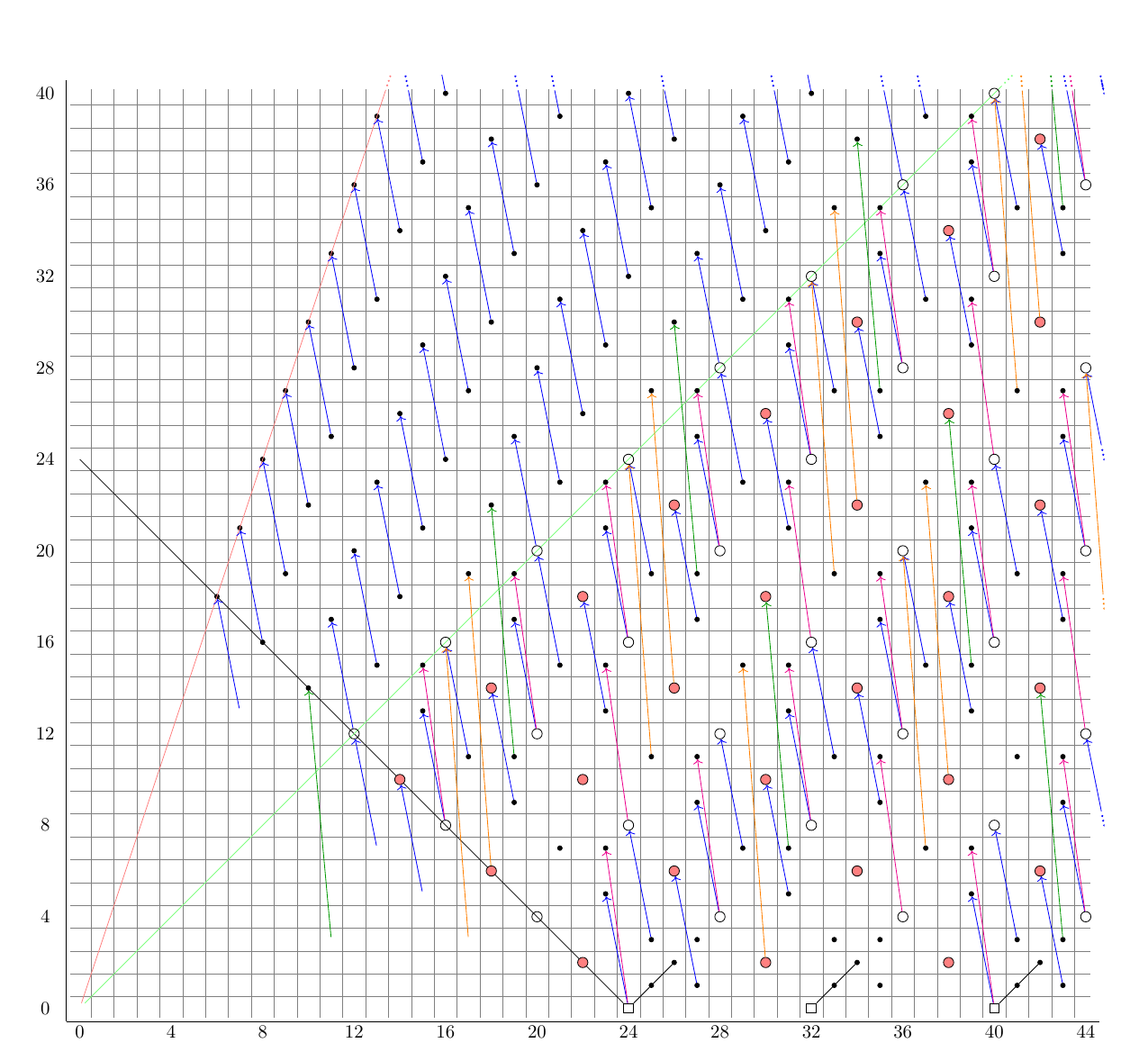}}
\end{center}
\begin{center}
\caption{Classes contributed by the slices in $\dthree^2 \BPone$.}
\hfill
\label{fig:E2C4trunclines24}
\end{center}
\end{figure}

When $i \neq j$, the monomial $\rthree^i \grthree^j \rone^k \grone^l$ contributes an induced slice of the form $(\rthree^i \grthree^j \rone^k \grone^l, \rthree^j \grthree^i \rone^l \grone^k)$.  By symmetry, let $i < j$.  The $d_3$-differential $d_3(u_{2\sigma_2}) = \bar{s}_1 a_{\sigma_2}^3 = (\rone + \grone)a_{\sigma_2}^3$ identifies $\rone$ and $\grone$ when the filtration is at least 3.  Similar to the situation for $\BPone$, define $\dthree:= N(\rthree)$ and $\sthree^i := tr(\rthree^i)$.  By an abuse of notation, we can rewrite this slice cell as
$$\rthree^i \grthree^j \rone^k \grone^l + \rthree^j \grthree^i \rone^k \grone^l = (\rthree^i \grthree^i) (\rthree^{j-i} + \grthree^{j-i})\rone^k \grone^l = \dthree^i \bar{s}_3^{j-i} \rone^k \grone^l.$$  
For a fixed pair $\{i, j\}$ with $i < j$, the classes of filtration $\geq 3$ that are contributed by these slice cells after the $d_3$-differentials are exactly the same as the classes on the $E_5$-page of $C_2$-$\SliceSS(i_{C_2}^*\BPone)$, truncated at the line $(t-s) + s = 6i + 6j$.  For this reason, we will call this collection of slices $\dthree^i \bar{s}_3^j  i_{C_2}^*\BPone$.  Figure~\ref{fig:E2C2SSSTrunc} shows the truncation lines for these slices.  Figures~\ref{fig:E2C2SSSTrunc6}, \ref{fig:E2C2SSSTrunc12}, and \ref{fig:E2C2SSSTrunc18} illustrate the classes contributed by some of these slices.  

\begin{figure}
\begin{center}
\makebox[\textwidth]{\includegraphics[trim={0cm 11cm 0cm 4cm}, clip, scale = 0.8, page = 1]{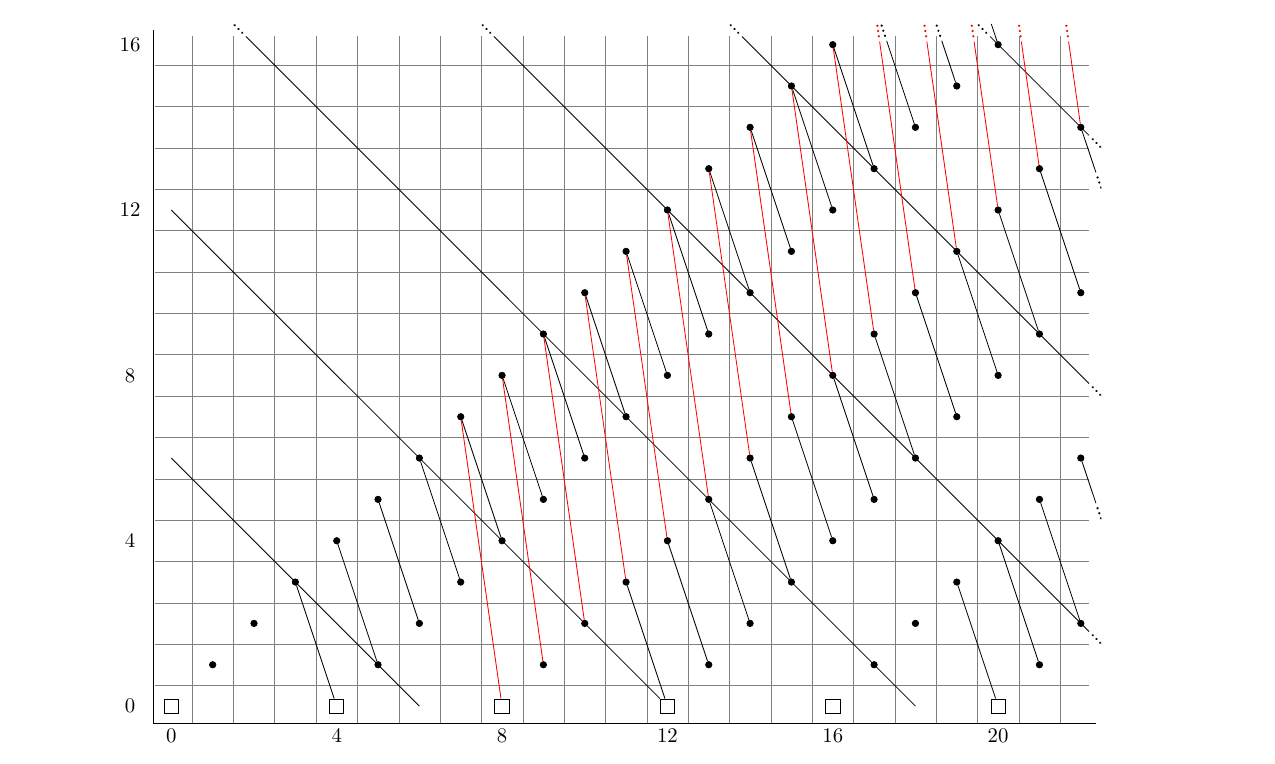}}
\end{center}
\begin{center}
\caption{Truncation lines for the slices $\dthree^i \bar{s}_3^j  i_{C_2}^*\BPone$.}
\hfill
\label{fig:E2C2SSSTrunc}
\end{center}
\end{figure}

\begin{figure}
\begin{center}
\makebox[\textwidth]{\includegraphics[trim={0cm 11cm 0cm 4cm}, clip, scale = 0.8, page = 2]{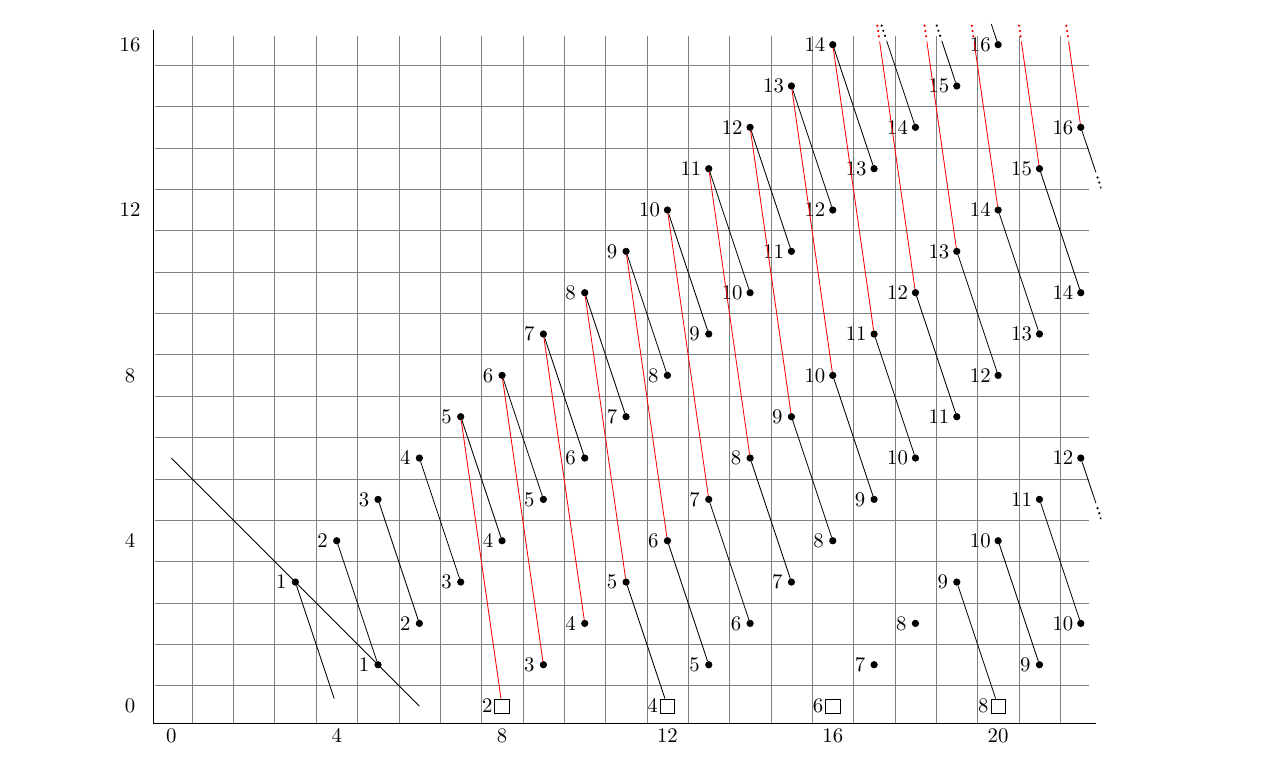}}
\end{center}
\begin{center}
\caption{Classes contributed by the slices in $\sthree i_{C_2}^* \BPone$. }
\hfill
\label{fig:E2C2SSSTrunc6}
\end{center}
\end{figure}

\begin{figure}
\begin{center}
\makebox[\textwidth]{\includegraphics[trim={0cm 11cm 0cm 4cm}, clip, scale = 0.8, page = 3]{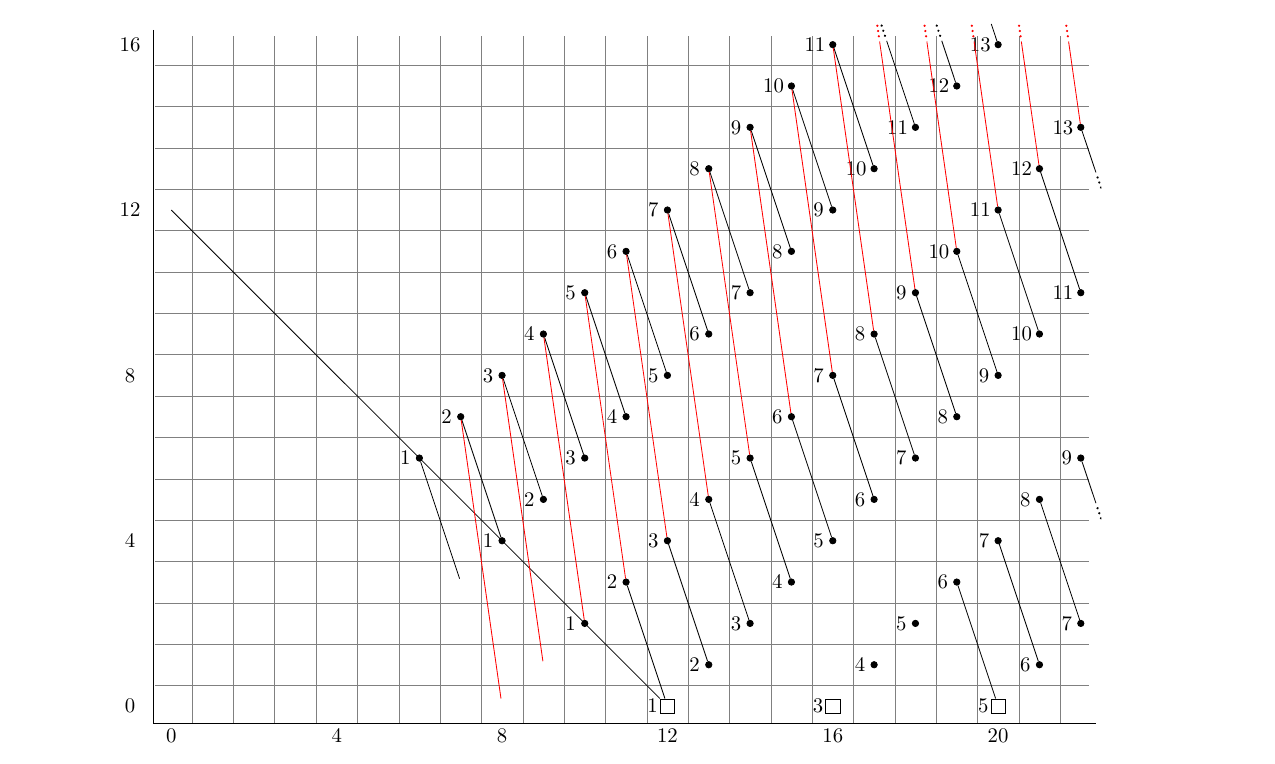}}
\end{center}
\begin{center}
\caption{Classes contributed by the slices in $\sthree^2 i_{C_2}^* \BPone$. }
\hfill
\label{fig:E2C2SSSTrunc12}
\end{center}
\end{figure}

\begin{figure}
\begin{center}
\makebox[\textwidth]{\includegraphics[trim={0cm 11cm 0cm 4cm}, clip, scale = 0.8, page = 4]{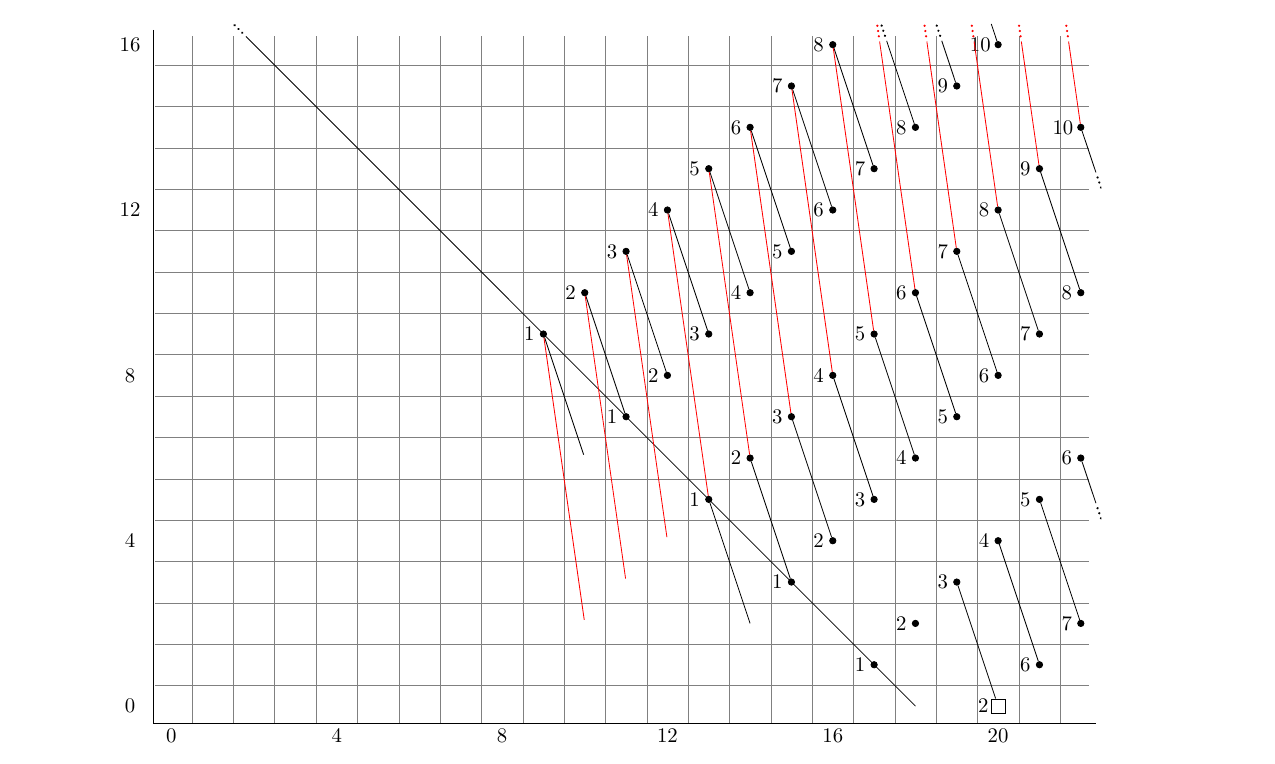}}
\end{center}
\begin{center}
\caption{Classes contributed by the slices in $\sthree^3 i_{C_2}^* \BPone$ or $\dthree \sthree i_{C_2}^* \BPone$. }
\hfill
\label{fig:E2C2SSSTrunc18}
\end{center}
\end{figure}

All the slices of $\BPtwo$ are organized into the following table, where the number inside the parenthesis indicates the truncation line.  For convenience, we will refer to each of the collections on the top row as a \textbf{$\BPone$-truncation}, and each of the other collections as a {\color{blue} $i_{C_2}^* \BPone$-truncation}. 

\begin{equation}
\begin{array}{c|c|c|c} \label{array:OrganizeSliceCells}
\bar{\mathfrak{d}}_3^0 \BPone - (0)  &\bar{\mathfrak{d}}_3^1 \BPone -(12) &\bar{\mathfrak{d}}_3^2 \BPone -(24)& \cdots \\ \hline
\color{blue} \bar{\mathfrak{d}}_3^0 \bar{s}_3^1 i_{C_2}^*\BPone -(6) & \color{blue}  \bar{\mathfrak{d}}_3^1 \bar{s}_3^1 i_{C_2}^*\BPone -(18)& \color{blue}  \bar{\mathfrak{d}}_3^2 \bar{s}_3^1 i_{C_2}^*\BPone -(30)& \cdots \\ 
\color{blue}  \bar{\mathfrak{d}}_3^0 \bar{s}_3^2 i_{C_2}^*\BPone -(12)& \color{blue} \bar{\mathfrak{d}}_3^1 \bar{s}_3^2 i_{C_2}^*\BPone -(24) & \color{blue}  \bar{\mathfrak{d}}_3^2 \bar{s}_3^2 i_{C_2}^*\BPone-(36)& \cdots \\ 
\vdots & \vdots & \vdots & \ddots
\end{array}
\end{equation}

%%%%%%%%%%%%%%%%%%%%%%%%
%%%%%%%%%%%%%%%%%%%%%%%%
\section{The $C_2$-slice spectral sequence of $i_{C_2}^* \BPC\langle 2 \rangle$} \label{sec:C2BPtwoSliceSS}
In this section, we will compute the $C_2$-slice spectral sequence for $i_{C_2}^*\BPC \langle 2 \rangle$.  The composition map 
$$\BPR \stackrel{i_L}{\longrightarrow} i_{C_2}^* \BPC \longrightarrow i_{C_2}^* \BPtwo$$
induces a map 
$$C_2\text{-}\SliceSS(\BPR) \longrightarrow C_2\text{-}\SliceSS(i_{C_2}^*\BPtwo)$$
of $C_2$-slice spectral sequences.  The formulas in Theorem~\ref{thm:fglFormulas} translate the differentials in $C_2\text{-}\SliceSS(\BPR)$ to the following differentials in $C_2\text{-}\SliceSS(i_{C_2}^*\BPtwo)$: 
\begin{eqnarray*}
d_3(u_{2\sigma_2}) &=& (\rone + \grone)a_{\sigma_2}^3\\
d_7(u_{4\sigma_2}) &=& (\rone^3 + \rthree + \grthree) a_{\sigma_2}^7\\
d_{15}(u_{8\sigma_2}) &=& \rone(\rthree^2 + \rthree \grthree + \grthree^2) a_{\sigma_2}^{15}\\
d_{31}(u_{16\sigma_2}) &=& \rthree^4 \grthree a_{\sigma_2}^{31}
\end{eqnarray*}
The class $u_{32\sigma_2}$ is a permanent cycle.  

On the $E_2$-page, the refinement 
$$S^0[\rone, \grone, \rthree, \grthree] \longrightarrow i_{C_2}^* \BPtwo$$
implies that all the slice cells are indexed by monomials of the form $\rthree^i \grthree^j \rone^k \grone^l$, where $i, j, k, l \geq 0$.  

We now give a step-by-step description of the surviving classes after each differential: 
\begin{enumerate}
\item After the $d_3$-differentials, the relation $\rone = \grone$ is introduced for classes with filtrations $\geq 3$.  Therefore, the slice cells corresponding to these classes can be written as a sum of monomials from the set $\{\rthree^i \grthree^j \rone^k \,|\, i, j, k \geq 0\}$.  
\item After the $d_7$-differentials, the relation $\rone^3 + \rthree + \grthree = 0$ is introduced for classes with filtrations $\geq 7$.  Therefore given a class with filtration at least 7, depending on its bidegree, its corresponding slice cell can be written as a sum of monomials from the set $\{ \rthree^i \grthree^j \rone \,|\, i, j \geq 0\}$, $\{ \rthree^i \grthree^j \rone^2 \,|\, i, j \geq 0\}$, or $\{ \rthree^i \grthree^j \rone^3, \rthree^{i+j} \grthree \,|\, i, j \geq 0\}$.
\item After the $d_{15}$-differentials, the relation $\rone(\rthree^2 + \rthree \grthree + \grthree^2) = 0$ is introduced for classes with filtrations $ \geq 15$.  Given a class with filtration at least 15, depending on its bidegree, its corresponding slice cell can be written as a sum of monomials from the set $\{\rthree^{i+1} \grthree^1 \rone, \rthree^{i} \grthree^2 \rone, \,|\, i \geq 4 \}$, $\{\rthree^{i+1} \grthree^1 \rone^2, \rthree^{i} \grthree^2 \rone^2, \,|\, i \geq 4 \}$, or $\{\rthree^{i+1} \grthree^1 \rone^3, \rthree^{i} \grthree^2 \rone^3, \rthree^{i+2} \grthree \,|\, i \geq 4 \}$. 
\item After the $d_{31}$-differentials, the relation $\rthree^4 \grthree = 0$ is introduced for classes with filtrations $\geq 31$.  Since all the classes with filtrations $\geq 31$ have slice cells divisible by $\rthree^4 \grthree$ on the $E_{31}$-page, they are all wiped out by the $d_{31}$-differentials.  The spectral sequence collapses afterwards and there is a horizontal vanishing line with filtration 31 on the $E_\infty$-page.  
\end{enumerate} 

\begin{exam}\rm
Consider all the classes at $(31, 31)$.  On the $E_2$-page, their names are of the form $x \cdot a_{31\sigma_2}$, where $x$ is a sum of slice cells of the form $\{\rthree^i \grthree^j \rone^k \grone^l\,|\, 3i + 3j + k + l = 31\}$.  
\begin{enumerate}
\item After the $d_3$-differentials, $x$ can be written as a sum of of classes whose corresponding slice cells are from the set $\{\rthree^i \grthree^j \rone^k \,|\, 3i + 3j + k = 31 \}$. 
\item After the $d_7$-differentials, $x$ can be written as a sum of of classes whose corresponding slice cells are from the set $\{\rthree^i \grthree^j \rone \,|\, i+j = 10 \}$.
\item After the $d_{15}$-differentials, $x$ can be written as a sum of of classes whose corresponding slice cells are from the set $\{\rthree^9\grthree\rone, \rthree^8 \grthree^2 \rone\}$. 
\item After the $d_{31}$-differentials, all the remaining classes are killed. 
\end{enumerate}
\end{exam}

\begin{exam}\rm
Consider all the classes at $(33, 33)$.  On the $E_2$-page, their names are of the form $x \cdot a_{33\sigma_2}$, where $x$ is a sum of classes whose corresponding slice cells are of the form $\{\rthree^i \grthree^j \rone^k \grone^l\,|\, 3i + 3j + k + l = 33\}$.  
\begin{enumerate}
\item After the $d_3$-differentials, $x$ can be written as a sum of of classes whose corresponding slice cells are from the set $\{\rthree^i \grthree^j \rone^k \,|\, 3i + 3j + k = 33 \}$. 
\item After the $d_7$-differentials, $x$ can be written as a sum of of classes whose corresponding slice cells are from the set $\{\rthree^i \grthree^j \rone^3, \rthree^{10} \grthree \,|\, i+j = 10 \}$.
\item After the $d_{15}$-differentials, $x$ can be written as a sum of of classes whose corresponding slice cells are from the set $\{\rthree^9\grthree\rone^3, \rthree^8 \grthree^2 \rone^3, \rthree^{10} \grthree \}$. 
\item After the $d_{31}$-differentials, all the remaining classes are killed. 
\end{enumerate}
\end{exam}

%%%%%%%%%%%%%%%%%%%%%%%%
%%%%%%%%%%%%%%%%%%%%%%%%
\section{Induced differentials from $\BPone$} \label{sec:InducedDiffBPone}
In section~\ref{sec:slicesBPtwo}, the slices of $\BPtwo$ are subdivided into collections of the form $\dthree^i \BPone$ ($\BPone$-truncation) and $\dthree^i \sthree^j i_{C_2}^* \BPone$ ($i_{C_2}^* \BPone$-truncation), where $i \geq 0$ and $j \geq 1$.  On the $E_2$-page of the $C_4$-slice spectral sequence, the classes contributed by the slices in $\dthree^i \BPone$ is a truncation of the $E_2$-page of $C_4\text{-}\SliceSS(\BPone)$, and the classes contributed by the slices in $\dthree^i \sthree^j i_{C_2}^*\BPone$ is a truncation of the $E_2$-page of $C_2\text{-}\SliceSS(i_{C_2}^* \BPone)$.  

Recall that in the computation of $\SliceSS(\BPone)$, we have also divided the slices into collections (they are the columns in Table~\ref{array:BPoneSlices}).  The computation was simplified by treating each collection individually with respect to the $d_3$-differentials.  After the $d_3$-differentials, we combined the $E_5$-pages of every collection together to form the $E_5$-page of $\SliceSS(\BPone)$.  

In light of this simplification for $\SliceSS(\BPone)$, it is natural to expect that in $\SliceSS(\BPtwo)$, each collection can be treated individually with respect to differentials of lengths up to 13 (the longest differential in $\BPone$).  Knowing this will allow us to compute the $E_{13}$-page of each collection individually, and then combine them together to form the $E_{13}$-page of $\SliceSS(\BPtwo)$.  

\begin{df}\label{df:predictedDiff}\rm
A \textit{predicted differential} is a differential whose leading terms for the source and the target belong to slices in the same collection and the position of that differential matches with a differential in $C_4\text{-}\SliceSS(\BPone)$ or $C_2\text{-}\SliceSS(i_{C_2}^*\BPone)$.  
\end{df}

For example, all of the differentials whose source and target are on or above the truncation lines in Figures~\ref{fig:E2C4trunclines12}, \ref{fig:E2C4trunclines24}, \ref{fig:E2C2SSSTrunc6}, \ref{fig:E2C2SSSTrunc12}, \ref{fig:E2C2SSSTrunc18} are predicted differentials.  

\begin{df}\label{df:interferingDiff}\rm
An \textit{interfering differential} is a differential whose source and target are in different collections.  
\end{df}

Given the definitions above, 
\begin{thm}\label{thm:inducedDiffBPCfourone}
The collections can be treated individually with respect to differentials of lengths up to 13.  More specifically: 
\begin{enumerate}
\item For $3 \leq r \leq 11$, all the predicted $d_r$-differentials occur and there are no interfering $d_r$-differentials. 
\item All the predicted $d_{13}$-differentials occur.  
\end{enumerate}
\end{thm}

The rest of this section is dedicated to the proof of Theorem~\ref{thm:inducedDiffBPCfourone}.  At each step of the proof, once we have proven that all the predicted $d_r$-differentials occur and there are no interfering $d_r$-differentials, a \textit{$d_r$-truncation class} is a class that is produced from a $d_r$-differential in $C_4$-$\SliceSS(\BPone)$ crossing a truncation line in $C_4$-$\SliceSS(\BPtwo)$.  Roughly speaking, this is a class that given its bidegree, would have been killed by a $d_r$-differential in $C_4$-$\SliceSS(\BPone)$ but survives on the $E_r$-page of $C_4$-$\SliceSS(\BPtwo)$ because the source of the $d_r$-differential has been truncated off. 

\subsection{$d_3$-differentials} \label{subsec:d3induced}
The quotient map $\BPtwo \to \BPone$ induces a map 
$$\SliceSS(\BPtwo) \longrightarrow \SliceSS(\BPone)$$
of $C_4$-slice spectral sequences.  In $\SliceSS(\BPone)$, the $d_3$-differentials are generated under multiplication by
$$d_3(u_\lambda) = \sone a_\lambda a_{\sigma_2}.$$
For natuality and degree reasons, the same differential occurs in $\SliceSS(\BPtwo)$ as well.  Moreover, by considering the restriction map 
$$C_4\text{-}\SliceSS(\BPtwo) \longrightarrow C_2\text{-}\SliceSS(i_{C_2}^* \BPtwo),$$
we deduce the $d_3$-differential 
$$d_3(u_{2\sigma_2}) = \sone a_{3\sigma_2}$$
as well. 
(Even though we are working with a $C_4$-slice spectral sequence, this differential applies to some of the classes in $i_{C_2}^*\BPone$-truncations because of our naming conventions). 
All the predicted differentials are generated by these two differentials.  Afterwards, there are no more $d_3$-differentials by degree reasons. 

\subsection{$d_5$-differentials}\label{subsec:d5induced}
In $\SliceSS(\BPone)$, all the $d_5$-differentials are generated under multiplication by the differentials 
$$d_5(u_{2\sigma})=\done a_\lambda a_{3\sigma}$$ and $$d_5(u_{2\lambda}) = \done u_\lambda a_{2\lambda} a_\sigma.$$
In $\SliceSS(\BPtwo)$, the first differential still exists by Hill--Hopkins--Ravenel's Slice Differential Theorem \cite[Theorem~9.9]{HHR}.  To prove that the second differential exists as well, consider again the map 
$$\SliceSS(\BPtwo) \longrightarrow \SliceSS(\BPone).$$
For natuality reasons, $u_{2\lambda}$ must support a differential of length at most 5 in $\SliceSS(\BPtwo)$.  Since $u_{\lambda}$ supports a nonzero $d_3$-differential, $u_{2\lambda}$ is a $d_3$-cycle.  This implies that $u_{2\lambda}$ must support a $d_5$-differential whose target maps to $\done u_\lambda a_{2\lambda} a_\sigma$ under the quotient map (which sends $\rthree$ and $\gamma \rthree$ to zero).  It follows that the only possible target is $\done u_\lambda a_{2\lambda} a_\sigma$, and the same $d_5$-differential on $u_{2\lambda}$ exists in $\SliceSS(\BPtwo)$. 

All the predicted $d_5$ differentials in $\SliceSS(\BPtwo)$ are generated by these two differentials.  

It remains to show that there are no interfering $d_5$-differentials.  There are two cases to consider: 

\noindent (1) \textit{The source is in a $i_{C_2}^* \BPone$-truncation.}  Every class in a $i_{C_2}^*\BPone$-truncation is in the image of the transfer map 
$$C_2\text{-}\SliceSS(i_{C_2}^*\BPtwo) \stackrel{tr}{\longrightarrow} C_4\text{-}\SliceSS(\BPtwo).$$ 
On the $E_5$-page of $C_2\text{-}\SliceSS(i_{C_2}^*\BPtwo)$, every class is a $d_5$-cycle because there are no $d_5$-differentials.  Therefore after applying the transfer map, all the images must be $d_5$-cycles as well.   

\noindent (2) \textit{The source is in a $\BPone$-truncation.}  If the source is in the image of the transfer, then by the same reasoning as above, it must be a $d_5$-cycle.  If the source is not in the image of the transfer, then it can be written as $\dthree^i \done^j u_\lambda^au_\sigma^b a_\lambda^c a_\sigma^d$ for some $i, j, a, b, c, d \geq 0$.  The only possibilities are the {\color{blue} blue classes} in Figure~\ref{fig:d5Interfering}.  These classes might support $d_5$-differentials whose targets are classes in $i_{C_2}^*\BPone$-truncations.  However, using the differentials $d_5(u_{2\lambda}) = \done u_\lambda a_{2\lambda} a_\sigma$ and $d_5(u_{2\sigma}) = \done a_\lambda a_{3\sigma}$, we can easily show that all of these classes are $d_5$-cycles.  
\begin{figure}
\begin{center}
\makebox[\textwidth]{\includegraphics[trim={0cm 5cm 0cm 3cm}, clip, page = 1, scale = 0.9]{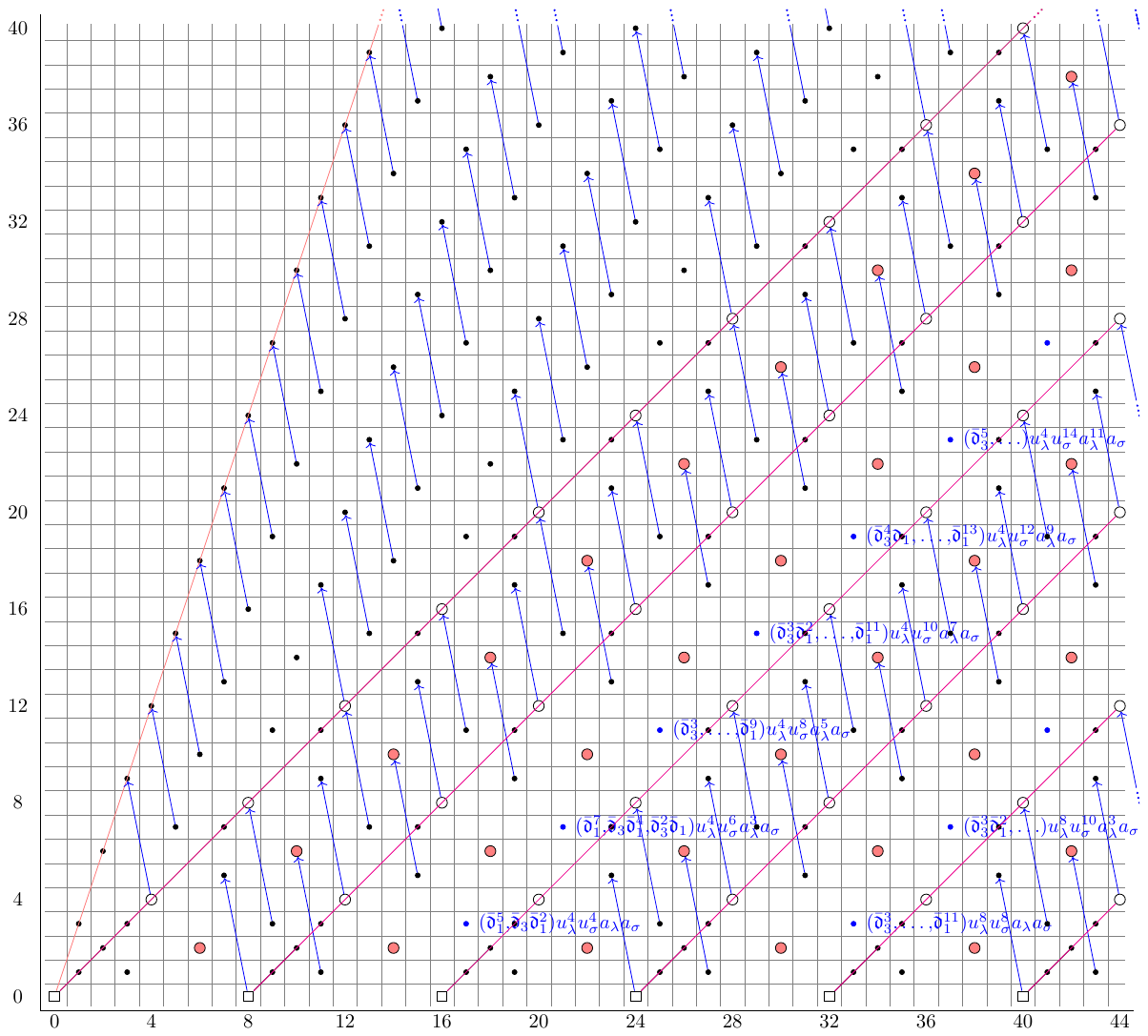}}
\end{center}
\begin{center}
\caption{Possible sources in $\BPone$-truncations that could support $d_5$-interfering differentials.  The {\color{magenta} magenta} lines indicate the locations of the classes in $i_{C_2}^*\BPone$-truncations. }
\hfill
\label{fig:d5Interfering}
\end{center}
\end{figure}

%%%%
\subsection{$d_7$-differentials} \label{subsec:d7induced}
In the slice spectral sequence for $\BPone$, the $d_7$-differentials are generated under multiplicative structure by three differentials: 
\begin{enumerate}
\item $d_7(2u_{2\lambda}) = \done \sone u_\sigma a_{3\lambda} a_{\sigma_2}$;
\item $d_7(2u_{2\lambda}u_{2\sigma}) = \done \sone u_{3\sigma} u_{3\lambda} a_{\sigma_2}$; 
\item $d_7(u_{4\lambda}) = \done \sone u_{2\lambda} u_{2\sigma} a_{3\lambda} a_{\sigma_2}$. 
\end{enumerate}
Using the natuality of the quotient map 
$$\SliceSS(\BPtwo) \longrightarrow \SliceSS(\BPone),$$
we deduce that the classes $2u_{2\lambda}$, $2u_{2\lambda}u_{2\sigma}$, and $u_{4\lambda}$ must all support differentials of length at most 7 in $\SliceSS(\BPtwo)$.  The formulas for the $d_5$-differentials on $u_{2\lambda}$ and $u_{2\sigma}$ imply that all three classes above are $d_5$-cycles.  Therefore, they must all support $d_7$-differentials.  It follows by natuality that we have the exact same $d_7$-differentials in $\SliceSS(\BPtwo)$.  These differentials generate the predicted $d_7$-differentials in all the $\BPone$-truncations.  

All of the predicted $d_7$-differentials in $i_{C_2}^*\BPone$-truncations are obtained by using the transfer map 
$$C_2\text{-}\SliceSS(\BPtwo) \stackrel{tr}{\longrightarrow} C_4\text{-}\SliceSS(\BPtwo).$$
More precisely, the transfer map takes in a $d_7$-differential in $\BPtwo$, which is generated by $d_7(u_{4\sigma_2}) = (\rthree + \gamma \rthree+\rone^3)a_{7\sigma_2} = (\sthree + \rone^3) a_{7\sigma_2}$, and produces a corresponding $d_7$-differential in a $i_{C_2}^*\BPone$-truncation.    

Note that in the $C_2$-slice spectral sequence for $\BPone$, the $d_7$-differentials are generated by $d_7(u_{4\sigma_2}) = \rone^3 a_{7\sigma_2}$, whereas in the $C_2$-slice spectral sequence for $\BPtwo$, they are generated by $d_7(u_{4\sigma_2}) = (\sthree + \rone^3) a_{7\sigma_2}$.  The readers should be warned that strictly speaking, the $d_7$-differentials are \textit{not} appearing independently within each $i_{C_2}^*\BPone$-truncations, but rather identifying classes between different $i_{C_2}^*\BPone$-truncations.  The exact formulas for this identification will be discussed in Section~\ref{sec:higherDifferentialsI}.  Nevertheless, since the leading terms are independent, the $d_7$-differentials do occur independently within each $i_{C_2}^*\BPone$-truncation.  

It remains to prove that there are no interfering $d_7$-differentials.  There are two cases to consider.  

\noindent (1) \textit{The source is in the image of the transfer} (in other words, the source is produced by an induced slice cell).  Denote the source by $tr(x)$, where $x$ is a class in the $C_2$-slice spectral sequence.  If $tr(x)$ supports a $d_7$-differential in the $C_4$-slice spectral sequence, then natuality of the transfer map implies that in the $C_2$-slice spectral sequence, $x$ must support a differential of length at most 7.  This means that $x$ either supports a $d_3$-differential or a $d_7$-differential.  

If $x$ supports a $d_3$-differential $d_3(x) = y$, then since the transfer map is faithful on the $E_3$-page, applying the transfer to this $d_3$-differential yields the nontrivial $d_3$-differential 
$$d_3(tr(x)) = tr(y)$$
in the $C_4$-slice spectral sequence.  This is a contradiction to the assumption that $d_7(tr(x)) \neq 0$.  

Therefore, $x$ must support a $d_7$-differential $d_7(x) = y$ in the $C_2$-slice spectral sequence.  Applying the transfer map to this $d_7$-differential gives $d_7(tr(x)) = tr(y)$, which must be the $d_7$-differential on $tr(x)$ by natuality.  However, this will not be an interfering $d_7$-differential because it is a predicted $d_7$-differential that is obtained via the transfer. 

\begin{exam}\rm
In Figure~\ref{fig:d7Interfering}, there is a possibility for a $d_7$-interfering differential with source a class at $(11,3)$ coming from a $\BPone$-truncation (it is supposed to support a predicted $d_{11}$-differential), and the target a class at $(10,10)$ coming from $i_{C_2}^*\BPone$-truncations (a pink class).  

The two possible sources at $(11,3)$ are 
$\dthree \sone u_{2\lambda}u_{3\sigma}a_\lambda a_{\sigma_2} = tr(\rthree \gamma \rthree \rone u_{4\sigma_2} a_{3\sigma_2})$ and $\done^3 \sone u_{2\lambda}u_{3\sigma}a_\lambda a_{\sigma_2} = tr(\rone^4 \gamma \rone^3 u_{4\sigma_2} a_{3\sigma_2})$.  By the discussion above, if any of these two classes support a $d_7$-differential hitting a class in the image of the transfer, then this differential must be obtained by applying the transfer map to a $d_7$-differential in the $C_2$-slice spectral sequence.  

In the $C_2$-slice spectral sequence, the relevant differentials are the following:
\begin{eqnarray*}
d_7(\rthree \gamma \rthree \rone u_{4\sigma_2} a_{3\sigma_2}) &=& \rthree \gamma \rthree \rone (\rthree + \grthree + \rone^3) a_{11\sigma_2} \\
d_7(\rone^4 \gamma \rone^3 u_{4\sigma_2} a_{3\sigma_2}) &=& \rone^4 \gamma \rone^3 (\rthree + \grthree + \rone^3) a_{11\sigma_2}.
\end{eqnarray*}
The transfer of the targets are $(\dthree \sthree \sone + \dthree \sone^4) a_{3\lambda}a_{4\sigma_2}$ and $(\done^3 \sthree \sone + \done^3 \sone^4) a_{3\lambda}a_{4\sigma_2}$, respectively.  They are both 0 on the $E_7$-page because they are targets of $d_3$-differentials.  It follows that the $d_7$-interfering differentials do not occur at $(11,3)$.  The same argument also shows that there are no $d_7$-interfering differentials with sources at $(19, 11)$, $(23, 15)$, $(31, 7)$, $\ldots$. 
\end{exam}

\noindent (2) \textit{The source is not in the image of the transfer} (in other words, the source is produced by a regular, non-induced slice cell).  As shown in Figure~\ref{fig:d7Interfering}, for degree reasons, there are possible interfering $d_7$-differentials with sources at 
\begin{enumerate}
\item $(20, 4)$, $(28, 12)$, $(36, 20)$, $\ldots$;
\item $(32,0)$, $(40,8)$, $(48, 16)$, $\ldots$; 
\item $(52, 4)$, $(60, 12)$, $(68, 20)$, $\ldots$; \\
$\ldots$.
\end{enumerate}
To prove that these $d_7$-differentials do not exist, it suffices to prove that all the classes at $(20, 4)$ are $d_7$-cycles.  Once we prove this, all the other possible sources above will be $d_7$-cycles as well by multiplicative reasons.  

The quotient map $\SliceSS(\BPtwo) \longrightarrow \SliceSS(\BPone)$ shows that both classes at $(11,3)$, $\done^3\sone u_{2\lambda}u_{3\sigma} a_\lambda a_{\sigma_2}$ and $\dthree\sone u_{2\lambda}u_{3\sigma} a_\lambda a_{\sigma_2}$, must support nontrivial $d_{11}$-differentials.  Multiplication by the permanent cycles at $(8, 8)$ ($\dthree \done u_{4\sigma} a_{4\lambda}$ and $\done^4 u_{4\sigma} a_{4\lambda}$) implies that all three classes at $(19, 11)$ (coming from the slice cells $\dthree^2 \done \sone$, $\dthree \done^4 \sone$, and $\done^7 \sone$) must support nontrivial $d_{11}$-differentials.  In fact, these are the predicted $d_{11}$-differentials.  

There are three classes at $(20, 4)$: $\dthree^2 u_{6\sigma}u_{4\lambda}a_{2\lambda}$, $\dthree \done^3u_{6\sigma}u_{4\lambda}a_{2\lambda}$, and $\done^6u_{6\sigma}u_{4\lambda}a_{2\lambda}$.  If any of these classes supports a nontrivial $d_7$-differential, the target would be a classes at $(19, 11)$, which, as we have shown in the previous paragraph, supports a nontrivial $d_{11}$-differential.  This is a contradiction because something killed on the $d_7$-page becomes trivial on the $d_{11}$-page, and cannot support a nontrivial $d_{11}$-differential (the target of that $d_{11}$-differential is not killed by a shorter differential and is on the $E_{11}$-page).

\begin{figure}
\begin{center}
\makebox[\textwidth]{\includegraphics[trim={0cm 10cm 0cm 0cm}, clip, page = 1, scale = 0.9]{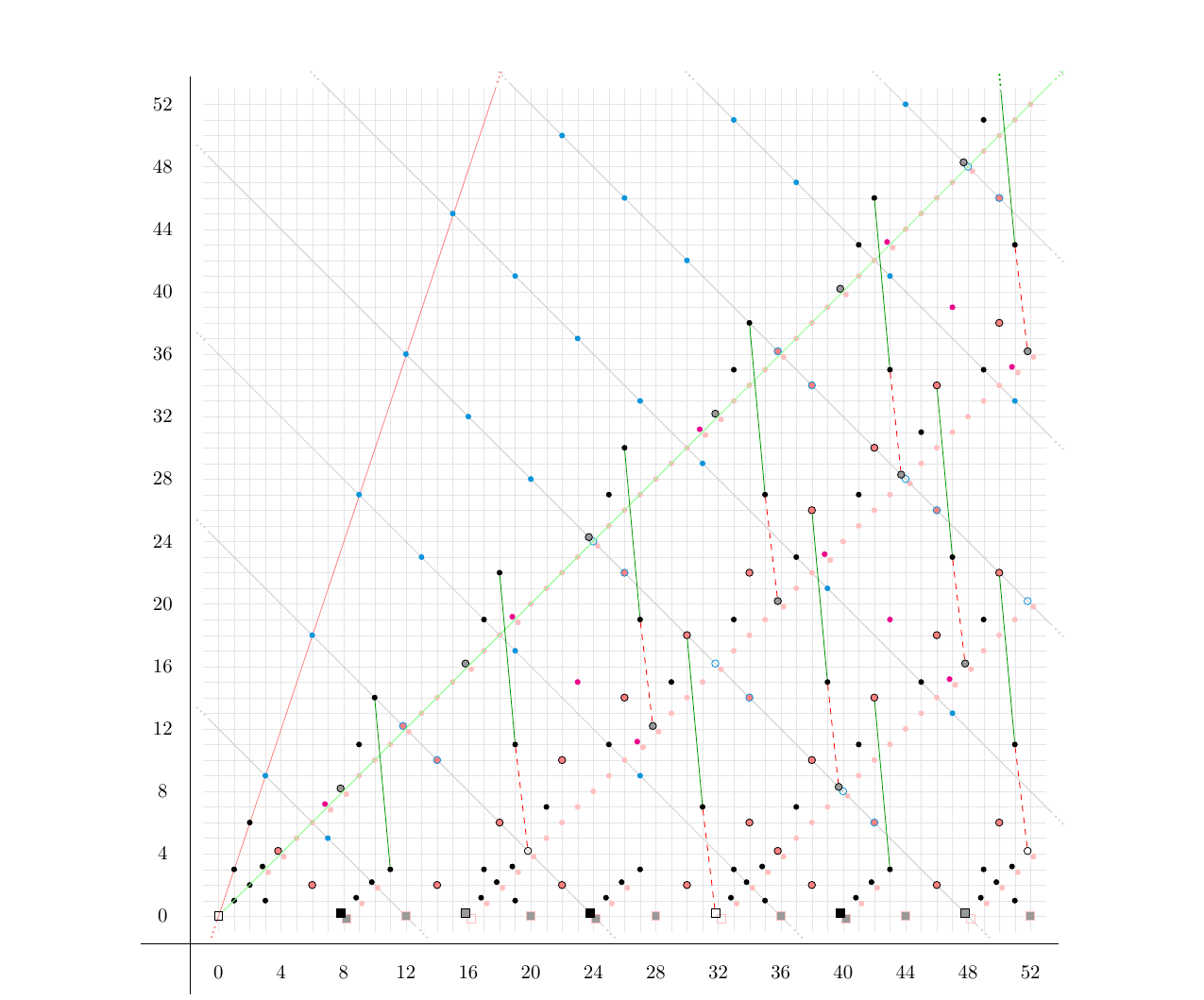}}
\end{center}
\begin{center}
\caption{The {\color{red}dashed red lines} are the possible $d_7$-interfering differentials.  The {\color{cyan} cyan classes} are the $d_5$-truncation classes, the {\color{magenta} magenta classes} are the $d_7$-truncation classes, and the {\color{pink} pink classes} are classes in $i_{C_2}^*\BPone$-truncations after the predicted $d_7$-differentials.  The {\color{green!60!black}green differentials} are the predicted $d_{11}$-differentials which all occur. }
\hfill
\label{fig:d7Interfering}
\end{center}
\end{figure}

%%%%
\subsection{$d_{11}$-differentials}\label{subsec:d11induced}
For degree reasons, there are no possible $d_9$-differentials.  The next possible differentials are the $d_{11}$-differentials.  

In the slice spectral sequence for $\BPone$, all the $d_{11}$-differentials are generated by the single $d_{11}$-differential
$$d_{11}(\sone u_{2\lambda} u_{3\sigma}a_{\sigma_2} ) = \done^3 u_{4\sigma} a_{5\lambda} a_{2\sigma}$$
under multiplication.  Using the quotient map 
$$\SliceSS(\BPtwo) \longrightarrow \SliceSS(\BPone),$$
we deduce that the class $\sone u_{2\lambda} u_{3\sigma}a_{\sigma_2}$ must support a differential of length at most 11 in $\SliceSS(\BPtwo)$.  

Our knowledge of the earlier differentials implies that this class is a $d_r$-cycle for $r \leq 10$, and hence it must support a $d_{11}$-differential.  Furthermore, the formula of the $d_{11}$-differential is of the form 
$$d_{11}(\sone u_{2\lambda} u_{3\sigma}a_{\sigma_2} ) = \done^3 u_{4\sigma} a_{5\lambda} a_{2\sigma} + \cdots,$$
where ``$\cdots$'' indicates terms that go to 0 under the quotient map (which sends $\bar{r}_3, \gamma \bar{r}_3 \mapsto 0$).  All of the predicted $d_{11}$-differentials are obtained using this $d_{11}$-differential under multiplication. 

Similar to situation of the $d_7$-differentials, strictly speaking, the $d_{11}$-differentials do not necessarily occur within each $\BPone$-truncation.  For instance, in the formula above, the ``$\cdots$" could be $\dthree u_{4\sigma} a_{5\lambda} a_{2\sigma}$.  If this happens, the $d_{11}$-differential would be identifying the two classes, $\done^3 u_{4\sigma} a_{5\lambda} a_{2\sigma}$ and $\dthree u_{4\sigma} a_{5\lambda} a_{2\sigma}$, which are located in different $\BPone$-truncations.  Given this, we can kill off the leading term and assume that the rest of the terms remain.  This will give us the same distribution of classes after the $d_{11}$-differentials and will not affect later computations.  

It remains to show that there are no $d_{11}$-interfering differentials.  Figure~\ref{fig:d11Interfering} shows all the possible $d_{11}$-interfering differentials.  We will prove that none of them exist.  

\noindent (1) {\color{blue} Blue differentials}.  These differentials have sources at 
\begin{itemize}
\item $\{(27,11), (39,23), (51,35), \ldots \}$;
\item $\{(35,3), (47,15), (59,27), \ldots \}$;
\item $\{(55,7), (67,19), (79,31), \ldots \}$;
\item $\{(75,11), (87,23), (99,35), \ldots \}$;
\item $\cdots$. 
\end{itemize}
The sources of these differentials are in the image of the transfer map.  Their pre-images in the $C_2$-slice spectral sequence are all $d_{11}$-cycles (more specifically, they all support differentials of length at least 15).  Therefore, their images under the transfer map cannot support nontrivial $d_{11}$-differentials.  

\noindent (2) {\color{gray} Gray differentials}.  These differentials have sources at 
\begin{itemize}
\item $\{(43,19), (67,43), (91,67), \ldots \}$;
\item $\{(63,23), (87,47), (111,71), \ldots \}$;
\item $\{(59,3), (83,27), (107,51), \ldots \}$;
\item $\{(79,7), (103,31), (127,55), \ldots \}$;
\item $\{(99,11), (123,35), (147,59), \ldots \}$;
\item $\{(119,15), (143,39), (167,63), \ldots \}$;
\item $\{(139,19), (163,43), (187,67), \ldots \}$;
\item $\cdots$. 
\end{itemize}
Each of the sources is a $d_7$-truncation class.  If any of these differentials exist, we will obtain a contradiction when we multiply this differential by the classes at $(8, 8)$ (either $\dthree \done u_{4\sigma}a_{4\lambda}$ or $\done^4 u_{4\sigma}a_{4\lambda}$).  

For example, suppose the class at $(43, 19)$ supports a nontrivial $d_{11}$-differential.  The target (a class at $(42,30)$), when multiplied by the class $\done^4 u_{4\sigma}a_{4\lambda}$, is a nonzero class at $(50,38)$.  The source, however, becomes 0.  This is a contradiction.  

\noindent (3) Black differentials.  These differentials have sources at 
\begin{itemize}
\item $\{(18, 6), (26, 14), (34, 22), \ldots\}$;
\item $\{(30, 2), (38, 10), (46, 18), \ldots\}$;
\item $\{(50, 6), (58, 14), (66, 22), \ldots\}$;
\item $\{(62, 2), (70, 10), (78, 18), \ldots\}$;
\item $\cdots$.
\end{itemize}
It suffices to show that all of the classes in the first set are $d_{11}$-cycles.  Once we have proven this, multiplication by the class $u_{8\lambda}u_{8\sigma}$ ($d_{11}$-cycle), the three classes at $(20, 4)$ ($(\dthree^2, \dthree \done^3, \done^6) u_{4\lambda} u_{6\sigma} a_{2\lambda}$, all $d_{11}$-cycles), and the two permanent cycles at $(8, 8)$ ($(\dthree \done, \done^4)u_{4\sigma}a_{4\lambda}$) will show that all the other classes are $d_{11}$-cycles as well. 

Now, for the first set, the names of the classes at each of the possible sources are as follows: 
\begin{itemize}
\item $(18, 6)$: $2(\dthree^2, \dthree \done^3, \done^6) u_{3\lambda} u_{6\sigma} a_{3\lambda} = (\dthree^2, \dthree \done^3, \done^6) u_{4\lambda} u_{4\sigma} a_{2\lambda} a_{2\sigma}$
\item $(26,14)$: $2(\dthree^3\done, \ldots, \done^{10}) u_{3\lambda} u_{10\sigma} a_{7\lambda}= (\dthree^3\done, \ldots, \done^{10}) u_{4\lambda} u_{8\sigma} a_{6\lambda} a_{2\sigma}$
\item $(34,22)$: $2(\dthree^4\done^2, \ldots, \done^{14}) u_{3\lambda} u_{14\sigma} a_{11\lambda}= (\dthree^4\done^2, \ldots, \done^{14}) u_{4\lambda} u_{12\sigma} a_{10\lambda} a_{2\sigma}$
\item $\cdots$. 
\end{itemize} 
The names can all be written as products of the following $d_{11}$-cycles: $\done$, $\dthree$, $a_\lambda$, $a_\sigma$, $u_{4\lambda} a_\sigma$ (supports $d_{13}$-differential), and $u_{4\sigma}$ (supports $d_{13}$-differential).  Therefore, there are no $d_{11}$-interfering differentials in this case. 

\noindent (4) {\color{red} Red differentials}.  These differentials have sources at 
\begin{itemize}
\item $\{(14, 2), (22, 10), (30, 18), \ldots \}$;
\item $\{(34, 6), (42, 14), (50, 22), \ldots \}$;
\item $\{(46, 2), (54, 10), (62, 18), \ldots \}$;
\item $\{(66, 6), (74, 14), (82, 22), \ldots \}$;
\item $\cdots$.
\end{itemize}
Similar to (3), it suffices to show that all of the classes in the first set are $d_{11}$-cycles.  Afterwards, all the other classes can be proven to be $d_{11}$-cycles via multiplication by the class $u_{8\lambda}u_{8\sigma}$, the classes at $(20, 4)$, and the classes at $(8, 8)$ (all of which are $d_{11}$-cycles).  

Now, for the first set, the classes at $(22, 10)$, $(30, 18)$, $(38, 26)$, $\ldots$ are all $d_{11}$-cycles because they can be written as products of classes at $(20, 4)$ and classes at $(2, 6)$, $(10, 14)$, $(18, 22)$, $\ldots$ (all of which are $d_{11}$-cycles).  Afterwards, we deduce that the classes at $(14, 2)$ are $d_{11}$-cycles as well because if they are not, then multiplying the $d_{11}$-differential by the classes at $(8, 8)$ would produce a nontrivial $d_{11}$-differential on the classes at $(22,10)$.  This is a contradiction because we have just proven that all the classes at $(22,10)$ are $d_{11}$-cycles. 

\begin{figure}
\begin{center}
\makebox[\textwidth]{\includegraphics[trim={0cm 10cm 0cm 0cm}, clip, page = 1, scale = 0.9]{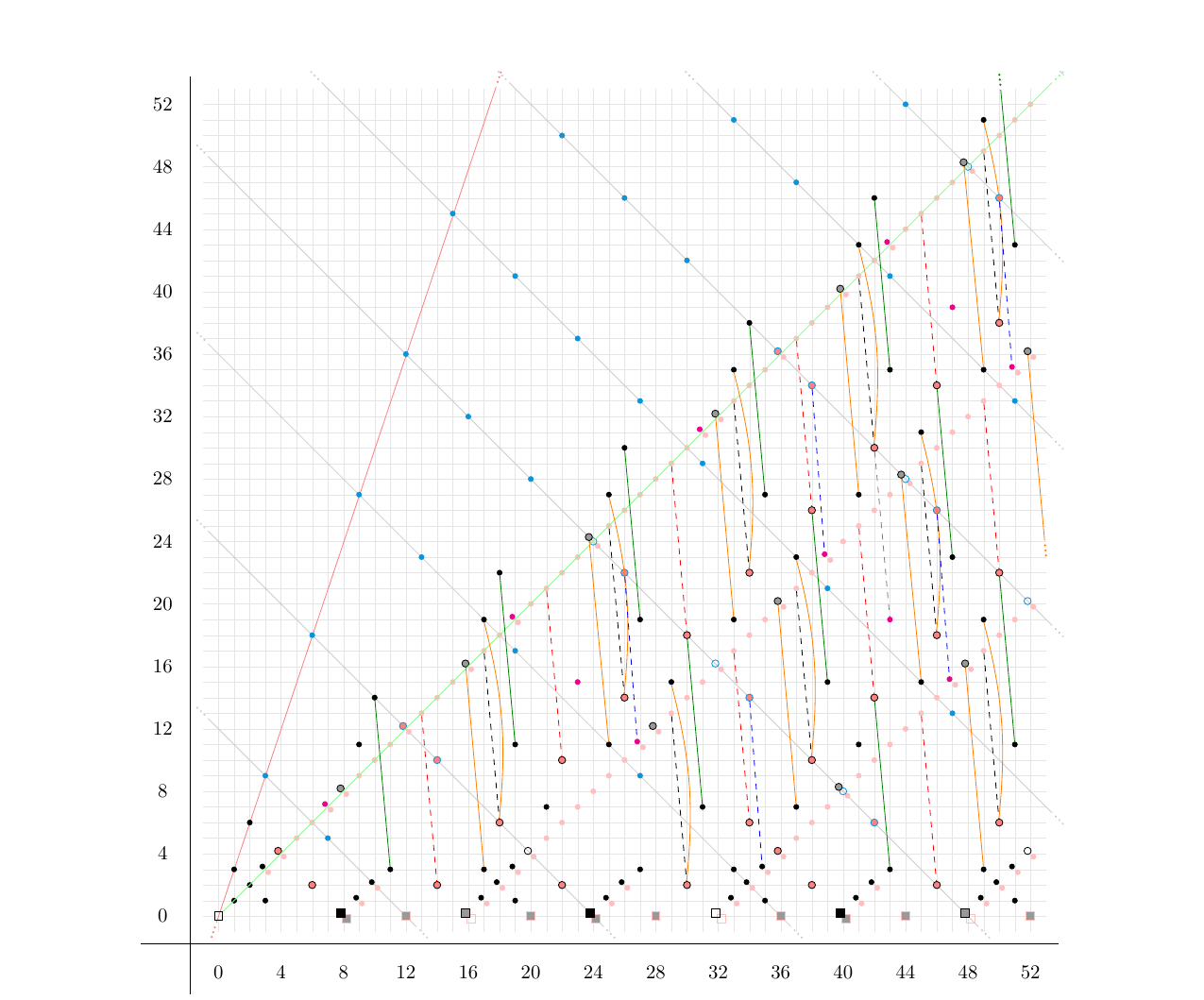}}
\end{center}
\begin{center}
\caption{The dashed lines are the possible $d_{11}$-interfering differentials.  The {\color{cyan} cyan classes} are the $d_5$-truncation classes, the {\color{magenta} magenta classes} are the $d_7$-truncation classes, and the {\color{pink} pink classes} are classes in $i_{C_2}^*\BPone$-truncations after the predicted $d_7$-differentials.  The {\color{green!60!black} green differentials} are the predicted $d_{11}$-differentials and the {\color{orange} orange differentials} are the predicted $d_{13}$-differentials. }
\hfill
\label{fig:d11Interfering}
\end{center}
\end{figure}

%%%%
\subsection{Predicted $d_{13}$-differentials} \label{subsec:Predictedd13Diff}
In the slice spectral sequence of $\BPone$, all the $d_{13}$-differentials are generated by $d_{13}(u_{4\lambda}a_\sigma) = \done^3 u_{4\sigma}a_{7\lambda}$ under multiplication.  This differential was proven by applying the norm formula (see Theorem~\ref{thm:NormFormula}, Theorem~\ref{thm:BPoned13} and Corollary~\ref{cor:BPoned13}).  In fact, we can also prove this differential in $\SliceSS(\BPtwo)$ by using the norm formula, and we will do so in Section~\ref{sec:Norm} when we discuss the norm in depth.  

Alternatively, we can analyze the quotient map 
$$\SliceSS(\BPtwo) \longrightarrow \SliceSS(\BPone) $$
again.  Since $u_{4\lambda}a_\sigma$ supports a $d_{13}$-differential in $\SliceSS(\BPone)$, it must support a differential of length at most 13 in $\SliceSS(\BPtwo)$.  Our knowledge of the earlier differentials implies that $u_{4\lambda}a_\sigma$ must be a $d_{11}$-cycle, and hence must support a $d_{13}$-differential.  More specifically, we can deduce this fact by analyzing the class $\dthree^3 u_{4\lambda} u_{8\sigma} a_{5\lambda}a_\sigma$ at $(25, 11)$.  If $u_{4\lambda}a_\sigma$ supports a $d_r$-differential of length $r < 13$, then $\dthree^3 u_{4\lambda} u_{8\sigma} a_{5\lambda}a_\sigma$ must support a $d_r$-differential as well, which is impossible by degree reasons.  

Since the $d_{13}$-differential on $u_{4\lambda}a_\sigma$ respects natuality under the quotient map, it must be of the form 
$$d_{13}(u_{4\lambda}a_\sigma) = \done^3 u_{4\sigma}  a_{7\lambda} + \cdots,$$
where ``$\cdots$'' denote terms that go to 0 under the quotient map sending $\rthree, \grthree \mapsto 0$ (in particular, it could contain $\dthree u_{4\sigma}a_{7\lambda}$, as we will see in Section~\ref{sec:Norm}).  All the predicted $d_{13}$-differentials are generated by this differential under multiplication.  

Similar to the cases for $d_7$ and $d_{11}$-differentials, the readers should be warned that the $d_{13}$-differentials are not necessarily occurring within each $\BPone$-truncation.  The above formula identifies the leading term, $\done^3 u_{4\sigma}a_{7\lambda}$, with the rest of the terms (possibly none).  Therefore, we can kill off the leading term and assume that the rest of the terms remain.

%there is a subtlety of whether the $d_{13}$-differentials are occurring within each $\BPone$-truncation.  Just as before, as long as the leading term is $\done^3 u_{4\sigma}  a_{7\lambda}$, the distribution of the classes and the later computations will not be affected.

%%%%
\subsection{$E_{13}$-page of $\SliceSS(\BPtwo)$}
Figure~\ref{fig:BPtwoE13} shows the $E_{13}$-page of $\SliceSS(\BPtwo)$ with the predicted $d_{13}$-differentials already taken out.  The truncation classes are color coded as follows:  
\begin{enumerate}
\item {\color{cyan!80!blue} Cyan classes}: $d_5$-truncation classes; 
\item {\color{magenta} Magenta classes}: $d_7$-truncation classes; 
\item {\color{green!80!black} Green classes}: $d_{11}$-truncation classes;
\item {\color{orange!80} Orange classes}: $d_{13}$-truncation classes; 
\item {\color{pink} Pink classes}: $i_{C_2}^*\BPone$-truncation classes. 
\end{enumerate}

%\begin{landscape}
\begin{figure}
\begin{center}
\makebox[\textwidth]{\includegraphics[trim={0cm 9cm 0cm 9cm}, clip, page = 1, scale = 0.45]{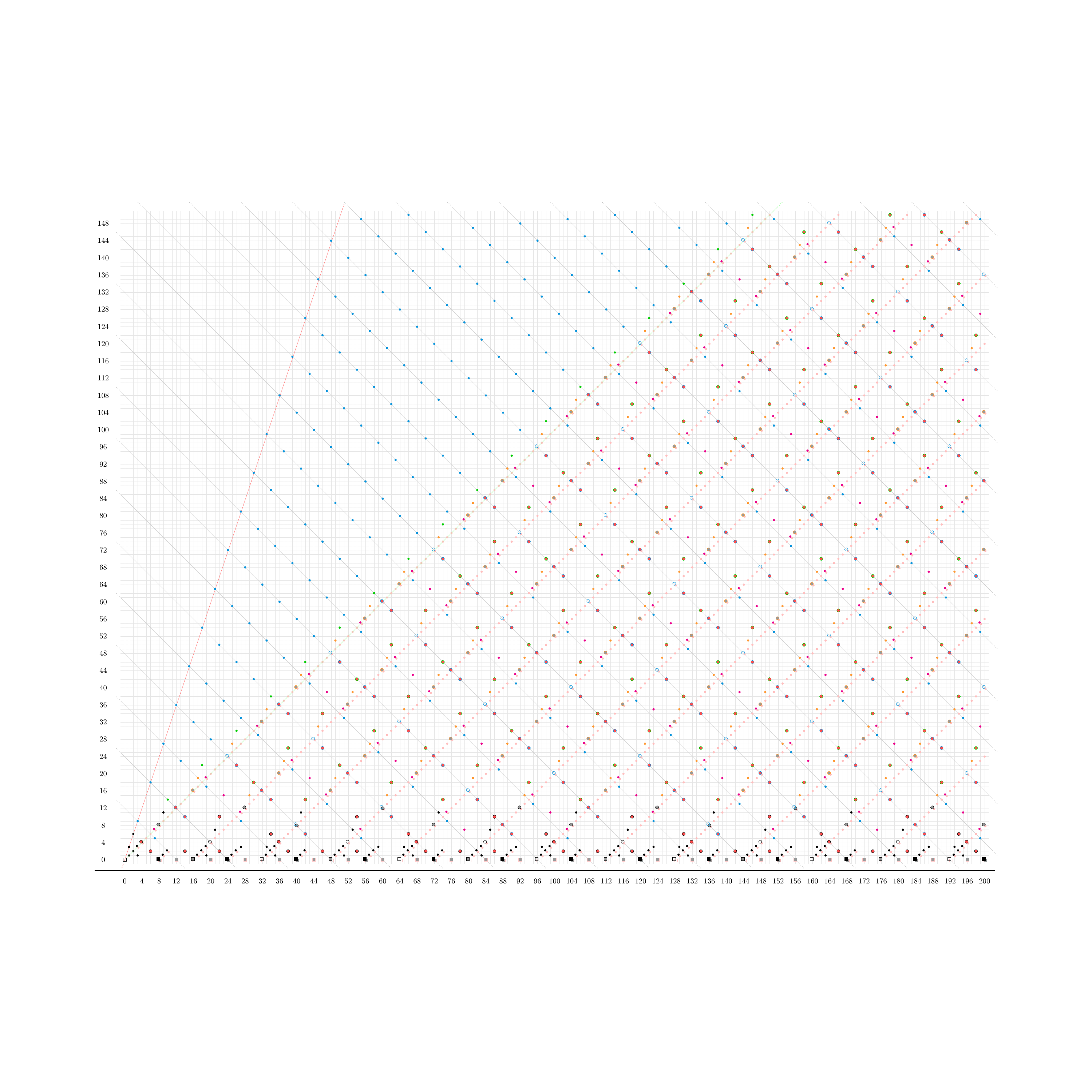}}
\end{center}
\begin{center}
\caption{The $E_{13}$-page of $\SliceSS(\BPtwo)$ with the predicted $d_{13}$-differentials already taken out.  The {\color{cyan!80!blue} cyan classes} are $d_5$-truncation classes, the {\color{magenta} magenta classes} are $d_7$-truncation classes, the {\color{green!80!black} green classes} are $d_{11}$-truncation classes, the {\color{orange!80} orange classes} are $d_{13}$-truncation classes, and the {\color{pink} pink classes} are $i_{C_2}^*\BPone$-truncation classes.}
\hfill
\label{fig:BPtwoE13}
\end{center}
\end{figure}
%\end{landscape}

%%%%%%%%%%%%%%%%%%%%%%%%
%%%%%%%%%%%%%%%%%%%%%%%%
\section{Higher differentials I: $d_{13}$ and $d_{15}$-differentials} \label{sec:higherDifferentialsI}
In this section, we prove all the $d_{13}$ and $d_{15}$-differentials in the slice spectral sequence of $\BPtwo$, as well as differentials between $i_{C_2}^*\BPone$-truncation classes.  

%%%%%%%%%%%%%%%
\subsection{$d_{13}$-differentials}
\begin{prop}\label{prop:d13slice}
The class $u_{4\sigma}$ supports the $d_{13}$-differential 
$$d_{13}(u_{4\sigma}) = \dthree a_{3\lambda} a_{7\sigma}.$$
\end{prop}
\begin{proof}
This is an immediate application of Hill--Hopkins--Ravenel's Slice Differential Theorem~\cite[Theorem 9.9]{HHR}.
\end{proof}

The $d_{13}$-differential in Proposition~\ref{prop:d13slice} generates all the $d_{13}$-differentials between the line of slope 1 and the line of slope 3 under multiplication (see Figure~\ref{fig:BPtwod13}).  

\begin{prop}\label{prop:BPtwod13(20,4)}
The class $\dthree^2 u_{4\lambda} u_{6\sigma}a_{2\lambda}$ at $(20, 4)$ supports the $d_{13}$-differential 
$$d_{13}(\dthree^2 u_{4\lambda} u_{6\sigma}a_{2\lambda}) = \dthree^3 u_{\lambda} u_{8\sigma}a_{8\lambda}a_\sigma.$$
\end{prop}
\begin{proof}
We will prove this differential by using the restriction map 
$$res: C_4\text{-}\SliceSS(\BPtwo) \longrightarrow C_2\text{-}\SliceSS(i_{C_2}^*\BPtwo).$$
The restriction of $\dthree^2 u_{4\lambda} u_{6\sigma}a_{2\lambda}$ is $\rthree^2 \grthree^2 u_{8\sigma_2}a_{4\sigma_2}$.  In the $C_2$-slice spectral sequence, this class supports the $d_{15}$-differential 
\begin{eqnarray*}
d_{15}( \rthree^2 \grthree^2 u_{8\sigma_2}a_{4\sigma_2}) &=& \rthree^2 \grthree^2 \rone (\rthree^2 + \rthree \grthree + \grthree^2) a_{19\sigma_2} \\
&=& \rthree^3 \grthree^3 \rone a_{19\sigma_2} + \rthree^2 \grthree^2 (\rthree^2 + \grthree^2) \rone a_{19\sigma_2}. 
\end{eqnarray*}
This implies that in the $C_4$-slice spectral sequence, $\dthree^2 u_{4\lambda} u_{6\sigma}a_{2\lambda}$ must support a differential of length at most 15.  

If $d_{15}(\dthree^2 u_{4\lambda} u_{6\sigma}a_{2\lambda}) = x$, then by natuality,
$$res(x) = \rthree^3 \grthree^3 \rone a_{19\sigma_2} + \rthree^2 \grthree^2 (\rthree^2 + \grthree^2) \rone a_{19\sigma_2}.$$
This is impossible because while the class $\rthree^2 \grthree^2 (\rthree^2 + \grthree^2) \rone a_{19\sigma_2}$ has a pre-image on the $E_{15}$-page ($\dthree^2 \sthree^2 \rone a_{6\lambda}a_{7\sigma_2}$), the class $\rthree^3 \grthree^3 \rone a_{19\sigma_2}$ does not.  

Therefore, the class $\dthree^2 u_{4\lambda} u_{6\sigma}a_{2\lambda}$ must support a $d_{13}$-differential.  There is one possible target, which is the class $\dthree^3 u_\lambda u_{8\sigma} a_{8\lambda}a_\sigma$ at $(19, 17)$.  This proves the desired differential.  
\end{proof}

Consider the following classes:
\begin{enumerate}
\item $\dthree u_{\lambda}u_{2\sigma}a_{2\lambda}a_\sigma$ at $(7,5)$.  This class is a permanent cycle by degree reasons.  
\item $\dthree^3 u_\lambda u_{8\sigma} a_{8\lambda}a_\sigma$ at $(19,17)$.  This class is a permanent cycle (it is the target of the $d_{13}$-differential in Proposition~\ref{prop:BPtwod13(20,4)}).  
\item $\dthree^3u_{8\sigma}a_{9\lambda}a_\sigma$ at $(17,19)$.  This class is a permanent cycle by degree reasons.  
\item $\dthree^4 u_{12\sigma}a_{12\lambda}$ at $(24,24)$.  This class supports the $d_{13}$-differential $d_{13}(\dthree^4 u_{12\sigma}a_{12\lambda}) = \dthree^5 u_{8\sigma}a_{15\lambda}a_{7\sigma}$. 
\end{enumerate} 
Using the Leibniz rule on the differential in Proposition~\ref{prop:BPtwod13(20,4)} with the classes above produces all the $d_{13}$-differentials under the line of slope 1 (see Figure~\ref{fig:BPtwod13}).

\begin{figure}
\begin{center}
\makebox[\textwidth]{\includegraphics[trim={0cm 9cm 0cm 9cm}, clip, page = 1, scale = 0.45]{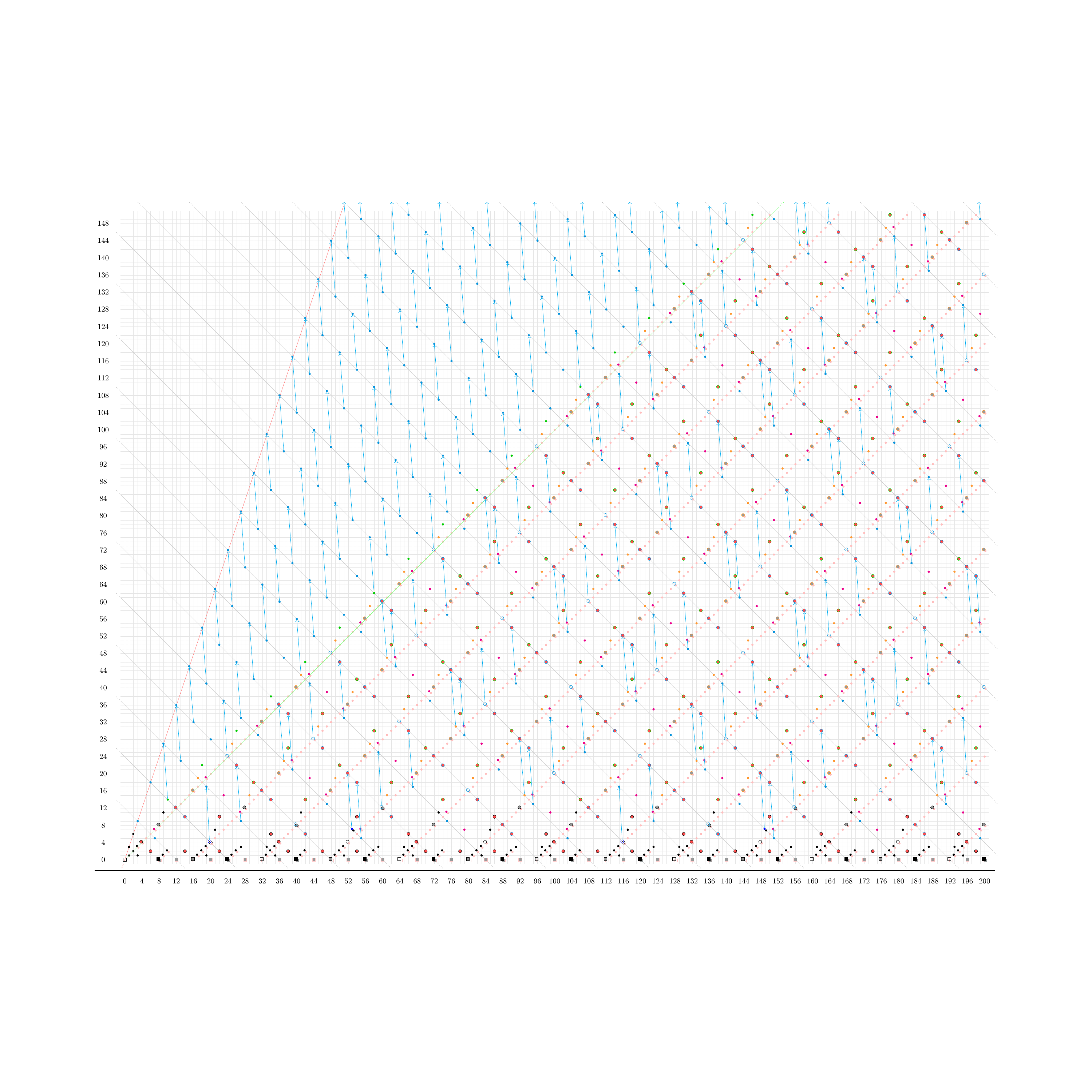}}
\end{center}
\begin{center}
\caption{$d_{13}$-differentials in $\SliceSS(\BPtwo)$.}
\hfill
\label{fig:BPtwod13}
\end{center}
\end{figure}

%%%%%%%%%%%%%%%
\subsection{Differentials between $i_{C_2}^*\BPone$-truncation classes} \label{subsection:inducedC2Differentials}
Using the restriction map
$$res:C_4\text{-}\SliceSS(\BPtwo) \longrightarrow C_2\text{-}\SliceSS(i_{C_2}^*\BPtwo)$$
and the transfer map 
$$tr: C_2\text{-}\SliceSS(i_{C_2}^*\BPtwo) \longrightarrow C_4\text{-}\SliceSS(\BPtwo),$$
we can prove all the $d_3$, $d_7$, $d_{15}$, and $d_{31}$-differentials between $i_{C_2}^*\BPone$-truncation classes ({\color{pink} pink classes}).  

The general argument goes as follows: suppose $tr(a)$ and $tr(b)$ are two $i_{C_2}^*\BPone$-truncation classes on the $E_r$-page, and $d_r(a) = b$ in $C_2\text{-}\SliceSS(\BPtwo)$. We want to prove the differential $d_r(tr(a)) = tr(b)$ in $C_4\text{-}\SliceSS(\BPtwo)$.  Since $d_r(a) = b$, $tr(b)$ must be killed by a differential of length at most $r$ (natuality).  Moreover, if $res(tr(a))$ and $res(tr(b))$ are both nonzero on the $E_r$-page, $d_r(a) = b$ implies $d_r(res(tr(a))) = res(tr(b))$.  By natuality again, $tr(a)$ must support a differential of length at most $r$.  In all the cases of interest, either our complete knowledge of all the shorter differentials (when $r = 3$, $7$, and $15$) or degree reasons will imply our desired differential.

\begin{notation}[Naming $i_{C_2}^*\BPone$-truncation classes]\rm 
From now on, we will name $i_{C_2}^*\BPone$-truncation classes by specifying the name of their corresponding slice cells and their bidegrees.  Furthermore, we will denote the induced slice cell $(\rthree^{i+j} \grthree^i \rone^k \grone^\ell, \rthree^i \grthree^{i+j}\rone^\ell \grone^k)$ by $\dthree^i \sthree^j \rone^k \grone^\ell$.  This convention reduces cluttering and improves the readability of our computations.  For example, consider the class $tr(\rthree^4 \grthree^2 \rone^2 u_{8\sigma_2}a_{12\sigma_2})$ at $(28, 12)$.  Instead of writing out its full name, we will write the class ``$\dthree^2\sthree^2\rone^2$ at $(28, 12)$'' or ``$tr(\rthree^4 \grthree^2 \rone^2)$ at $(28, 12)$'' instead.
\end{notation}

\begin{exam} \rm
On the $E_{15}$-page, there are three classes at $(28, 12)$ ($\sthree^6 \rone^2$, $\dthree \sthree^4 \rone^2$, $\dthree^2 \sthree^2 \rone^2$) and five classes at $(27,27)$ ($\sthree^9$, $\dthree \sthree^7$, $\dthree^2\sthree^5$, $\dthree^3 \sthree^3$, $\dthree^4 \sthree$) coming from $i_{C_2}^*\BPone$-truncations.  In the $C_2$-slice spectral sequence, there are $d_{15}$-differentials (as a reminder, note that from the definition of $\sthree^i$, $\sthree^i \neq (\sthree)^i$)
\begin{eqnarray*}
d_{15}(\bar{s}_3^6 \bar{r}_1^2) &=& \bar{s}_3^6 \bar{r}_1^2 \cdot \bar{r}_1 (\dthree + \bar{s}_3^2) = \bar{s}_3^6 \bar{s}_3 (\dthree + \bar{s}_3^2) \\
&=& \dthree \sthree^6 \sthree + (\dthree \sthree^6\sthree + \sthree^6 \sthree^3) =\bar{s}_3^6 \bar{s}_3^3 = \bar{s}_3^9 + \dthree^3 \bar{s}_3^3, \\
d_{15}(\dthree \bar{s}_3^4 \bar{r}_1^2) &=&  \dthree \bar{s}_3^4 \bar{r}_1^2 \cdot  \bar{r}_1 (\dthree + \bar{s}_3^2)= \dthree \bar{s}_3^4 \sthree (\dthree + \bar{s}_3^2)=\dthree \bar{s}_3^4 \cdot \bar{s}_3^3 = \dthree \bar{s}_3^7 + \dthree^4 \bar{s}_3,\\
d_{15}(\dthree^2 \bar{s}_3^2 \bar{r}_1^2) &=&\dthree^2 \bar{s}_3^2 \bar{r}_1^2 \cdot\bar{r}_1 (\dthree + \bar{s}_3^2)= \dthree^2 \bar{s}_3^2 \sthree (\dthree + \bar{s}_3^2)= \dthree^2 \bar{s}_3^2 \cdot \bar{s}_3^3 = \dthree^2 \bar{s}_3^5 + \dthree^4 \bar{s}_3.   
\end{eqnarray*}
This implies that the three classes $\sthree^6 \rone^2$, $\dthree \sthree^4 \rone^2$, $\dthree^2 \sthree^2 \rone^2$ all support differentials of length at most 15 in the $C_4$-slice spectral sequence.  Since we have complete knowledge of all the shorter differentials, the $d_{15}$-differentials above must occur.  

Alternatively, we can use the transfer.  The first differential can be rewritten as 
$$d_{15}(tr(\rthree^6 \rone^2)) = tr(\rthree^9 +\rthree^6 \grthree^3).$$
In the $C_2$-slice spectral sequence, we have the $d_{15}$-differential 
\begin{eqnarray*}
d_{15}(\bar{r}_3^6 \bar{r}_1^2) &=& \bar{r}_3^6 \bar{r}_1^2 \cdot \bar{r}_1 (\dthree + \bar{s}_3^2)  \\
&=& \bar{r}_3^6 \bar{s}_3 (\dthree + \bar{s}_3^2) \\
&=& \dthree \bar{s}_3 \bar{r}_3^6 + (\dthree \bar{s}_3 \bar{r}_3^6 + \bar{r}_3^6 \bar{s}_3^3) \\
&=& \bar{r}_3^6 \bar{s}_3^3 \\
&=& \bar{r}_3^9 + \bar{r}_3^6 \grthree^3.
\end{eqnarray*}
Applying the transfer shows that the class $tr(\bar{r}_3^9 + \bar{r}_3^6 \grthree^3)$ must be killed by a differential of length at most 15.  Our knowledge of the previous differentials again proves the desired differential.  The other two differentials above can be proved in the same way by using the transfer. 
\end{exam}

The formulas in Section~\ref{sec:C2BPtwoSliceSS} describe explicitly the surviving $i_{C_2}^*\BPone$-truncation classes on each page.  The $d_3$-differentials introduce the relation $\rone = \grone$ for $i_{C_2}^*\BPone$-truncation classes with filtrations at least 3.  After the $d_3$-differentials, their corresponding slice cells can all be written as 
$$\dthree^i \sthree^j \rone^k,$$ 
where $j > 0$.  

The $d_7$-differentials introduce the relation $\rone^3 + \rthree + \grthree = 0$ for classes with filtrations at least 7.  In other words, $\rone^3 = \sthree$ for $i_{C_2}^*\BPone$-truncation classes with filtrations at least 7.  After the $d_7$-differentials, their corresponding slice cells can all be written as 
$$\dthree^i \sthree^j \rone^k,$$
where $j > 0$ and $0 \leq k \leq 2$.  Figure~\ref{fig:BPtwopinkd7} shows the $d_7$-differentials between $i_{C_2}^*\BPone$-truncation classes. 

\begin{figure}
\begin{center}
\makebox[\textwidth]{\includegraphics[trim={0cm 10.5cm 0cm 3.5cm}, clip, scale = 0.8, page = 1]{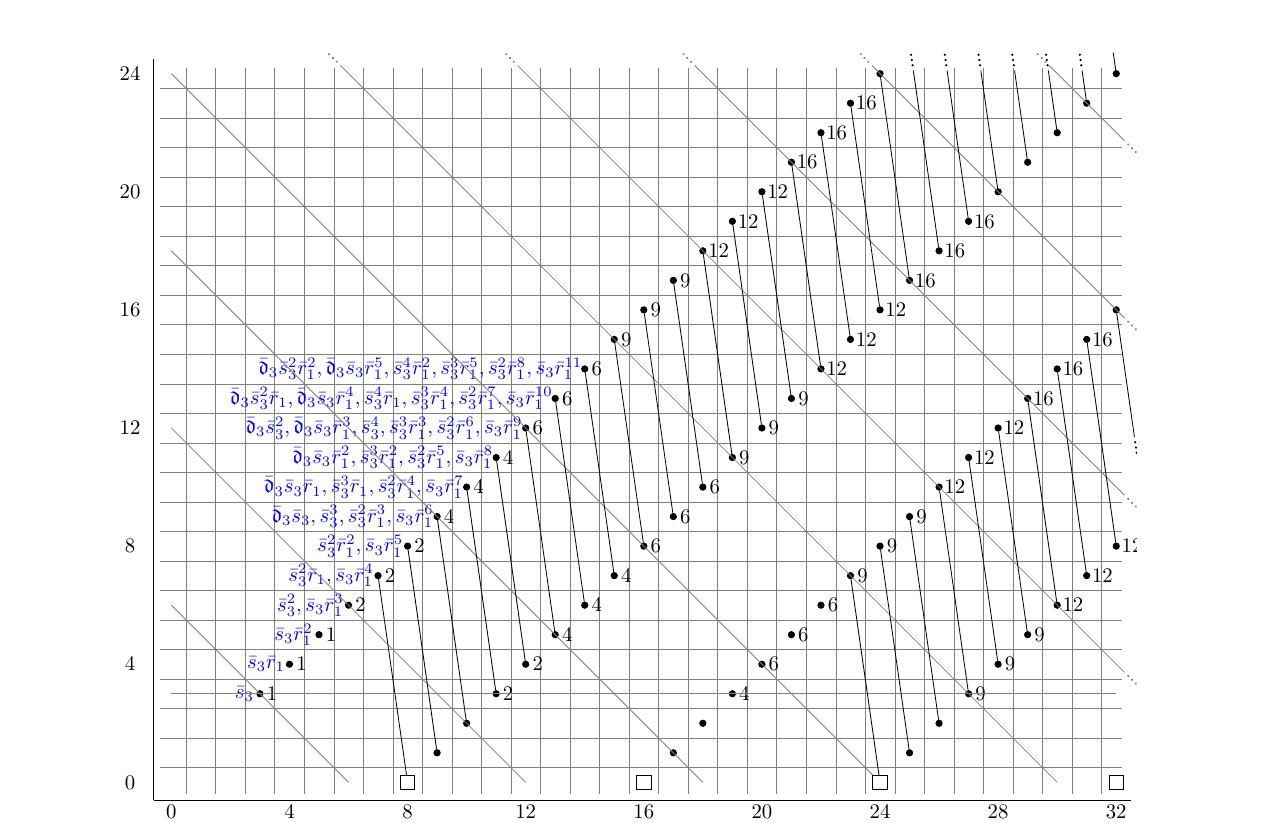}}
\end{center}
\begin{center}
\caption{$d_7$-differentials between $i_{C_2}^*\BPone$-truncation classes. }
\hfill
\label{fig:BPtwopinkd7}
\end{center}
\end{figure}

\begin{prop}\label{prop:DthreeSthreeRelations} 
After the $d_{15}$-differentials between $i_{C_2}^*\BPone$-truncation classes, the following relations hold for the classes in filtrations at least 15: 
\begin{enumerate}
\item $\sthree^3 \rone = \sthree^3 \rone^2 = 0$;
\item $\sthree^{6m+1} = \dthree^{3m}\sthree$ for all $m \geq 0$; 
\item $\sthree^{6m+2} = \dthree^{3m} \sthree^2$ for all $m \geq 0$; 
\item $\sthree^{6m+3} = \dthree^{3m} \sthree^3$ for all $m \geq 0$; 
\item $\sthree^{6m+4} = \dthree^{3m+1} \sthree^2$ for all $m \geq 0$; 
\item $\sthree^{6m+5} = \dthree^{3m+2} \sthree$ for all $m \geq 0$;
\item $\sthree^{6m} = 2 \dthree^{3m}$ for all $m \geq 1$.
\end{enumerate}
\end{prop}
\begin{proof}
The $d_{15}$-differential in the $C_2$-slice spectral sequence multiplies the slice cell of the source by $\rone(\rthree^2 + \rthree \grthree + \grthree^2 ) = \rone (\dthree + \sthree)$.  

\noindent (1) We have the equality 
$$\rthree^3 + \grthree^3 = (\rthree + \grthree)(\rthree^2 + \rthree \grthree + \grthree^2) = \sthree (\dthree + \sthree^2). $$
Therefore, 
\begin{eqnarray*}
res(\sthree^3 \rone) &=& (\rthree^3 + \grthree^3) \rone \\
&=& (\rthree + \grthree) \cdot \rone(\rthree^2 + \rthree \grthree + \grthree^2) \\
&=& res(tr(\rthree)) \cdot \rone(\rthree^2 + \rthree \grthree + \grthree^2).
\end{eqnarray*}
Consider the class $\rthree u_{8\sigma_2}$ in the $C_2$-spectral sequence.  It supports the $d_{15}$-differential \begin{eqnarray*}
d_{15}(\rthree u_{8\sigma_2}) &=& \rthree \cdot \rone(\rthree^2 + \rthree \grthree + \grthree^2) a_{15\sigma_2} \\
&=& (\rthree^3 \rone + \rthree^2 \grthree \rone + \rthree \grthree^2 \rone) a_{15\sigma_2}\\ 
&=& (\rthree^3 \rone + res(\dthree \sthree \rone)) a_{15\sigma_2}.
\end{eqnarray*}
Applying the transfer to this $d_{15}$-differential and using natuality implies the $d_{15}$-differential 
$$d_{15}(\sthree) = tr(\rthree^3 \rone a_{15\sigma_2}) + tr(res(\dthree \sthree \rone a_{15\sigma_2})) = \sthree^3 \rone$$
in the $C_4$-slice spectral sequence.  Therefore, $\sthree^3 \rone = 0$ after the $d_{15}$-differentials. 

For $\sthree^3 \rone^2$, the proof is exactly the same.  The exact same argument as above shows the $d_{15}$-differential 
$$d_{15}(\sthree \rone ) = \sthree^3 \rone^2$$
in the $C_4$-slice spectral sequence.  

\vspace{0.1in}

\noindent (2) The statement holds trivially when $m = 0$.  When $m \geq 1$, we have the equality $\rthree^{6m+1} + \rthree^{6m-2} \grthree^3 = \rthree^{6m-2} (\rthree^3 + \grthree^3)$.  This implies the $d_{15}$-differential 
$$d_{15}(\rthree^{6m-2} \rone^2 u_{8\sigma_2}) = (\rthree^{6m+1} + \rthree^{6m-2} \grthree^3)a_{15\sigma_2} $$
in the $C_2$-slice spectral sequence.  Applying the transfer and using natuality, we obtain the $d_{15}$-differential 
$$d_{15}(\sthree^{6m-2} \rone^2) = tr(\rthree^{6m+1}) + tr(\rthree^{6m-2} \grthree^3) = \sthree^{6m+1} + \dthree^3 \sthree^{6m-5}$$
in the $C_4$-slice spectral sequence.  This produces the relation $$\sthree^{6m+1} = \dthree^3 \sthree^{6m-5}$$ for all $m \geq 1$.  Induction on $m$ proves the desired equality.  

\vspace{0.1in}

\noindent (3) The statement holds trivially when $m = 0$.  When $m \geq 1$, we have the equality 
$$\rthree^{6m+2} + \rthree^{6m-1} \grthree^3 = \rthree^{6m-1} (\rthree^3 + \grthree^3).$$  This implies the $d_{15}$-differential 
$$d_{15}(\rthree^{6m-1} \rone^2 u_{8\sigma_2}) = (\rthree^{6m+2} + \rthree^{6m-1} \grthree^3)a_{15\sigma_2}$$
in the $C_2$-slice spectral sequence.  Applying the transfer and using natuality produces the $d_{15}$-differential 
$$d_{15}(\sthree^{6m-1} \rone^2) =tr(\rthree^{6m+2}) + tr(\rthree^{6m-1} \grthree^3) = \sthree^{6m+2} + \dthree^3 \sthree^{6m-4}$$
in the $C_4$-slice spectral sequence.  Induction on $m$ proves the desired equality. 

\vspace{0.1in}
\noindent (4) The statement holds trivially when $m = 0$.  When $m \geq 1$, we have the equality 
$$\rthree^{6m+3} + \rthree^{6m} \grthree^3 = \rthree^{6m} (\rthree^3 + \grthree^3).$$
This implies the $d_{15}$-differential 
$$d_{15}(\rthree^{6m} \rone^2 u_{8\sigma_2}) = (\rthree^{6m+3} + \rthree^{6m} \grthree^3)a_{15\sigma_2}$$
in the $C_2$-slice spectral sequence.  Applying the transfer and using natuality produces the $d_{15}$-differential 
$$d_{15}(\sthree^{6m}\rone^2) = tr(\rthree^{6m+3}) + tr( \rthree^{6m} \grthree^3) = \sthree^{6m+3} + \dthree^3 \sthree^{6m-3}$$
in the $C_4$-slice spectral sequence.  Induction on $m$ proves the desired equality. 

\vspace{0.1in}
\noindent (5) We have the equality 
$$\rthree^{6m+4} + \rthree^{6m+1} \grthree^3 = \rthree^{6m+1} (\rthree^3 + \grthree^3).$$  This implies the $d_{15}$-differential 
$$d_{15}(\rthree^{6m+1}\rone^2 u_{8\sigma_2}) = (\rthree^{6m+4} + \rthree^{6m+1} \grthree^3)a_{15\sigma_2}$$
in the $C_2$-slice spectral sequence.  Applying the transfer and using natuality produces the $d_{15}$-differential 
$$d_{15}(\sthree^{6m+1} \rone^2) = tr(\rthree^{6m+4}) + tr(\rthree^{6m+1} \grthree^3)$$
in the $C_4$-slice spectral sequence.  When $m = 0$, the target is $\sthree^4 + \dthree \sthree^2$, from which we get the relation $\sthree^4 = \dthree \sthree^2$.  For $m \geq 1$, the target is $\sthree^{6m+4} + \dthree^3 \sthree^{6m-2}$, from which we get the relation $\sthree^{6m+4} = \dthree^3 \sthree^{6m-2}$.  Induction on $m$ proves the desired equality.  
\vspace{0.1in}

\noindent (6) We have the equality 
$$\rthree^{6m+5} + \rthree^{6m+2} \grthree^3 = \rthree^{6m+2} (\rthree^3 + \grthree^3).$$  This implies the $d_{15}$-differential 
$$d_{15}(\rthree^{6m+2} \rone^2 u_{8\sigma_2}) = (\rthree^{6m+5} + \rthree^{6m+2} \grthree^3) a_{15\sigma_2}$$
in the $C_2$-slice spectral sequence.  Applying the transfer and using natuality produces the $d_{15}$-differential 
$$d_{15}(\sthree^{6m+2} \rone^2) = tr(\rthree^{6m+5}) + tr(\rthree^{6m+2} \grthree^3)$$
in the $C_4$-slice spectral sequence.  When $m = 0$, the target is $\sthree^5 + \dthree^2 \sthree$, from which we get the relation $\sthree^5 = \dthree^2 \sthree$.  For $m \geq 1$, the target is $\sthree^{6m+5} + \dthree^3 \sthree^{6m-1}$, from which we get the relation $\sthree^{6m+5} = \dthree^3 \sthree^{6m-1}$.  Induction on $m$ proves the desired equality. 
\vspace{0.1in}

\noindent (7) Since 
$$\rthree^{6m} + \rthree^{6m-3} \grthree^3 = \rthree^{6m-3}(\rthree^3 + \grthree^3),$$ there is the $d_{15}$-differential 
$$d_{15}(\rthree^{6m-3}\rone^2 u_{8\sigma_2}) = \rthree^{6m-3}(\rthree^3 + \grthree^3) a_{15\sigma_2}$$
in the $C_2$-slice spectral sequence.  Applying the transfer and using natuality produces the $d_{15}$-differential 
$$d_{15}(\sthree^{6m-3} \rone^2) =  tr(\rthree^{6m}) + tr(\rthree^{6m-3} \grthree^3) = \sthree^{6m} + tr(\rthree^{6m-3} \grthree^3)$$
in the $C_4$-slice spectral sequence.  

We will now use induction on $m$.  When $m=1$, the target is $\sthree^6 + 2\dthree^3$, from which we deduce $\sthree^6 = 2\dthree^3$.  When $m >1$, the target is $\sthree^{6m} + \dthree^3 \sthree^{6m-6}$, from which we deduce $\sthree^{6m} = \dthree^3 \sthree^{6m-6}$.  Induction on $m$ shows that $\sthree^{6m} = 2 \dthree^{3m}$.

\begin{comment}
Relations (3), (4), (5), (6) follow from the following equalities: 
\begin{eqnarray*}
\sthree^{6m+1} - \dthree^{3m} \sthree &=& \sthree^{6m+1} + \dthree^{3m} \sthree = \sthree^{3m} \cdot \sthree^{3m+1}, \\ 
\sthree^{6m+2} - \dthree^{3m} \sthree^2 &=& \sthree^{6m+2} + \dthree^{3m} \sthree^2 = \sthree^{3m} \cdot \sthree^{3m+2}, \\ 
\sthree^{6m+4} - \dthree^{3m+1} \sthree^2 &=& \sthree^{6m+4} + \dthree^{3m+1} \sthree^2 = \sthree^{3m+1} \cdot \sthree^{3m+3}, \\ 
\sthree^{6m+5} - \dthree^{3m+2} \sthree &=& \sthree^{6m+5} + \dthree^{3m+2} \sthree = \sthree^{3m+2} \cdot \sthree^{3m+3}.
\end{eqnarray*}
\end{comment}
\end{proof}

\begin{warn}\rm
The class $\sthree^3$ is not 0 after the $d_{15}$-differentials between $i_{C_2}^*\BPone$-truncation classes.  In particular, the classes $\dthree \sthree^3$ at $(15, 15)$, $\dthree^2 \sthree^3$ at $(21,21)$, $\dthree^3 \sthree^3$ at $(27,27)$, $\ldots$ are not targets of $d_{15}$-differentials with sources coming from $i_{C_2}^*\BPone$-truncation classes.  However, some of these classes ($\dthree^3 \sthree^3 \in (27, 27)$ and $\dthree^7 \sthree^3 \in (51,51)$, for example) are still targets of $d_{15}$-differentials with sources coming from $\BPone$-truncation classes.  We will discuss this in the next subsection.  
\end{warn}

Figures~\ref{fig:BPtwopinkd15diagram1} and \ref{fig:BPtwopinkd15diagram2} illustrate the $d_{15}$-differentials between $i_{C_2}^*\BPone$-truncation classes. 

\begin{figure}
\begin{center}
\makebox[\textwidth]{\includegraphics[trim={0cm 9cm 0cm 9cm}, clip, page = 2, scale = 0.45]{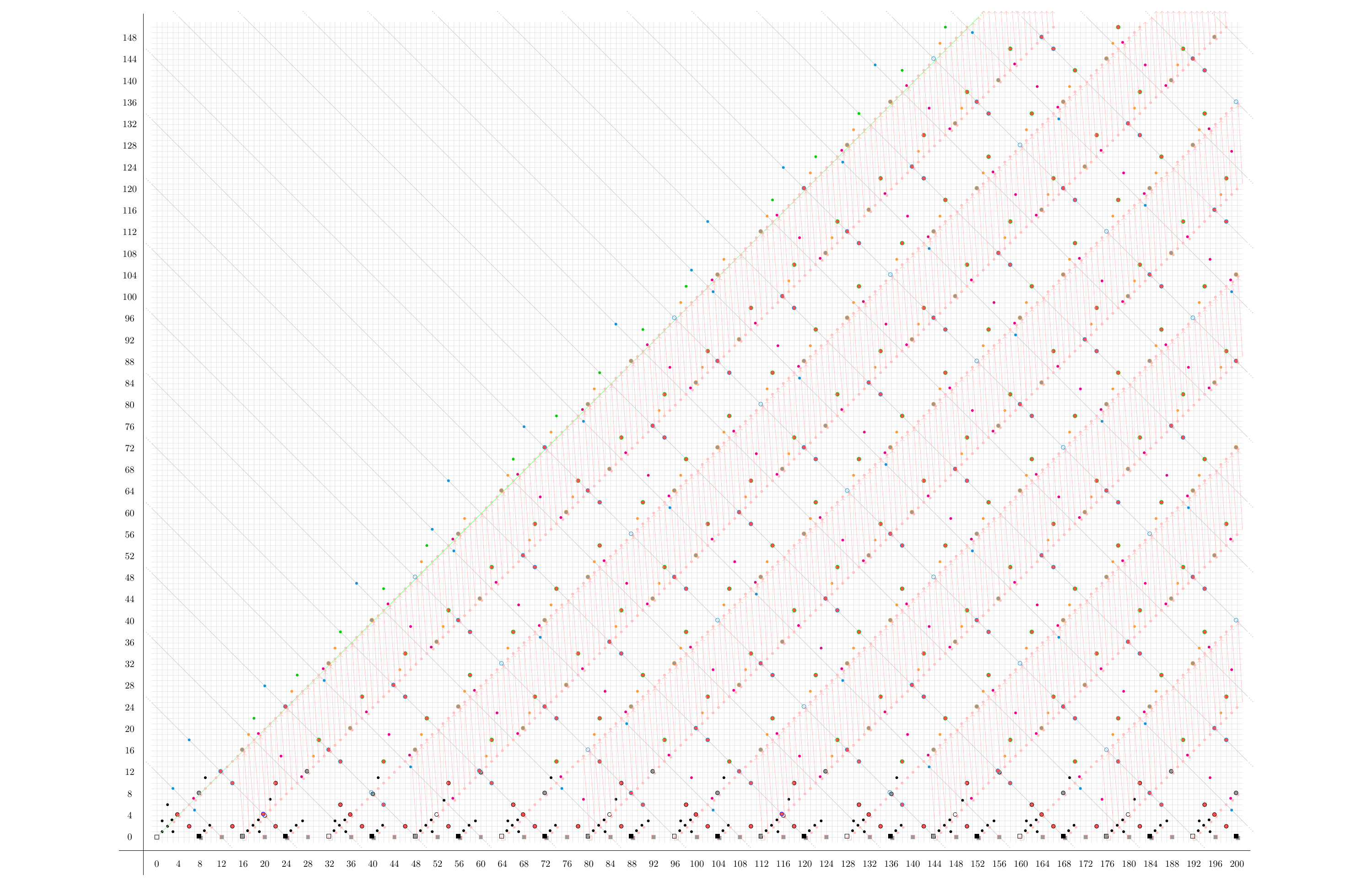}}
\end{center}
\begin{center}
\caption{$d_{15}$-differentials between $i_{C_2}^*\BPone$-truncation classes.}
\hfill
\label{fig:BPtwopinkd15diagram1}
\end{center}
\end{figure}

\begin{figure}
\begin{center}
\makebox[\textwidth]{\includegraphics[trim={0cm 9.5cm 0cm 3.5cm}, clip, scale = 0.8, page = 1]{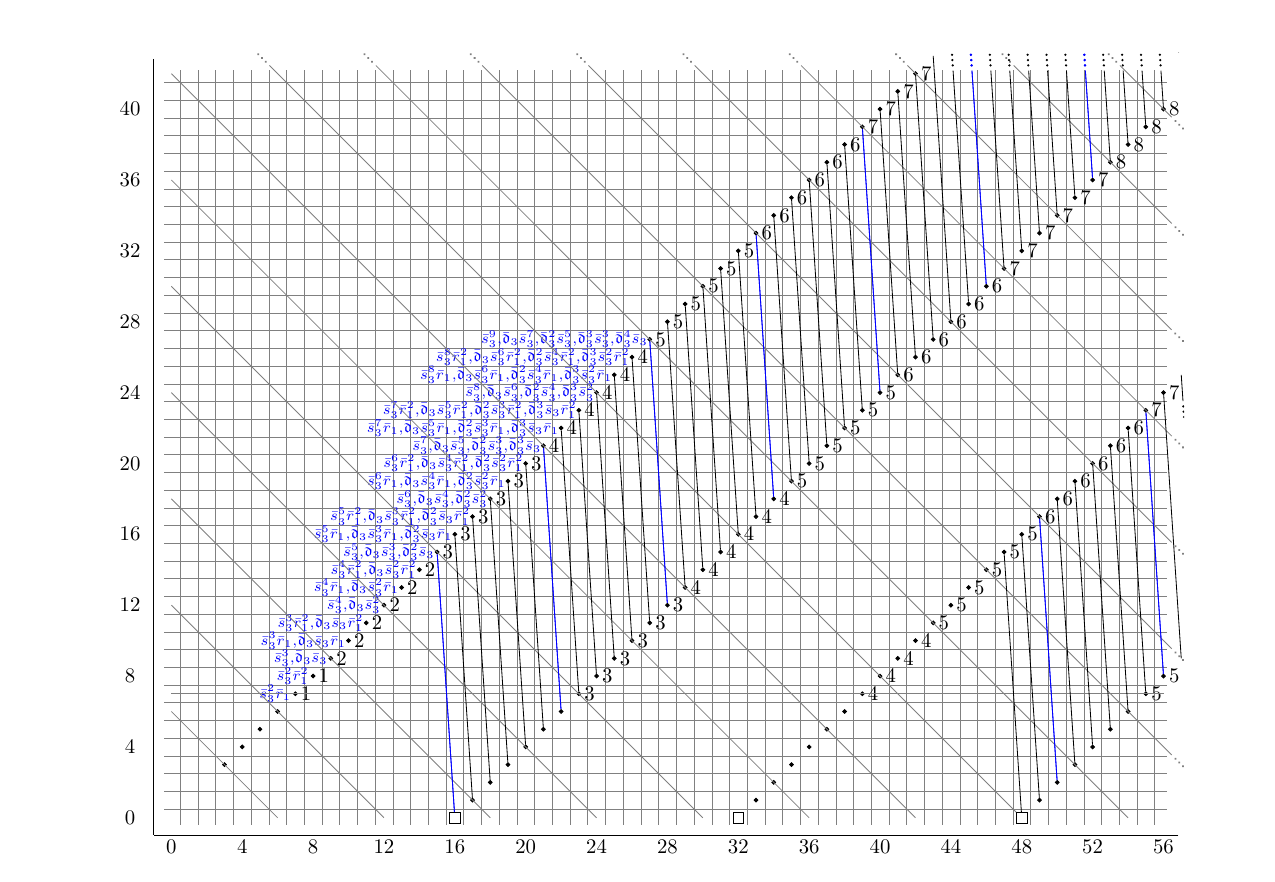}}
\end{center}
\begin{center}
\caption{$d_{15}$-differentials between $i_{C_2}^*\BPone$-truncation classes.  The targets of the {\color{blue} blue differentials} have two surviving classes instead of one.}
\hfill
\label{fig:BPtwopinkd15diagram2}
\end{center}
\end{figure}

%%%%%%%%%%%%%%%
\subsection{All the other $d_{15}$-differentials and some $d_{31}$-differentials.} \label{subsec:d15ResTr}

We will now prove the rest of the $d_{15}$-differentials (see Figure~\ref{fig:BPtwod15Rest}).

\begin{figure}
\begin{center}
\makebox[\textwidth]{\includegraphics[trim={0cm 9cm 0cm 9cm}, clip, page = 1, scale = 0.45]{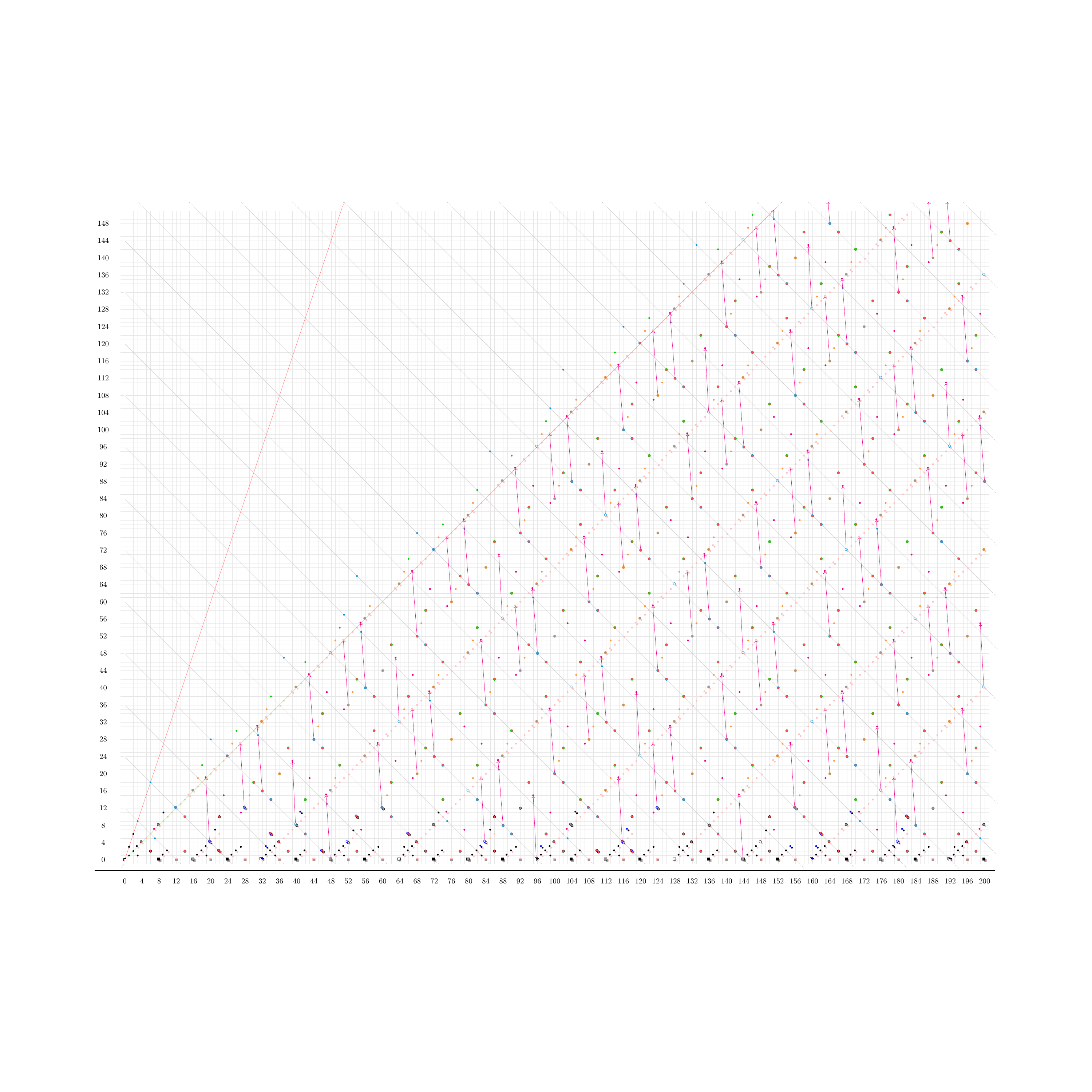}}
\end{center}
\begin{center}
\caption{The rest of the $d_{15}$-differentials.}
\hfill
\label{fig:BPtwod15Rest}
\end{center}
\end{figure}

\begin{prop}\label{prop:d15(20,4)}
The class $2\dthree^2u_{4\lambda}u_{6\sigma} a_{2\lambda}$ at $(20, 4)$ supports the $d_{15}$-differential 
$$d_{15}(2\dthree^2u_{4\lambda}u_{6\sigma} a_{2\lambda}) = \dthree^3 \sone a_{9\lambda}a_{\sigma_2}.$$
(Under our naming convention, the target is abbreviated as $\dthree^3 \sone$ at $(19, 19)$).
\end{prop}
\begin{proof}
In the $C_2$-slice spectral sequence, the restriction of the class $\dthree^2u_{4\lambda}u_{6\sigma} a_{2\lambda}$ at $(20,4)$ supports the $d_{15}$-differential 
\begin{eqnarray*}
d_{15}(res(\dthree^2u_{4\lambda}u_{6\sigma} a_{2\lambda})) &=& \rthree^2 \grthree^2 \cdot \rone(\rthree^2 + \rthree \grthree + \grthree^2) \\ 
&=& \dthree^2 \rone (\dthree + \sthree^2) \\ 
&=& \dthree^3 \rone + \dthree^2 \sthree^2 \rone\\
&=& \dthree^3 \rone + res(tr(\dthree^2 \rthree^2 \rone)).  
\end{eqnarray*}
Applying the transfer map shows that the class
$$tr(\dthree^3 \rone) + tr(res(tr(\dthree^2 \rthree^2 \rone))) = \dthree^3 \sone + 2 \cdot tr(\dthree^2 \rthree^2 \rone) = \dthree^3 \sone$$
must be killed by a differential of length at most 15.  By degree reasons, the differential must of length 15, and the source must be 
$$tr(res(\dthree^2u_{4\lambda}u_{6\sigma} a_{2\lambda})) = 2 \dthree^2u_{4\lambda}u_{6\sigma} a_{2\lambda}$$
by natuality.  This proves the desired differential.  
\end{proof}

Using the exact same method as the proof of Proposition~\ref{prop:d15(20,4)}, we can prove $d_{15}$-differentials on the following classes: 
\begin{enumerate}
\item $\{(20, 4), (32, 16), (44, 28), \ldots\}$, $\{(116,4), (128, 16), (140, 28), \ldots\}$, $\ldots$.
\item $\{(48, 0), (60, 12), (72, 24), \ldots\}$, $\{(144, 0), (156, 12), (168, 24), \ldots\}$, $\ldots$.
\item $\{(88, 8), (100, 20), (112, 32), \ldots\}$, $\{(184, 8), (196, 20), (208, 32), \ldots\}$, $\ldots$.
\end{enumerate}

\begin{rem}\rm
The restrictions of the classes $\dthree \done^3 u_{4\lambda}u_{6\sigma}a_{2\lambda}$ and $\done^6 u_{4\lambda}u_{6\sigma}a_{2\lambda}$ at $(20,4)$ support the following $d_{15}$-differentials in the $C_2$-spectral sequence:
\begin{eqnarray*}
d_{15}(\rthree \grthree \rone^6) &=& \rthree \grthree \rone^6 \cdot \rone (\rthree^2 + \rthree \grthree + \grthree^2) = \dthree \sthree^2 \cdot \rone (\sthree^2 + \dthree) = (\dthree^2 \sthree^2 + \dthree \sthree^4) \rone,\\
d_{15}(\rone^{12}) &=& \rone^{12} \cdot \rone (\rthree^2 + \rthree \grthree + \grthree^2) = \sthree^4 \cdot \rone (\sthree^2 + \dthree) = (\dthree \sthree^4 + \dthree^2 \sthree^2 + \sthree^6) \rone.
\end{eqnarray*}
(In the formulas above, we used the relation $\rone^3 = \rthree + \grthree = \sthree$.)  By natuality and degree reasons, there exist $d_{15}$-differentials 
$$d_{15}(\dthree \done^3) = (\dthree^2 \sthree^2 + \dthree \sthree^4) \rone$$
and 
$$d_{15}(\done^6) = (\dthree \sthree^4 + \dthree^2 \sthree^2 + \sthree^6) \rone$$
on the classes at $(20, 4)$ in the $C_4$-slice spectral sequence.  

We also have the following $d_{15}$-differentials on the classes $\dthree \sthree^2$ and $\sthree^4$ at $(20,4)$: 
\begin{eqnarray*}
d_{15}(\dthree \sthree^2) &=& (\dthree^2 \sthree^2 + \dthree \sthree^4) \rone, \\ 
d_{15}(\sthree^4) &=& (\dthree \sthree^4 + \dthree^2 \sthree^2 + \sthree^6) \rone.
\end{eqnarray*}

After the $d_{15}$-differentials, the surviving classes at $(20,4)$ are $2\dthree \done^3$, $2 \done^6$, $\dthree \done^3 + \dthree \sthree^2$, and $\done^6  + \sthree^4$.  There is a slight subtlety here because from the way we are organizing the $d_{15}$-differentials, the surviving leading terms should be $\dthree \done^3$, $2\dthree \done^3$, $\done^6$, and $2\done^6$.  Our presentation of the surviving classes will not affect later computations.  
\end{rem}

\begin{prop} \label{prop:d15(28,12)}
The class $\dthree^3 \done u_{4\lambda} u_{10\sigma} a_{6\lambda}$ at $(28, 12)$ supports the $d_{15}$-differential
$$d_{15}(\dthree^3 \done u_{4\lambda} u_{10\sigma} a_{6\lambda}) = \dthree^3 \sthree^3 u_{9\sigma}a_{9\lambda} a_{9\sigma_2}. $$
\end{prop}
\begin{proof}
The restriction of the class $\dthree^3 \done u_{4\lambda} u_{10\sigma} a_{6\lambda}$ supports the $d_{15}$-differential 
\begin{eqnarray*}
d_{15}(res(\dthree^3 \done u_{4\lambda} a_{10\sigma} a_{6\lambda})) &=& \dthree^3 \rone^2 \cdot \rone (\dthree + \sthree^2) \\
&=& \dthree^3 \sthree (\dthree + \sthree^2) \\
&=& \dthree^4 \sthree + \dthree^4 \sthree + \dthree^3 \sthree^3 \\
&=& \dthree^3 \sthree^3
\end{eqnarray*}
in the $C_2$-slice spectral sequence.  This implies that the class $\dthree^3 \done u_{4\lambda} u_{10\sigma} a_{6\lambda}$ must support a differential of length at most 15 in the $C_4$-slice spectral sequence.  For degree reasons, it must support a $d_{15}$-differential, and the target must be $\dthree^3 \sthree^3$ by natuality.  
\end{proof}

The proof of Proposition~\ref{prop:d15(28,12)} can be used to prove $d_{15}$-differentials on the following classes: 
\begin{enumerate}
\item $\{(28, 12),(52, 36),(76, 60), \ldots\}$, $\{(124,12),(148, 36),(172, 60),\ldots\}$, $\ldots$; 
\item $\{(68,20),(92, 44),(116, 68),\ldots\}$, $\{(164, 20),(188, 44),(212, 68),\ldots\}$, $\ldots$;
\item $\{(84, 4),(108, 28),(132, 52),\ldots\}$, $\{(180, 4),(204, 28),(228,52),\ldots\}$, $\ldots$.
\end{enumerate}

\begin{prop} \label{prop:d15(40,8)}
The class $\dthree^4 u_{8\lambda}u_{12\sigma}a_{4\lambda}$ at $(40, 8)$ supports the $d_{15}$-differential 
$$d_{15}(\dthree^4 u_{8\lambda}u_{12\sigma}a_{4\lambda}) = \dthree^5 \sone u_{4\lambda}u_{15\sigma}a_{11\lambda}a_{\sigma_2}.$$
\end{prop}
\begin{proof}
In the $C_2$-slice spectral sequence, the restriction of $\dthree^4 u_{8\lambda}u_{12\sigma}a_{4\lambda}$ supports the $d_{31}$-differential 
$$d_{31}(res(\dthree^4 u_{8\lambda}u_{12\sigma}a_{4\lambda})) = \dthree^4 \cdot (\dthree \rthree^3) = \dthree^5 \rthree^3. $$
This implies that the class $\dthree^4 u_{8\lambda}u_{12\sigma}a_{4\lambda}$ must support a differential of length at most 31 in the $C_4$-slice spectral sequence.  By degree reasons, the target can either be at $(39, 39)$ ($d_{31}$-differential) or at $(39, 23)$ ($d_{15}$-differential).  

There are two classes at $(39, 39)$ --- $\dthree^6\sthree$ and $\dthree^5 \sthree^3$.  Since neither class restricts to $\dthree^5 \rthree^3$, the target cannot be at $(39, 39)$ by natuality.  The only possibility left is the class $\dthree^5 \sone$ at $(39, 23)$.  This is the desired differential.  
\end{proof}

\begin{prop} \label{prop:d31(40,8)}
The class $2\dthree^4 u_{8\lambda}u_{12\sigma}a_{4\lambda}$ at $(40,8)$ supports the $d_{31}$-differential
$$d_{31}(2\dthree^4 u_{8\lambda}u_{12\sigma}a_{4\lambda}) = \dthree^5 \sthree^3 u_{15\sigma}a_{15\lambda}a_{9\sigma_2}.$$
\end{prop}
\begin{proof}
As in the proof of Proposition~\ref{prop:d15(40,8)}, we have the differential 
$$d_{31}(res(\dthree^4 u_{8\lambda}u_{12\sigma}a_{4\lambda})) = \dthree^4 \cdot (\dthree \rthree^3) = \dthree^5 \rthree^3$$
in the $C_2$-spectral sequence.  Applying the transfer to the target of this differential shows that the class $tr(\dthree^5 \rthree^3) = \dthree^5 \sthree^3$ must be killed in the $C_4$-slice spectral sequence by a differential of length at most 31.  

For degree reasons, the target can only be killed by a differential of length 31.  Therefore, by natuality, the source must be 
$$tr(res(\dthree^4 u_{8\lambda}u_{12\sigma}a_{4\lambda})) = 2\dthree^4 u_{8\lambda}u_{12\sigma}a_{4\lambda}.$$
\end{proof}

The proofs of Proposition~\ref{prop:d15(40,8)} and ~\ref{prop:d31(40,8)} can be used to prove $d_{15}$ and $d_{31}$-differentials on the following classes: 
\begin{enumerate}
\item $\{(40, 8),(64, 32),(88, 56), \ldots \}$, $\{(232, 8),(256,32),(280, 56), \ldots \}$, $\ldots$;
\item $\{(96, 0),(120, 24),(144, 48), \ldots \}$, $\{(288, 0),(312, 24),(336, 48), \ldots \}$, $\ldots$;
\item $\{(176, 16),(200, 40),(224,64), \ldots \}$, $\{(368, 16),(392,40),(416,64), \ldots \}$, $\ldots$.
\end{enumerate}

\begin{rem}\rm
All the other classes at $(40, 8)$ support $d_{31}$-differentials hitting the same target, $\dthree^6 \sthree$.  In the formulas below, we use the relation $\rthree^3 = \grthree^3$, which is produced by the $d_{15}$-differentials, as well as Proposition~\ref{prop:DthreeSthreeRelations}.  
\begin{enumerate}
\item $d_{31}(res(\dthree^3 \done^3)) = d_{31} (res(\dthree^3 \sthree^2)) = \dthree^3 \sthree^2 \cdot (\dthree \rthree^3) = \dthree^4(\rthree^5 + \rthree^3 \grthree^2) = \dthree^4 \sthree^5 = \dthree^6 \sthree$.
\item $d_{31}(res(\dthree^2\done^6)) = d_{31}(res(\dthree^2 \sthree^4)) = \dthree^2 \sthree^4 \cdot (\dthree \rthree^3) = \dthree^3 (\rthree^7 + \rthree^3\grthree^4) = \dthree^3 \sthree^7 = \dthree^6 \sthree$. 
\item $d_{31}(res(\dthree \done^9)) = d_{31}(res(\dthree (\sthree)^6)) = \dthree \sthree^4 \sthree^2 \cdot (\dthree \rthree^3) =\dthree(\dthree^2 \sthree^2 + \sthree^6) \cdot (\dthree \rthree^3) = \dthree^4 \sthree^2 \rthree^3+ \dthree^2 \sthree^6 \rthree^3  =\dthree^4 \sthree^5+ \dthree^2 \sthree^9 = \dthree^6 \sthree +0= \dthree^6 \sthree$.
\item $d_{31}(res(\done^{12})) = d_{31}(res(\sthree^8)) = \sthree^8(\dthree \rthree^3) = \dthree (\rthree^{11} + \rthree^3\grthree^8) = \dthree \sthree^{11} = \dthree^6 \sthree$.
\end{enumerate}
This, combined with Proposition~\ref{prop:d31(40,8)}, shows that both remaining $i_{C_2}^*\BPone$-truncation classes ($\dthree^6 \sthree$ and $\dthree^5 \sthree^3$) at $(39, 39)$ are killed by $d_{31}$-differentials.  The same phenomenon occurs at the following bidegrees as well: 
\begin{enumerate}
\item $\{(39, 39),(63, 63),(87, 87), \ldots \}$, $\{(231, 39),(255,63),(279, 87), \ldots \}$, $\ldots$;
\item $\{(95, 31),(119, 55),(143, 79), \ldots \}$, $\{(287, 31),(311, 55),(335, 79), \ldots \}$, $\ldots$;
\item $\{(175, 47),(199, 71),(223,95), \ldots \}$, $\{(367, 47),(391,71),(415,95), \ldots \}$, $\ldots$.
\end{enumerate}
\end{rem}

Now we turn to prove the nonexistence of $d_{15}$-differentials on classes.  
\begin{prop}\label{prop:nod15(36,20)}
The class $\dthree^4 \done^2 u_{4\lambda} u_{14\sigma}a_{10\lambda}$ at $(36, 20)$ is the leading term of a $d_{15}$-cycle. 
\end{prop}
\begin{proof}
The restriction of $\dthree^4 \done^2 u_{4\lambda} u_{14\sigma}a_{10\lambda}$ supports the $d_{15}$-differential 
\begin{eqnarray*}
d_{15}(res(\dthree^4 \done^2 u_{4\lambda} u_{14\sigma}a_{10\lambda})) &=& \dthree^4 \rone^4 \cdot \rone(\dthree + \sthree^2) \\ 
&=& \dthree^4 \sthree \rone^2 (\dthree + \sthree^2) \\ 
&=& \dthree^5 \sthree \rone^2 + (\dthree^5 \sthree \rone^2 + \dthree^4 \sthree^3 \rone^2) \\ 
&=& \dthree^4 \sthree^3 \rone^2. 
\end{eqnarray*}
The class $\dthree^4 \sthree^3 \rone^2$ is also killed by the $d_{15}$-differential supported by the $i_{C_2}^*\BPone$-truncation class $\dthree^4 \sthree \rone$: 
\begin{eqnarray*}
d_{15}(\dthree^4 \sthree \rone) &=& \dthree^4 \sthree \rone \cdot \rone (\dthree + \sthree^2) \\ 
&=& \dthree^5 \sthree \rone^2 + (\dthree^5 \sthree \rone^2 + \dthree^4 \sthree^3 \rone^2) \\ 
&=& \dthree^4 \sthree^3 \rone^2. 
\end{eqnarray*} 
Technically, after these two $d_{15}$-differentials, the surviving class is $\dthree^4 \done^2 + \dthree^4 \sthree \rone$.  However, since the second differential has already been accounted for, there is no more $d_{15}$-differential on the class $\dthree^4 \done^2 u_{4\lambda} u_{14\sigma}a_{10\lambda}$.
\end{proof}

The proof of Proposition~\ref{prop:nod15(36,20)} can be used to show that the following classes are the leading terms of $d_{15}$-cycles: 
\begin{enumerate}
\item $\{(36, 20),(60, 44),(84, 68), \ldots \}$, $\{(132, 20),(156,44),(180, 68), \ldots \}$, $\ldots$;
\item $\{(52, 4),(76, 28),(100, 52), \ldots \}$, $\{(148, 4),(172, 28),(196, 52), \ldots \}$, $\ldots$;
\item $\{(92, 12),(116, 36),(140, 60), \ldots \}$, $\{(188, 12),(212,36),(236,60), \ldots \}$, $\ldots$.
\end{enumerate}

\begin{prop}\label{prop:nod15(27,11)}
There is no $d_{15}$-differential on the class $\sthree^3 \sone u_{4\lambda} u_{9\sigma} a_{5\lambda} a_{\sigma_2}$ at $(27, 11)$.  
\end{prop}
\begin{proof}
The class $\sthree^3 \sone u_{4\lambda} u_{9\sigma} a_{5\lambda} a_{\sigma_2}$ is in  the image of the transfer because $\dthree^3 \sone = tr(\dthree^3 \rone)$.  In the $C_2$-slice spectral sequence, the class $\dthree^3 \rone$ supports the $d_{15}$-differential 
$$d_{15}(\dthree^3 \rone) = \dthree^3 \rone \cdot \rone (\dthree + \sthree^2) = \dthree^4 \rone^2 + \dthree^3 \sthree^2 \rone^2.$$
If the class $\dthree^3 \sone$ does support a $d_{15}$-differential, then by natuality, the target must be 
$$tr(\dthree^4 \rone^2 + \dthree^3 \sthree^2 \rone^2) = tr(\dthree^4 \rone^2) + tr(\dthree^3 \sthree^2 \rone^2) = \dthree^4 \sone^2 +0 = \dthree^4 \sone^2.$$
This is impossible because the class $\dthree^4 \sone^2$ is killed by a $d_3$-differential and no longer exists on the $d_{15}$-page (the only class left at $(26, 26)$ is $\dthree^3 \sthree^2$).  
\end{proof}

The proof of Proposition~\ref{prop:nod15(27,11)} can be used to show that there are no $d_{15}$-differentials on the following classes: 
\begin{enumerate}
\item $\{(27, 11),(39, 23),(51,35), \ldots \}$, $\{(123, 11),(135,23),(147,35), \ldots \}$, $\ldots$;
\item $\{(55,7),(67,19),(79,31), \ldots \}$, $\{(151,7),(163,19),(175,31), \ldots \}$, $\ldots$;
\item $\{(83,3),(95,15),(107,27), \ldots \}$, $\{(179,3),(191,15),(203,27), \ldots \}$, $\ldots$.
\end{enumerate}

%%%%%%%%%%%%%%%%%%%%%%%%
%%%%%%%%%%%%%%%%%%%%%%%%
\section{Higher Differentials II: the Norm} \label{sec:Norm}

In this section, we will use Theorem~\ref{thm:NormFormula} to norm up differentials in $C_2$-$\SliceSS(\BPtwo)$ to differentials in $C_4$-$\SliceSS(\BPtwo)$.  

Recall that in the $C_2$-slice spectral sequence of $\BPtwo$, all the differentials are generated under multiplication by the following differentials: 
\begin{eqnarray*}
d_3(u_{2\sigma_2})  &=& \bar{v}_1 a_{3\sigma_2} = (\rone + \grone)a_{3\sigma_2}\\
d_7(u_{4\sigma_2})  &=& \bar{v}_2 a_{7\sigma_2} = (\rone^3 + \rthree + \grthree) a_{7\sigma_2}\\
d_{15}(u_{8\sigma_2})  &=& \bar{v}_3 a_{15\sigma_2} = \rone(\rthree^2 + \rthree \grthree + \grthree^2) a_{15\sigma_2}\\
d_{31}(u_{16\sigma_2})  &=& \bar{v}_4 a_{31\sigma_2}=\rthree^4 \grthree a_{31\sigma_2}
\end{eqnarray*}

\begin{thm}\label{thm:NormedDiffd5}
In the $C_4$-slice spectral sequence of $\BPtwo$, the class $u_{2\lambda}a_\sigma$ supports the $d_5$-differential 
$$d_5(u_{2\lambda}a_\sigma) = 2\done u_{2\sigma}a_{3\lambda}. $$
\end{thm}
\begin{proof}
Applying Theorem~\ref{thm:NormFormula} to the $d_3$-differential $d_3(u_{2\sigma_2}) = (\rone + \grone) a_{3\sigma_2}$ predicts the $d_5$-differential 
$$d_5\left(\frac{u_{2\lambda}}{u_{2\sigma}} a_\sigma\right) = N(\rone + \grone) a_{3\lambda}.$$
In particular, the proof of Theorem~\ref{thm:NormFormula} shows that the class $N(\rone + \grone) a_{3\lambda}$ must be killed on or before the $E_5$-page.  

We claim that the class $N(\rone + \grone) u_{2\sigma} a_{3\lambda}$ must be killed on or before the $E_5$-page as well.  This is because if $N(\rone + \grone) a_{3\lambda}$ is killed by a differential of length $r < 5$, then since $u_{2\sigma}$ is a $d_r$-cycle, $N(\rone + \grone) u_{2\sigma} a_{3\lambda}$ is killed by a differential of length at most $r$.  Otherwise, if $N(\rone + \grone) a_{3\lambda}$ is killed by a differential of length exactly 5, then the only possible source is $\frac{u_{2\lambda}}{u_{2\sigma}} a_\sigma$ (see \cite[Figure~5]{ZengThesis}), and the predicted $d_5$-differential occurs.  Applying the Leibniz rule yields the differential 
\begin{eqnarray*}
d_5(u_{2\lambda} a_\sigma) = d_5\left( u_{2\sigma} \cdot \frac{u_{2\lambda}a_\sigma }{u_{2\sigma}}\right) &=& d_5(u_{2\sigma}) \cdot \frac{u_{2\lambda}a_\sigma }{u_{2\sigma}} + u_{2\sigma} \cdot d_5\left(\frac{u_{2\lambda}a_\sigma }{u_{2\sigma}}\right) \\ 
&=& \done a_\lambda a_\sigma^3 \cdot \frac{u_{2\lambda}a_\sigma }{u_{2\sigma}}  + u_{2\sigma} \cdot N(\rone + \grone) a_{3\lambda} \\ 
&=& 0 + N(\rone + \grone) u_{2\sigma}  a_{3\lambda} \,\,\, \text{($u_{\lambda} a_{3\sigma} = 2u_{2\sigma}a_\lambda a_\sigma = 0$)}\\ 
&=& N(\rone + \grone) u_{2\sigma}  a_{3\lambda}.
\end{eqnarray*}

Since $u_\lambda$ supports a $d_3$-differential, $u_{2\lambda}$ is a $d_3$-cycle and the class $u_{2\lambda}a_\sigma$ survives to the $E_5$-page.  To identify the target and show that it is nonzero on the $E_5$-page, note that
\begin{eqnarray*}
res(N(\rone + \grone)) &=& (\rone + \grone)(\grone - \rone) \\
&=& -(\rone^2 - \grone^2) \\
&=& -(1 + \gamma)\rone^2 u_{-\sigma}.
\end{eqnarray*}
This implies that $N(\rone + \grone) = -tr(\rone^2 u_{-\sigma})$.  Note that $u_{-\sigma}$ is used as a placeholder to indicate that target of the transfer is in degree $2 + \lambda - (1 - \sigma) = 1 + \sigma + \lambda$. 

The target of the predicted differential is 
$$-tr(\rone^2 u_\sigma^{-1})u_{2\sigma} a_{3\lambda} = -tr(\rone^2 u_\sigma a_{6\sigma_2}) = tr(\rone^2 u_\sigma a_{6\sigma_2}
). $$
To identify this with $2\done u_{2\sigma} a_{3\lambda}$, consider the equality 
\begin{eqnarray*}
tr(\rone^2 u_\sigma a_{6\sigma_2}) + tr(\rone \grone u_\sigma a_{6\sigma_2}) &=& tr(\rone (\rone + \grone) u_\sigma a_{6\sigma_2}
) \\ 
&=& tr(\rone u_\sigma a_{\sigma_2} res(tr(\rone a_{\sigma_2})a_{2\lambda})) \\
&=&tr(\rone u_\sigma a_{\sigma_2})tr(\rone a_{\sigma_2}) a_{2\lambda}.
\end{eqnarray*}
The last expression is 0 because $d_3(u_\lambda) = \sone a_\lambda a_{\sigma_2} = tr(\rone a_{\sigma_2}) a_\lambda$.  It follows that 
\begin{eqnarray*}
tr(\rone^2 u_\sigma a_{6\sigma_2}) &=& - tr(\rone \grone u_\sigma a_{6\sigma_2}) \\ 
&=& tr(\rone \grone u_\sigma a_{6\sigma_2}) \\ 
&=&tr(res(\done u_{2\sigma} a_{3\lambda})) \\ 
&=& 2\done u_{2\sigma} a_{3\lambda}.
\end{eqnarray*}
This class is not zero on the $E_5$-page.  Therefore, the predicted $d_5$-differential on $u_{2\lambda}a_\sigma$ occurs.
\end{proof}

\begin{rem}\rm
In the integer graded spectral sequence, the normed $d_5$-differential can be seen on the class $\done^3 u_{2\lambda}u_{2\sigma}a_\lambda a_\sigma$ at $(9,3)$:
$$d_{5}(\done^3 u_{2\lambda}u_{2\sigma}a_\lambda a_\sigma) = 2\done^4 u_{4\lambda}a_{4\lambda}.$$
This is the product of the differential in Theorem~\ref{thm:NormedDiffd5} and $\done^3 u_{2\sigma}a_{\lambda}$.  An alternative, perhaps easier way to identify the target is to note that 
\begin{eqnarray*}
\done^3 u_{2\sigma}a_{\lambda} tr(\rone^2 u_\sigma a_{6\sigma_2}) &=& tr(res(\done^3 u_{2\sigma}a_{\lambda}) \rone^2 u_\sigma a_{6\sigma_2}) \\ 
&=& tr(\rone^5 \grone^3 a_{8\sigma_2}) \\ 
&=& tr(\rone^4 \grone^4 a_{8\sigma_2}) \, \, \, \text{($\rone = \grone$ after the $d_3$-differentials)}\\ 
&=& tr(res(\done^4 u_{4\sigma}a_{4\lambda})) \\
&=& 2 \done^4 u_{4\sigma} a_{\lambda}. 
\end{eqnarray*}
\end{rem}

\begin{thm}\label{thm:NormedDiffd13}
In the $C_4$-slice spectral sequence of $\BPtwo$, the class $u_{4\lambda}a_\sigma$ supports the $d_{13}$-differential 
$$d_{13}(u_{4\lambda}a_\sigma) = \done^3 u_{4\sigma} a_{7\lambda} + tr(\rthree^2 u_\sigma a_{14\sigma_2}).$$
\end{thm}
\begin{proof}
Applying Theorem~\ref{thm:NormFormula} to the $d_7$-differential $d_7(u_{4\sigma_2}) = (\rone^3+ \rthree + \grthree) a_{7\sigma_2}$ predicts the $d_{13}$-differential 
$$d_{13}\left(\frac{u_{4\lambda}}{u_{4\sigma}}a_\sigma \right) = N(\rone^3 + \rthree + \grthree)a_{7\lambda}.$$
In particular, the proof of Theorem~\ref{thm:NormFormula} shows that the class $N(\rone^3 + \rthree + \grthree)a_{7\lambda}$ must be killed on or before the $E_{13}$-page.  

We claim that the class $N(\rone^3 + \rthree + \grthree)u_{4\sigma}a_{7\lambda}$ must also be killed on or before the $E_{13}$-page.  This is because if $N(\rone^3 + \rthree + \grthree)a_{7\lambda}$ is killed by a differential of length $r < 13$, then since $u_{4\sigma}$ is a $d_r$-cycle, $N(\rone^3 + \rthree + \grthree)u_{4\sigma}a_{7\lambda}$ is killed by a differential of length at most $r$.  Otherwise, if $N(\rone^3 + \rthree + \grthree)a_{7\lambda}$ is killed by a differential of length exactly 13, then the only possible source is $\frac{u_{4\lambda}}{u_{4\sigma}} a_\sigma$ (see \cite[Figure~5]{ZengThesis}), and the predicted $d_{13}$-differential occurs.  Since $d_{13}(u_{4\sigma}) = \dthree a_{3\lambda}a_{7\sigma}$, applying the Leibniz rule yields the differential 
\begin{eqnarray*}
d_{13}(u_{4\lambda}a_\sigma) = d_{13}\left(u_{4\sigma} \cdot \frac{u_{4\lambda}}{u_{4\sigma}}a_\sigma\right) &=& d_{13}(u_{4\sigma}) \cdot \frac{u_{4\lambda}}{u_{4\sigma}}a_\sigma + u_{4\sigma} \cdot d_{13}\left(\frac{u_{4\lambda}}{u_{4\sigma}}a_\sigma \right) \\ 
&=& \dthree a_{3\lambda}a_{7\sigma} \cdot \frac{u_{4\lambda}}{u_{4\sigma}}a_\sigma+ u_{4\sigma} \cdot N(\rone^3 + \rthree + \grthree)a_{7\lambda} \\
&=& 0 + u_{4\sigma} \cdot N(\rone^3 + \rthree + \grthree)a_{7\lambda} \,\,\, \text{($u_{\lambda} a_{3\sigma} = 2u_{2\sigma}a_\lambda a_\sigma = 0$)}\\
&=& N(\rone^3 + \rthree + \grthree)u_{4\sigma}a_{7\lambda}. 
\end{eqnarray*}
To prove that this $d_{13}$-differential occurs, it suffices to show that the predicted target, $N(\rone^3 + \rthree + \grthree)u_{4\sigma}a_{7\lambda}$, is not zero on the $E_{13}$-page.  Once we show this, it would imply that this class must be killed by a $d_{13}$-differential (and the only possible source is $u_{4\lambda}a_\sigma$).

The restriction of $N(\rone^3 + \rthree + \grthree)$ is 
\begin{eqnarray*}
res(N(\rone^3 + \rthree + \grthree)) &=& (\rone^3 + \rthree + \grthree)(\grone^3 + \grthree - \rthree) \\ 
&=& \rone^3 \grone^3 + (\grthree \rone^3 + \rthree \grone^3 ) + (\grthree \grone^3 - \rthree \rone^3 ) - (\rthree^2 - \grthree^2) \\ 
&=& res(\done^3) +  (1 + \gamma) \grthree \rone^3u_{-3\sigma} +  (1 + \gamma)\rthree \grone^3u_{-3\sigma}- (1 + \gamma)\rthree^2 u_{-3\sigma}.
\end{eqnarray*}
Therefore we have
\begin{eqnarray*}
N(\rone^3 + \rthree + \grthree)u_{4\sigma}a_{7\lambda} &=& \done^3 u_{4\sigma}a_{7\lambda} + tr(\grthree \rone^3 u_{\sigma} a_{14\sigma_2}) + tr(\grthree \grone^3 u_{\sigma} a_{14\sigma_2}) - tr(\rthree^2u_{\sigma} a_{14\sigma_2} )\\ 
&=& \done^3 u_{4\sigma}a_{7\lambda} + tr(\grthree \rone^3 u_{\sigma} a_{14\sigma_2}) + tr(\grthree \rone^3 u_{\sigma} a_{14\sigma_2}) + tr(\rthree^2u_{\sigma} a_{14\sigma_2} )\\ 
&=&\done^3 u_{4\sigma}a_{7\lambda} + 2tr(\grthree \rone^3 u_{\sigma} a_{14\sigma_2}) + tr(\rthree^2u_{\sigma} a_{14\sigma_2} )\\ 
&=&\done^3 u_{4\sigma}a_{7\lambda} + tr(\rthree^2u_{\sigma} a_{14\sigma_2} ). 
\end{eqnarray*}
To show that this class is not zero, we multiply the predicted differential on $u_{4\lambda}a_\sigma$ by $\done^9u_{8\sigma} a_{5\lambda}$ (a permanent cycle) to bring it to the integer graded part of the slice spectral sequence: 
$$d_{13}(\done^9 u_{4\lambda}u_{8\sigma} a_{5\lambda} a_{\sigma}) = \done^9u_{8\sigma} a_{5\lambda} \cdot (\done^3 u_{4\sigma}a_{7\lambda} + tr(\rthree^2u_{\sigma} a_{14\sigma_2} )).$$
The source of this differential is at $(25, 11)$ and the target is at $(24, 24)$.  Once we verify that the target of this differential is not zero on the $E_{13}$-page, we can then conclude that $N(\rone^3 + \rthree + \grthree)u_{4\sigma}a_{7\lambda}$ is also not zero on the $E_{13}$-page.  Indeed, 
\begin{eqnarray*}
&&\done^9u_{8\sigma} a_{5\lambda} \cdot (\done^3 u_{4\sigma}a_{7\lambda} + tr(\rthree^2u_{\sigma} a_{14\sigma_2} )) \\
&=& \done^{12}u_{12\sigma} a_{12\lambda} + tr(res(\done^9 u_{8\sigma} a_{5\lambda})\rthree^2u_{\sigma} a_{14\sigma_2} ) \\ 
&=& \done^{12}u_{12\sigma} a_{12\lambda} + tr(\rthree^2 \rone^{9} \grone^{9} a_{24\sigma_2})\\ 
&=& \done^{12}u_{12\sigma} a_{12\lambda} + tr(\rthree^2 \rone^{18} a_{24\sigma_2})\\ 
&=& \done^{12}u_{12\sigma} a_{12\lambda} + tr(\rthree^2 (\rthree + \grthree)^6 a_{24\sigma_2}) \,\,\, \text{(in the $C_2$-spectral sequence, $\rone^3 = \rthree + \grthree$)} \\ 
&=& \done^{12}u_{12\sigma} a_{12\lambda} +  tr(\rthree^2(\rthree^6 + \rthree^4 \grthree^2 + \rthree^2 \grthree^4 + \grthree^6)a_{24\sigma_2})\\ 
&=& \done^{12}u_{12\sigma} a_{12\lambda} + tr(\rthree^4 \grthree^4 a_{24\sigma_2})+ tr(\rthree^8 a_{24\sigma_2}) + tr((\rthree^6 \grthree^2+ \rthree^2 \grthree^6)a_{24\sigma_2})   \\ 
&=& \done^{12}u_{12\sigma} a_{12\lambda} + tr(res(\dthree^4 u_{12\sigma}a_{12\lambda}))+ tr(\rthree^8 a_{24\sigma_2}) + tr(res(\dthree^2 u_{6\sigma}a_{6\lambda} tr(\rthree^4 a_{12\sigma_2}))  \\ 
&=& \done^{12}u_{12\sigma} a_{12\lambda} +2\dthree^4 u_{12\sigma}a_{12\lambda}+  \sthree^8 + 2 \dthree^2 u_{6\sigma}a_{6\lambda} tr(\rthree^4 a_{12\sigma_2}) \\ 
&=& \done^{12}u_{12\sigma} a_{12\lambda} +2\dthree^4 u_{12\sigma}a_{12\lambda}+  \sthree^8 ,
\end{eqnarray*}
which is not zero on the $E_{13}$-page.  Therefore, the normed $d_{13}$-differential on $u_{4\lambda}a_\sigma$ occurs. 
\end{proof}

\begin{rem}\rm \label{rem:Normd13Intuition}
The term $tr(\rthree^2 u_{\sigma}a_{14\sigma_2})$ in the expression of the target can also be rewritten as 
\begin{eqnarray*}
tr(\rthree^2 u_{\sigma}a_{14\sigma_2}) &=& tr(\rthree(\rthree + \grthree)u_{\sigma}a_{14\sigma_2} + \rthree \grthree u_{\sigma}a_{14\sigma_2}) \\
&=& tr(res(\dthree u_{4\sigma}a_{7\lambda})) + tr( res(tr(\rthree a_{3\sigma_2})a_{4\lambda}) \rthree u_\sigma a_{3\sigma_2}) \\
&=& 2\dthree u_{4\sigma}a_{7\lambda} + tr(\rthree u_\sigma a_{3\sigma_2}) tr(\rthree a_{3\sigma_2})a_{4\lambda}.
\end{eqnarray*}
Intuitively, this is saying that after the $d_{13}$-differentials, 
$$\done^3 = 2\dthree + \text{terms in $i_{C_2}^*\BPone$-truncations}.$$
This intuition will be useful in the proof of the next theorem.  
\end{rem}

\begin{thm}\label{thm:NormedDiffd29}
In the $C_4$-slice spectral sequence of $\BPtwo$, the class $u_{8\lambda}a_\sigma$ supports the $d_{29}$-differential 
$$d_{29}(u_{8\lambda}a_\sigma) = \dthree^2 \done u_{8\sigma}a_{15\lambda} + tr(\rthree^4 \rone^2 u_\sigma a_{15\lambda}).$$
\end{thm}
\begin{proof}
Applying Theorem~\ref{thm:NormFormula} to the $d_{15}$-differential ${d_{15}(u_{8\sigma_2}) = \rone(\rthree^2 + \rthree \grthree + \grthree^2) a_{15\sigma_2}}$ predicts the $d_{29}$-differential 
$$d_{29}\left(\frac{u_{8\lambda}}{u_{8\sigma}}a_\sigma \right) = N(\rone (\rthree^2 + \rthree \grthree + \grthree^2))a_{15\lambda}.$$
In particular, the proof of Theorem~\ref{thm:NormFormula} shows that the class ${N(\rone (\rthree^2 + \rthree \grthree + \grthree^2))a_{15\lambda}}$ must be killed on or before the $E_{29}$-page.  

We claim that the class $N(\rone (\rthree^2 + \rthree \grthree + \grthree^2))u_{8\sigma}a_{15\lambda}$ must also be killed on or before the $E_{29}$-page.  This is because if $N(\rone (\rthree^2 + \rthree \grthree + \grthree^2))a_{15\lambda}$ is killed by a differential of length $r < 29$, then since $u_{8\sigma}$ is a $d_r$-cycle, ${N(\rone (\rthree^2 + \rthree \grthree + \grthree^2))u_{8\sigma}a_{15\lambda}}$ is killed by a differential of length at most $r$.  Otherwise, if ${N(\rone (\rthree^2 + \rthree \grthree + \grthree^2))a_{15\lambda}}$ is killed by a differential of length exactly 29, then the only possible source is $\frac{u_{8\lambda}}{u_{8\sigma}} a_\sigma$ (see \cite[Figure~5]{ZengThesis}), and the predicted $d_{29}$-differential occurs.  Since $d_{29}(u_{8\sigma}) = 0$, applying the Leibniz rule yields the $d_{29}$-differential 
$$d_{29}(u_{8\lambda}a_\sigma) = N(\rone (\rthree^2 + \rthree \grthree + \grthree^2))u_{8\sigma}a_{15\lambda}.$$  

To prove that this $d_{29}$-differential on $u_{8\lambda}a_\sigma$ exists, it suffices to show that the predicted target, ${N(\rone (\rthree^2 + \rthree \grthree + \grthree^2))u_{8\sigma}a_{15\lambda}}$, is not zero on the $E_{29}$-page.  Once we show this, it would imply that this class must be killed by a $d_{29}$-differential (and the only possible source is $u_{8\lambda}a_\sigma$). 

The restriction of $N(\rone (\rthree^2 + \rthree \grthree + \grthree^2))$ is 
\begin{eqnarray*}
res(N(\rone (\rthree^2 + \rthree \grthree + \grthree^2))) &=& \rone (\rthree^2 + \rthree \grthree + \grthree^2) \cdot \grone (\grthree^2 -\rthree \grthree + \rthree^2) \\ 
&=& (\rthree \grthree)^2 \rone \grone + (\rthree^4 + \grthree^4) \rone \grone \\ 
&=& res(\dthree^2 \done) + (1+\gamma)\rthree^4 \rone \grone u_{-7\sigma}. 
\end{eqnarray*}
Therefore we have
\begin{eqnarray*}
{N(\rone (\rthree^2 + \rthree \grthree + \grthree^2))u_{8\sigma}a_{15\lambda}}&=&\dthree^2 \done u_{8\sigma}a_{15\lambda} + tr(\rthree^4 \rone \grone u_{-7\sigma})u_{8\sigma}a_{15\lambda} \\
&=& \dthree^2 \done u_{8\sigma}a_{15\lambda} + tr(\rthree^4 \rone^2 u_\sigma a_{30\sigma_2}). 
\end{eqnarray*}
To show that this class is not zero, we will multiply the predicted differential by $\dthree^5\done^2 u_{16\sigma}a_{9\lambda}$ (a permanent cycle) to bring it to the integer graded part of the spectral sequence:
$$d_{29}(\dthree^5 \done^2 u_{8\lambda} u_{16\sigma} a_{9\lambda}a_\sigma) =  \dthree^7 \done^3 u_{24\sigma} a_{24\lambda} + \dthree^5 \done^2 u_{16\sigma} a_{9\lambda} tr(\rthree^4 \rone^2 u_\sigma a_{30\sigma_2}).$$
The source of this differential is at $(49, 19)$ and the target is at $(48, 48)$ (see Figure~\ref{fig:NormedDiffd29d61}).  Once we verify that the target of this new differential is not zero on the $E_{29}$-page, we can then conclude that the original target, ${N(\rone (\rthree^2 + \rthree \grthree + \grthree^2))u_{8\sigma}a_{15\lambda}}$, is also not zero on the $E_{29}$-page.

The new target is equal to 
\begin{eqnarray*}
&&\dthree^7 \done^3 u_{24\sigma} a_{24\lambda} + \dthree^5 \done^2 u_{16\sigma} a_{9\lambda} tr(\rthree^4 \rone^2 u_\sigma a_{30\sigma_2})  \\
&=& \dthree^7 \done^3 u_{24\sigma} a_{24\lambda} +  tr((\rthree \grthree)^5\rthree^4 \rone^4 \grone^2 a_{48\sigma_2})\\
&=& \dthree^7 \done^3 u_{24\sigma} a_{24\lambda} +  tr((\rthree \grthree)^5\rthree^4 \rone^6 a_{48\sigma_2})\\
&=& \dthree^7 \done^3 u_{24\sigma} a_{24\lambda} + tr((\rthree \grthree)^5\rthree^4 (\rthree + \grthree)^2 a_{48\sigma_2}) \, \, \, \text{(in the $C_2$-spectral sequence, $\rone^3 = \rthree + \grthree$)}\\
&=& \dthree^7 \done^3 u_{24\sigma} a_{24\lambda} + tr((\rthree \grthree)^5\rthree^6 a_{18\sigma_2}) + tr((\rthree \grthree)^5 \rthree^4 \grthree^2 a_{48\sigma_2})\\
&=& \dthree^7 \done^3 u_{24\sigma} a_{24\lambda} + \dthree^5 \sthree^6 + \dthree^7 \sthree^2. \\
\end{eqnarray*}
To further simplify this, take the $d_{13}$-differential on $u_{4\lambda}a_\sigma$ and multiply it by $\dthree^7u_{20\sigma}a_{17\lambda}$.  This produces the $d_{13}$-differential 
\begin{eqnarray*}
d_{13}(\dthree^7 u_{4\lambda}u_{20\sigma}a_{17\lambda}a_\sigma) &=& \dthree^7u_{20\sigma}a_{17\lambda}(\done^3 u_{4\sigma}a_{7\lambda} + tr(\rthree^2u_{\sigma} a_{14\sigma_2} ) ) \\ 
&=& \dthree^7 \done^3 u_{24\sigma}a_{24\lambda} + tr((\rthree \grthree)^7 \rthree^2 a_{48\sigma_2})\\ 
&=& \dthree^7 \done^3 u_{24\sigma}a_{24\lambda} + \dthree^7 \sthree^2.
\end{eqnarray*}
This is a differential whose source is at $(49, 35)$ and target is at $(48, 48)$.  It introduces the relation $\dthree^7 \done^3 u_{24\sigma}a_{24\lambda} + \dthree^7 \sthree^2 = 0$ after the $E_{13}$-page.  Therefore, the new target can be identified with $\dthree^5 \sthree^6 = tr((\rthree \grthree)^5 \sthree^6 a_{48\sigma_2})$.  

We will now show that $\dthree^5 \sthree^6 \neq 0$ on the $E_{29}$-page.  Recall that in the $C_2$-slice spectral sequence, we have the $d_{15}$-differential 
\begin{eqnarray*}
d_{15}((\rthree \grthree)^5 \rthree^3 \rone^2 u_{8\sigma_2}a_{33\sigma_2}) &=& (\rthree \grthree)^5 \rthree^3 \rone^2 \cdot \rone (\rthree^2 + \rthree \grthree + \grthree^2) a_{48\sigma_2} \\
&=& (\rthree \grthree)^5 \rthree^3 \cdot (\rthree + \grthree)(\rthree^2 + \rthree \grthree + \grthree^2) a_{48\sigma_2} \\
&=& (\rthree \grthree)^5 \rthree^3(\rthree^3 + \grthree^3) a_{48\sigma_2} \\ 
&=& (\rthree \grthree)^5 \rthree^6 a_{48\sigma_2} + (\rthree \grthree)^8a_{48\sigma_2}.
\end{eqnarray*}
Applying the transfer to this differential yields the $d_{15}$-differential 
$$d_{15}(\dthree^5 \sthree^3 \rone^2) = \dthree^5 \sthree^6 + 2 \dthree^8 u_{24\sigma} a_{24\lambda}$$
in the $C_4$-spectral sequence (cf. Proposition~\ref{prop:DthreeSthreeRelations}).  Therefore, the new target of our normed $d_{29}$-differential can be identified with the class $2 \dthree^8 u_{24\sigma} a_{24\lambda}$ on the $E_{29}$-page, which is not zero.  It follows that the original target, ${N(\rone (\rthree^2 + \rthree \grthree + \grthree^2))u_{8\sigma}a_{15\lambda}}$, is also not zero on the $E_{29}$-page.  This completes the proof of the theorem.  
\end{proof}

\begin{thm}\label{thm:NormedDiffd61}
In the $C_4$-slice spectral sequence of $\BPtwo$, the class $u_{16\lambda}a_\sigma$ supports the $d_{61}$-differential 
$$d_{61}(u_{16\lambda}a_\sigma) = \dthree^5 u_{16\sigma}a_{31\lambda}.$$
\end{thm}
\begin{proof}
Applying Theorem~\ref{thm:NormFormula} to the $d_{31}$-differential $d_{31}(u_{16\sigma_2}) = \rthree^4 \grthree a_{31\sigma_2}$ predicts the $d_{61}$-differential 
$$d_{61}\left(\frac{u_{16\lambda}}{u_{16\sigma}}a_\sigma \right) =  N(\rthree^4 \grthree) a_{31\lambda} = \dthree^5 a_{31\lambda}. $$
In particular, the proof of Theorem~\ref{thm:NormFormula} shows that the class $\dthree^5 a_{31\lambda}$ must be killed on or before the $E_{61}$-page.  

We claim that the class $\dthree^5 u_{16\sigma}a_{31\lambda}$ must also be killed on or before the $E_{61}$-page.  This is because if $\dthree^5 a_{31\lambda}$ is killed by a differential of length $r < 61$, then since $u_{16\sigma}$ is a $d_r$-cycle, $\dthree^5 u_{16\sigma}a_{31\lambda}$ is killed by a differential of length at most $r$.  Otherwise, if $\dthree^5 a_{31\lambda}$ is killed by a differential of length exactly 61, then the only possible source is $\frac{u_{16\lambda}}{u_{16\sigma}} a_\sigma$ (see \cite[Figure~5]{ZengThesis}), and the predicted $d_{61}$-differential occurs.  Since $d_{61}(u_{16\sigma}) = 0$, applying the Leibniz rule yields the $d_{61}$-differential 
$$d_{61}(u_{16\lambda}a_\sigma) = \dthree^5 u_{16\sigma}a_{31\lambda}.$$  

To prove that this $d_{61}$-differential on $u_{16\lambda}a_\sigma$ occurs, it suffices to show that the predicted target, $\dthree^5 u_{16\sigma}a_{31\lambda}$, is not zero on the $E_{61}$-page.  Once we show this, it would imply that this class must be killed by a $d_{61}$-differential (and the only possible source is $u_{16\lambda}a_\sigma$).

To show that the predicted target is not zero on the $E_{61}$-page, we will multiply the predicted $d_{61}$-differential on $u_{16\lambda}a_\sigma$ by $\dthree^{11}u_{32\sigma}a_{17\lambda}$ (a permanent cycle) to move it to the integer graded spectral sequence:
$$d_{61}(\dthree^{11}u_{16\lambda}u_{32\sigma}a_{17\lambda}a_\sigma) = \dthree^{16} u_{48\sigma}a_{48\lambda}. $$
This is a predicted $d_{61}$-differential whose source is at $(97, 65)$ and whose target is at $(96,96)$.  Once we verify that the target of this differential is not zero on the $E_{61}$-page, we can then conclude that the original target, $\dthree^5 u_{16\sigma}a_{31\lambda}$, is also not zero on the $E_{61}$-page.

Theorem~\ref{thm:NormedDiffd29} shows that there is a $d_{29}$-differential 
$$d_{29}(\dthree^{13} \done^2 u_{8\lambda} u_{40\sigma}a_{33\lambda} a_\sigma) = 2\dthree^{16} u_{48\sigma} a_{48\lambda}.$$
For degree reasons, if the class $\dthree^{16} u_{48\sigma}a_{48\lambda}$ is zero on the $E_{61}$-page, then the only possibility is for it to be killed by a $d_{31}$-differential (see Figure~\ref{fig:NormedDiffd29d61}).  This is impossible by considering the class $\dthree^{17}u_{48\sigma}a_{51\lambda} a_{3\sigma} = \dthree^{16} u_{48\sigma} a_{48\lambda} \cdot \dthree a_{3\lambda} a_{3\sigma}$ at $(99, 105)$.  If the class $\dthree^{16} u_{48\sigma}a_{48\lambda}$ becomes zero after the $d_{31}$-differentials, then $\dthree^{17}u_{48\sigma}a_{51\lambda} a_{3\sigma}$ will also become zero after the $E_{31}$-page (this is because $\dthree a_{3\lambda} a_{3\sigma}$ is a permanent cycle).  However, by degree reasons, $\dthree^{17}u_{48\sigma}a_{51\lambda} a_{3\sigma}$ cannot be killed by a differential of length at most 31.  

Therefore, the class $\dthree^{16} u_{48\sigma} a_{48\lambda}$ is not zero on the $E_{61}$-page.  It follows that $\dthree^5 u_{16\sigma}a_{31\lambda}$ is also not zero on the $E_{61}$-page and the normed $d_{61}$-differential on $u_{8\lambda}a_\sigma$ occurs. 
\end{proof}

\begin{figure}
\begin{center}
\makebox[\textwidth]{\includegraphics[trim={0cm 9cm 0cm 9cm}, clip, page = 1, scale = 0.45]{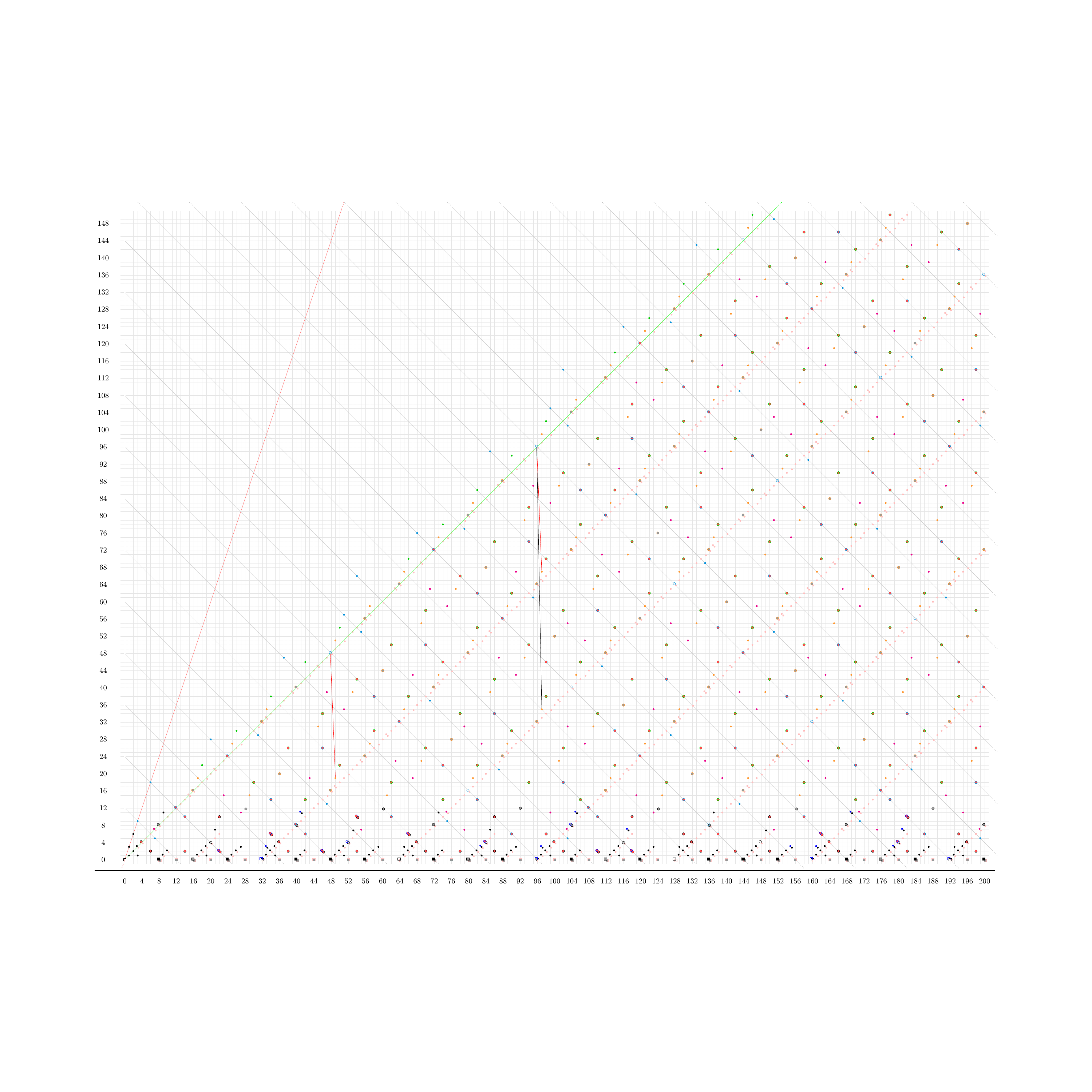}}
\end{center}
\begin{center}
\caption{Normed $d_{29}$ and $d_{61}$-differentials. }
\hfill
\label{fig:NormedDiffd29d61}
\end{center}
\end{figure}

%%%%%%%%%%%%%%%%%%%%%%%%
%%%%%%%%%%%%%%%%%%%%%%%%
\section{Higher Differentials III: The Vanishing Theorem} %
\subsection{$\alpha$ and $\alpha^2$}
Let $\alpha$ be the class $\dthree^8 u_{24\sigma} u_{24\lambda}$ at $(48, 48)$.  Theorem~\ref{thm:NormedDiffd29} and Theorem~\ref{thm:NormedDiffd61} show that 
\begin{enumerate}
\item The class $\alpha$ is a permanent cycle that survives to the $E_\infty$-page; 
\item The class $\alpha^2 = \dthree^{16} u_{48\sigma} a_{48\lambda}$ at $(96, 96)$ is killed by the $d_{61}$-differential 
$$d_{61}(\dthree^{11}u_{16\lambda}u_{32\sigma} a_{17\lambda}a_\sigma) = \alpha^2. $$
\end{enumerate}

In the $C_4$-slice spectral sequence of $\BPtwo$, $\alpha$ will be playing the role of $\epsilon$ in $C_4\text{-}\SliceSS(\BPone)$.  The following lemma is the higher height analogue of Lemma~\ref{lem:BPoneMainLem} and is proven using the exact same method.  

\begin{lem}\label{lem:BPtwoMainLem}
Let $d_r(x) = y$ be a nontrivial differential in $C_4$-$\SliceSS(\BPtwo)$.
\begin{enumerate}
\item The classes $\alpha x$ and $\alpha y$ are both nonzero on the $E_r$-page, and $d_r(\alpha x ) = \alpha y$.
\item If both $x$ and $y$ are divisible by $\alpha$ on the $E_2$-page, then $x/\alpha$ and $y/\alpha$ both survive to the $E_r$-page, and $d_r(x/\alpha) = y/\alpha$.  
\end{enumerate}
\end{lem}
\begin{proof}
We will prove both statements by using induction on $r$, the length of the differential.  Both claims are true when $r \leq 15$.  

Now, suppose that both statements hold for all differentials of length smaller than $r$.  Given a nontrivial differential $d_r(x) = y$, we will first show that $\alpha y$ survives to the $E_r$-page.  

If $\alpha y$ supports a differential, then $y$ must support a differential as well.  This is a contradiction because $y$ is the target of a differential.  Therefore if $\alpha y$ does not survive to the $E_r$-page, it must be killed by a differential $d_k(z) = \alpha y$, where $k <r$.  

We claim that $z$ is divisible by $\alpha$.  If $k \leq 15$, then this is true because we have characterized completely all the differentials of length $\leq 15$, and in all the cases $z$ will be divisible by $\alpha$.  If $k > 15$, then $k$ will be divisible by $\alpha$ as well because it is a class on or under the line of slope 1 with filtration at least 48, and all such classes are divisible by $\alpha$ starting from the $E_{16}$-page.  

The inductive hypothesis, applied to the differential $d_k(z) = \alpha y$, shows that $d_k(z/\alpha) = y$.  This is a contradiction because $d_r(x) = y$ is a nontrivial $d_r$-differential.  Therefore, $\alpha y$ survives to the $E_r$-page.  

$$\begin{tikzcd}
& \alpha y &\\
y \ar[ru, dashed, dash, "\cdot \alpha"] && \\
&& \\ 
&&z \ar[luuu, swap, "d_k"]\\ 
&z/\alpha \ar[ru, dashed, dash, "\cdot \alpha"]\ar[luuu, swap, "d_k"]& \alpha x \\ 
& x \ar[luuuu, "d_r"] \ar[ru, dashed, dash, "\cdot \alpha"]& 
\end{tikzcd}$$

If $\alpha x$ does not survive to the $E_r$-page, then it must be killed by a shorter differential.  This shorter differential will introduce the relation $\epsilon x = 0$ on the $E_r$-page.  However, the Leibniz rule, applied to the differential $d_r(x) = y$, shows that 
$$d_r(\alpha x) = \alpha y \neq 0$$
on the $E_r$-page.  This is a contradiction.  Therefore, $\alpha x$ must survive to the $E_r$-page as well, and it supports the differential 
$$d_r(\alpha x) = \alpha y.$$ 
This proves (1). 

To prove (2), note that if $y/\alpha$ supports a differential of length smaller than $r$, then the induction hypothesis would imply that $y$ also supports a differential of the same length.  Similarly, if $y/\alpha$ is killed by a differential of length smaller than $r$, then the induction hypothesis would imply that $y$ is also killed a by a differential of the same length.  Both scenarios lead to contradictions.  Therefore, $y/\alpha$ survives to the $E_r$-page.  

We will now show that $x/\alpha$ also survives to the $E_r$-page.  Since $x$ supports a $d_r$-differential, $x/\alpha$ must also support a differential of length at most $r$.  Suppose that $d_k (x/\alpha) = z$, where $k < r$.  The induction hypothesis, applied to this $d_k$-differential, implies the existence of the differential $d_k(x) = \alpha z$.  This is a contradiction because $d_r(x) = y$.  
$$\begin{tikzcd}
&y& \\
y/\alpha \ar[ru, dashed, dash, "\cdot \alpha"]&\alpha z&  \\ 
z\ar[ru, dashed, dash, "\cdot \alpha"]&& \\
&&\\
&& x \ar[luuu, "d_k"]\ar[luuuu, swap, "d_r"]\\
&x/\alpha \ar[luuu, "d_k"]\ar[ru, dashed, dash, "\cdot \alpha"]&
\end{tikzcd}$$

It follows that $x/\alpha$ survives to the $E_r$-page, and it supports a nontrivial $d_r$-differential.  Since $y/\alpha$ also survives to the $E_r$-page, the Leibniz rule shows that 
$$d_r(x/\alpha) = y/\alpha,$$
as desired. 
\end{proof}

\begin{thm}[Vanishing Theorem] \label{thm:aboveFiltration61die}
Any class of the form $\alpha^2x$ on the $E_2$-page of $C_4$-$\SliceSS(\BPtwo)$ must die on or before the $E_{61}$-page.  
\end{thm}
\begin{proof}
If $x$ is a $d_{61}$-cycle, then the class $\alpha^2 x$ is a $d_{61}$-cycle as well.  Since $\alpha^2$ is killed by a $d_{61}$-differential, $\alpha^2 x$ must also be killed by a differential of length at most 61.  

Now suppose that the class $x$ is not a $d_{61}$-cycle and it supports the differential $d_r(x) = y$, where $r \leq 61$.  Applying Lemma~\ref{lem:BPtwoMainLem}(1), we deduce that the class $\alpha^2 x$ must support the nontrivial $d_r$-differential 
$$d_r(\alpha^2 x) = \alpha^2 y.$$
Therefore, it cannot survive past the $E_{61}$-page.  
\end{proof}

%%%%%%%%%%%%%%%%%%%%%%%%%%%%%%%
\subsection{Important permanent cycles}
\begin{figure}
\begin{center}
\makebox[\textwidth]{\includegraphics[trim={0cm 10cm 0cm 10cm}, clip, page = 1, scale = 0.23]{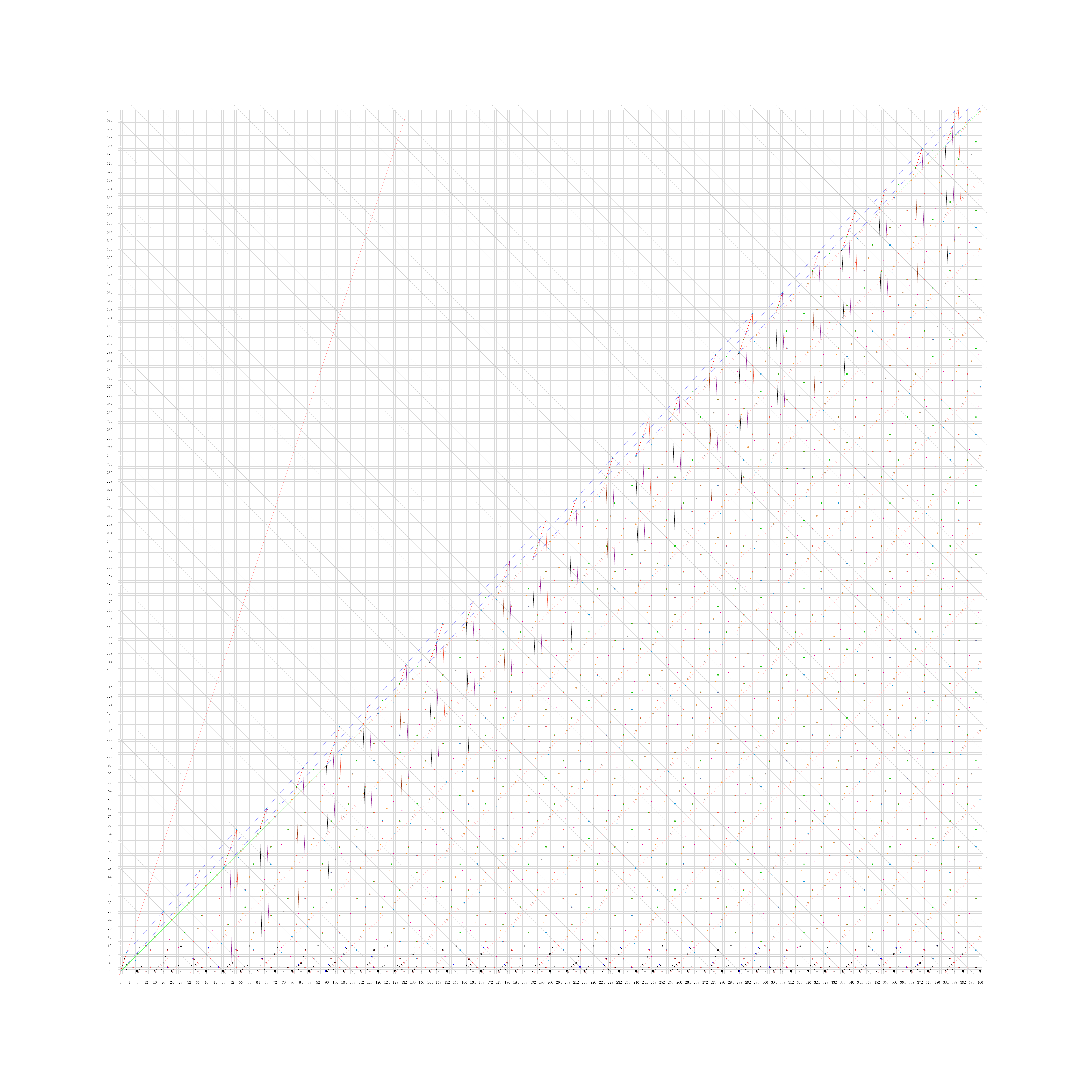}}
\end{center}
\begin{center}
\caption{$E_{16}$-page of $\SliceSS(\BPtwo)$.  The differentials shown are the long differentials that cross the vanishing line of slope 1.  The black differentials are the $d_{61}$-differentials; the {\color{RawSienna} sienna differentials} are the $d_{59}$-differentials; the {\color{Plum} plum differentials} are the $d_{53}$-differentials; and the {\color{RedOrange} red-orange differentials} are the $d_{43}$-differentials.} 
\hfill
\label{fig:E4C4E16page}
\end{center}
\end{figure}

\begin{prop}
The following classes are permanent cycles that survive to the $E_\infty$-page of $C_4\text{-}\SliceSS(\BPtwo)$.  
\begin{itemize}
\item $\eta':= \done a_\lambda a_\sigma$ at $(1,3)$;
\item $\xi:= \dthree a_{3\lambda} a_{3\sigma}$ at $(3, 9)$;
\item $\epsilon':= $ at $(8, 8)$; 
\item $\beta:= \dthree^3 u_{8\sigma} a_{9\lambda}a_\sigma$ at $(17, 19)$.
\end{itemize}
\end{prop}
\begin{proof}
All the classes are clearly permanent cycles.  It is also immediately clear that $\eta'$ and $\xi$ survive to the $E_\infty$-page.  

Suppose that $\epsilon'$ is killed by a differential.  For degree reasons, the length of that differential must be 7.  This implies that $\alpha \cdot \epsilon'$ at $(56, 56)$ must also be killed on or before the $E_7$-page.  This is impossible for degree reasons.  This shows that $\epsilon'$ survives to the $E_\infty$-page.  

Now, suppose that $\beta$ is killed by a differential.  For degree reasons, the length of that differential must be 17.  This implies that $\alpha \cdot \beta$ at $(65, 67)$ must also be killed on or before the $E_{17}$-page.  This is again impossible because of degree reasons. 
\end{proof}

In Figure~\ref{fig:E4C4E16page}, we have drawn some multiplications by $\xi$ ({\color{red} red structure lines}) and $\beta$ ({\color{blue} blue structure lines}).  These multiplications will be useful later when we prove long differentials that cross the vanishing line of slope 1.  

\begin{prop}\label{prop:gammaPowers}
Let $\gamma:=\dthree^2 u_{4\lambda}u_{6\sigma}a_{2\lambda}$ at $(20, 4)$.  Then 
\begin{enumerate}
\item $d_{13}(\gamma) = \dthree^3 u_{\lambda} u_{8\sigma} a_{8\lambda}a_\sigma$ ($d_{13}(20,4) = (19, 17)$) ; 
\item $d_{15}(\gamma^2) =\dthree^5 \sone u_{4\lambda}u_{15\sigma}a_{11\lambda}a_{\sigma_2}$ ($d_{15}(40, 8) = (39, 23)$); 
\item $d_{31}(\gamma^4) = \dthree^9 \sthree^3 = tr(\rthree^{12} \grthree^9 u_{16\sigma_2}a_{47\sigma_2})$ ($d_{31}(80, 16) = (79, 47)$); 
\item The class $\gamma^8 = \dthree^{16} u_{32\lambda}u_{48\sigma}a_{16\lambda}$ at $(160,32)$ is a permanent cycle.  
\end{enumerate}
\end{prop}
\begin{proof}
(1) and (2) are Proposition~\ref{prop:BPtwod13(20,4)} and Propostion~\ref{prop:d15(40,8)}, respectively.  

To prove (3), note that $\gamma^4 \cdot \beta$ is the class $\dthree^{11}u_{16\lambda}u_{32\sigma}a_{17\lambda}a_\sigma$ at $(97, 35)$, which supports the $d_{61}$-differential killing $\alpha^2$.  Therefore $\gamma^4$ must support a differential of length at most 61.  The possible targets are the following classes:  
\begin{itemize}
\item $\dthree^9 \sthree^3$ at $(79, 47)$; 
\item $\dthree^{10}\sthree$ at $(79, 47)$;
\item $\dthree^{13}u_{\lambda}u_{38\sigma}a_{38\lambda}a_\sigma$ at $(79, 77)$;
\item $\dthree^{12}\sthree^2 \rone$ at $(79, 79)$.
\end{itemize}
We know all the $d_{31}$-differentials between $i_{C_2}^* \BPone$-truncation classes.  In particular, the class $\dthree^{10}\sthree$ at $(79, 47)$ supports the $d_{31}$-differential 
$$d_{31}(\dthree^{10}\sthree) = \dthree^{12}\sthree^2 \,\,\,(d_{31}(79,47) = (78, 78)),$$
and the class $\dthree^{12}\sthree^2 \rone$ at $(79, 79)$ is killed by the $d_{31}$-differential 
$$d_{31}(\dthree^{10}\sthree \rone) = \dthree^{12}\sthree^2 \rone \,\,\,(d_{31}(80,48) = (79, 79)).$$

Now, consider the class $\dthree^{10}\done^2 u_{8\lambda} u_{32\sigma} a_{24\lambda}$ at $(80, 48)$.  The product of this class with $\beta = \dthree^3u_{8\sigma}a_{9\lambda}a_\sigma$ is the class
$$\dthree^{10}\done^2 u_{8\lambda} u_{32\sigma} a_{24\lambda} \cdot \dthree^3u_{8\sigma}a_{9\lambda}a_\sigma = \dthree^{13} \done^2 u_{8\lambda}u_{40\sigma} a_{33\lambda}a_\sigma$$
at $(97, 67)$.  This class supports the $d_{29}$-differential 
$$d_{29}(\dthree^{13} \done^2 u_{8\lambda}u_{40\sigma} a_{33\lambda}a_\sigma) = 2\dthree^{16}u_{48\sigma}a_{48\lambda} \,\,\, (d_{29}(97, 67) = 2(96,96)).$$
Therefore, the class $\dthree^{10}\done^2 u_{8\lambda} u_{32\sigma} a_{24\lambda}$ must support a differential of length at most 29.  The only possibility is the class $\dthree^{13}u_{\lambda}u_{38\sigma}a_{38\lambda}a_\sigma$ at $(79, 77)$.  

It follows that the only possibility for the target of the differential supported by $\gamma^4$ is the class $\dthree^9 \sthree^3$ at $(79, 47)$.  The differential will be the $d_{31}$-differential that we claimed.  This proves (3).   

For (4), since the classes $\dthree^{16}$, $u_{32\lambda}$, $u_{48\sigma}$, and $a_{16\lambda}$ are all permanent cycles, their product is a permanent cycle as well.  
\end{proof}

Proposition~\ref{prop:gammaPowers} shows that whenever we have proved a $d_r$-differential of length $r < 31$, we can multiply that differential by $\alpha$ and $\gamma^4$ and use Lemma~\ref{lem:BPtwoMainLem} to deduce more $d_r$-differentials.  If the length of the differential is $r \geq 31$, then we can multiply that differential by $(48, 48)$ and $(160,32)$ and use Lemma~\ref{lem:BPtwoMainLem} to produce more $d_r$-differentials.  

\subsection{Long Differentials Crossing the Line of Slope 1}
\begin{thm}\label{thm:LongCrossingDifferentials}
The following differentials exist:
\begin{enumerate}
\item $d_{61}(\overline{\mathfrak{d}}_3^{11} u_{16\lambda} u_{32\sigma} a_{17\lambda} a_{\sigma}) = \alpha^2 $ ($d_{61}(97,35) =(96,96)$);
\item $d_{61}(2\overline{\mathfrak{d}}_3^{14} u_{15\lambda} u_{42\sigma} a_{27\lambda}) = \alpha^2\beta$ ($d_{61}(2(114,54)) =(113,115)$);
\item $d_{59}(\overline{\mathfrak{d}}_3^{17} \overline{s}_1 u_{14\lambda} u_{51\sigma} a_{37\lambda} a_{\sigma_2}) = \alpha^2\beta^2$ ($d_{59}(131,75) =(130,134)$);
\item $d_{53}(\overline{\mathfrak{d}}_3^{20} \overline{\mathfrak{d}}_1^{2} u_{12\lambda} u_{62\sigma} a_{50\lambda}) = \alpha^2\beta^3$ ($d_{53}(148,100) =(147,153)$);
\item $d_{53}(\overline{\mathfrak{d}}_3^{23} \overline{\mathfrak{d}}_1^{2} u_{12\lambda} u_{70\sigma} a_{59\lambda} a_{\sigma}) = \alpha^2\beta^4$ ($d_{53}(165,119) =(164,172)$);
\item $d_{53}(2\overline{\mathfrak{d}}_3^{26} \overline{\mathfrak{d}}_1^{2} u_{11\lambda} u_{80\sigma} a_{69\lambda}) = \alpha^2\beta^5$ ($d_{53}(2(182,138)) =(181,191)$);
\item $d_{43}(\overline{\mathfrak{d}}_3^{29} \overline{s}_3^{3}) = \alpha^2\beta^6$ ($d_{43}(199,167) =(198,210)$).
\end{enumerate}
\end{thm}
\begin{proof}
Since $\beta$ is a permanent cycle, the classes $\alpha^2 \beta^i$ ($1\leq i \leq 6$) are all killed on or before the $E_{61}$-page by differentials of decreasing length.  More precisely, if a $d_r$-differential kills $\alpha^2 \beta^i$ and a $d_{r'}$-differential kills $\alpha^2 \beta^{i'}$ with $0 \leq i < i' \leq 6$, then $r \geq r'$.  

Consider the class $\alpha^2\beta^6$.  The shortest differential that can kill this class is a $d_{43}$-differential.  Therefore, all the differentials killing the class $\alpha^2 \beta^i$ for $1 \leq i \leq 6$ must all be of length at least 43 and at most 61. 

(1) follows directly from Theorem~\ref{thm:NormedDiffd61}.   

For (2), the only differential that can kill $\alpha^2 \beta$ that's of length $43 \leq r \leq 61$ is the claimed $d_{61}$-differential.  

For (3), the only differential that can kill $\alpha^2 \beta^2$ that's of length $43 \leq r \leq 61$ is the claimed $d_{59}$-differential.  

For (4), the only differential that can kill $\alpha^2 \beta^3$ that's of length $43 \leq r \leq 59$ is the claimed $d_{53}$-differential.  

For (5), the only differential that can kill $\alpha^2 \beta^4$ that's of length $43 \leq r \leq 53$ is the claimed $d_{53}$-differential.  

For (6), the only differential that can kill $\alpha^2 \beta^5$ that's of length $43 \leq r \leq 53$ is the claimed $d_{53}$-differential.  

Lastly, for (7), the only differential that can kill $\alpha^2 \beta^6$ that's of length $43 \leq r \leq 53$ is the claimed $d_{43}$-differential.  
\end{proof}
Applying Lemma~\ref{lem:BPtwoMainLem} to the differentials in Theorem~\ref{thm:LongCrossingDifferentials}, we obtain all the other long differentials crossing the line of slope 1.  These differentials are shown in Figure~\ref{fig:E4C4E16page}. 

%%%%%%%%%%%%%%%%%%%%%%%%
%%%%%%%%%%%%%%%%%%%%%%%%
\section{Higher differentials IV: Everything until the $E_{29}$-page} \label{sec:HigherDifferentialsIV}
\begin{figure}
\begin{center}
\makebox[\textwidth]{\includegraphics[trim={0cm 10cm 0cm 10cm}, clip, page = 1, scale = 0.23]{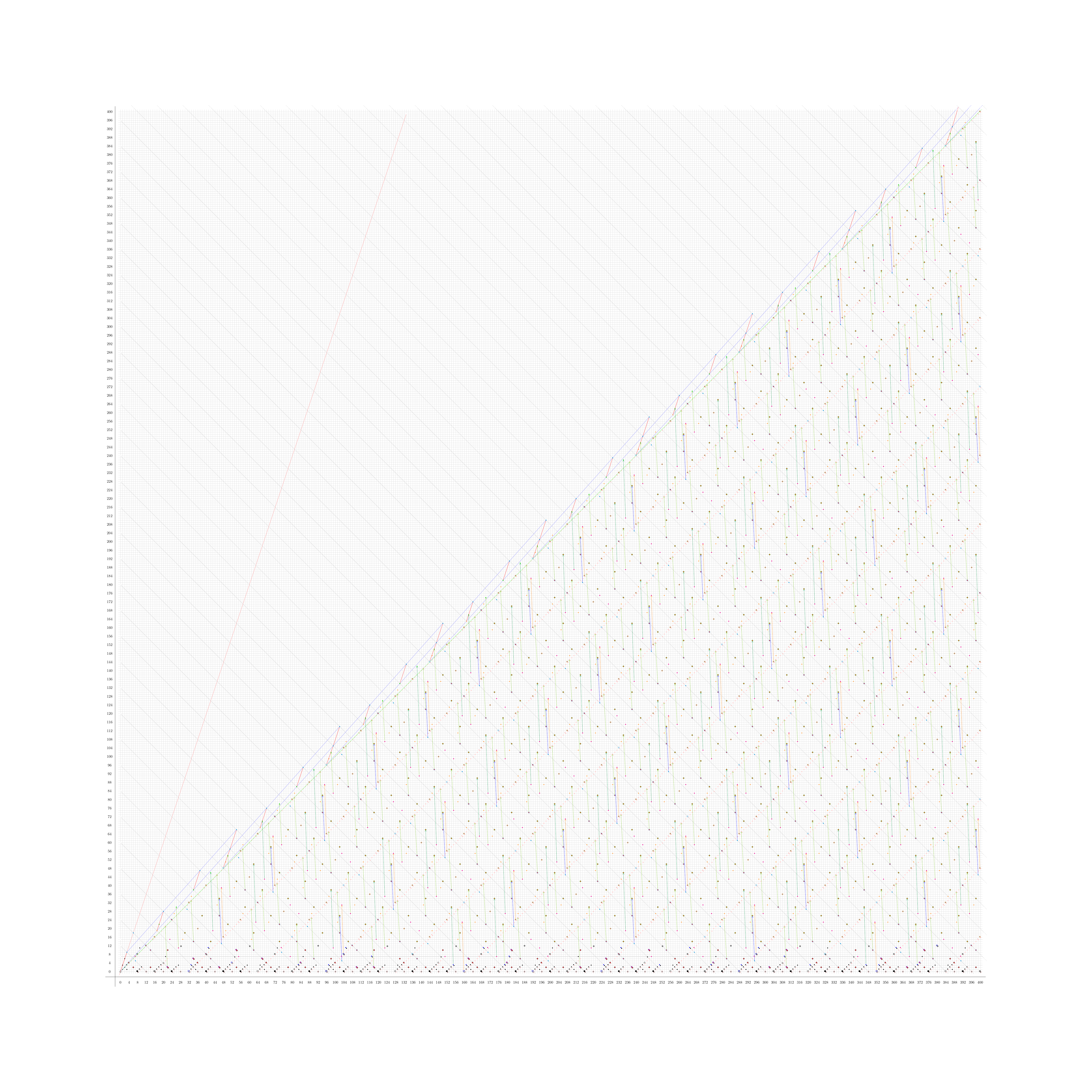}}
\end{center}
\begin{center}
\caption{$d_r$-differentials of lengths $15 < r < 29$.  The $d_{19}$-differentials are shown in {\color{LimeGreen} lime-green}, the $d_{21}$-differentials are shown in {\color{blue} blue}, the $d_{23}$-differentials are shown in {\color{orange} orange}, and the $d_{27}$-differentials are shown in {\color{ForestGreen} forest-green}. }
\hfill
\label{fig:E4C4E19toE27page}
\end{center}
\end{figure}

\subsection{$d_{21}$-differentials}
\begin{prop}\label{prop:d21Differential}
The following $d_{21}$-differentials exist: 
\begin{eqnarray*}
d_{21}(\dthree^{21} u_{9\lambda}u_{62\sigma}a_{54\lambda}a_\sigma) &=&2\dthree^{22} \done^2 u_{3\lambda} u_{68\sigma}a_{65\lambda}  \,\,\,(d_{21}(143,109) = 2(142, 130)), \\
d_{21}(\overline{\mathfrak{d}}_3^{25} u_{9\lambda} u_{74\sigma} a_{66\lambda} a_{\sigma}) &=& 2\overline{\mathfrak{d}}_3^{26} \overline{\mathfrak{d}}_1^{2} u_{3\lambda} u_{80\sigma} a_{77\lambda}\,\,\,(d_{21}(167,133) = 2(166, 154)).
\end{eqnarray*}
\end{prop}
\begin{proof}
The Vanishing Theorem (Theorem~\ref{thm:aboveFiltration61die}) shows that the class $2(142, 30)$ must die on or before the $E_{61}$-page.  For degree reasons, the only possibility is for it to be killed.  The only possibilities for the source of the differential are the following classes: 
\begin{enumerate}
\item $\dthree^{21} u_{9\lambda}u_{62\sigma}a_{54\lambda}a_\sigma$ at $(143, 109)$;
\item $\dthree^{18}\sthree$ at $(143,79)$; 
\item $\dthree^{17}\sthree^3$ at $(143, 79)$.
\end{enumerate} 
Class (2) is killed by the class $\dthree^{15}\sthree^2$ at $(144,48)$ by a $d_{31}$-differential (this is a $d_{31}$-differential between $i_{C_2}^*\BPone$-truncation classes).  Class (3) is killed by the class $2\dthree^{16} u_{24\lambda}u_{48\sigma}a_{24\lambda}$ at $(144,48)$ via a $d_{31}$-differential (see the discussion after Proposition~\ref{prop:d31(40,8)}).  

Therefore, the only possibility for the source is class (1), and we deduce our desired $d_{21}$-differential.  
\end{proof}
All the other $d_{21}$-differentials are obtained from the differentials in Proposition~\ref{prop:d21Differential} by using product structures with the classes $\alpha$ and $\gamma^4$ (see the discussion after Propositon~\ref{prop:gammaPowers}).  These differentials are the {\color{blue} blue differentials} in Figure~\ref{fig:E4C4E19toE27page}.

\subsection{$d_{23}$-differentials}
\begin{prop}\label{prop:d23Differential}
The following $d_{23}$-differentials exist:
\begin{eqnarray*}
d_{23}(\overline{\mathfrak{d}}_3^{5} \overline{\mathfrak{d}}_1 u_{8\lambda} u_{16\sigma} a_{8\lambda}) &=& \overline{\mathfrak{d}}_3^{7} \overline{s}_1 u_{2\lambda} u_{21\sigma} a_{19\lambda} a_{\sigma_2}\,\,\,(d_{23}(48,16)=(47,39)), \\ 
d_{23}(\overline{\mathfrak{d}}_3^{9} \overline{\mathfrak{d}}_1 u_{8\lambda} u_{28\sigma} a_{20\lambda}) &=&\overline{\mathfrak{d}}_3^{11} \overline{s}_1 u_{2\lambda} u_{33\sigma} a_{31\lambda} a_{\sigma_2} \,\,\,(d_{23}(72,40)=(71,63)).
\end{eqnarray*}
\begin{proof}
We will prove the first $d_{23}$-differential.  The proof of the second $d_{23}$-differential is exactly the same.  The restriction of the class $\overline{\mathfrak{d}}_3^{5} \overline{\mathfrak{d}}_1 u_{8\lambda} u_{16\sigma} a_{8\lambda}$ is $\rthree^5 \grthree^5\rone \grone u_{16\sigma_2}a_{16\sigma_2}$ in the $C_2$-slice spectral sequence.  It supports the $d_{31}$-differential 
$$d_{31}(\rthree^5 \grthree^5\rone \grone u_{16\sigma_2}a_{16\sigma_2}) = \rthree^5 \grthree^5 \rone^2 \cdot (\rthree^4 \grthree) a_{47\sigma_2} = \rthree^9 \grthree^6 \rone^2 a_{47\sigma_2}. $$
This implies that in the $C_4$-spectral sequence, the class $\overline{\mathfrak{d}}_3^{5} \overline{\mathfrak{d}}_1 u_{8\lambda} u_{16\sigma} a_{8\lambda}$ must support a differential of length at most 31.  There are two possible choices for the target: 
\begin{enumerate}
\item $\dthree^7 \sthree \rone^2$ at (47, 47) ($d_{31}$-differential);
\item $\overline{\mathfrak{d}}_3^{7} \overline{s}_1 u_{2\lambda} u_{21\sigma} a_{19\lambda} a_{\sigma_2}$ at $(47, 39)$ ($d_{23}$-differential).
\end{enumerate}
Class (1) is impossible for natuality reasons because the class $\dthree^7 \sthree \rone^2$ does not restrict to $\rthree^9 \grthree^6 \rone^2 a_{47\sigma_2}$.  Therefore, the target must be class (2) and we deduce the desired $d_{23}$-differential.  
\end{proof}
\end{prop}
All the other $d_{23}$-differentials are obtained from the differentials in Proposition~\ref{prop:d23Differential} by using product structures with the classes $\alpha$ and $\gamma^4$ (see the discussion after Propositon~\ref{prop:gammaPowers}).  These differentials are the {\color{orange} orange differentials} in Figure~\ref{fig:E4C4E19toE27page}.

\subsection{$d_{19}$ and $d_{27}$-differentials}

%In order to prove all the $d_{19}$ and $d_{27}$-differentials, we need to prove some $d_{29}$ and $d_{35}$-differentials first.  
\begin{lem}\label{lem:tempd29Differentials}
The following $d_{29}$-differentials exist: 
\begin{enumerate}
%\item $d_{29}(2\overline{\mathfrak{d}}_3^{7} \overline{\mathfrak{d}}_1 u_{7\lambda} u_{22\sigma} a_{15\lambda}) = \overline{\mathfrak{d}}_3^{9} \overline{\mathfrak{d}}_1^{2} u_{28\sigma} a_{29\lambda} a_{\sigma}\,\,\, (d_{29}(2(58,30)) = (57,59))$  
\item $d_{29}(2\dthree^3 \done u_{7\lambda} u_{10\sigma} a_{3\lambda}) =\overline{\mathfrak{d}}_3^{5} \overline{\mathfrak{d}}_1^{2} u_{16\sigma} a_{17\lambda} a_{\sigma} \,\,\, (d_{29}(2(34,6)) = (33, 35))$; 
\item $d_{29}(\dthree^2\done^2 u_{8\lambda} u_{8\sigma}) = \overline{\mathfrak{d}}_3^{5} u_{\lambda} u_{14\sigma} a_{14\lambda} a_{\sigma}\,\,\, (d_{29}((32,0)) = (31,29))$; 
\item $d_{29}(\overline{\mathfrak{d}}_3^{22} \overline{\mathfrak{d}}_1^{2} u_{8\lambda} u_{68\sigma} a_{60\lambda}) = \overline{\mathfrak{d}}_3^{25} u_{\lambda} u_{74\sigma} a_{74\lambda} a_{\sigma}\,\,\, (d_{29}((152,120)) = (151,149))$; 
\item $d_{29}(\overline{\mathfrak{d}}_3^{9} \overline{\mathfrak{d}}_1^{2} u_{8\lambda} u_{28\sigma} a_{21\lambda} a_{\sigma}) = 2\overline{\mathfrak{d}}_3^{12} u_{36\sigma} a_{36\lambda}\,\,\, (d_{29}((73,43)) = 2(72, 72))$;
\item $d_{29}(\overline{\mathfrak{d}}_3^{5} \overline{\mathfrak{d}}_1^{2} u_{8\lambda} u_{16\sigma} a_{9\lambda} a_{\sigma}) = 2\overline{\mathfrak{d}}_3^{8} u_{24\sigma} a_{24\lambda}\,\,\, (d_{29}((49,19)) = 2(48,48))$.
\end{enumerate}
\end{lem}
\begin{proof}
\noindent (1): This follows directly from Theorem~\ref{thm:NormedDiffd29}. 

\vspace{0.1in}

\noindent (2): If we multiply the class $\dthree^2\done^2 u_{8\lambda} u_{8\sigma}$ at $(32, 0)$ by $\beta$ and use Theorem~\ref{thm:NormedDiffd29}, we deduce the $d_{29}$-differential 
$$d_{29}(\overline{\mathfrak{d}}_3^{5} \overline{\mathfrak{d}}_1^{2} u_{8\lambda} u_{16\sigma} a_{9\lambda} a_{\sigma}) = 2\overline{\mathfrak{d}}_3^{8} u_{24\sigma} a_{24\lambda} \,\,\,(d_{29}(49,19) = 2(48,48)).$$
Therefore, the class $\dthree^2\done^2 u_{8\lambda} u_{8\sigma}$ at $(32, 0)$ must support a differential of length at most 29.  For degree reasons, the only possible target is the class at $(31, 29)$.

\vspace{0.1in}

\noindent (3): Consider the class $\overline{\mathfrak{d}}_3^{22} \overline{\mathfrak{d}}_1^{2} u_{8\lambda} u_{68\sigma} a_{60\lambda}$ at $(152, 120)$.  We will show that this class must support a differential of length at most 29.  Once we have shown this, the only possible target will be the class $\overline{\mathfrak{d}}_3^{25} u_{\lambda} u_{74\sigma} a_{74\lambda} a_{\sigma}$ at $(151, 149)$.

By Theorem~\ref{thm:NormedDiffd29}, we have the $d_{29}$-differential. 
$$d_{29}(\overline{\mathfrak{d}}_3^{24} \overline{\mathfrak{d}}_1 u_{8\lambda} u_{72\sigma} a_{65\lambda} a_{\sigma}) =\overline{\mathfrak{d}}_3^{26} \overline{\mathfrak{d}}_1^{2} u_{80\sigma} a_{80\lambda} \,\,\,(d_{29}(161,131) = (160,160)).$$
If we multiply the target of this differential by $\epsilon'$, we get the class $2\overline{\mathfrak{d}}_3^{28} u_{84\sigma} a_{84\lambda}$ at $(168,168)$.  This class must be killed by a differential of length at most 29.  For degree reasons, the only possibility is the $d_{29}$-differential 
$$d_{29}(\overline{\mathfrak{d}}_3^{25} \overline{\mathfrak{d}}_1^{2} u_{8\lambda} u_{76\sigma} a_{69\lambda} a_{\sigma}) = 2\overline{\mathfrak{d}}_3^{28} u_{84\sigma} a_{84\lambda}\,\,\,(d_{29}(169,139) = 2(168,168)).$$
The source of this differential is equal to $\beta \cdot (152,120)$.  Therefore, the class $\overline{\mathfrak{d}}_3^{22} \overline{\mathfrak{d}}_1^{2} u_{8\lambda} u_{68\sigma} a_{60\lambda}$ at $(152, 120)$ must support a differential of length 29, as desired. 

\begin{comment}
By Theorem~\ref{thm:aboveFiltration61die}, this class must die on or before the $E_{61}$-page.  If it supports a differential, then the only possible target is the class at $(151, 149)$, and we would be done.  
Now, suppose this class is killed by a differential.  For degree reasons, the possible sources are the following classes: 
\begin{itemize}
\item at $(153,91)$;
\item at $(153,59)$. 
\end{itemize}
We claim that the class at $(153,91)$ is the target of the $d_{29}$-differential
$$d_{29}() = \,\,\,(d_{29}(2(154,62)) = (153,91)).$$
To prove this differential, consider the class at $(154,62)$.  If we multiply this class by $\epsilon'$ and use Theorem~\ref{thm:NormedDiffd29}, we get the $d_{29}$-differential 
$$d_{29}() = \,\,\,(d_{29}(2(162,70)) = (161,99)).$$
This implies that the class at $(154,62)$ must support a differential of length at most 29.  The only possible differential is the one we claimed.  
Similarly, we claim that the class at $(153,59)$ is the source of the $d_{29}$-differential 
$$d_{29}() = \,\,\,(d_{29}((153,59)) = 2(152,88)).$$
To prove this differential, note that by Theorem~\ref{thm:NormedDiffd29}, we have the $d_{29}$-differential 
$$d_{29}() = \,\,\,(d_{29}((145,51)) = (144,80)).$$
If we multiply the target of this differential by $\epsilon'$, we get the class at $(152,88)$.  Therefore, the class at $(152, 88)$ must be killed by a differential of length at most 29.  The only possibility is the $d_{29}$-differential that we claimed. 
\end{comment}

\vspace{0.1in}

\noindent (4): By Theorem~\ref{thm:NormedDiffd29}, we have the $d_{29}$-differential 
$$d_{29}(\overline{\mathfrak{d}}_3^{8} \overline{\mathfrak{d}}_1 u_{8\lambda} u_{24\sigma} a_{17\lambda} a_{\sigma}) = \overline{\mathfrak{d}}_3^{10} \overline{\mathfrak{d}}_1^{2} u_{32\sigma} a_{32\lambda}\,\,\,(d_{29}(65,35) = (64,64)).$$
If we multiply the target of this differential by $\epsilon'$, we get the class $2\overline{\mathfrak{d}}_3^{12} u_{36\sigma} a_{36\lambda}$ at $(72,72)$.  Therefore, this class at $(72,72)$ must be killed by a differential of length at most 29.  The only possibility is the $d_{29}$-differential that we claimed.  

\vspace{0.1in}

\noindent (5): This differential is proven in the proof of Theorem~\ref{thm:NormedDiffd29}.  
\end{proof}

\begin{lem}\label{lem:tempd35Differentials}
The following $d_{35}$-differentials exist: 
\begin{eqnarray*}
d_{35}(\overline{\mathfrak{d}}_3^{17} \overline{s}_1 u_{12\lambda} u_{51\sigma} a_{39\lambda} a_{\sigma_2}) =  2\overline{\mathfrak{d}}_3^{20} u_{3\lambda} u_{60\sigma} a_{57\lambda}\,\,\,(d_{35}(127,79) = 2(126, 114)), \\
d_{35}(\overline{\mathfrak{d}}_3^{21} \overline{s}_1 u_{12\lambda} u_{63\sigma} a_{51\lambda} a_{\sigma_2}) =2\overline{\mathfrak{d}}_3^{24} u_{3\lambda} u_{72\sigma} a_{69\lambda} \,\,\,(d_{35}(151,103) = 2(150,138)).
\end{eqnarray*}
\end{lem}
\begin{proof}
Consider the target of the first differential.  By Theorem~\ref{thm:aboveFiltration61die}, it must be killed on or before the $E_{61}$-page.  The possibilities for the sources are
\begin{enumerate}
\item $\dthree^{17} \sthree^3$ at $(127, 95)$; 
\item $\dthree^{18} \sthree$ at $(127, 95)$;
\item $\overline{\mathfrak{d}}_3^{17} \overline{s}_1 u_{12\lambda} u_{51\sigma} a_{39\lambda} a_{\sigma_2}$ at $(127, 79)$;
\item $\dthree^{15}\sthree \rone^2$ at $(127, 63)$;
\item $\overline{\mathfrak{d}}_3^{15} \overline{s}_1 u_{18\lambda} u_{45\sigma} a_{27\lambda} a_{\sigma_2}$ at $(127, 55)$;
\end{enumerate}

(1) is impossible because by Proposition~\ref{prop:gammaPowers}, the class $\dthree^{17} \sthree^3$ at $(127, 95)$ is the target of the $d_{31}$-differential 
$$d_{31}(\overline{\mathfrak{d}}_3^{16} u_{16\lambda} u_{48\sigma} a_{32\lambda}) = \dthree^{17} \sthree^3 \,\,\,(d_{31}(128, 64) = (127, 95)).$$

(2) is impossible because $\dthree^{18}\sthree$ at $(127, 95)$ supports the $d_{31}$-differential 
$$d_{31}(\dthree^{18}\sthree) = \dthree^{20} \sthree^2  \,\,\,(d_{31}(127,95) = (126, 126)).$$
This is a $d_{31}$-differential between $i_{C_2}^*\BPone$-truncation classes. 

(4) is impossible because $\dthree^{15}\sthree \rone^2$ at $(127, 63)$ is the target of the $d_{31}$-differential 
$$d_{31}(\dthree^{12}\sthree^2 \rone^2) = \dthree^{15}\sthree \rone^2  \,\,\,(d_{31}(128, 32) = (127, 63))$$
between $i_{C_2}^*\BPone$-truncation classes. 

(5) is impossible because by Proposition~\ref{prop:d23Differential}, this class is the target of the $d_{23}$-differential 
$$d_{23}(\overline{\mathfrak{d}}_3^{13} \overline{\mathfrak{d}}_1 u_{24\lambda} u_{40\sigma} a_{16\lambda}) = \overline{\mathfrak{d}}_3^{15} \overline{s}_1 u_{18\lambda} u_{45\sigma} a_{27\lambda} a_{\sigma_2} \,\,\,(d_{23}(128, 32) = (127, 63))$$

It follows that the only possibility for the source is (3), and we deduce our claimed $d_{35}$-differential.  

The second differential is proven in the exact same way, except that we just need the extra fact that the class $\overline{\mathfrak{d}}_3^{21} \overline{s}_3^{3}$ at $(151,119)$ supports a $d_{43}$-differential (Theorem~\ref{thm:LongCrossingDifferentials}).
\end{proof}

\begin{thm}\label{prop:d19d27Differential}
The following differentials exist:
\begin{enumerate}
\item $d_{19}(2\overline{\mathfrak{d}}_3^{30} u_{5\lambda} u_{90\sigma} a_{85\lambda}) = \overline{\mathfrak{d}}_3^{30} \overline{s}_3^{3}\,\,\, (d_{19}(2(190,170)) = (189,189))$; \\
$d_{19}(2\overline{\mathfrak{d}}_3^{26} u_{5\lambda} u_{78\sigma} a_{73\lambda}) = \overline{\mathfrak{d}}_3^{26} \overline{s}_3^{3}\,\,\,(d_{19}(166,146) = (165,165))$;
\item $d_{19}(\overline{\mathfrak{d}}_3^{31} \overline{s}_1 u_{4\lambda} u_{93\sigma} a_{89\lambda} a_{\sigma_2}) = \overline{\mathfrak{d}}_3^{32} \overline{\mathfrak{d}}_1^{2} u_{96\sigma} a_{98\lambda} a_{2\sigma}\,\,\,(d_{19}(195,179) = (194,198))$;\\
$d_{19}(\overline{\mathfrak{d}}_3^{27} \overline{s}_1 u_{4\lambda} u_{81\sigma} a_{77\lambda} a_{\sigma_2}) = \overline{\mathfrak{d}}_3^{28} \overline{\mathfrak{d}}_1^{2} u_{84\sigma} a_{86\lambda} a_{2\sigma}\,\,\,(d_{19}(171,155) = (170,174))$;
\item $d_{19}(2\overline{\mathfrak{d}}_3^{32} u_{5\lambda} u_{96\sigma} a_{91\lambda}) = \overline{\mathfrak{d}}_3^{32} \overline{s}_3^{3}\,\,\,(d_{19}(2(202,182)) = (201,201))$;\\
$d_{19}(2\overline{\mathfrak{d}}_3^{28} u_{5\lambda} u_{84\sigma} a_{79\lambda}) =\overline{\mathfrak{d}}_3^{28} \overline{s}_3^{3} \,\,\,(d_{19}(2(178,158)) = (177,177))$;
\item $d_{19}(2\overline{\mathfrak{d}}_3^{30} u_{13\lambda} u_{90\sigma} a_{77\lambda}) =\overline{\mathfrak{d}}_3^{30} \overline{s}_3^{3} \,\,\,(d_{19}(2(206,154)) = (205,173))$;\\
$d_{19}(2\overline{\mathfrak{d}}_3^{26} u_{13\lambda} u_{78\sigma} a_{65\lambda}) = \overline{\mathfrak{d}}_3^{26} \overline{s}_3^{3}\,\,\,(d_{19}(2(182,130)) = (181,149))$;
\item $d_{19}(\overline{\mathfrak{d}}_3^{31} \overline{s}_1 u_{12\lambda} u_{93\sigma} a_{81\lambda} a_{\sigma_2}) = 2\overline{\mathfrak{d}}_3^{32} \overline{\mathfrak{d}}_1^{2} u_{7\lambda} u_{98\sigma} a_{91\lambda}\,\,\,(d_{19}(211,163) = 2(210,182))$;\\
$d_{19}(\overline{\mathfrak{d}}_3^{27} \overline{s}_1 u_{12\lambda} u_{81\sigma} a_{69\lambda} a_{\sigma_2}) =2\overline{\mathfrak{d}}_3^{28} \overline{\mathfrak{d}}_1^{2} u_{7\lambda} u_{86\sigma} a_{79\lambda} \,\,\,(d_{19}(187,139) = 2(186,158))$;
\item $d_{19}(2\overline{\mathfrak{d}}_3^{32} u_{13\lambda} u_{96\sigma} a_{83\lambda}) =\overline{\mathfrak{d}}_3^{32} \overline{s}_3^{3} \,\,\,(d_{19}(2(218,166)) = (217,185))$;\\
$d_{19}(2\overline{\mathfrak{d}}_3^{28} u_{13\lambda} u_{84\sigma} a_{71\lambda}) =\overline{\mathfrak{d}}_3^{28} \overline{s}_3^{3} \,\,\,(d_{19}(2(194,142)) = (193,161))$;
\item $d_{27}(\overline{\mathfrak{d}}_3^{31} \overline{s}_1 u_{10\lambda} u_{93\sigma} a_{83\lambda} a_{\sigma_2}) = 2\overline{\mathfrak{d}}_3^{33} \overline{\mathfrak{d}}_1 u_{3\lambda} u_{100\sigma} a_{97\lambda}\,\,\,(d_{27}(207,167) = 2(206,194))$;\\
$d_{27}(\overline{\mathfrak{d}}_3^{27} \overline{s}_1 u_{10\lambda} u_{81\sigma} a_{71\lambda} a_{\sigma_2}) = 2\overline{\mathfrak{d}}_3^{29} \overline{\mathfrak{d}}_1 u_{3\lambda} u_{88\sigma} a_{85\lambda}\,\,\,(d_{27}(183,143) = 2(182,170))$;
\item $d_{27}(\overline{\mathfrak{d}}_3^{33} \overline{s}_1 u_{6\lambda} u_{99\sigma} a_{93\lambda} a_{\sigma_2}) = \overline{\mathfrak{d}}_3^{35} \overline{\mathfrak{d}}_1 u_{104\sigma} a_{106\lambda} a_{2\sigma}\,\,\,(d_{27}(211,187) = (210,214))$;\\
$d_{27}(\overline{\mathfrak{d}}_3^{29} \overline{s}_1 u_{6\lambda} u_{87\sigma} a_{81\lambda} a_{\sigma_2}) = \overline{\mathfrak{d}}_3^{31} \overline{\mathfrak{d}}_1 u_{92\sigma} a_{94\lambda} a_{2\sigma}\,\,\,(d_{27}(187,163) = (186,190))$.
\end{enumerate}
\end{thm}

\begin{proof}
\noindent (1): Consider the class $2\overline{\mathfrak{d}}_3^{30} u_{5\lambda} u_{90\sigma} a_{85\lambda}$ at $(190, 170)$.  By Theorem~\ref{thm:aboveFiltration61die}, this class must die on or before the $E_{61}$-page.  If this class supports a differential, then we are done.  

If it is the target of a differential, then for degree reasons, the only possibility for the source is the class $\overline{\mathfrak{d}}_3^{25} \overline{s}_3^{3}$ at $(191,127)$.  However, by Proposition~\ref{prop:d31(40,8)} and the discussion afterwards, this class is the target of the $d_{31}$-differential 
$$d_{31}(2\overline{\mathfrak{d}}_3^{24} u_{24\lambda} u_{72\sigma} a_{48\lambda}) =\overline{\mathfrak{d}}_3^{25} \overline{s}_3^{3} \,\,\,(d_{31}(2(192,96)) = (191,127)).$$
Therefore, this class cannot be the target of a differential.  This proves the first differential in (1). 

The second differential in (1) is proven by the exact same method.  

\vspace{0.1in}

\noindent (2): To prove the first differential in (2), consider the class $\overline{\mathfrak{d}}_3^{32} \overline{\mathfrak{d}}_1^{2} u_{96\sigma} a_{98\lambda} a_{2\sigma}$ at $(194, 198)$.  By Theorem~\ref{thm:aboveFiltration61die}, this class must die on or before the $E_{61}$-page.  For degree reasons, the only possible source is the class $\overline{\mathfrak{d}}_3^{31} \overline{s}_1 u_{4\lambda} u_{93\sigma} a_{89\lambda} a_{\sigma_2}$ at $(195, 179)$.  This proves the first differential in (2).  The second differential is  proven in the exact same way.  

\vspace{0.1in}

\noindent (5): For the first differential, consider the class $2\overline{\mathfrak{d}}_3^{32} \overline{\mathfrak{d}}_1^{2} u_{7\lambda} u_{98\sigma} a_{91\lambda}$ at $2(210, 182)$.  By Theorem~\ref{thm:aboveFiltration61die}, this class must die on or before the $E_{61}$-page.  If this class is the source of a differential, the target must be the class $\overline{\mathfrak{d}}_3^{35} u_{104\sigma} a_{105\lambda} a_{\sigma}$ at $(209,211)$.  This is impossible because we have proven in Theorem~\ref{thm:LongCrossingDifferentials} that the class at $(209, 211)$ is the target of a $d_{61}$-differential.  

Therefore, this class must be killed by a differential of length at most 61.  For degree reasons, the only possible source is the class $\overline{\mathfrak{d}}_3^{31} \overline{s}_1 u_{12\lambda} u_{93\sigma} a_{81\lambda} a_{\sigma_2}$ at $(211,163)$.  This proves the first differential in (5).  The second differential is proven in the exact same way.  

\vspace{0.1in}

\noindent (3): For the first differential, consider the class $\overline{\mathfrak{d}}_3^{32} \overline{s}_3^{3}$ at $(201, 201)$.  By Theorem~\ref{thm:aboveFiltration61die}, this class must be killed on or before the $E_{61}$-page.  For degree reasons, the only possible sources are the following classes:
\begin{itemize}
\item $2\overline{\mathfrak{d}}_3^{32} u_{5\lambda} u_{96\sigma} a_{91\lambda}$ at $(202, 182)$;
\item $2\overline{\mathfrak{d}}_3^{31} \overline{\mathfrak{d}}_1 u_{7\lambda} u_{94\sigma} a_{87\lambda}$ at $(202, 174)$;
\item $2\overline{\mathfrak{d}}_3^{28} \overline{\mathfrak{d}}_1^{2} u_{15\lambda} u_{86\sigma} a_{71\lambda}$ at $(202, 142)$. 
\end{itemize}
If the class $2\overline{\mathfrak{d}}_3^{31} \overline{\mathfrak{d}}_1 u_{7\lambda} u_{94\sigma} a_{87\lambda}$ at $(202, 174)$ is the source, then the differential will be a $d_{27}$-differential.  However, by Theorem~\ref{thm:NormedDiffd29} the class $2\overline{\mathfrak{d}}_3^{30} u_{7\lambda} u_{90\sigma} a_{83\lambda}$ at $(194, 166)$ support the $d_{29}$-differential 
$$d_{29}(2\overline{\mathfrak{d}}_3^{30} u_{7\lambda} u_{90\sigma} a_{83\lambda}) = \overline{\mathfrak{d}}_3^{32} \overline{\mathfrak{d}}_1 u_{96\sigma} a_{97\lambda} a_{\sigma} \,\,\,(d_{29}(2(194, 166)) =(193,195)).$$
Since $2(194, 166) \cdot \epsilon' = 2(202, 174)$, this is a contradiction.  

The class $2\overline{\mathfrak{d}}_3^{28} \overline{\mathfrak{d}}_1^{2} u_{15\lambda} u_{86\sigma} a_{71\lambda}$ at $(202, 142)$ cannot be the source either because it is the target of the $d_{19}$-differential 
$$d_{19}(\overline{\mathfrak{d}}_3^{27} \overline{s}_1 u_{20\lambda} u_{81\sigma} a_{61\lambda} a_{\sigma_2}) = 2\overline{\mathfrak{d}}_3^{28} \overline{\mathfrak{d}}_1^{2} u_{15\lambda} u_{86\sigma} a_{71\lambda} \,\,\,(d_{19}(203, 123) =2(202,142)).$$
This $d_{19}$-differential can be deduced from the second differential of (2) by using multiplicative structures with the classes $\alpha$ and $\gamma^4$.  

Therefore, the only possibility for the source is the class $2\overline{\mathfrak{d}}_3^{32} u_{5\lambda} u_{96\sigma} a_{91\lambda}$ at $(202, 182)$.  This proves the first $d_{19}$-differential.  

For the second differential, consider the class $\overline{\mathfrak{d}}_3^{28} \overline{s}_3^{3}$ at $(177,177)$.  By Theorem~\ref{thm:aboveFiltration61die}, it must be killed by a differential of length at most 61.  The possible sources are the following classes: 
\begin{itemize}
\item $2\overline{\mathfrak{d}}_3^{28} u_{5\lambda} u_{84\sigma} a_{79\lambda}$ at $(178, 158)$;
\item $2\overline{\mathfrak{d}}_3^{27} \overline{\mathfrak{d}}_1 u_{7\lambda} u_{82\sigma} a_{75\lambda}$ at $(178, 150)$;
\item $2\overline{\mathfrak{d}}_3^{24} \overline{\mathfrak{d}}_1^{2} u_{15\lambda} u_{74\sigma} a_{59\lambda}$ at $(178,118)$.
\end{itemize}
By using Lemma~\ref{lem:tempd29Differentials} (1) and multiplicative structures with $\alpha$, the class $2\overline{\mathfrak{d}}_3^{28} u_{5\lambda} u_{84\sigma} a_{79\lambda}$ at $(178, 150)$ support a $d_{29}$-differential, and therefore cannot be the source.  The class $2\,\overline{\mathfrak{d}}_3^{24} \overline{\mathfrak{d}}_1^{2} u_{15\lambda} u_{74\sigma} a_{59\lambda}$ at $(178, 118)$ is the target of the $d_{19}$-differential 
$$d_{19}(\overline{\mathfrak{d}}_3^{23} \overline{s}_1 u_{20\lambda} u_{69\sigma} a_{49\lambda} a_{\sigma_2}) = 2\overline{\mathfrak{d}}_3^{24} \overline{\mathfrak{d}}_1^{2} u_{15\lambda} u_{74\sigma} a_{59\lambda} \,\,\,(d_{19}(179, 99) =2(178,118)).$$
This $d_{19}$-differential can be deduced from the first differential of (2) by using multiplicative structures with the classes $\alpha$ and $\gamma^4$.  

Therefore, there is only one possible source left, and this leads to the desired $d_{19}$-differential.  

\vspace{0.1in}

\noindent (4): To prove the first differential in (4), we will first multiply the source by $\gamma^4$ and prove the $d_{19}$-differential 
$$d_{19}(2\overline{\mathfrak{d}}_3^{38} u_{29\lambda} u_{114\sigma} a_{85\lambda}) = \overline{\mathfrak{d}}_3^{38} \overline{s}_3^{3}\,\,\,(d_{19}(2(286,170)) = (285,189)).$$
Once we have proven this, we can immediately deduce the first differential.  

Consider the class $\overline{\mathfrak{d}}_3^{38} \overline{s}_3^{3}$ at $(285, 189)$.  By Theorem~\ref{thm:aboveFiltration61die}, this class must die on or before the $E_{61}$-page.  For degree reasons, it cannot support a differential (because the length of that differential must be $15 <r \leq 61$).  Therefore, it must be the target of a differential.  For degree reasons, the possible sources are the following classes:
\begin{itemize}
\item $2\overline{\mathfrak{d}}_3^{38} u_{29\lambda} u_{114\sigma} a_{85\lambda}$ at $(286,170)$; 
\item $2\overline{\mathfrak{d}}_3^{36} u_{35\lambda} u_{108\sigma} a_{73\lambda}$ at $(286,146)$.
\end{itemize}
Using the first differential in Lemma~\ref{lem:tempd35Differentials} and multiplicative structures with $\gamma^8$, we deduce that the class $2\overline{\mathfrak{d}}_3^{36} u_{35\lambda} u_{108\sigma} a_{73\lambda}$ at $(286, 146)$ is the target of the $d_{35}$-differential
$$d_{35}(\overline{\mathfrak{d}}_3^{33} \overline{s}_1 u_{44\lambda} u_{99\sigma} a_{55\lambda} a_{\sigma_2}) = 2\overline{\mathfrak{d}}_3^{36} u_{35\lambda} u_{108\sigma} a_{73\lambda}\,\,\,(d_{35}(287,111) = 2(286,146)).$$
Therefore, the only possible source left is the class $2\overline{\mathfrak{d}}_3^{38} u_{29\lambda} u_{114\sigma} a_{85\lambda}$ at $(286, 170)$.  This proves our desired differential.  

The proof of the second differential is exactly the same (except near the end we use the second differential in Lemma~\ref{lem:tempd35Differentials} and multiplicative structures with $\gamma^8$ and $\alpha$).  

\vspace{0.1in}

\noindent (7): For the first differential, consider the class $2\overline{\mathfrak{d}}_3^{33} \overline{\mathfrak{d}}_1 u_{3\lambda} u_{100\sigma} a_{97\lambda}$ at $(206, 194)$.  By Theorem~\ref{thm:aboveFiltration61die} and degree reasons, this class must be killed on or before the $E_{61}$-page.  The only possibilities for the source are the following classes: 
\begin{itemize}
\item $\overline{\mathfrak{d}}_3^{31} \overline{s}_1 u_{10\lambda} u_{93\sigma} a_{83\lambda} a_{\sigma_2}$ at $(207, 167)$;
\item $\overline{\mathfrak{d}}_3^{29} u_{17\lambda} u_{86\sigma} a_{70\lambda} a_{\sigma}$ at $(207, 141)$. 
\end{itemize}
Using Lemma~\ref{lem:tempd29Differentials} (2) and multiplicative structures with $\alpha$ and $\gamma^4$, we deduce that the class $\overline{\mathfrak{d}}_3^{29} u_{17\lambda} u_{86\sigma} a_{70\lambda} a_{\sigma}$ at $(207, 141)$ is the target of the $d_{29}$-differential 
$$d_{29}(\overline{\mathfrak{d}}_3^{26} \overline{\mathfrak{d}}_1^{2} u_{24\lambda} u_{80\sigma} a_{56\lambda}) =\overline{\mathfrak{d}}_3^{29} u_{17\lambda} u_{86\sigma} a_{70\lambda} a_{\sigma} \,\,\,(d_{29}(208,112)=(207,141)).$$
Therefore, the source must be the class $\overline{\mathfrak{d}}_3^{31} \overline{s}_1 u_{10\lambda} u_{93\sigma} a_{83\lambda} a_{\sigma_2}$ at $(207, 167)$. 

The second differential is proven in the exact same way, except near the end we use Lemma~\ref{lem:tempd29Differentials} (3) and multiplicative structures with $\alpha$ and $\gamma^4$ to deduce a $d_{29}$-differential.  

\vspace{0.1in}

\noindent (8): To prove the first differential, consider the class $\overline{\mathfrak{d}}_3^{35} \overline{\mathfrak{d}}_1 u_{104\sigma} a_{106\lambda} a_{2\sigma}$ at $(210, 214)$.  By Theorem~\ref{thm:aboveFiltration61die}, it must be killed on or before the $E_{61}$-page.  The only possibilities for the sources are the following classes: 
\begin{itemize}
\item $\overline{\mathfrak{d}}_3^{33} \overline{s}_1 u_{6\lambda} u_{99\sigma} a_{93\lambda} a_{\sigma_2}$ at $(211,187)$;
\item $\overline{\mathfrak{d}}_3^{31} \overline{s}_1 u_{12\lambda} u_{93\sigma} a_{81\lambda} a_{\sigma_2}$ at $(211,163)$.
\end{itemize}
By (5), the class $(211,163)$ supports the $d_{19}$-differential 
$$d_{19}(\overline{\mathfrak{d}}_3^{31} \overline{s}_1 u_{12\lambda} u_{93\sigma} a_{81\lambda} a_{\sigma_2}) =2\overline{\mathfrak{d}}_3^{32} \overline{\mathfrak{d}}_1^{2} u_{7\lambda} u_{98\sigma} a_{91\lambda} \,\,\,(d_{19}(211,163)=2(210,182)).$$
Therefore, the only possibility for the source is the class $\overline{\mathfrak{d}}_3^{33} \overline{s}_1 u_{6\lambda} u_{99\sigma} a_{93\lambda} a_{\sigma_2}$ at $(211, 187)$.  This proves our desired differential. 

The proof of the second differential is exactly the same.  

\vspace{0.1in}

\noindent (6): For the first differential, consider the class $\overline{\mathfrak{d}}_3^{32} \overline{s}_3^{3}$ at $(217, 185)$.  By Theorem~\ref{thm:aboveFiltration61die}, this class must die on or before the $E_{61}$-page.  If this class supports a differential, then the only possible target is the class $2\overline{\mathfrak{d}}_3^{36} u_{108\sigma} a_{108\lambda}$ at $(216, 216)$.  This is impossible because this class at $(216,216)$ is the target of the $d_{29}$-differential 
$$d_{29}(\overline{\mathfrak{d}}_3^{33} \overline{\mathfrak{d}}_1^{2} u_{8\lambda} u_{100\sigma} a_{93\lambda} a_{\sigma}) = 2\overline{\mathfrak{d}}_3^{36} u_{108\sigma} a_{108\lambda}\,\,\,(d_{29}(217,187) = 2(216,216)).$$
We can deduce this differential from Lemma~\ref{lem:tempd29Differentials} (4) and multiplicative structures with $\alpha$.  

Therefore, this class must be killed by a differential of length at most 61.  The only possibilities for the source are the following classes: 
\begin{itemize}
\item $2\overline{\mathfrak{d}}_3^{32} u_{13\lambda} u_{96\sigma} a_{83\lambda}$ at $(218,166)$;
\item $2\overline{\mathfrak{d}}_3^{31} \overline{\mathfrak{d}}_1 u_{15\lambda} u_{94\sigma} a_{79\lambda}$ at $(218,158)$;
\item $2\overline{\mathfrak{d}}_3^{28} \overline{\mathfrak{d}}_1^{2} u_{23\lambda} u_{86\sigma} a_{63\lambda}$ at $(218,126)$.
\end{itemize}

The class $2\overline{\mathfrak{d}}_3^{31} \overline{\mathfrak{d}}_1 u_{15\lambda} u_{94\sigma} a_{79\lambda}$ at $(218, 158)$ is the target of the $d_{27}$-differential 
$$d_{27}(\overline{\mathfrak{d}}_3^{29} \overline{s}_1 u_{22\lambda} u_{87\sigma} a_{65\lambda} a_{\sigma_2}) = 2\overline{\mathfrak{d}}_3^{31} \overline{\mathfrak{d}}_1 u_{15\lambda} u_{94\sigma} a_{79\lambda}\,\,\,(d_{27}(219,131) = 2(218,158)).$$
We can deduce this differential from the second differential in (8) and multiplication with $\alpha$ and $\gamma^4$. 

The class $2\overline{\mathfrak{d}}_3^{28} \overline{\mathfrak{d}}_1^{2} u_{23\lambda} u_{86\sigma} a_{63\lambda}$ at $(218, 126)$ is the target of the $d_{19}$-differential 
$$d_{19}(\overline{\mathfrak{d}}_3^{27} \overline{s}_1 u_{28\lambda} u_{81\sigma} a_{53\lambda} a_{\sigma_2}) = 2\overline{\mathfrak{d}}_3^{28} \overline{\mathfrak{d}}_1^{2} u_{23\lambda} u_{86\sigma} a_{63\lambda}\,\,\,(d_{19}(219,107) = 2(218,126)).$$
We can deduce this differential from the second differential in (5) and multiplication with $\alpha$ and $\gamma^4$.  

It follows that the only possibility left for the source is the class $2\overline{\mathfrak{d}}_3^{32} u_{13\lambda} u_{96\sigma} a_{83\lambda}$ at $(218, 166)$, as desired. 

The proof of the second differential is exactly the same.  
\end{proof}

All the other $d_{19}$- and $d_{27}$-differentials are obtained from the differentials in Proposition~\ref{prop:d19d27Differential} by using product structures with the classes $\alpha$ and $\gamma^4$ (see the discussion after Propositon~\ref{prop:gammaPowers}).  These differentials are the {\color{LimeGreen} lime-green differentials} and the {\color{ForestGreen} forest-green differentials} in Figure~\ref{fig:E4C4E19toE27page}, respectively.

%%%%%%%%%%%%%%%%%%%%%%%%
%%%%%%%%%%%%%%%%%%%%%%%
\section{Higher differentials V: $d_{29}$-differentials and $d_{31}$-differentials} \label{sec:HigherDifferentialsV}
\subsection{$d_{29}$-differentials}

\begin{thm}\label{thm:E4C4d29Differentials}
The following $d_{29}$-differentials exist: 
\begin{enumerate}
\item $d_{29}(\overline{\mathfrak{d}}_3^{26} \overline{\mathfrak{d}}_1^{2} u_{8\lambda} u_{80\sigma} a_{72\lambda}) =\overline{\mathfrak{d}}_3^{29} u_{\lambda} u_{86\sigma} a_{86\lambda} a_{\sigma} \,\,\,(d_{29}(176,144) = (175,173))$;
\item $d_{29}(\overline{\mathfrak{d}}_3^{27} u_{8\lambda} u_{80\sigma} a_{73\lambda} a_{\sigma}) = \overline{\mathfrak{d}}_3^{29} \overline{\mathfrak{d}}_1 u_{88\sigma} a_{88\lambda}\,\,\,(d_{29}(177,147) = (176,176))$;
\item $d_{29}(2\overline{\mathfrak{d}}_3^{27} \overline{\mathfrak{d}}_1 u_{7\lambda} u_{82\sigma} a_{75\lambda}) =\overline{\mathfrak{d}}_3^{29} \overline{\mathfrak{d}}_1^{2} u_{88\sigma} a_{89\lambda} a_{\sigma} \,\,\,(d_{29}(2(178,150)) = (177,179))$;
\item $d_{29}(\overline{\mathfrak{d}}_3^{28} \overline{\mathfrak{d}}_1 u_{8\lambda} u_{84\sigma} a_{77\lambda} a_{\sigma}) = \overline{\mathfrak{d}}_3^{30} \overline{\mathfrak{d}}_1^{2} u_{92\sigma} a_{92\lambda}\,\,\,(d_{29}(185,155) = (184,184))$;
\item $d_{29}(\overline{\mathfrak{d}}_3^{29} \overline{\mathfrak{d}}_1^{2} u_{8\lambda} u_{88\sigma} a_{81\lambda} a_{\sigma}) = 2\overline{\mathfrak{d}}_3^{32} u_{96\sigma} a_{96\lambda}\,\,\,(d_{29}(193,163) = 2(192,192))$;
\item $d_{29}(2\overline{\mathfrak{d}}_3^{30} u_{7\lambda} u_{90\sigma} a_{83\lambda}) = \overline{\mathfrak{d}}_3^{32} \overline{\mathfrak{d}}_1 u_{96\sigma} a_{97\lambda} a_{\sigma}\,\,\,(d_{29}(2(194,166)) = (193,195))$;
\item $d_{29}(\overline{\mathfrak{d}}_3^{30} \overline{\mathfrak{d}}_1^{2} u_{8\lambda} u_{92\sigma} a_{84\lambda}) = \overline{\mathfrak{d}}_3^{33} u_{\lambda} u_{98\sigma} a_{98\lambda} a_{\sigma}\,\,\,(d_{29}(200,168) = (199,197)$;
\item $d_{29}(\overline{\mathfrak{d}}_3^{31} u_{8\lambda} u_{92\sigma} a_{85\lambda} a_{\sigma}) =\overline{\mathfrak{d}}_3^{33} \overline{\mathfrak{d}}_1 u_{100\sigma} a_{100\lambda} \,\,\,(d_{29}(201,171) = (200,200))$;
\item $d_{29}(2\overline{\mathfrak{d}}_3^{31} \overline{\mathfrak{d}}_1 u_{7\lambda} u_{94\sigma} a_{87\lambda}) = \overline{\mathfrak{d}}_3^{33} \overline{\mathfrak{d}}_1^{2} u_{100\sigma} a_{101\lambda} a_{\sigma}\,\,\,(d_{29}(2(202,174)) = (201,203))$;
\item $d_{29}(\overline{\mathfrak{d}}_3^{32} \overline{\mathfrak{d}}_1 u_{8\lambda} u_{96\sigma} a_{89\lambda} a_{\sigma}) = \overline{\mathfrak{d}}_3^{34} \overline{\mathfrak{d}}_1^{2} u_{104\sigma} a_{104\lambda}\,\,\,(d_{29}(209,179) = (208,208))$;
\item $d_{29}(\overline{\mathfrak{d}}_3^{33} \overline{\mathfrak{d}}_1^{2} u_{8\lambda} u_{100\sigma} a_{93\lambda} a_{\sigma}) = 2\overline{\mathfrak{d}}_3^{36} u_{108\sigma} a_{108\lambda}\,\,\,(d_{29}(217,187) = 2(216,216))$;
\item $d_{29}(2\overline{\mathfrak{d}}_3^{34} u_{7\lambda} u_{102\sigma} a_{95\lambda}) = \overline{\mathfrak{d}}_3^{36} \overline{\mathfrak{d}}_1 u_{108\sigma} a_{109\lambda} a_{\sigma} \,\,\,(d_{29}(2(218,190)) = (217,219))$;
\item $d_{29}(2\overline{\mathfrak{d}}_3^{28} u_{11\lambda} u_{84\sigma} a_{73\lambda}) = \overline{\mathfrak{d}}_3^{30} \overline{\mathfrak{d}}_1 u_{4\lambda} u_{90\sigma} a_{87\lambda} a_{\sigma}\,\,\,(d_{29}(2(190,146)) = (189,175))$;
\item $d_{29}(2\overline{\mathfrak{d}}_3^{29} \overline{\mathfrak{d}}_1 u_{11\lambda} u_{88\sigma} a_{77\lambda}) = \overline{\mathfrak{d}}_3^{31} \overline{\mathfrak{d}}_1^{2} u_{4\lambda} u_{94\sigma} a_{91\lambda} a_{\sigma}\,\,\,(d_{29}(2(198,154)) = (197,183))$;
\item $d_{29}(\overline{\mathfrak{d}}_3^{30} \overline{\mathfrak{d}}_1 u_{12\lambda} u_{90\sigma} a_{79\lambda} a_{\sigma}) = \overline{\mathfrak{d}}_3^{32} \overline{\mathfrak{d}}_1^{2} u_{4\lambda} u_{98\sigma} a_{94\lambda}\,\,\,(d_{29}(205,159) = (204,188))$;
\item $d_{29}(2\overline{\mathfrak{d}}_3^{32} u_{11\lambda} u_{96\sigma} a_{85\lambda}) =\overline{\mathfrak{d}}_3^{34} \overline{\mathfrak{d}}_1 u_{4\lambda} u_{102\sigma} a_{99\lambda} a_{\sigma} \,\,\,(d_{29}(2(214,170)) = (213,199))$;
\item $d_{29}(2\overline{\mathfrak{d}}_3^{33} \overline{\mathfrak{d}}_1 u_{11\lambda} u_{100\sigma} a_{89\lambda}) = \overline{\mathfrak{d}}_3^{35} \overline{\mathfrak{d}}_1^{2} u_{4\lambda} u_{106\sigma} a_{103\lambda} a_{\sigma}\,\,\,(d_{29}(2(222,178)) = (221,207))$;
\item $d_{29}(\overline{\mathfrak{d}}_3^{34} \overline{\mathfrak{d}}_1 u_{12\lambda} u_{102\sigma} a_{91\lambda} a_{\sigma}) = \overline{\mathfrak{d}}_3^{36} \overline{\mathfrak{d}}_1^{2} u_{4\lambda} u_{110\sigma} a_{106\lambda}\,\,\,(d_{29}(229,183) = (228,212))$.
\end{enumerate}
\end{thm}
\begin{proof}
The differentials (2), (3), (5), (10) are immediate from Theorem~\ref{thm:NormedDiffd29}.  \\

\noindent (1): Consider the class $\overline{\mathfrak{d}}_3^{26} \overline{\mathfrak{d}}_1^{2} u_{8\lambda} u_{80\sigma} a_{72\lambda}$ at $(176, 144)$.  If we multiply this class by $\beta$, we get the class $(193, 163)$, which supports the $d_{29}$-differential (5).  Therefore, the class $\overline{\mathfrak{d}}_3^{26} \overline{\mathfrak{d}}_1^{2} u_{8\lambda} u_{80\sigma} a_{72\lambda}$ at $(176, 144)$ must support a differential of length at most 29.  The only possibility is the $d_{29}$-differential that we claimed. 

\vspace{0.1in}
\noindent (4): This differential follows from (2) via multiplication by $\epsilon'$.  

\vspace{0.1in}
\noindent (6): Consider the class $\overline{\mathfrak{d}}_3^{32} \overline{\mathfrak{d}}_1 u_{96\sigma} a_{97\lambda} a_{\sigma}$ at $(193, 195)$.  By Theorem~\ref{thm:aboveFiltration61die}, this class must be killed by a differential of length at most 61.  By degree reasons, the only possibility is the $d_{29}$-differential we claimed. 

\vspace{0.1in}
\noindent (8): Consider the class $\overline{\mathfrak{d}}_3^{31} u_{8\lambda} u_{92\sigma} a_{85\lambda} a_{\sigma}$ at $(201, 171)$.  If we multiply this class by $\epsilon'$, we get the class $\overline{\mathfrak{d}}_3^{32} \overline{\mathfrak{d}}_1 u_{8\lambda} u_{96\sigma} a_{89\lambda} a_{\sigma}$ at $(209, 179)$, which supports differential (10).  Therefore, the class $\overline{\mathfrak{d}}_3^{31} u_{8\lambda} u_{92\sigma} a_{85\lambda} a_{\sigma}$ at $(201, 171)$ must support a differential of length at most 29.  The only possibility is the differential that we claimed. 

\vspace{0.1in}
\noindent (9): This differential follows form (6) via multiplication by $\epsilon'$.  

\vspace{0.1in}
\noindent (11): This differential follows from (10) via multiplication by $\epsilon'$.  

\vspace{0.1in}
\noindent (7): Consider the class $\overline{\mathfrak{d}}_3^{30} \overline{\mathfrak{d}}_1^{2} u_{8\lambda} u_{92\sigma} a_{84\lambda}$ at $(200, 168)$.  If we multiply this class by $\beta$, we get the class $\overline{\mathfrak{d}}_3^{33} \overline{\mathfrak{d}}_1^{2} u_{8\lambda} u_{100\sigma} a_{93\lambda} a_{\sigma}$ at $(217, 187)$, which supports differential (11).  Therefore, the class $\overline{\mathfrak{d}}_3^{30} \overline{\mathfrak{d}}_1^{2} u_{8\lambda} u_{92\sigma} a_{84\lambda}$ at $(200, 168)$ must support a differential of length at most 29.  The only possibility is the differential that we claimed. 

\vspace{0.1in}
\noindent (12): Consider the class $\overline{\mathfrak{d}}_3^{36} \overline{\mathfrak{d}}_1 u_{108\sigma} a_{109\lambda} a_{\sigma}$ at $(217, 219)$.  By Theorem~\ref{thm:aboveFiltration61die}, this class must be killed by a differential of length at most 61.  For degree reasons, the only possible differential is the $d_{29}$-differential that we claimed. 

\vspace{0.1in}
\noindent (13): Consider the class $\overline{\mathfrak{d}}_3^{30} \overline{\mathfrak{d}}_1 u_{4\lambda} u_{90\sigma} a_{87\lambda} a_{\sigma}$ at $(189,175)$.  By Theorem~\ref{thm:aboveFiltration61die}, this class must die on or before the $E_{61}$-page.  For degree reasons, the only possible way for this to happen is for the claimed $d_{29}$-differential to exist.  

\vspace{0.1in}
\noindent (14): This differential follows from (13) via multiplication by $\epsilon'$ (alternatively, we can also use Theorem~\ref{thm:aboveFiltration61die}). 

\vspace{0.1in}
\noindent (15): Consider the class $\overline{\mathfrak{d}}_3^{30} \overline{\mathfrak{d}}_1 u_{12\lambda} u_{90\sigma} a_{79\lambda} a_{\sigma}$ at $(205, 159)$.  If we multiply this class by $\epsilon'$, we get the class $\overline{\mathfrak{d}}_3^{31} \overline{\mathfrak{d}}_1^{2} u_{12\lambda} u_{94\sigma} a_{83\lambda} a_{\sigma}$ at $(213,167)$.  This class support a $d_{53}$-differential by Theorem~\ref{thm:LongCrossingDifferentials}.  Therefore, the class $\overline{\mathfrak{d}}_3^{30} \overline{\mathfrak{d}}_1 u_{12\lambda} u_{90\sigma} a_{79\lambda} a_{\sigma}$ at $(205, 159)$ must support a differential of length at most 53.  For degree reasons, this implies the $d_{29}$-differential that we claimed.  

\vspace{0.1in}
\noindent (16): Consider the class $\overline{\mathfrak{d}}_3^{34} \overline{\mathfrak{d}}_1 u_{4\lambda} u_{102\sigma} a_{99\lambda} a_{\sigma}$ at $(213,199)$.  By Theorem~\ref{thm:aboveFiltration61die}, this class must die on or before the $E_{61}$-page.  For degree reasons, the only way for this to happen is for the claimed $d_{29}$-differential to exist.  

\vspace{0.1in}
\noindent (17): This differential follows from (16) via multiplication by $\epsilon'$.  

\vspace{0.1in}
\noindent (18): Consider the class $\overline{\mathfrak{d}}_3^{34} \overline{\mathfrak{d}}_1 u_{12\lambda} u_{102\sigma} a_{91\lambda} a_{\sigma}$ at $(229,183)$.  If we multiply this class by $\eta'$, we get the class $2\overline{\mathfrak{d}}_3^{34} \overline{\mathfrak{d}}_1^{2} u_{11\lambda} u_{104\sigma} a_{93\lambda}$ at $(230,186)$.  By Theorem~\ref{thm:LongCrossingDifferentials}, this class supports a $d_{53}$-differential.  Therefore, the class $\overline{\mathfrak{d}}_3^{34} \overline{\mathfrak{d}}_1 u_{12\lambda} u_{102\sigma} a_{91\lambda} a_{\sigma}$ at $(229,183)$ must support a differential of length at most 53.  For degree reasons, we deduce the claimed $d_{29}$-differential. 
\end{proof}

Theorem~\ref{thm:E4C4d29Differentials}, combined with using multiplicative structures on the classes $\alpha$ and $\gamma^4$, produces all the $d_{29}$-differentials in $\SliceSS(\BPtwo)$ (see our discussion after Proposition~\ref{prop:gammaPowers}).  These differentials are shown in Figure~\ref{fig:E4C4d29Differentials}. 

\begin{figure}
\begin{center}
\makebox[\textwidth]{\includegraphics[trim={0cm 10cm 0cm 10cm}, clip, page = 1, scale = 0.23]{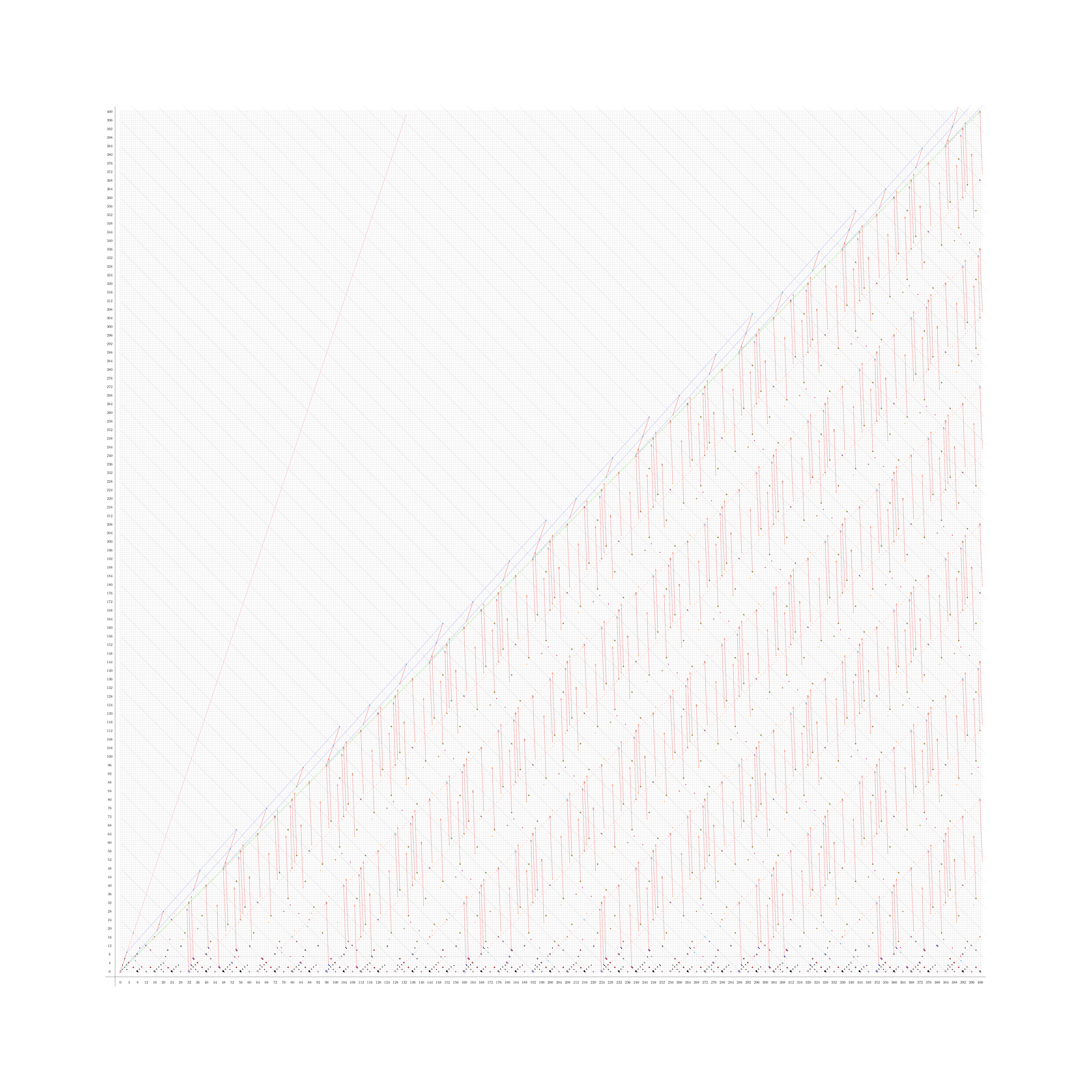}}
\end{center}
\begin{center}
\caption{$d_{29}$-differentials in $\SliceSS(\BPtwo)$.}
\hfill
\label{fig:E4C4d29Differentials}
\end{center}
\end{figure}

\subsection{$d_{31}$-differentials} \label{subsec:d31Diff}

Almost all of the $d_{31}$-differentials are induced $d_{31}$-differentials from $i_{C_2}^*\BPone$-truncation classes, and they can be proven by using the transfer and the restriction map (see Section~\ref{subsection:inducedC2Differentials}).  

The rest of the $d_{31}$-differentials follows from Proposition~\ref{prop:d31(40,8)} (and the discussion afterwards), Proposition~\ref{prop:gammaPowers}, and multiplication with the following classes:
\begin{itemize}
\item $\alpha$ at $(48,48)$ (permanent cycle); 
\item $\gamma^8$ at $(160,32)$ (permanent cycle);
\item $\overline{\mathfrak{d}}_3^{12} u_{16\lambda} u_{36\sigma} a_{20\lambda}$ at $(104,40)$ ($d_{31}$-cycle).  
\end{itemize}
The $d_{31}$-differentials are shown in Figure~\ref{fig:E4C4d31Differentials}.

\begin{figure}
\begin{center}
\makebox[\textwidth]{\includegraphics[trim={0cm 10cm 0cm 10cm}, clip, page = 1, scale = 0.23]{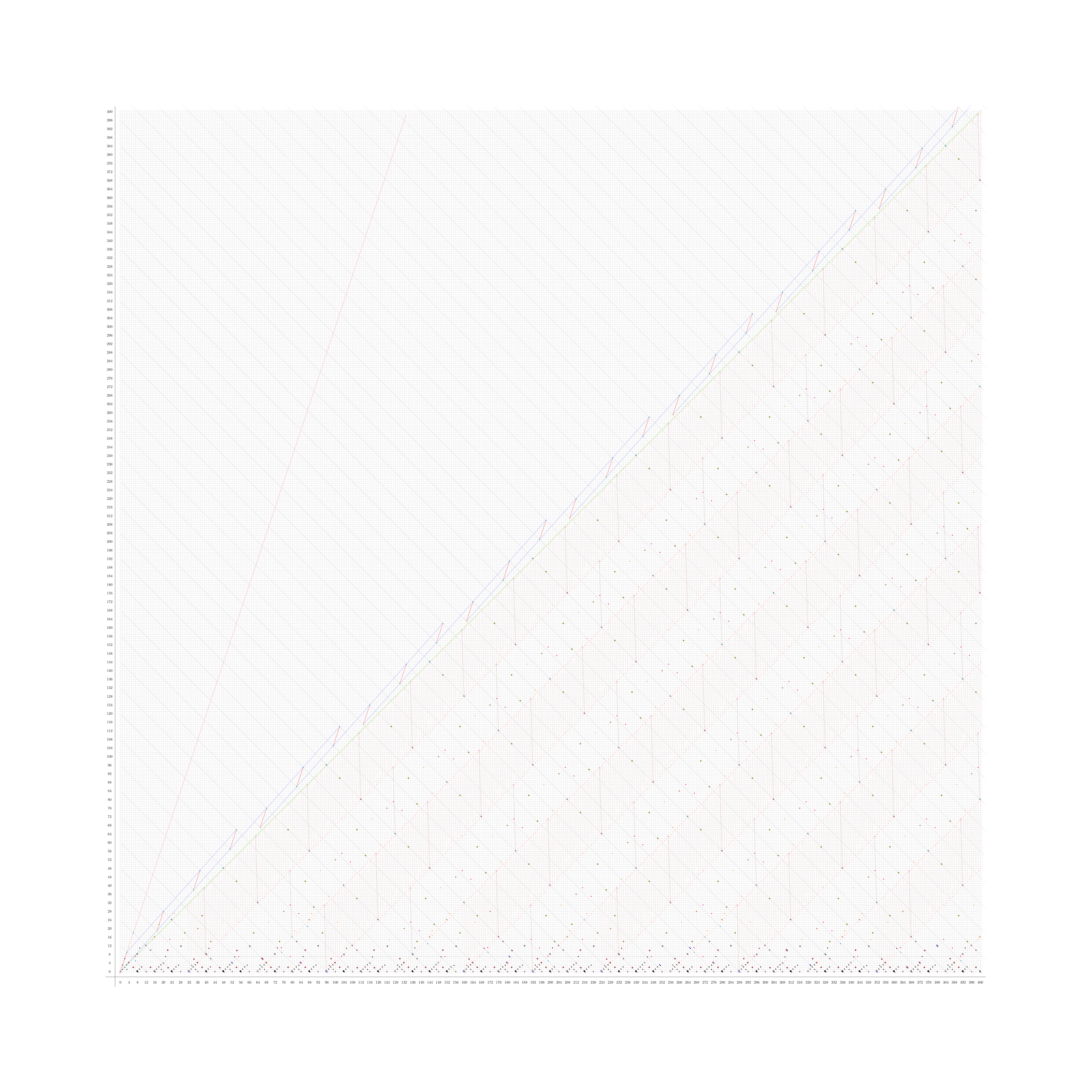}}
\end{center}
\begin{center}
\caption{$d_{31}$-differentials in $\SliceSS(\BPtwo)$.}
\hfill
\label{fig:E4C4d31Differentials}
\end{center}
\end{figure}

%%%%%%%%%%%%%%%%%%%%%%%
%%%%%%%%%%%%%%%%%%%%%%%%
\section{Higher differentials VI: $d_{35}$ to $d_{61}$-differentials} \label{sec:HigherDifferentialsVI}

\begin{prop}\label{prop:E4C4d35Differentials}
The following $d_{35}$-differentials exist: 
\begin{enumerate}
\item $d_{35}(\overline{\mathfrak{d}}_3^{17} \overline{s}_1 u_{12\lambda} u_{51\sigma} a_{39\lambda} a_{\sigma_2}) =  2\overline{\mathfrak{d}}_3^{20} u_{3\lambda} u_{60\sigma} a_{57\lambda}\,\,\,(d_{35}(127,79) = 2(126, 114))$;
\item $d_{35}(\overline{\mathfrak{d}}_3^{21} \overline{s}_1 u_{12\lambda} u_{63\sigma} a_{51\lambda} a_{\sigma_2}) =2\overline{\mathfrak{d}}_3^{24} u_{3\lambda} u_{72\sigma} a_{69\lambda} \,\,\,(d_{35}(151,103) = 2(150,138))$;
\item $d_{35}(\overline{\mathfrak{d}}_3^{21} \overline{s}_1 u_{28\lambda} u_{63\sigma} a_{35\lambda} a_{\sigma_2}) = 2\overline{\mathfrak{d}}_3^{24} u_{19\lambda} u_{72\sigma} a_{53\lambda}\,\,\, (d_{35}(183,71) = 2(182,106))$;
\item $d_{35}(\overline{\mathfrak{d}}_3^{25} \overline{s}_1 u_{28\lambda} u_{75\sigma} a_{47\lambda} a_{\sigma_2}) = 2\overline{\mathfrak{d}}_3^{28} u_{19\lambda} u_{84\sigma} a_{65\lambda}\,\,\, (d_{35}(207,95) = 2(206,130))$.
\end{enumerate}
\end{prop}
\begin{proof}
(1) and (2) are proven in Lemma~\ref{lem:tempd35Differentials}.  For (3), consider the class $2\overline{\mathfrak{d}}_3^{24} u_{19\lambda} u_{72\sigma} a_{53\lambda}$ at $(182,106)$.  By Theorem~\ref{thm:aboveFiltration61die}, this class must die on or before the $E_{61}$-page.  For degree reasons, the only way this can happen is for this class to be killed by the class $\overline{\mathfrak{d}}_3^{21} \overline{s}_1 u_{28\lambda} u_{63\sigma} a_{35\lambda} a_{\sigma_2}$ at $(183,71)$.  This proves the desired $d_{35}$-differential. 

For (4), first consider the class $2\overline{\mathfrak{d}}_3^{30} \overline{\mathfrak{d}}_1^{2} u_{11\lambda} u_{92\sigma} a_{81\lambda}$ at $(206,162)$.  By Theorem~\ref{thm:aboveFiltration61die}, this class must die on or before the $E_{61}$-page.  For degree reasons, this class must be killed by the class $\overline{\mathfrak{d}}_3^{25} \overline{s}_3^{3}$ at $(207, 111)$.  

Now, consider the class $2\overline{\mathfrak{d}}_3^{28} u_{19\lambda} u_{84\sigma} a_{65\lambda}$ at $(206, 130)$.  By Theorem~\ref{thm:aboveFiltration61die} again, this class must die on or before the $E_{61}$-page.  For degree reasons, this class must be killed by one of the following classes: 
\begin{itemize}
\item $\overline{\mathfrak{d}}_3^{25} \overline{s}_3^{3}$ at $(207, 111)$;
\item $\overline{\mathfrak{d}}_3^{25} \overline{s}_1 u_{28\lambda} u_{75\sigma} a_{47\lambda} a_{\sigma_2}$ at $(207, 95)$.
\end{itemize}
Since we have already shown in the previous paragraph the class $\overline{\mathfrak{d}}_3^{25} \overline{s}_3^{3}$ at $(207, 111)$ supports a $d_{51}$-differential, the source must be the class $\overline{\mathfrak{d}}_3^{25} \overline{s}_1 u_{28\lambda} u_{75\sigma} a_{47\lambda} a_{\sigma_2}$ at $(207, 95)$.  This proves the desired $d_{35}$-differential. 
\end{proof}

Using product structures with $\alpha$ and $\gamma^8$, the $d_{35}$-differentials in Proposition~\ref{prop:E4C4d35Differentials} produce all the other $d_{35}$-differentials.  They are shown in Figure~\ref{fig:E4C4d35Differentials}.

The $d_{43}$-differentials follow from Theorem~\ref{thm:LongCrossingDifferentials} (7) and using product structures with $\alpha$.  They are shown in Figure~\ref{fig:E4C4d43Differentials}.

\begin{prop}\label{prop:E4C4d51Differentials}
The following $d_{51}$-differentials exist: 
\begin{enumerate}
\item $d_{51}(\overline{\mathfrak{d}}_3^{25} \overline{s}_3^{3}) =2\overline{\mathfrak{d}}_3^{30} \overline{\mathfrak{d}}_1^{2} u_{11\lambda} u_{92\sigma} a_{81\lambda} \,\,\, (d_{51}(207,111) = 2(206,162))$;
\item $d_{51}(\overline{\mathfrak{d}}_3^{29} \overline{s}_3^{3}) = 2\overline{\mathfrak{d}}_3^{34} \overline{\mathfrak{d}}_1^{2} u_{27\lambda} u_{104\sigma} a_{77\lambda}\,\,\, (d_{51}(263,103) = 2(262,154))$.
\end{enumerate}
\end{prop}
\begin{proof}
For (1), consider the class $2\overline{\mathfrak{d}}_3^{30} \overline{\mathfrak{d}}_1^{2} u_{11\lambda} u_{92\sigma} a_{81\lambda}$ at $(206, 162)$.  By Theorem~\ref{thm:aboveFiltration61die}, this class must die on or before the $E_{61}$-page.  For degree reasons, the only way for this to happen is for the claimed $d_{51}$-differential to exist.  

For (2), consider the class $2\overline{\mathfrak{d}}_3^{34} \overline{\mathfrak{d}}_1^{2} u_{27\lambda} u_{104\sigma} a_{77\lambda}$ at $(262,154)$.  By Theorem~\ref{thm:aboveFiltration61die} again, this class must die on or before the $E_{61}$-page.  There are two possibilities: 
\begin{itemize}
\item This class supports a $d_{61}$-differential and kills the class $\overline{\mathfrak{d}}_3^{39} \overline{\mathfrak{d}}_1^{2} u_{12\lambda} u_{118\sigma} a_{107\lambda} a_{\sigma}$ at $(261,215)$;
\item This class is killed by a $d_{51}$-differential coming from the class $d_{51}(\overline{\mathfrak{d}}_3^{29} \overline{s}_3^{3})$ at $(263,103)$. 
\end{itemize}
The first case is impossible because by Theorem~\ref{thm:LongCrossingDifferentials} (5) and multiplication with $\alpha$, the class $\overline{\mathfrak{d}}_3^{39} \overline{\mathfrak{d}}_1^{2} u_{12\lambda} u_{118\sigma} a_{107\lambda} a_{\sigma}$ at $(261, 215)$ supports a $d_{53}$-differential.  It follows that the claimed $d_{51}$-differential exists. 
\end{proof}

Using product structures with $\alpha$ and $\gamma^8$, the $d_{51}$-differentials in Proposition~\ref{prop:E4C4d51Differentials} produce all the other $d_{51}$-differentials.  They are shown in Figure~\ref{fig:E4C4d51Differentials}.

All the $d_{53}$-differentials are obtained from the $d_{53}$-differentials in Theorem~\ref{thm:LongCrossingDifferentials} and using product structures with $\alpha$.  They are shown in Figure~\ref{fig:E4C4d53Differentials}

\begin{prop}\label{prop:E4C4d55Differentials}
The following $d_{55}$-differentials exist: 
\begin{enumerate}
\item $d_{55}(\overline{\mathfrak{d}}_3^{24} \overline{\mathfrak{d}}_1^{2} u_{28\lambda} u_{74\sigma} a_{46\lambda}) = \overline{\mathfrak{d}}_3^{29} \overline{s}_1 u_{14\lambda} u_{87\sigma} a_{73\lambda} a_{\sigma_2}\,\,\, (d_{55}(204,92) = (203,147))$;
\item $d_{55}(\overline{\mathfrak{d}}_3^{28} \overline{\mathfrak{d}}_1^{2} u_{44\lambda} u_{86\sigma} a_{42\lambda}) = \overline{\mathfrak{d}}_3^{33} \overline{s}_1 u_{30\lambda} u_{99\sigma} a_{69\lambda} a_{\sigma_2}\,\,\, (d_{55}(260,84) = (259,139))$.
\end{enumerate}
\end{prop}
\begin{proof}
For (1), consider the class $\overline{\mathfrak{d}}_3^{29} \overline{s}_1 u_{14\lambda} u_{87\sigma} a_{73\lambda} a_{\sigma_2}$ at $(203, 147)$.  By Theorem~\ref{thm:aboveFiltration61die}, this class must die on or before the $E_{61}$-page.  For degree reasons, the only way this can happen is for the claimed $d_{55}$-differential to exist.  

For (2), consider the class $\overline{\mathfrak{d}}_3^{33} \overline{s}_1 u_{30\lambda} u_{99\sigma} a_{69\lambda} a_{\sigma_2}$ at $(259, 139)$.  By Theorem~\ref{thm:aboveFiltration61die}, this class must die on or before the $E_{61}$-page.  For degree reasons, there are two possibilities: 
\begin{itemize}
\item This class supports a $d_{59}$-differential to kill the class $2\overline{\mathfrak{d}}_3^{38} u_{15\lambda} u_{114\sigma} a_{99\lambda}$ at $(258, 198)$;
\item this class is killed by a $d_{55}$-differential from the class $\overline{\mathfrak{d}}_3^{28} \overline{\mathfrak{d}}_1^{2} u_{44\lambda} u_{86\sigma} a_{42\lambda}$ at $(260, 84)$. 
\end{itemize}
The first case is impossible because by Theorem~\ref{thm:LongCrossingDifferentials} (2) and multiplication with $\alpha$, the class $2\overline{\mathfrak{d}}_3^{38} u_{15\lambda} u_{114\sigma} a_{99\lambda}$ at $(258, 198)$ supports a $d_{61}$-differential.  Therefore, the second possibility must occur, and we get our desired $d_{55}$-differential.  
\end{proof}

Using product structures with $\alpha$ and $\gamma^8$, the $d_{55}$-differentials in Proposition~\ref{prop:E4C4d55Differentials} produce all the other $d_{55}$-differentials.  They are shown in Figure~\ref{fig:E4C4d55Differentials}.

\begin{prop}\label{prop:E4C4d59Differentials}
The following $d_{59}$-differentials exist: 
\begin{enumerate}
\item $d_{59}(\overline{\mathfrak{d}}_3^{17} \overline{s}_1 u_{14\lambda} u_{51\sigma} a_{37\lambda} a_{\sigma_2}) = \overline{\mathfrak{d}}_3^{22} u_{64\sigma} a_{66\lambda} a_{2\sigma}\,\,\, (d_{59}(131,75) = (130,134))$;
\item $d_{59}(\overline{\mathfrak{d}}_3^{21} \overline{s}_1 u_{30\lambda} u_{63\sigma} a_{33\lambda} a_{\sigma_2}) = 2\overline{\mathfrak{d}}_3^{26} u_{15\lambda} u_{78\sigma} a_{63\lambda}\,\,\, (d_{59}(187,67) = 2(186,126))$.
\end{enumerate}
\end{prop}
\begin{proof}
(1) is Theorem~\ref{thm:LongCrossingDifferentials} (3).  To prove (2), consider the class $2\overline{\mathfrak{d}}_3^{26} u_{15\lambda} u_{78\sigma} a_{63\lambda}$ at $(186, 126)$.  By Theorem~\ref{thm:aboveFiltration61die}, this class must die on or before the $E_{61}$-page.  For degree reasons, the only way this can happen is for the claimed $d_{59}$-differential to exist.  
\end{proof}
Using product structures with $\alpha$ and $\gamma^8$, the $d_{59}$-differentials in Proposition~\ref{prop:E4C4d59Differentials} produce all the other $d_{59}$-differentials.  They are shown in Figure~\ref{fig:E4C4d59Differentials}.

\begin{prop}\label{prop:E4C4d61Differentials}
The following $d_{61}$-differentials exist: 
\begin{enumerate}
\item $d_{61}(2\overline{\mathfrak{d}}_3^{26} u_{31\lambda} u_{78\sigma} a_{47\lambda}) = \overline{\mathfrak{d}}_3^{31} u_{16\lambda} u_{92\sigma} a_{77\lambda} a_{\sigma}\,\,\, (d_{61}(2(218,194)) = (217,155))$;
\item $d_{61}(\overline{\mathfrak{d}}_3^{27} \overline{\mathfrak{d}}_1^{2} u_{28\lambda} u_{82\sigma} a_{55\lambda} a_{\sigma}) = \overline{\mathfrak{d}}_3^{32} \overline{\mathfrak{d}}_1^{2} u_{12\lambda} u_{98\sigma} a_{86\lambda}\,\,\, (d_{61}(221,111) = (220,172))$;
\item $d_{61}(2\overline{\mathfrak{d}}_3^{30} \overline{\mathfrak{d}}_1^{2} u_{27\lambda} u_{92\sigma} a_{65\lambda}) =\overline{\mathfrak{d}}_3^{35} \overline{\mathfrak{d}}_1^{2} u_{12\lambda} u_{106\sigma} a_{95\lambda} a_{\sigma} \,\,\, (d_{61}(2(238,130)) = (237,191))$;
\item $d_{61}(\overline{\mathfrak{d}}_3^{31} u_{32\lambda} u_{92\sigma} a_{61\lambda} a_{\sigma}) = \overline{\mathfrak{d}}_3^{36} u_{16\lambda} u_{108\sigma} a_{92\lambda}\,\,\, (d_{61}(249,123) = (248,184))$;
\item $d_{61}(2\overline{\mathfrak{d}}_3^{30} u_{47\lambda} u_{90\sigma} a_{43\lambda}) =\overline{\mathfrak{d}}_3^{35} u_{32\lambda} u_{104\sigma} a_{73\lambda} a_{\sigma} \,\,\, (d_{61}(2(274,86)) = (273,147))$;
\item $d_{61}(\overline{\mathfrak{d}}_3^{31} \overline{\mathfrak{d}}_1^{2} u_{44\lambda} u_{94\sigma} a_{51\lambda} a_{\sigma}) =\overline{\mathfrak{d}}_3^{36} \overline{\mathfrak{d}}_1^{2} u_{28\lambda} u_{110\sigma} a_{82\lambda} \,\,\, (d_{61}(277,103) = (276,164))$;
\item $d_{61}(2\overline{\mathfrak{d}}_3^{34} \overline{\mathfrak{d}}_1^{2} u_{43\lambda} u_{104\sigma} a_{61\lambda}) = \overline{\mathfrak{d}}_3^{39} \overline{\mathfrak{d}}_1^{2} u_{28\lambda} u_{118\sigma} a_{91\lambda} a_{\sigma}\,\,\, (d_{61}(2(294,122)) = (293,183))$;
\item $d_{61}(\overline{\mathfrak{d}}_3^{35} u_{48\lambda} u_{104\sigma} a_{57\lambda} a_{\sigma}) =\overline{\mathfrak{d}}_3^{40} u_{32\lambda} u_{120\sigma} a_{88\lambda} \,\,\, (d_{61}(305,115) = (304,176))$.
\end{enumerate}
\end{prop}
\begin{proof}
All of these differentials are proven by using the same method: we first consider the target, which, by Theorem~\ref{thm:aboveFiltration61die}, must die on or before the $E_{61}$-page.  Once we know this, then for degree reasons, we deduce the claimed $d_{61}$-differential.  
\end{proof}
Using product structures with $\alpha$ and $\gamma^8$, the $d_{61}$-differentials in Proposition~\ref{prop:E4C4d61Differentials} produce all the other $d_{61}$-differentials.  They are shown in Figure~\ref{fig:E4C4d61Differentials}.

\begin{figure}
\begin{center}
\makebox[\textwidth]{\includegraphics[trim={0cm 10cm 0cm 10cm}, clip, page = 1, scale = 0.23]{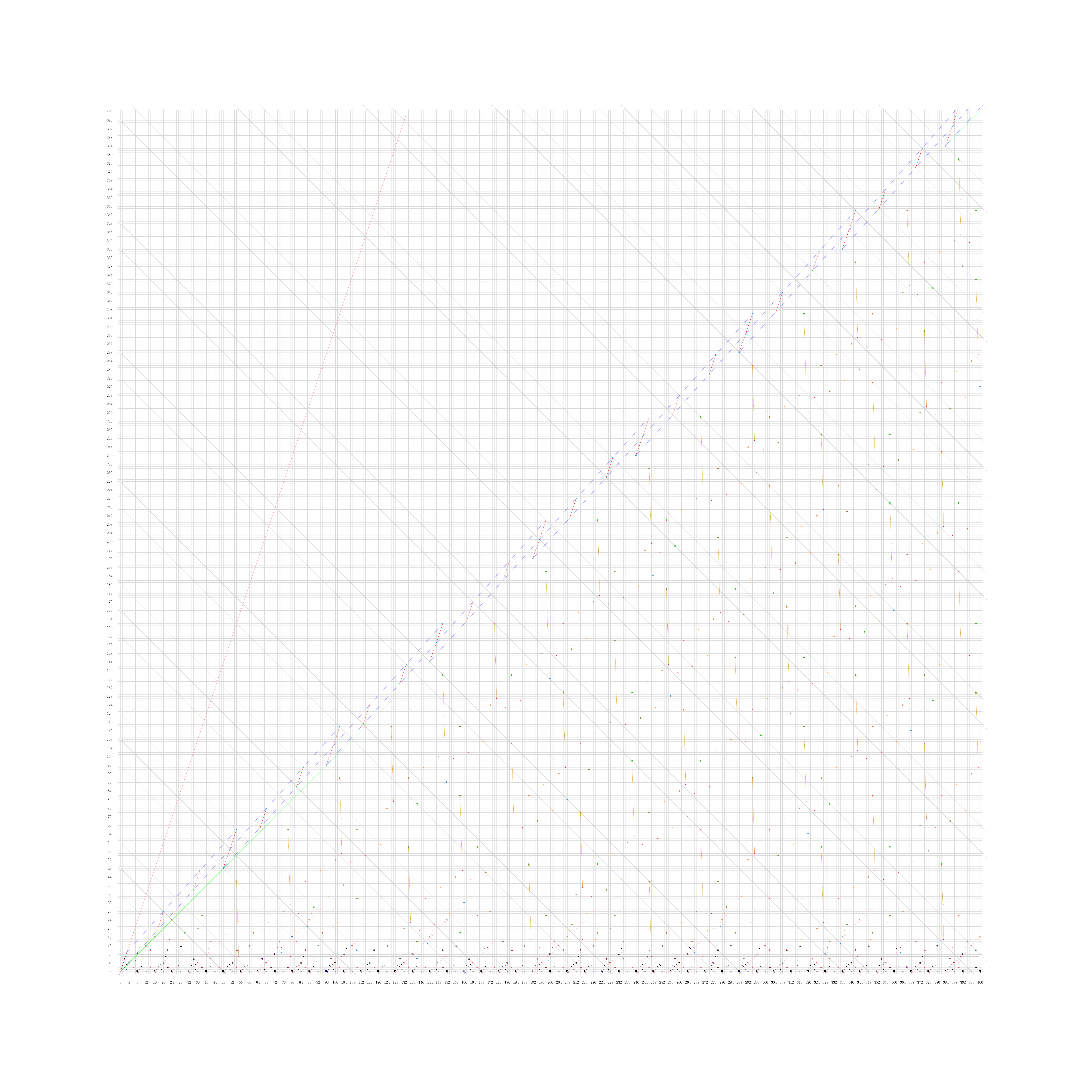}}
\end{center}
\begin{center}
\caption{$d_{35}$-differentials in $\SliceSS(\BPtwo)$.}
\hfill
\label{fig:E4C4d35Differentials}
\end{center}
\end{figure}

\begin{figure}
\begin{center}
\makebox[\textwidth]{\includegraphics[trim={0cm 10cm 0cm 10cm}, clip, page = 2, scale = 0.23]{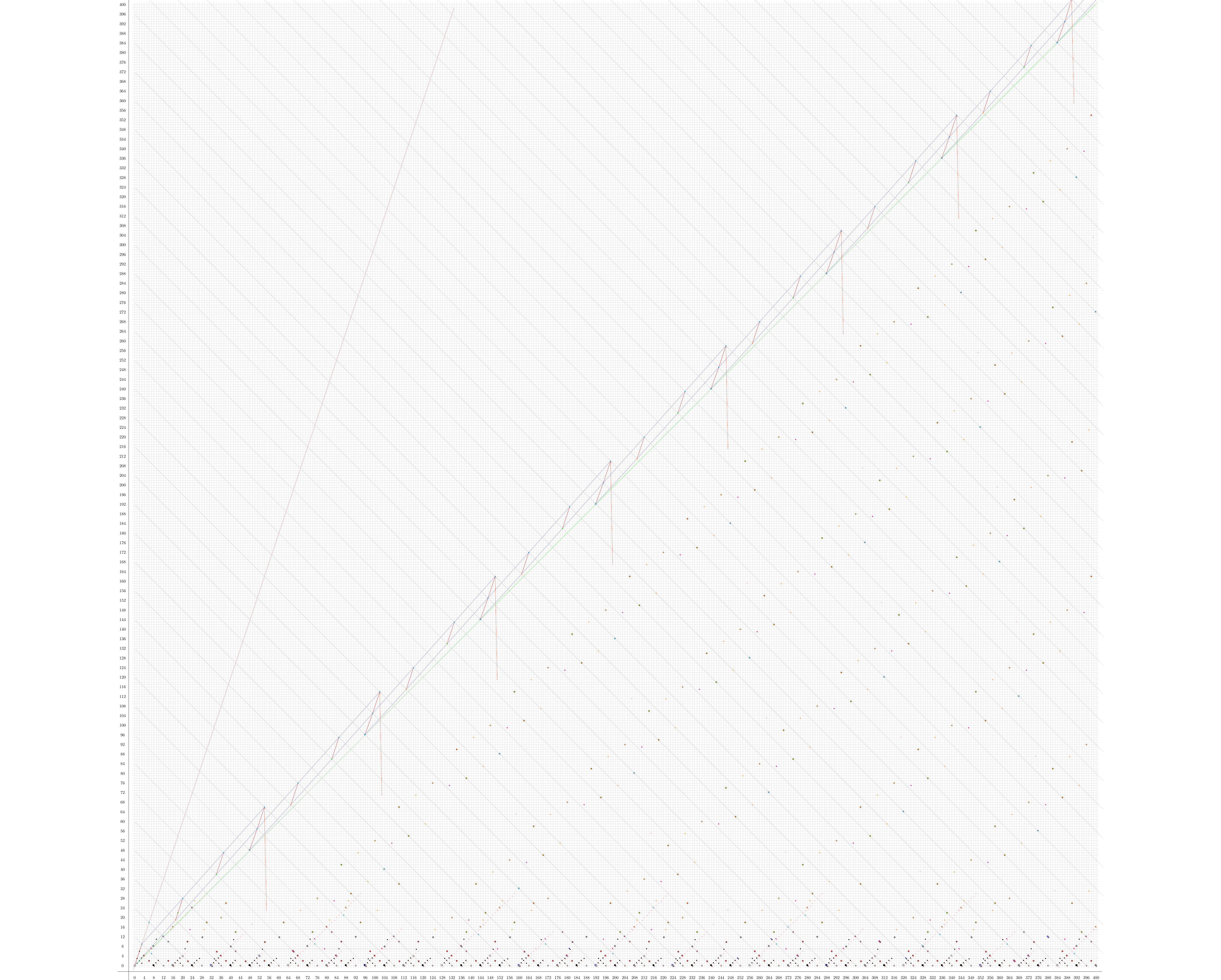}}
\end{center}
\begin{center}
\caption{$d_{43}$-differentials in $\SliceSS(\BPtwo)$.}
\hfill
\label{fig:E4C4d43Differentials}
\end{center}
\end{figure}

\begin{figure}
\begin{center}
\makebox[\textwidth]{\includegraphics[trim={0cm 10cm 0cm 10cm}, clip, page = 3, scale = 0.23]{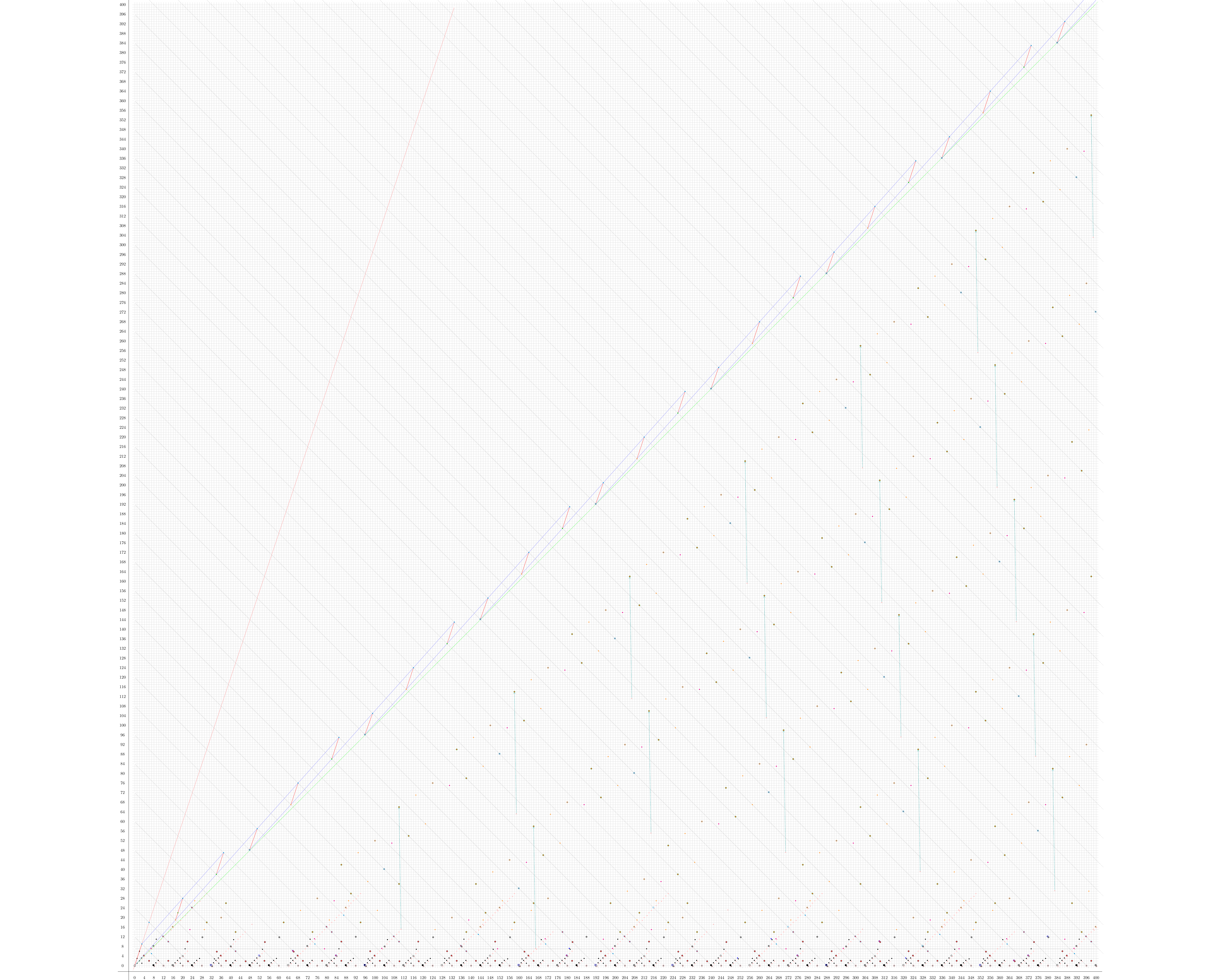}}
\end{center}
\begin{center}
\caption{$d_{51}$-differentials in $\SliceSS(\BPtwo)$.}
\hfill
\label{fig:E4C4d51Differentials}
\end{center}
\end{figure}

\begin{figure}
\begin{center}
\makebox[\textwidth]{\includegraphics[trim={0cm 10cm 0cm 10cm}, clip, page = 4, scale = 0.23]{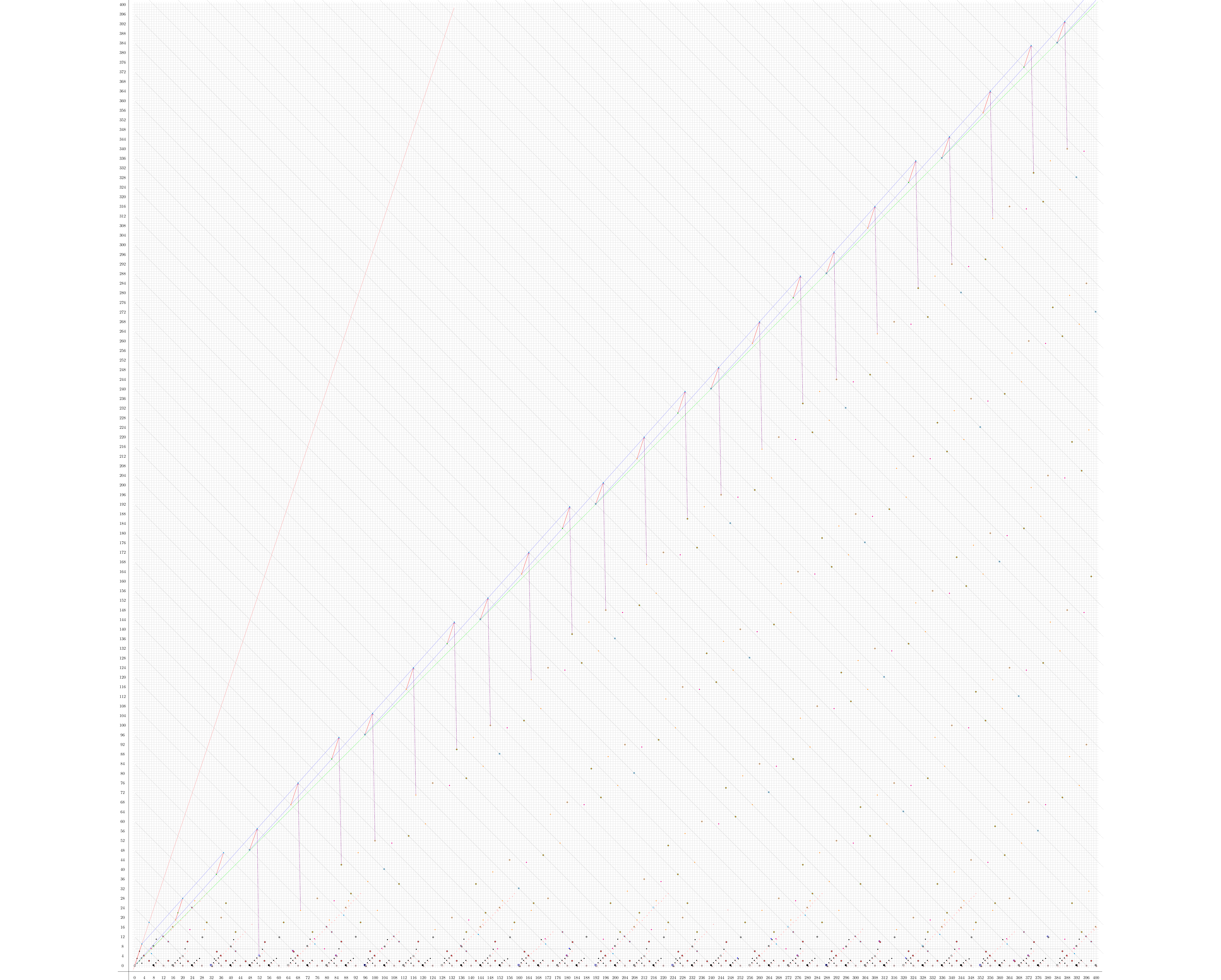}}
\end{center}
\begin{center}
\caption{$d_{53}$-differentials in $\SliceSS(\BPtwo)$.}
\hfill
\label{fig:E4C4d53Differentials}
\end{center}
\end{figure}

\begin{figure}
\begin{center}
\makebox[\textwidth]{\includegraphics[trim={0cm 10cm 0cm 10cm}, clip, page = 5, scale = 0.23]{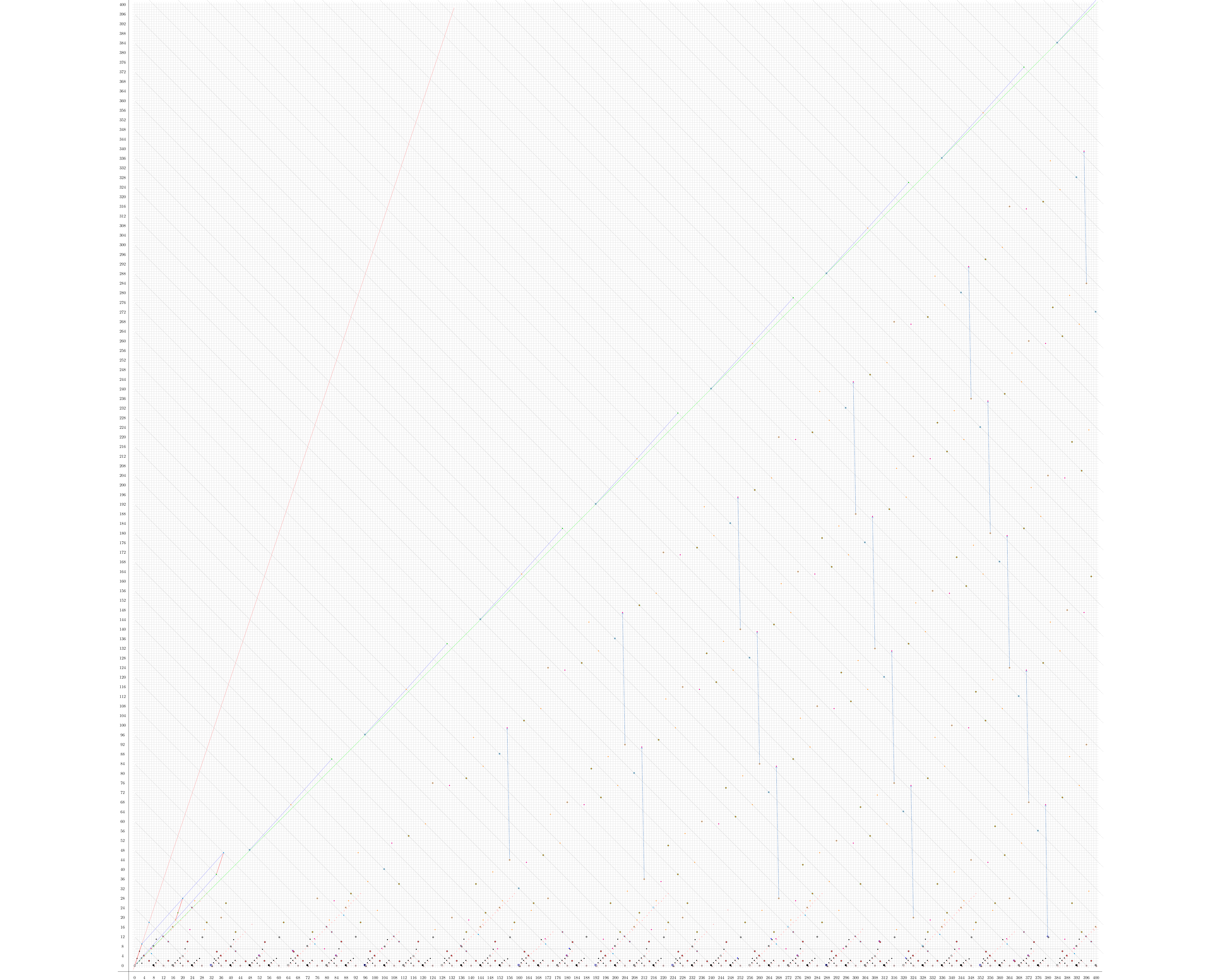}}
\end{center}
\begin{center}
\caption{$d_{55}$-differentials in $\SliceSS(\BPtwo)$.}
\hfill
\label{fig:E4C4d55Differentials}
\end{center}
\end{figure}

\begin{figure}
\begin{center}
\makebox[\textwidth]{\includegraphics[trim={0cm 10cm 0cm 10cm}, clip, page = 6, scale = 0.23]{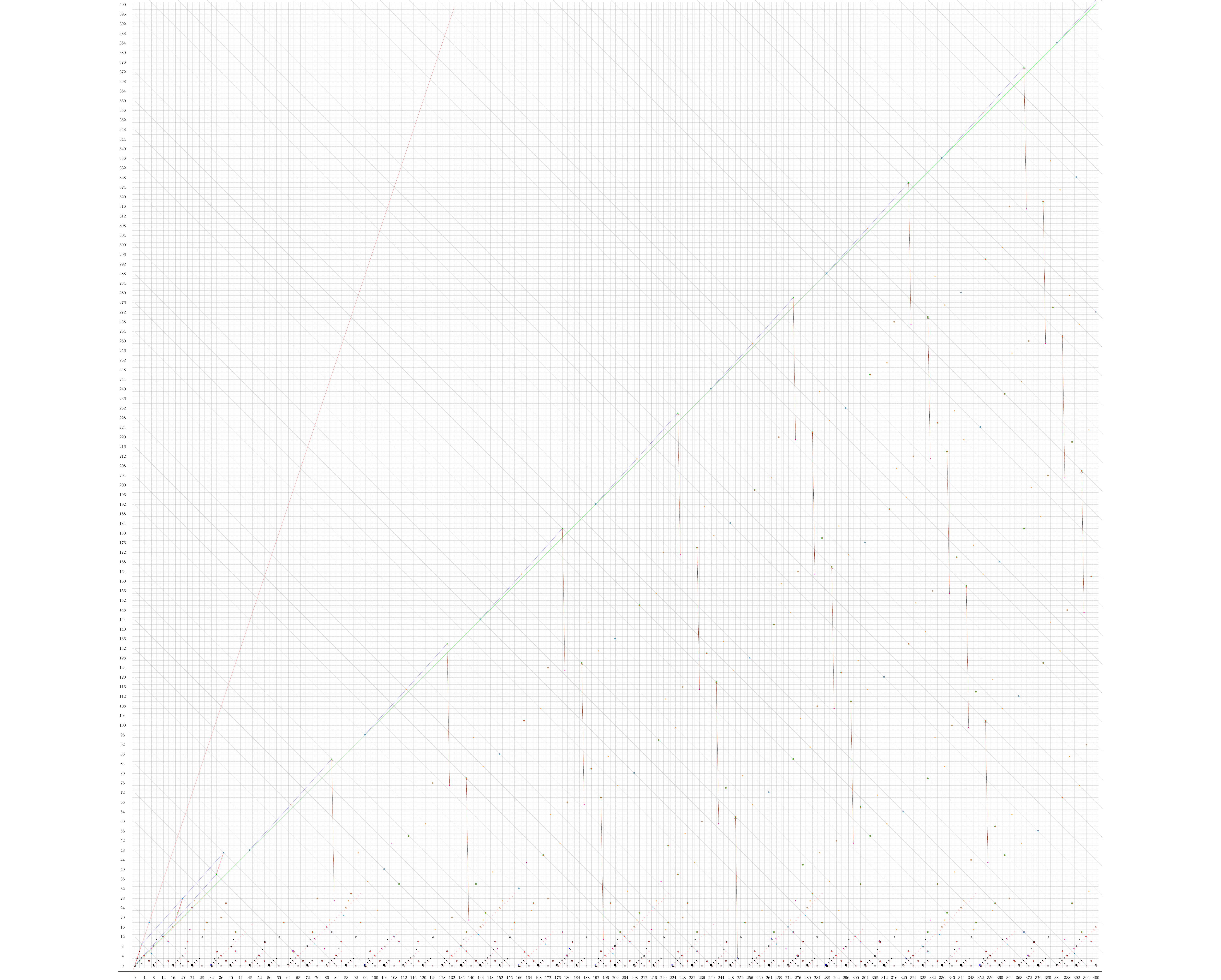}}
\end{center}
\begin{center}
\caption{$d_{59}$-differentials in $\SliceSS(\BPtwo)$.}
\hfill
\label{fig:E4C4d59Differentials}
\end{center}
\end{figure}

\begin{figure}
\begin{center}
\makebox[\textwidth]{\includegraphics[trim={0cm 10cm 0cm 10cm}, clip, page = 7, scale = 0.23]{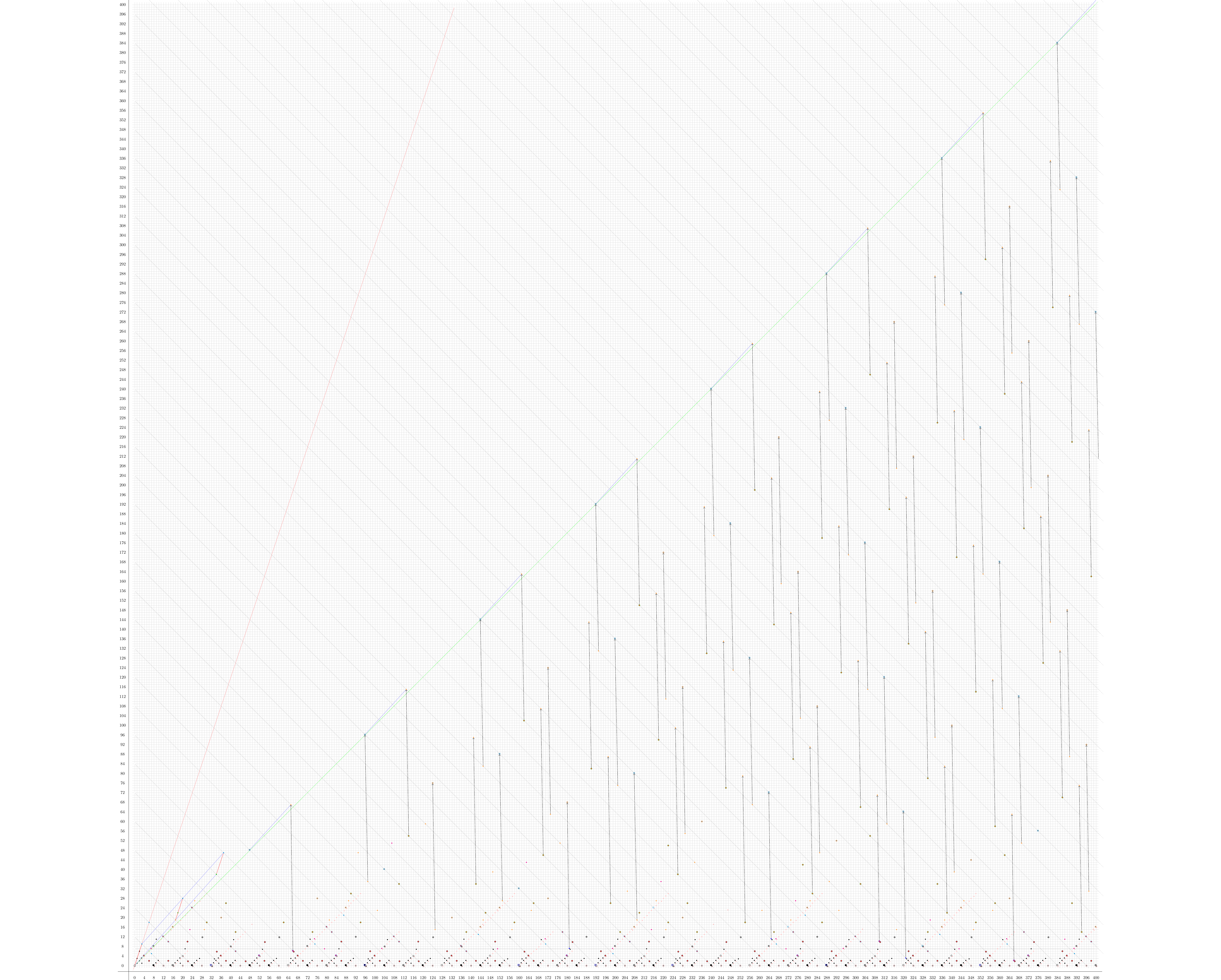}}
\end{center}
\begin{center}
\caption{$d_{61}$-differentials in $\SliceSS(\BPtwo)$.}
\hfill
\label{fig:E4C4d61Differentials}
\end{center}
\end{figure}

%%%%%%%%%%%%%%%%%%%%%%%%%%%%%%
%%%%%%%%%%%%%%%%%%%%%%%%%%%%%%
%%%%%%%%%%%%%%%%%%%%%%%%%%%%%%
\section{Summary of Differentials} \label{sec:SummaryOfDifferentials}
In this section, we summarize all the differentials in the slice spectral sequence of $\BPtwo$.  Table~\ref{table:summaryDifferentials} summarizes all the $RO(C_4)$-graded differentials that we can deduce from the integer graded spectral sequence, as well as their proofs.  Figure~\ref{fig:E4C4AllDifferentials} shows all the differentials from $d_{13}$ to $d_{61}$.  To better organize the differentials visually, Figure~\ref{fig:E4C4d13-d15Differentials} shows the $d_{13}$ to $d_{15}$-differentials; Figure~\ref{fig:E4C4d19-d31Differentials} shows the $d_{19}$ to $d_{31}$-differentials; and Figure~\ref{fig:E4C4d35-d61Differentials} shows the $d_{35}$ to $d_{61}$-differentials.  The $E_\infty$-page is shown in Figure~\ref{fig:E4C4EinftyPage}.  We observe that there is a horizontal vanishing line at filtration 60.  

%%%%%%%%%%%%%%%%%%%%%%%%%%

\begin{center}
\begin{longtable}{| p{.15\textwidth} | p{.50\textwidth} |p{.45\textwidth} |} 
\hline
Differential & Formula & Proof \\ \hline 
$d_3$ & $d_3(u_\lambda) = \bar{s}_1a_\lambda a_{\sigma_2}$ 
\newline $d_3(u_{2\sigma_2}) = (\rone + \grone) a_{3\sigma_2}$ & Section~\ref{subsec:d3induced} \\ \hline
{\color{blue} $d_5$} & $d_5(u_{2\sigma}) = \done a_\lambda a_{3\sigma}$
\newline $d_5(u_{2\lambda}) = \done u_\lambda a_{2\lambda}a_\sigma$ 
\newline $d_5(u_{3\lambda}a_\sigma) = 2\done u_\lambda u_{2\sigma}a_{3\lambda} $
&Section~\ref{subsec:d5induced} \\ \hline
${\color{Magenta} d_7}$ &$d_7(2u_{2\lambda}) = \done \sone u_\sigma a_{3\lambda} a_{\sigma_2}$ \newline
$d_7(2u_{2\lambda} u_{2\sigma}) = \done \sone u_{3\sigma} a_{3\lambda} a_{\sigma_2}$ \newline
$d_7(u_{4\lambda}) = \done \sone u_{2\lambda} u_\sigma a_{3\lambda} a_{\sigma_2}$ \newline
$d_7(u_{4\sigma_2}) = (\rone^3 + \rthree + \grthree)a_{7\sigma_2}$& Section~\ref{subsec:d7induced} \\ \hline
${\color{green!60!black}d_{11}}$ & $d_{11}(\sone u_{2\lambda}u_{3\sigma}a_{\sigma_2}) = \done^3 u_{4\sigma}a_{5\lambda}a_{2\sigma} $
\newline $d_{11}(\sone u_{6\lambda}u_\sigma a_{\sigma_2}) = 2\done^3 u_{3\lambda} u_{4\sigma}a_{6\lambda}$& Section~\ref{subsec:d11induced} \\ \hline

${\color{cyan} d_{13}}$ & 
$d_{13}(u_{4\sigma})= \dthree a_{3\lambda}a_{7\sigma}$ \newline
$d_{13}(u_{4\lambda}a_\sigma)= \done^3 u_{4\sigma}a_{7\lambda}$ \newline
$d_{13}(u_{4\lambda}u_{2\sigma})= \dthree u_\lambda u_{4\sigma}a_{6\lambda}a_\sigma$ \newline
$d_{13}(u_{4\lambda}u_{2\sigma}a_\sigma)= 2\dthree u_{6\sigma}a_{7\lambda}$ \newline
$d_{13}(u_{5\lambda}a_\sigma)= 2\dthree u_\lambda u_{4\sigma}a_{7\lambda}$ \newline
%$d_{13}(u_{4\lambda}u_{2\sigma}a_\sigma)= 2\dthree u_{6\sigma}a_{7\lambda}$ \newline
$d_{13}(u_{5\lambda}u_{2\sigma}a_\sigma)= 2\dthree u_{\lambda} u_{6\sigma}a_{7\lambda}$
%$d_{13}(u_{8\lambda}u_{2\sigma}a_\sigma)= \done^3 u_{4\lambda}u_{6\sigma} a_{7\lambda}$ cab be obtained from the two differentials before
& Section~\ref{subsec:Predictedd13Diff} \newline 
Proposition~\ref{prop:d13slice} \newline Proposition~\ref{prop:BPtwod13(20,4)} \newline 
Theorem~\ref{thm:NormedDiffd13} \\ \hline 

${\color{Magenta} d_{15}}$ & $d_{15}(u_{8\sigma_2}) = \rone (\rthree \grthree + \rthree^2 + \grthree^2) a_{15\sigma_2}$ \newline 
$d_{15}(2u_{4\lambda}) = \dthree \sone u_{3\sigma}a_{7\lambda}a_{\sigma_2}$ \newline
$d_{15}(2u_{4\lambda}u_{2\sigma}) = \dthree \sone u_{5\sigma}a_{7\lambda}a_{\sigma_2}$ \newline
$d_{15}(2u_{4\lambda}u_{4\sigma}) = \dthree \sone u_{7\sigma}a_{7\lambda}a_{\sigma_2}$ \newline
$d_{15}(2u_{4\lambda}u_{6\sigma}) = \dthree \sone u_{9\sigma}a_{7\lambda}a_{\sigma_2}$ \newline
$d_{15}(\done u_{4\lambda}u_{2\sigma}) = \sthree^3 u_\sigma a_{3\lambda}a_{9\sigma_2}$ \newline
$d_{15}(\done u_{4\lambda}u_{6\sigma}) = \sthree^3 u_{5\sigma} a_{3\lambda}a_{9\sigma_2}$ \newline
$d_{15}(u_{8\lambda}) = \dthree \sone u_{4\lambda} u_{3\sigma} a_{7\lambda} a_{\sigma_2}$ \newline
$d_{15}(u_{8\lambda}u_{4\sigma}) =\dthree \sone u_{4\lambda} u_{7\sigma} a_{7\lambda} a_{\sigma_2} $ & Section~\ref{subsection:inducedC2Differentials} \newline Section~\ref{subsec:d15ResTr} \\ \hline
${\color{LimeGreen}d_{19}}$ & $d_{19}(2u_{5\lambda}) = \sthree^3 a_{5\lambda}a_{9\sigma_2}$ \newline
$d_{19}(2u_{5\lambda}u_{2\sigma}) =\sthree^3 u_{2\sigma}a_{5\lambda}a_{9\sigma_2}$ \newline
$d_{19}(2u_{5\lambda}u_{4\sigma}) =\sthree^3 u_{4\sigma}a_{5\lambda}a_{9\sigma_2}$ \newline
$d_{19}(2u_{5\lambda}u_{6\sigma}) =\sthree^3 u_{\sigma}a_{5\lambda}a_{9\sigma_2}$ \newline
$d_{19}(2u_{13\lambda}) = \sthree^3 u_{8\lambda}a_{5\lambda}a_{9\sigma_2}$ \newline
$d_{19}(2u_{13\lambda}u_{2\sigma}) = \sthree^3 u_{8\lambda} u_{2\sigma}a_{5\lambda}a_{9\sigma_2} $ \newline
$d_{19}(2u_{13\lambda}u_{4\sigma}) =\sthree^3 u_{8\lambda} u_{4\sigma}a_{5\lambda}a_{9\sigma_2} $ \newline
$d_{19}(2u_{13\lambda}u_{6\sigma}) =\sthree^3 u_{8\lambda} u_{6\sigma}a_{5\lambda}a_{9\sigma_2} $ \newline
$d_{19}(\sone u_{4\lambda}u_{\sigma}a_{\sigma_2}) =\dthree \done^2 u_{4\sigma}a_{9\lambda}a_{2\sigma}$ \newline
$d_{19}(\sone u_{4\lambda}u_{5\sigma}a_{\sigma_2}) =\dthree \done^2 u_{8\sigma}a_{9\lambda}a_{2\sigma}$ \newline
$d_{19}(\sone u_{12\lambda}u_{\sigma}a_{\sigma_2}) = 2 \dthree \done^2 u_{7\lambda}u_{6\sigma}a_{10\lambda}$ \newline
$d_{19}(\sone u_{12\lambda}u_{5\sigma}a_{\sigma_2}) =2 \dthree \done^2 u_{7\lambda}u_{10\sigma}a_{10\lambda}$ & Proposition~\ref{prop:d19d27Differential} \newline Proven together with $d_{27}$-differentials \newline Uses Theorem~\ref{thm:aboveFiltration61die} and some $d_{29}$- and $d_{35}$-differentials\\ \hline
${\color{blue}d_{21}}$ &$d_{21}(u_{9\lambda}u_{2\sigma}a_\sigma) = 2 \dthree \done^2 u_{3\lambda} u_{8\sigma} a_{11\lambda}$ \newline 
$d_{21}(u_{9\lambda}u_{6\sigma}a_\sigma) =2 \dthree \done^2 u_{3\lambda} u_{12\sigma} a_{11\lambda}$ & Proposition~\ref{prop:d21Differential} (uses Theorem~\ref{thm:aboveFiltration61die}) \\ \hline
${\color{orange}d_{23}}$ & $d_{23}(\done u_{8\lambda}) = \dthree^2 \sone u_{2\lambda} u_{5\sigma} a_{11\lambda} a_{\sigma_2}$ \newline 
$d_{23}(\done u_{8\lambda}u_{4\sigma}) =\dthree^2 \sone u_{2\lambda} u_{9\sigma} a_{11\lambda} a_{\sigma_2}$ & Proposition~\ref{prop:d23Differential} \\ \hline
${\color{ForestGreen}d_{27}}$ &$d_{27}(\sone u_{6\lambda}u_{3\sigma}a_{\sigma_2}) = \dthree^2\done u_{8\sigma} a_{13\lambda}a_{2\sigma}$ \newline 
$d_{27}(\sone u_{6\lambda}u_{7\sigma}a_{\sigma_2}) = \dthree^2\done u_{12\sigma} a_{13\lambda}a_{2\sigma}$ \newline
$d_{27}(\sone u_{10\lambda} u_{\sigma}a_{\sigma_2}) = 2\dthree^2 \done u_{3\lambda} u_{8\sigma} a_{14\lambda}$ \newline
$d_{27}(\sone u_{10\lambda} u_{5\sigma}a_{\sigma_2}) = 2\dthree^2 \done u_{3\lambda} u_{12\sigma} a_{14\lambda}$ & Proposition~\ref{prop:d19d27Differential} \newline Proven together with $d_{19}$-differentials \newline Uses Theorem~\ref{thm:aboveFiltration61die} and some $d_{29}$- and $d_{35}$-differentials\\ \hline
${\color{Red}d_{29}}$ & $d_{29}(2u_{7\lambda}u_{2\sigma}) = \dthree^2 \done u_{8\sigma} a_{14\lambda} a_\sigma$ \newline
$d_{29}(2\done u_{7\lambda}u_{2\sigma}) = \dthree^2 \done^2 u_{8\sigma} a_{14\lambda} a_\sigma$ \newline
$d_{29}(2u_{7\lambda}u_{6\sigma}) = \dthree^2 \done u_{12\sigma} a_{14\lambda} a_\sigma$ \newline
$d_{29}(2\done u_{7\lambda}u_{6\sigma}) = \dthree^2 \done^2 u_{12\sigma} a_{14\lambda} a_\sigma$ \newline
$d_{29}(u_{8\lambda}a_\sigma) = \dthree^2 \done u_{8\sigma} a_{15\lambda}$ \newline
$d_{29}(\done u_{8\lambda}a_\sigma) = \dthree^2 \done^2 u_{8\sigma} a_{15\lambda}$ \newline
$d_{29}(\done^2 u_{8\lambda}a_\sigma) = 2\dthree^3 u_{8\sigma} a_{15\lambda}$ \newline
$d_{29}(u_{8\lambda}u_{4\sigma}a_\sigma) = \dthree^2 \done u_{12\sigma} a_{15\lambda}$ \newline
$d_{29}(\done u_{8\lambda}u_{4\sigma}a_\sigma) = \dthree^2 \done^2 u_{12\sigma} a_{15\lambda}$ \newline
$d_{29}(\done^2 u_{8\lambda}u_{4\sigma}a_\sigma) = 2\dthree^3 u_{12\sigma} a_{15\lambda} $ \newline
$d_{29}(\done^2 u_{8\lambda}) = \dthree^3 u_\lambda u_{6\sigma} a_{14\lambda} a_\sigma$ \newline 
$d_{29}(\done^2 u_{8\lambda} u_{4\sigma}) = \dthree^3 u_\lambda u_{10\sigma} a_{14\lambda} a_\sigma$ \newline
$d_{29}(2u_{11\lambda}) = \dthree^2 \done u_{4\lambda} u_{6\sigma} a_{14\lambda} a_\sigma$ \newline
$d_{29}(2\done u_{11\lambda}) = \dthree^2 \done^2 u_{4\lambda} u_{6\sigma} a_{14\lambda} a_\sigma$ \newline
$d_{29}(2u_{11\lambda}u_{4\sigma}) = \dthree^2 \done u_{4\lambda} u_{10\sigma} a_{14\lambda} a_\sigma$ \newline
$d_{29}(2\done u_{11\lambda}u_{4\sigma}) = \dthree^2 \done^2 u_{4\lambda} u_{10\sigma} a_{14\lambda} a_\sigma$ \newline 
$d_{29}(\done u_{12\lambda}u_{2\sigma}a_\sigma) = \dthree^2 \done^2 u_{4\lambda} u_{10\sigma} a_{15\lambda}$ \newline
$d_{29}(\done u_{12\lambda}u_{6\sigma}a_\sigma) = \dthree^2 \done^2 u_{4\lambda} u_{14\sigma} a_{15\lambda}$ & Theorem~\ref{thm:E4C4d29Differentials} (uses Theorem~\ref{thm:NormedDiffd29} and Theorem~\ref{thm:aboveFiltration61die}) \\ \hline
${\color{pink!80!black}d_{31}}$ & $d_{31}(u_{16\sigma_2}) = \rthree^4 \grthree a_{31\sigma_2}$ \newline
$d_{31}(2u_{8\lambda}) = \dthree \sthree^3 u_{3\sigma} a_{11\lambda} a_{9\sigma_2}$ \newline
$d_{31}(2u_{8\lambda}u_{4\sigma}) = \dthree \sthree^3 u_{7\sigma} a_{11\lambda} a_{9\sigma_2}$ \newline
$d_{31}(u_{16\lambda}) = \dthree \sthree^3 u_{8\lambda} u_{3\sigma} a_{11\lambda} a_{9\sigma_2}$ & Section~\ref{subsection:inducedC2Differentials} \newline Proposition~\ref{prop:d31(40,8)} \newline Proposition~\ref{prop:gammaPowers} \newline Section~\ref{subsec:d31Diff} \\ \hline
${\color{Dandelion!90!black}d_{35}}$ & $d_{35}(\sone u_{12\lambda} u_{3\sigma} a_{\sigma_2}) = 2 \dthree^3 u_{3\lambda} u_{12\sigma} a_{18\lambda}$ \newline
$d_{35}(\sone u_{12\lambda} u_{7\sigma} a_{\sigma_2}) = 2 \dthree^3 u_{3\lambda} u_{16\sigma} a_{18\lambda}$ \newline
$d_{35}(\sone u_{28\lambda} u_{3\sigma} a_{\sigma_2}) = 2 \dthree^3 u_{19\lambda} u_{12\sigma} a_{18\lambda}$ \newline
$d_{35}(\sone u_{28\lambda} u_{7\sigma} a_{\sigma_2}) = 2 \dthree^3 u_{19\lambda} u_{16\sigma} a_{18\lambda}$  & Proposition~\ref{prop:E4C4d35Differentials} (uses Theorem~\ref{thm:aboveFiltration61die})\\ \hline
${\color{RedOrange}d_{43}}$ &$d_{43}(\sthree^3 u_{8\lambda} u_{7\sigma} a_{9\sigma_2} ) = \dthree^5 u_{16\sigma} a_{23\lambda} a_{6\sigma}$ & Theorem~\ref{thm:LongCrossingDifferentials} (uses Theorem~\ref{thm:aboveFiltration61die})\\ \hline
${\color{TealBlue!90!black}d_{51}}$ & $d_{51}(\sthree^3 u_{24\lambda} u_{3\sigma} a_{9\sigma_2}) = 2\dthree^5 \done^2 u_{11\lambda} u_{20\sigma} a_{30\lambda} $ \newline
$d_{51}(\sthree^3 u_{40\lambda} u_{7\sigma} a_{9\sigma_2}) = 2\dthree^5 \done^2 u_{27\lambda} u_{24\sigma} a_{30\lambda}$  & Proposition~\ref{prop:E4C4d51Differentials} (uses Theorem~\ref{thm:aboveFiltration61die}) \\ \hline
${\color{Plum}d_{53}}$ & $d_{53}(\done^2 u_{12\lambda} u_{6\sigma}) = \dthree^5 u_{16\sigma} a_{25\lambda} a_{3\sigma}$ & Theorem~\ref{thm:LongCrossingDifferentials} (uses Theorem~\ref{thm:aboveFiltration61die})\\ \hline
${\color{NavyBlue}d_{55}}$ & $d_{55}(\done^2 u_{28\lambda}u_{2\sigma}) = \dthree^5 \sone u_{14\lambda} u_{15\sigma} a_{27\lambda} a_{\sigma_2}$ \newline
$d_{55}(\done^2 u_{44\lambda} u_{6\sigma}) = \dthree^5 \sone u_{30\lambda} u_{19\sigma} a_{27\lambda} a_{\sigma_2}$ & Proposition~\ref{prop:E4C4d55Differentials} (uses Theorem~\ref{thm:aboveFiltration61die})\\ \hline
${\color{RawSienna}d_{59}}$ & $d_{59}(\sone u_{14\lambda} u_{3\sigma} a_{\sigma_2}) = \dthree^5 u_{16\sigma} a_{29\lambda} a_{2\sigma}$ \newline 
$d_{59}(\sone u_{30\lambda} u_{7\sigma} a_{\sigma_2}) = \dthree^5 u_{16\lambda} u_{20\sigma} a_{29\lambda} a_{2\sigma}$ & Proposition~\ref{prop:E4C4d59Differentials} (uses Theorem~\ref{thm:aboveFiltration61die}) \\ \hline
${\color{black}d_{61}}$ & $d_{61}(2u_{15\lambda}u_{2\sigma}) = \dthree^5 u_{16\sigma}a_{30\lambda} a_\sigma$ \newline
$d_{61}(2u_{31\lambda}u_{6\sigma}) = \dthree^5 u_{16\lambda} u_{20\sigma} a_{30\lambda} a_\sigma$ \newline
$d_{61}(u_{16\lambda}a_\sigma) = \dthree^5 u_{16\sigma}a_{31\lambda}$ \newline
$d_{61}(u_{32\lambda}u_{4\sigma} a_\sigma) = \dthree^5 u_{16\lambda} u_{20\sigma} a_{31\lambda}$ \newline
$d_{61}(2\done^2 u_{27\lambda} u_{4\sigma}) = \dthree^5 \done^2 u_{12\lambda} u_{18\sigma} a_{30\lambda} a_\sigma$ \newline
$d_{61}(2\done^2 u_{43\lambda}) = \dthree^5\done^2 u_{28\lambda} u_{14\sigma} a_{30\lambda} a_\sigma$ \newline
$d_{61}(\done^2 u_{28\lambda} u_{2\sigma} a_\sigma) = \dthree^5 \done^2 u_{12\lambda} u_{18\sigma} a_{31\lambda}$ \newline
$d_{61}(\done^2 u_{44\lambda} u_{6\sigma} a_\sigma) = \dthree^5 \done^2 u_{28\lambda} u_{22\sigma} a_{31\lambda} $  & Proposition~\ref{prop:E4C4d61Differentials} (uses Theorem~\ref{thm:NormedDiffd61} and Theorem~\ref{thm:aboveFiltration61die}) \\ \hline
\caption{Summary of Differentials.} 
\label{table:summaryDifferentials}
\end{longtable}
\end{center}

%%%%%%%%%%%%
\begin{figure}
\begin{center}
\makebox[\textwidth]{\includegraphics[trim={0cm 10cm 0cm 10cm}, clip, page = 1, scale = 0.23]{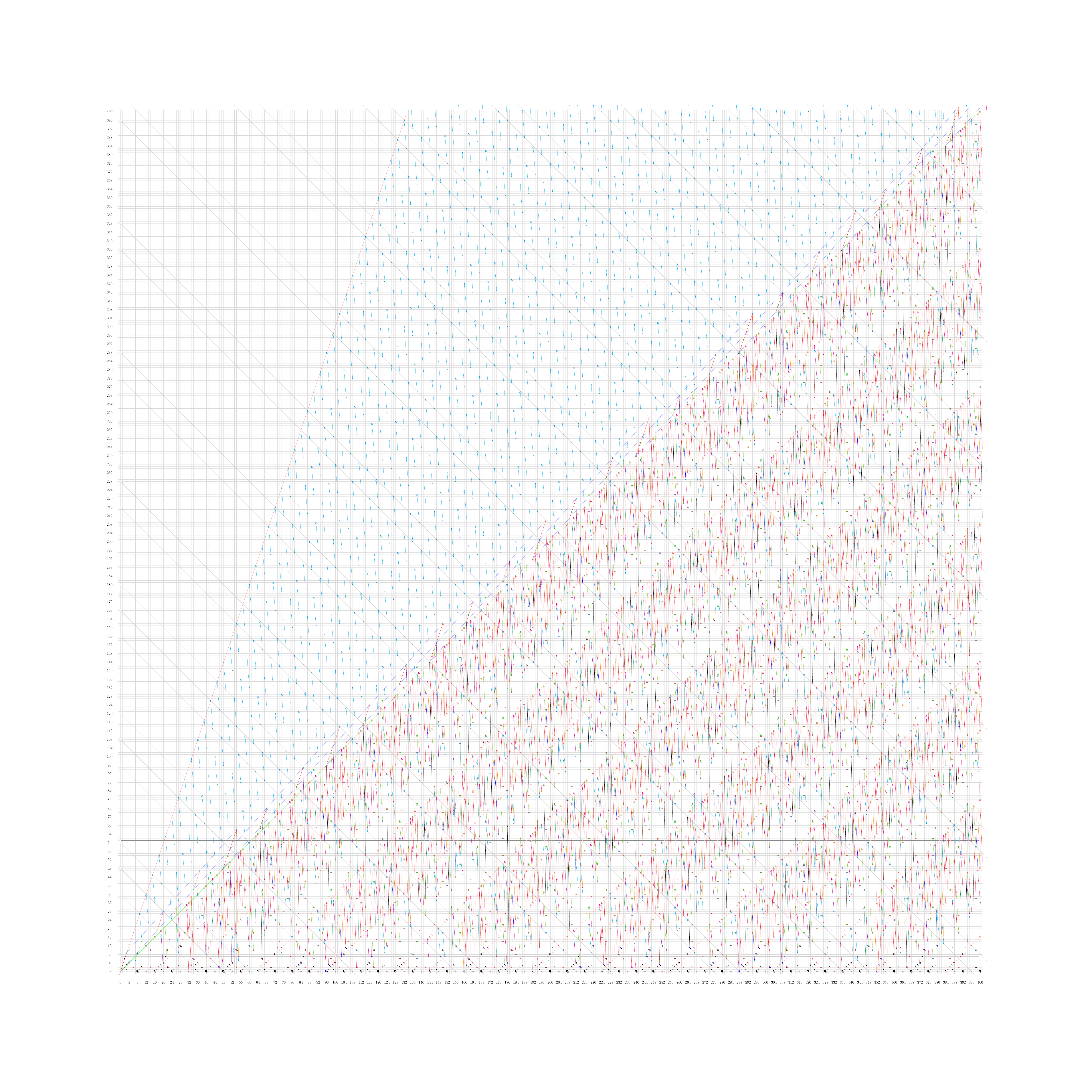}}
\end{center}
\begin{center}
\caption{$d_{13}$ to $d_{61}$-differentials in $\SliceSS(\BPtwo)$.}
\hfill
\label{fig:E4C4AllDifferentials}
\end{center}
\end{figure}

\begin{figure}
\begin{center}
\makebox[\textwidth]{\includegraphics[trim={0cm 10cm 0cm 10cm}, clip, page = 1, scale = 0.23]{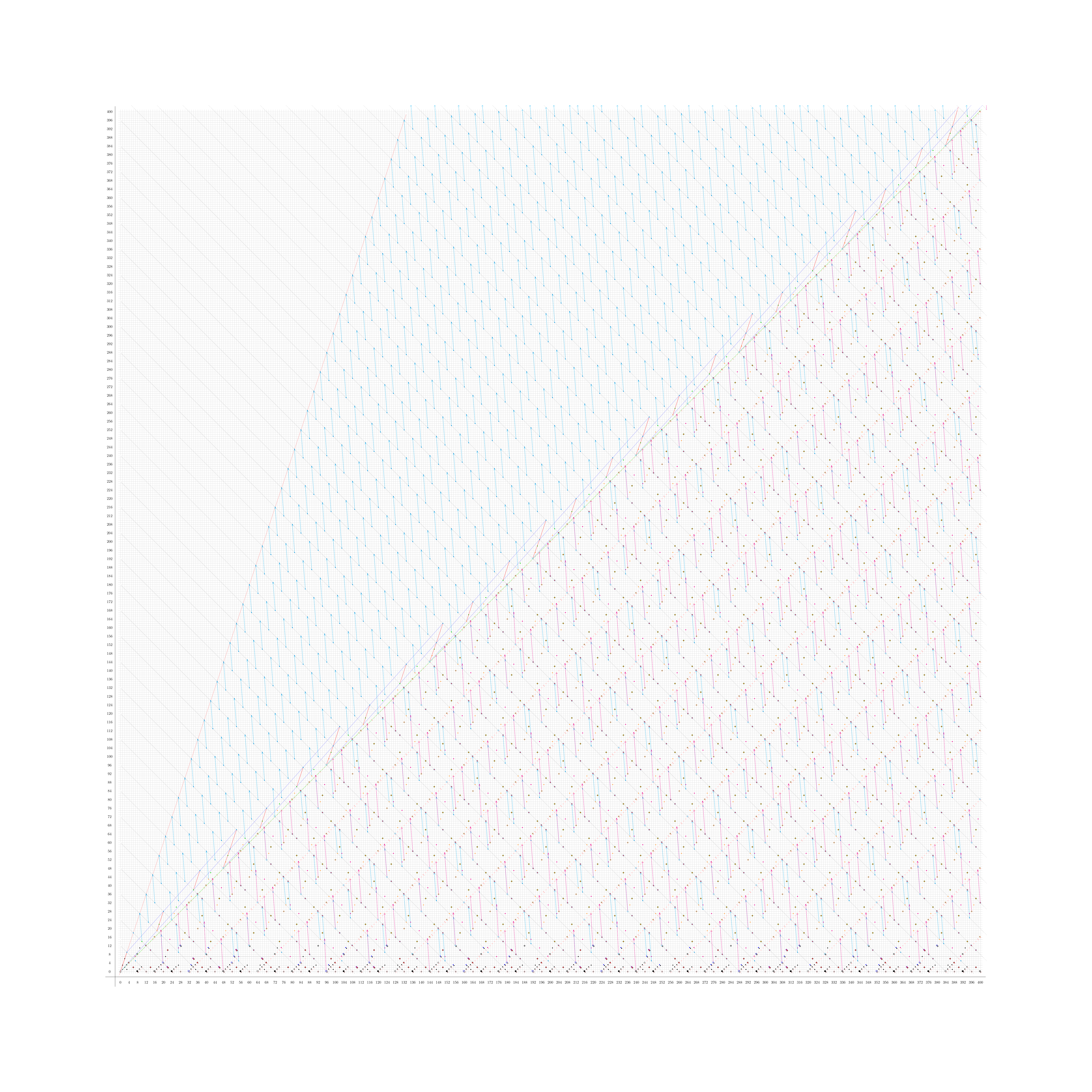}}
\end{center}
\begin{center}
\caption{$d_{13}$ to $d_{15}$-differentials in $\SliceSS(\BPtwo)$.}
\hfill
\label{fig:E4C4d13-d15Differentials}
\end{center}
\end{figure}

\begin{figure}
\begin{center}
\makebox[\textwidth]{\includegraphics[trim={0cm 10cm 0cm 10cm}, clip, page = 1, scale = 0.23]{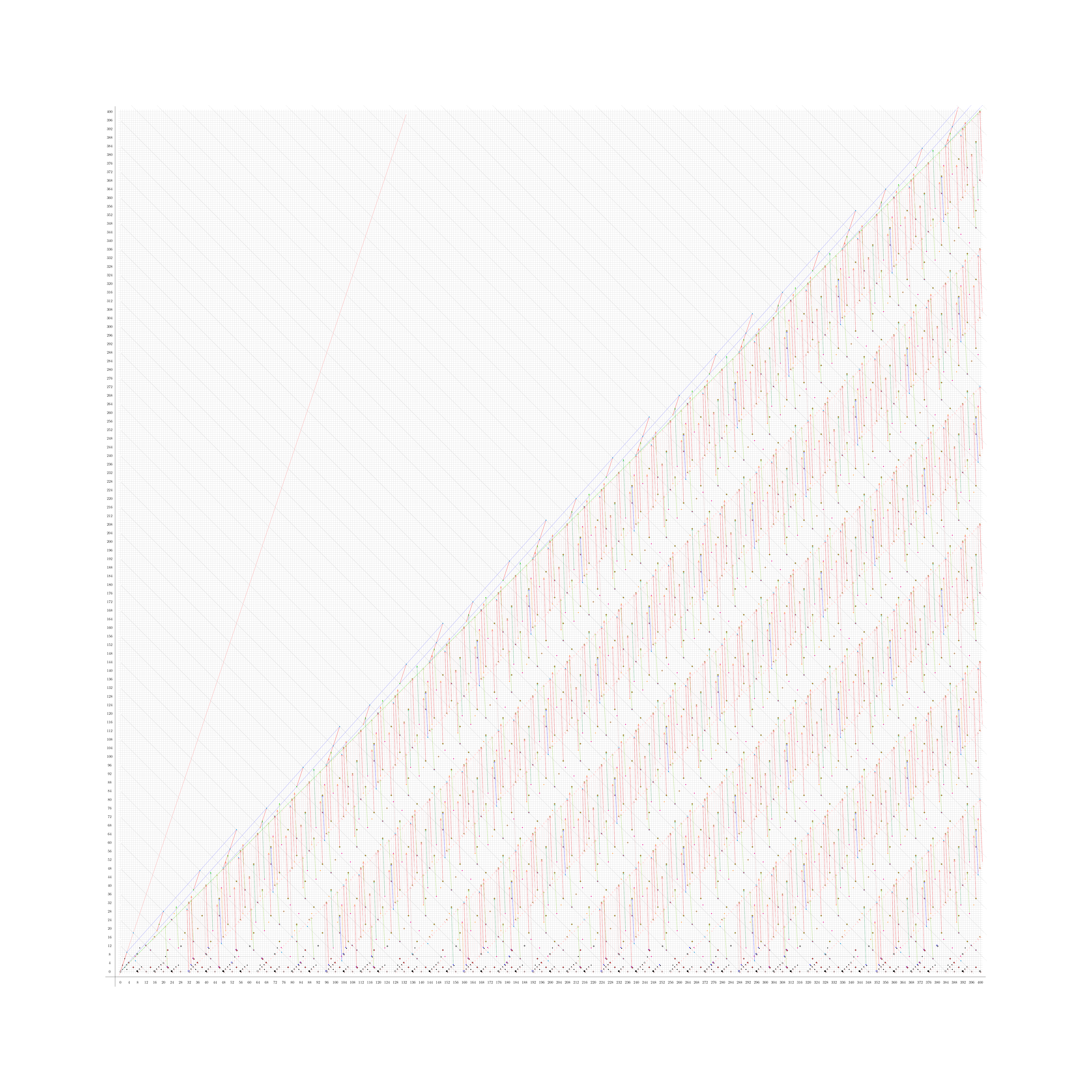}}
\end{center}
\begin{center}
\caption{$d_{19}$ to $d_{31}$-differentials in $\SliceSS(\BPtwo)$.}
\hfill
\label{fig:E4C4d19-d31Differentials}
\end{center}
\end{figure}

\begin{figure}
\begin{center}
\makebox[\textwidth]{\includegraphics[trim={0cm 10cm 0cm 10cm}, clip, page = 1, scale = 0.23]{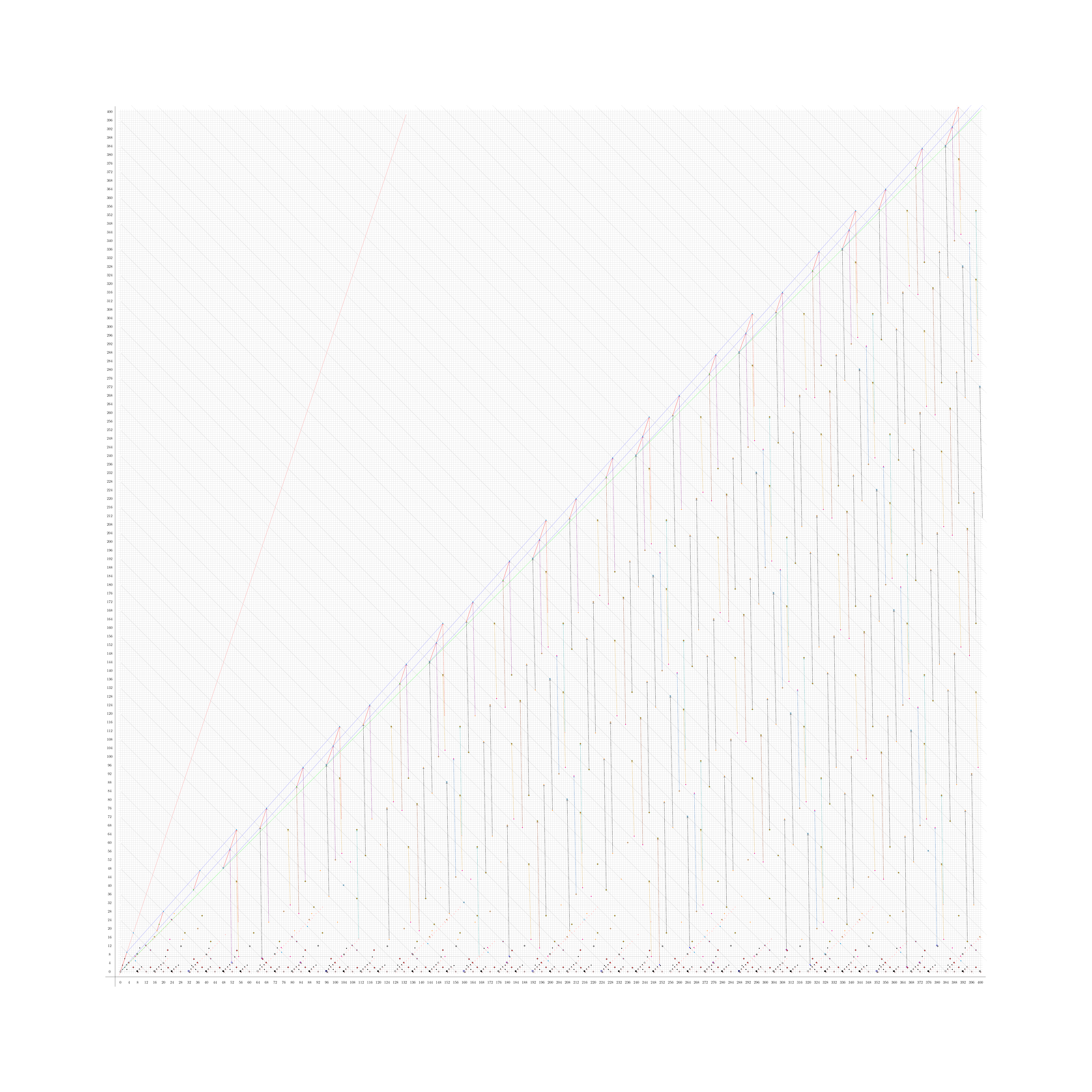}}
\end{center}
\begin{center}
\caption{$d_{35}$ to $d_{61}$-differentials in $\SliceSS(\BPtwo)$.}
\hfill
\label{fig:E4C4d35-d61Differentials}
\end{center}
\end{figure}

\begin{figure}
\begin{center}
\makebox[\textwidth]{\includegraphics[trim={0cm 10cm 0cm 10cm}, clip, page = 2, scale = 0.23]{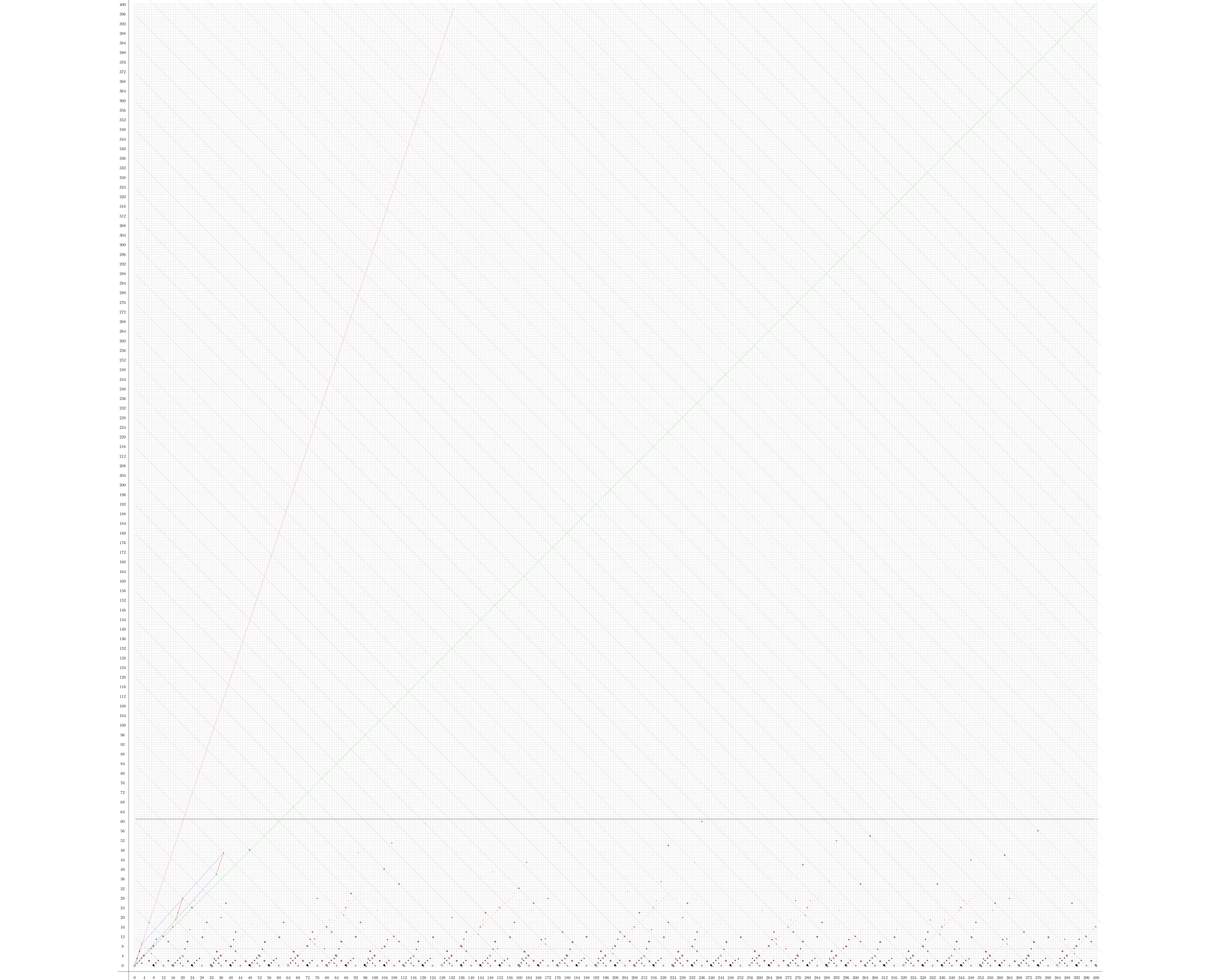}}
\end{center}
\begin{center}
\caption{$E_\infty$-page of $\SliceSS(\BPtwo)$.}
\hfill
\label{fig:E4C4EinftyPage}
\end{center}
\end{figure}

%%%%%%%%%%%%%%%%%%%%
\bibliographystyle{alpha}
\bibliography{Bibliography}

\end{document}